\documentclass[11pt]{article}
\usepackage[margin=1in]{geometry}
\usepackage[T1]{fontenc}
\usepackage{lmodern,microtype}
\usepackage{amsmath,amssymb,amsthm}
\usepackage{booktabs}
\usepackage{needspace}
\usepackage{tikz}
\usetikzlibrary{arrows.meta,calc,shapes.misc}
\usepackage[colorlinks=true,linkcolor=blue!50!black,citecolor=blue!50!black,urlcolor=blue!50!black]{hyperref}
\usepackage{bookmark}
\hypersetup{pdftitle={Vertical and reentrant ferromagnetic boundaries below the Nishimori point},pdfauthor={Yan Ru Pei},pdfkeywords={spin glass, Nishimori line, reentrance, Parisi formula, Sherrington-Kirkpatrick model, p-spin model, random-bond Ising model, decorated lattice}}
\newtheorem{theorem}{Theorem}[section]
\newtheorem{proposition}[theorem]{Proposition}
\newtheorem{lemma}[theorem]{Lemma}
\newtheorem{corollary}[theorem]{Corollary}
\theoremstyle{definition}

\theoremstyle{remark}
\newtheorem{remark}[theorem]{Remark}
\newcommand{\E}{\mathbb E}
\newcommand{\Pp}{\mathbb P}
\newcommand{\R}{\mathbb R}
\newcommand{\Z}{\mathbb Z}
\newcommand{\sech}{\operatorname{sech}}
\newcommand{\lc}{\log\cosh}
\newcommand{\Par}{\mathcal P}
\newcommand{\Om}{\Omega}
\newcommand{\gibbs}[1]{\langle #1\rangle}
\newcommand{\ind}{\mathbf 1}
\newcommand{\dd}{\mathrm d}

\title{Vertical and reentrant ferromagnetic boundaries\\below the Nishimori point}
\author{Yan Ru Pei\\[0.3em]\small\href{mailto:yanrpei@gmail.com}{\texttt{yanrpei@gmail.com}}}
\date{September 27, 2026}
\begin{document}
\maketitle

\begin{abstract}
We prove that below the multicritical (Nishimori) point the ferromagnetic boundary of a biased
Ising spin glass, argued by Nishimori (1986) to be exactly vertical, can be vertical and can also
be reentrant, so that cooling at fixed disorder destroys order. It is vertical in the Sherrington--Kirkpatrick
(SK) model with a Curie--Weiss bias $J_0$, and reentrant near the triple point of fully connected
$p$-spin glasses and, at one disorder strength, on a decorated planar $\Z^2$-periodic lattice.
For every $T<J$ the SK magnetization vanishes for $J_0\le J$, continuously at $J_0=J$,
and is macroscopic for $J_0>J$.
For every integer $p\ge3$, at each bias slightly larger than the triple-point bias, the $p$-spin
magnetization is macroscopic at the Nishimori temperature and vanishes at an explicit lower
temperature; this is Nishimori's recent replica prediction in two-temperature form.
On the lattice ($\Z^2$ with each bond replaced by a fixed series--parallel gadget; maximum degree
four), with iid $\pm J$ couplings at $\Pp(J_e=-1)=9/10000$, the Gibbs state is almost surely
unique at high temperature, non-unique in a window containing the Nishimori temperature, and
unique again at low temperature.
As Dey and Kang observed, the SK boundary is set by the zero-field susceptibility together with
an envelope bound on the free-energy gain in a field, which follows from the gauge inequality.
We prove that this susceptibility is the curvature at the origin of the zero-field Parisi PDE
solution, which Lopatto evaluated for every $\beta>1$ via his theorem that the support of the
Parisi measure accumulates at the origin; we give a short alternative proof of that theorem.
The $p$-spin and lattice results and the alternative proof are computer-assisted. The models are
proxies for the nearest-neighbour $\pm J$ model on $\Z^2$ and say nothing about $\Z^d$.

\end{abstract}

\smallskip

\smallskip

\noindent\textbf{Code.} The certificate programs, frozen inputs and outputs are available at
\url{https://github.com/PeaBrane/nishimori-boundaries} (tag \texttt{arxiv-v1.1}); see
Appendix~\ref{app:repro}.

\section{Introduction}\label{sec:intro}

\subsection{The question}

Consider an Ising spin glass whose couplings are independent with a positive mean, say
Gaussian with mean $J_0$ and variance $J^2$ (suitably scaled), or $\pm J$ with a majority of
ferromagnetic bonds. Its phase diagram in the plane of bias and temperature contains a
paramagnetic phase, a ferromagnetic phase and, in high dimension or mean field, a spin-glass
phase. A distinguished curve, the Nishimori line, is where the inverse temperature matches the
disorder law: $\beta=J_0/J^2$ in the Gaussian case, and $\tanh\beta=\E J_e$ for couplings
$J_e\in\{\pm1\}$.
There, a gauge symmetry makes the internal energy and many correlation identities exactly
computable~\cite{Nishimori1981,NishimoriBook}. Nishimori's inequality (the gauge inequality) compares correlations at an
arbitrary inverse temperature $\beta$ with those at the point of the Nishimori line with the same
disorder (we call it the Nishimori point of that bias), at inverse temperature $\beta_N$:
\[
\bigl|\E\gibbs{\sigma_i\sigma_j}_{\beta}\bigr|\le \E\bigl|\gibbs{\sigma_i\sigma_j}_{\beta_N}\bigr|.
\]
On the Nishimori line $\E\gibbs{\sigma_i\sigma_j}_{\beta_N}=\E\gibbs{\sigma_i\sigma_j}_{\beta_N}^2$, so
the right side is at most $(\E\gibbs{\sigma_i\sigma_j}_{\beta_N})^{1/2}$, which is small for
distant pairs when the Nishimori point carries no ferromagnetic order. In mean field no pair is
distant; there, averaging over all pairs bounds the magnetization $m=N^{-1}\sum_i\sigma_i$ instead:
$\E\gibbs{m^2}_\beta\le(\E\gibbs{m^2}_{\beta_N})^{1/2}$ (Lemma~\ref{lem:gauge-transfer}). Consequently no
temperature is ferromagnetic at a bias whose Nishimori point is not. If, as in the mean-field
models below, the ordered part of the Nishimori line is a half-line in the bias, the ferromagnetic
region therefore lies on the ferromagnetic side of the vertical line through the multicritical
point $M$ where the Nishimori line meets the phase boundary (Figure~\ref{fig:question}). The
literature, and our title, also call $M$ itself the Nishimori point.

Is the boundary below $M$ exactly vertical, or does it bend toward larger bias, so that cooling
at fixed disorder destroys ferromagnetic order? The second behaviour is called reentrance.
Nishimori argued that the temperature-independent singularity of the frustration entropy at $M$
(the entropy of the distribution of frustrated plaquettes, a property of the disorder that depends
on the bias but not on the temperature) makes the boundary between the ferromagnetic and the non-ferromagnetic (paramagnetic or
spin-glass) phase strictly vertical below $M$ ``in any dimension (two or
larger)''~\cite{Nishimori1986}; he later described the argument as ``not a rigorous
proof''~\cite[p.~63]{NishimoriBook}. Kitatani proved the verticality of the ferromagnetic--spin-glass
boundary of the $\pm J$ model ``assuming an appropriate condition''~\cite{Kitatani1992}. Numerical
studies of the nearest-neighbour two-dimensional $\pm J$ model indicate a slight
reentrance~\cite{Nobre2001,ThomasKatzgraber2011}, and so do studies of hierarchical
lattices~\cite{HinczewskiBerker2005,Guven2008}.\footnote{Manouchehri and Shtengel recently argued the
opposite: a percolation of zero-weight paths, which depends on the disorder but not on the
temperature, suggests that the boundary of the nearest-neighbour two-dimensional model is
vertical below $M$~\cite{ManouchehriShtengel2026}. Nishimori, Ohzeki and Okuyama argued that, if a
spin-glass phase exists at finite temperature, reentrance forces temperature chaos in that
phase~\cite{NOO2025}, which gives the question a further physical consequence. See
Section~\ref{sec:discussion}.}
Rigorous results
are scarce. For the Bethe lattice, Carlson, Chayes, Chayes, Sethna and Thouless proved a
reentrant local bifurcation of the magnetization recursion with fixed boundary
conditions~\cite{CCCST1990}: near the multicritical point, magnetized fixed points of the recursion
for the distribution of cavity magnetizations branch off only beyond a boundary that bends toward
larger bias on cooling. This is a statement about local fixed points of the recursion: it does
not determine which fixed point is reached from a given boundary condition, and the tree has no
cycles, so its frustration comes from the boundary rather than the bulk.

Even for Ising spins in mean field the question had not been settled. For the Sherrington--Kirkpatrick (SK)
model with a ferromagnetic bias, the full replica-symmetry-breaking solution predicts a vertical
boundary $J_0=J$~\cite{Toulouse1980,MPV1987}. Dey and Kang recently observed that this vertical
boundary is equivalent to a zero-field susceptibility identity of the Parisi solution,
Toulouse's marginality identity, together with an envelope bound, and they noted that the
identity was open mathematically and the exact boundary unresolved~\cite[Remark~1.3]{DeyKang2026}.
Lopatto's recent theorem that, for $\beta>1$, the origin is an accumulation point of the support of
the zero-field Parisi measure implies that the solution of the Parisi equation for that measure has second
derivative $1/\beta$ at the origin~\cite{Lopatto2026}; that this second derivative is the
susceptibility is one of our results (Section~\ref{sec:intro-prior}). For spherical spins the analogous
questions have rigorous answers.\footnote{In the spherical SK model with a Curie--Weiss term the free
energy is computed by random-matrix methods, and the spin-glass and ferromagnetic regimes are
separated by a vertical line~\cite[Theorem~1.4 and Figure~1]{BaikLee2017}. In spherical tensor
PCA, the maximum-likelihood estimator, a zero-temperature state, is correlated with the signal
down to the information-theoretic threshold~\cite[Theorem~1.2]{JLM2020}, which is a
zero-temperature analogue of a vertical boundary; Jagannath, Lopatto and Miolane remark that this
is not expected for every prior~\cite[p.~3]{JLM2020}, citing~\cite{GNS2001}.}
For fully connected $p$-spin
models with a ferromagnetic bias and $p\ge3$~\cite{NishimoriWong1999,GNS2001}, Nishimori recently
computed, within the replica description, a reentrant bending of the ferromagnet--spin-glass
boundary near the triple point where the Nishimori line meets the paramagnetic, spin-glass and
ferromagnetic phases~\cite{Nishimori2026}. Its curvature is proportional to the difference
$C_{\rm SG}-C_{\rm FM}$ of the specific heats of the two phases at the triple point, and the
boundary is reentrant when $C_{\rm SG}>C_{\rm FM}$. He checked
the sign of this difference numerically for integer $3\le p\le12$ and in a scan over real $p$
from near $2$ to about $12$, and asymptotically as $p\to\infty$ and $p\to2^+$, and he left a
proof uniform in $p$ open~\cite[p.~19]{Nishimori2026}.

\begin{figure}[t]
\centering
\begin{tikzpicture}[x=6.2cm,y=4cm,>=Stealth]
\fill[gray!13] (0,0) rectangle (0.8,1.38);
\draw[->] (0,0) -- (1.64,0) node[below left]{bias};
\draw[->] (0,0) -- (0,1.44) node[below left]{$T$};
\draw[thick,dashed,domain=0.47:1.6,smooth,variable=\x] plot ({\x},{0.64/\x});
\node[anchor=west] at (0.49,1.37) {\small Nishimori line};
\draw[thick] (0,0.8) -- (0.8,0.8);
\draw[thick] (0.8,0.8) .. controls (1.0,0.95) and (1.3,1.1) .. (1.6,1.18);
\draw[thick,dotted] (0.8,0.8) -- (0.8,0);
\draw[thick] (0.8,0.8) .. controls (0.8,0.45) and (0.9,0.25) .. (0.97,0);
\node[anchor=east] at (0.79,0.3) {\small vertical};
\node[anchor=west] at (0.955,0.3) {\small reentrant};
\fill (0.8,0.8) circle (1.3pt) node[below left]{\small $M$};
\draw[gray] (0.88,0.02) -- (0.88,0.9);
\fill[blue!60!black] (0.88,0.7273) circle (1.7pt);
\node[blue!60!black,anchor=west] (ord) at (1.08,0.64) {\small order};
\draw[->,blue!60!black] (ord.west) -- (0.895,0.722);
\fill[red!70!black] (0.88,0.18) circle (1.7pt);
\node[red!70!black,anchor=west] (dis) at (1.08,0.12) {\small no order};
\draw[->,red!70!black] (dis.west) -- (0.895,0.176);
\node at (0.38,1.12) {\small PM};
\node at (0.38,0.42) {\small SG};
\node at (1.4,0.84) {\small FM};
\end{tikzpicture}
\caption{The question, schematically, in the plane of bias and temperature. Dashed: the
Nishimori line through the multicritical point $M$. The gauge inequality excludes ferromagnetic
order at every temperature to the left of the vertical line through $M$ (shaded). Below $M$ the
ferromagnetic boundary is either this vertical line (dotted) or bends toward larger bias (solid,
reentrance). At a bias right of $M$ (grey), order at the Nishimori point of that bias (blue) together
with its absence at a lower temperature (red) proves that the boundary bends;
Theorems~\ref{thm:intro-p} and~\ref{thm:intro-lat} are statements of this form. The labels are
those of the mean-field models; in the lattice example of Theorem~\ref{thm:intro-lat} the cold
disordered phase is paramagnetic.}
\label{fig:question}
\end{figure}
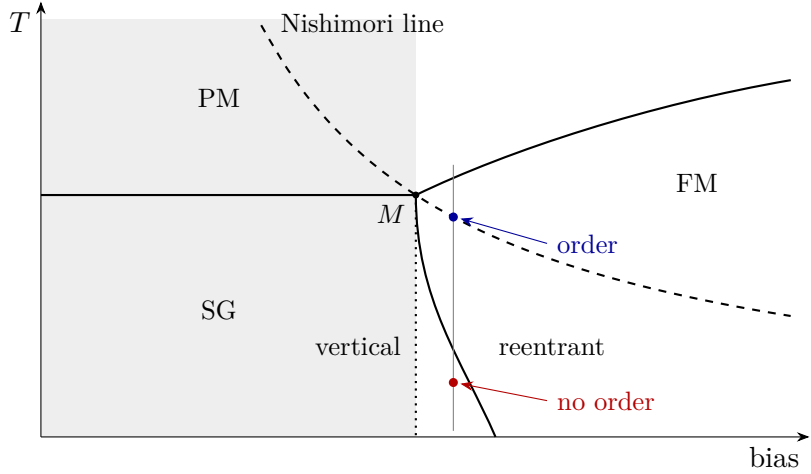

Gauge symmetry compares two temperatures at one bias, and only in one direction. To decide the
shape of the boundary one must show, at biases just beyond $M$, either that order persists at
every lower temperature or that it disappears at some lower temperature. Near $M$ the competing
phases differ by little, and neither the gauge identities nor the standard rigorous criteria
resolve the difference.

\subsection{Main results}\label{sec:intro-results}

We prove that both shapes occur. The boundary is vertical in the SK model. It is reentrant, in
the two-point sense of Figure~\ref{fig:question} (order at the Nishimori point of a bias and none
at a lower temperature at the same bias), in the $p$-spin models near the triple point for every integer
$p\ge3$, and on one planar periodic lattice with independent $\pm J$ couplings at one disorder
strength. Order is measured differently in the three settings: by $\E\gibbs{m^2}$ for SK, by the
Gibbs probability of $|m|>m_1$ and by $\E\gibbs{|m|}$ for the $p$-spin models, and by the number of
infinite-volume Gibbs states and by $\E\gibbs{M_L^2}$ on the lattice. The three theorems below are
informal versions; each points to its exact statement. Figure~\ref{fig:roadmap} shows how they are proved;
Section~\ref{sec:intro-prior} relates them to earlier work, and Section~\ref{sec:intro-ideas}
explains the proofs.

\needspace{10\baselineskip}
\begin{theorem}[SK: a vertical boundary with a continuous transition; informal version of
Theorems~\ref{thm:sk} and~\ref{thm:skline} and Remark~\ref{rem:parisi-free}]\label{thm:intro-sk}
Consider the SK model at inverse temperature $\beta$, with Gibbs weight proportional to
$\exp(\beta\sum_{i<j}J_{ij}\sigma_i\sigma_j)$ and independent couplings $J_{ij}\sim N(J_0/N,1/N)$,
i.e.\ with a Curie--Weiss bias $J_0$, at temperature $T=1/\beta$. Let
$m=N^{-1}\sum_i\sigma_i$. Then $\E\gibbs{m^2}\to0$ for $J_0\le\max(1,T)$ and
$\liminf_N\E\gibbs{m^2}>0$ for $J_0>\max(1,T)$. In particular, below the multicritical point
$(J_0,T)=(1,1)$ the boundary is the vertical line $J_0=1$. The magnetization vanishes
continuously there: for every $T>0$ and $1\le J_0\le21/20$,
\[
\limsup_{N\to\infty}\E\gibbs{m^2}\le\bigl(12(J_0-1)\bigr)^{1/2}.
\]
\end{theorem}

Parts of Theorem~\ref{thm:intro-sk} are known; Section~\ref{sec:intro-prior} says which, and what
we add.

For the $p$-spin models, let $p\ge3$ be an integer and let the couplings $J_I$, indexed by the
$p$-subsets $I$ of $\{1,\dots,N\}$, be independent Gaussians with mean $j_0\,p!/N^{p-1}$ and
variance $p!/(2N^{p-1})$. The Gibbs measure at inverse temperature $\beta$ is proportional to
$\exp(\beta\sum_IJ_I\prod_{i\in I}\sigma_i)$, and the Nishimori line is $\beta=2j_0$. Let
$\beta_1$ be the inverse temperature at which replica symmetry breaks in the pure $p$-spin
model~\cite{Talagrand2000,Chen2019,Zhou2024}. The triple point is
$M=(j_{0M},T_M)=(\beta_1/2,1/\beta_1)$ (Section~\ref{sec:nm-triple}). On the Nishimori line the
model is, after a gauge transformation and up to the passage from ordered $p$-tuples to
$p$-subsets, the spiked Rademacher tensor model, and $\beta_1$ is also its weak-recovery threshold
(Remark~\ref{rem:nm-chen}).

\begin{theorem}[$p$-spin: reentrance near the triple point; informal version of
Theorem~\ref{thm:nm-main}]\label{thm:intro-p}
For every integer $p\ge3$ there are $\bar\delta>0$, $c>0$ and $m_1\in(0,1)$ such that for every
bias $j_0=j_{0M}+\delta$ with $0<\delta<\bar\delta$ the following hold.
\begin{itemize}
\item At the Nishimori temperature $T_{\rm w}=1/(2j_0)$ (the warm point),
$\E\gibbs{\ind\{|m|>m_1\}}\to1$, and $\E\gibbs{\ind\{m>m_1\}}\to1$ if $p$ is odd (for even $p$
the Gibbs measure is invariant under $\sigma\mapsto-\sigma$).
\item At the lower temperature $T_{\rm cold}=1/(\beta_1+c\,\delta^{1/2})<T_{\rm w}$, one has
$\E\gibbs{|m|}\to0$.
\end{itemize}
The constant $m_1$ is explicit, $c$ is an explicit function of the constants of the triple point,
and $\bar\delta$ is not explicit.
\end{theorem}

Cooling from $T_{\rm w}$ to $T_{\rm cold}$ at one disorder law thus destroys ferromagnetic order,
at every bias slightly larger than the triple-point bias $j_{0M}$. The first part of
Theorem~\ref{thm:intro-p} is, in substance, W.-K.~Chen's weak-recovery theorem for the spiked
Rademacher tensor model~\cite{Chen2019}. Chen's proof keeps the magnetization above a threshold
set by the left derivative of a limiting free energy; this threshold depends on the bias and is
shown to be positive by convexity, but it is not evaluated. We reprove the first part directly for
the Hamiltonian over $p$-subsets, with the explicit threshold
$m_1=\min\{(1-2/p)^{1/2},4^{-1/(p-2)}\}$, which depends only on $p$: at the Nishimori temperature
the expected Gibbs mass of $\{|m|\le m_1\}$ (for odd $p$, of $\{m\le m_1\}$) is exponentially
small in $N$
(Theorem~\ref{thm:nm-main}(i) and Section~\ref{sec:intro-ideas}). The reentrance rests on the
second part.

Theorem~\ref{thm:intro-p} is computer-assisted.
Its inputs that depend on $p$, among them $C_{\rm SG}>C_{\rm FM}$, are verified in interval
arithmetic for $3\le p\le25$ and follow for $p\ge26$ from explicit bounds whose constants are
also certified. It is a two-point statement at each fixed bias: it constructs no boundary curve,
and $\bar\delta$, which bounds the biases in both parts, depends on $p$ and is not explicit. It concerns integer $p$ only, so Nishimori's question for real
$p>2$ is not resolved.

For $p=3$ and $p=4$ we also certify explicit instances far below the
triple point in temperature (Theorems~\ref{thm:p3} and~\ref{thm:p4}; Section~\ref{sec:certificates}
gives $j_{0M}$ and $T_M$ for comparison). For $p=3$ and
$j_0=0.7682$, the magnetization exceeds $0.44$ with Gibbs probability tending to one at the
Nishimori point of this bias ($T\approx0.6509$), while $\E\gibbs{|m|}\to0$ at $T=0.4$. For $p=4$ and
$j_0=0.8107$, one has $|m|>0.64$ with Gibbs probability tending to one at the Nishimori point of
this bias ($T\approx0.61675$), while $\E\gibbs{|m|}\to0$ at $T=0.3$.

The third result concerns a two-dimensional graph. Replace every bond of $\Z^2$ by a copy of a
fixed two-terminal series--parallel gadget $H$ (Figure~\ref{fig:lat-gadget}). The gadget is a
chain of ten units in series between two terminal edges. Each unit is a direct edge in parallel
with a branch that consists of a connector edge, two parallel paths of $1000$ edges, and a second
connector edge, in series. A unit thus has $2003$ edges, and $H$ has $2+10\cdot2003=20032$. The
resulting graph $\Z^2[H]$ is planar, $\Z^2$-periodic and of maximum degree four. The long paths
are weak at the Nishimori temperature but freeze on cooling, and the chain of ten units makes it
likely that some unit is frustrated at low temperature; the mechanism is explained in
Section~\ref{sec:intro-ideas}.

\begin{theorem}[A decorated square lattice; informal version of
Theorem~\ref{thm:lat-main}]\label{thm:intro-lat}
Put independent couplings $J_e\in\{\pm1\}$ on the edges of $\Z^2[H]$, with
$\Pp(J_e=-1)=p_0:=9/10000$, and let $\beta_N=\frac12\log((1-p_0)/p_0)$ be the Nishimori inverse
temperature. There are explicit $\beta_{\rm hot}<\beta_-<\beta_N<\beta_+<\beta_{\rm cold}$, with
$\beta_{\rm hot}\approx0.440\,\beta_N$, $\beta_-\approx0.561\,\beta_N$,
$\beta_+\approx1.401\,\beta_N$ and $\beta_{\rm cold}\approx1.992\,\beta_N$, such that for each
fixed $\beta$:
\begin{itemize}
\item almost surely the Gibbs state is unique if $\beta\le\beta_{\rm hot}$ or
$\beta\ge\beta_{\rm cold}$;
\item almost surely there are at least two Gibbs states if $\beta_-\le\beta\le\beta_+$;
\item at $\beta=\beta_N$, the magnetization $M_L$ (the average of all spins) of a box of side $L$,
with free or periodic boundary conditions, satisfies $\liminf_L\E\gibbs{M_L^2}\ge0.08$.
\end{itemize}
Moreover, at the Nishimori temperature, $\liminf_L\E\gibbs{M_L^2}>0$ for every disorder strength
$0<\Pp(J_e=-1)\le1/1000$, and almost surely there are at least two Gibbs states for every
$0<\Pp(J_e=-1)\le3/2500$. So $p_0$ lies strictly inside the ordered part of the Nishimori line.
\end{theorem}

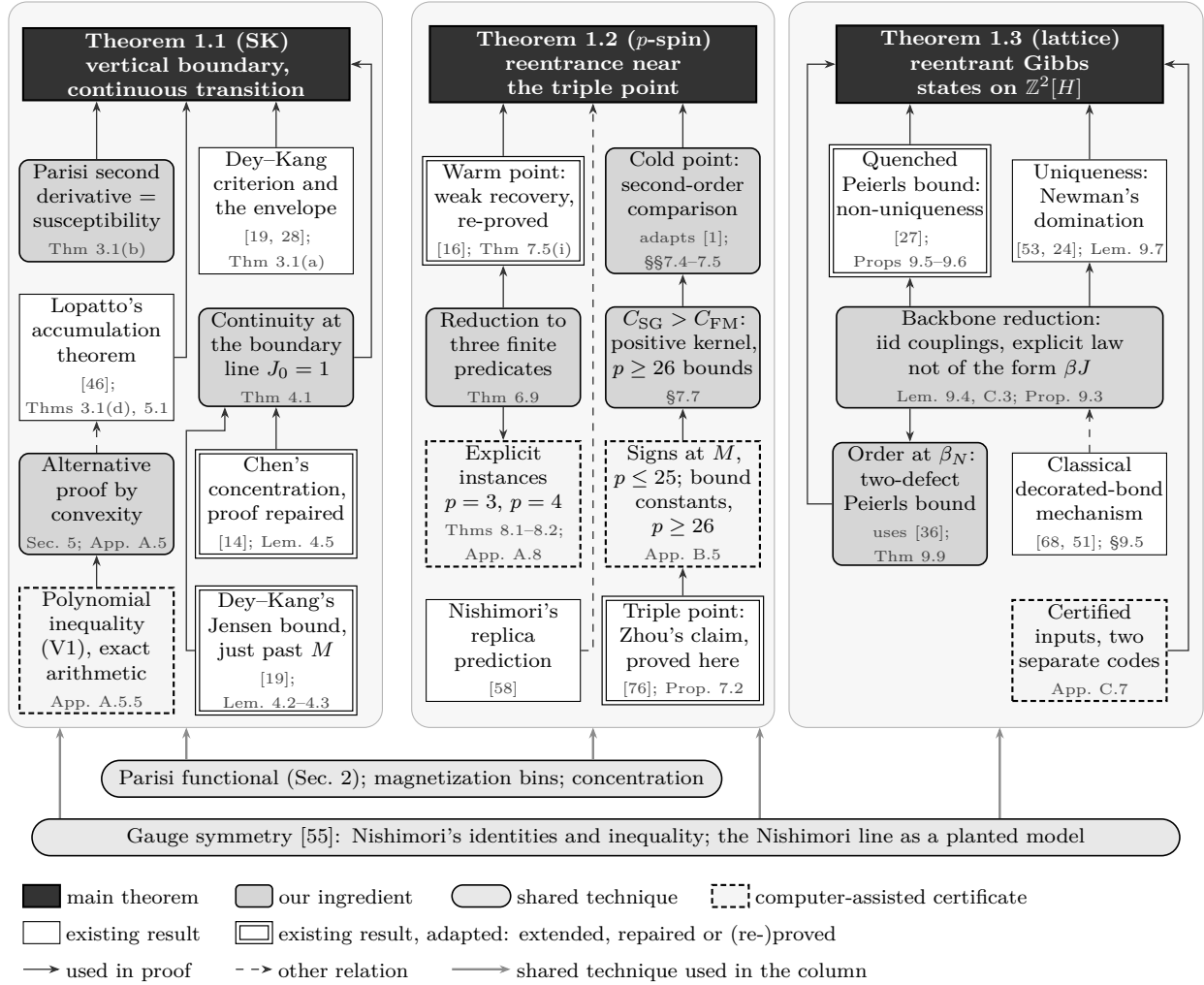
\begin{figure}[tp]
\centering
\newcommand*{\rmloc}[1]{\\[1pt]{\tiny\color{black!75}#1}}%
\begin{NoHyper}%
\begin{tikzpicture}[
  x=1cm, y=1cm, font=\scriptsize,
  execute at begin node={\thickmuskip=5mu\medmuskip=4mu\relax},
  >={Stealth[length=1.5mm,width=1.1mm]},
  thmlook/.style={fill=black!80, draw=black, line width=0.8pt},
  ourslook/.style={rounded corners=1.3mm, fill=black!16, draw=black, line width=0.6pt},
  knownlook/.style={fill=white, draw=black, line width=0.4pt},
  redolook/.style={fill=white, draw=black, line width=0.45pt,
    path picture={\draw[black, line width=0.45pt]
      ([shift={(1.6pt,1.6pt)}]path picture bounding box.south west) rectangle
      ([shift={(-1.6pt,-1.6pt)}]path picture bounding box.north east);}},
  certlook/.style={fill=black!4, draw=black, line width=0.8pt, dash pattern=on 2.2pt off 1.3pt},
  toollook/.style={rounded rectangle, fill=black!9, draw=black, line width=0.6pt},
  box/.style={rectangle, align=center, text width=2.02cm, inner xsep=1.2pt, inner ysep=1.5pt,
    minimum height=0.9cm},
  wide/.style={text width=4.37cm},
  thm/.style={box, wide, thmlook, text=white, font=\scriptsize\bfseries},
  ours/.style={box, ourslook},
  known/.style={box, knownlook},
  redo/.style={box, redolook, inner xsep=2.6pt, inner ysep=2.8pt},
  cert/.style={box, certlook},
  tool/.style={toollook, align=center, inner sep=2.5pt, minimum height=0.5cm},
  sw/.style={rectangle, minimum width=0.5cm, minimum height=0.3cm, inner sep=0pt},
  lg/.style={anchor=west, inner sep=2pt, text height=1.6ex, text depth=0.3ex},
  use/.style={->, line width=0.5pt, draw=black!80},
  rel/.style={use, dashed},
  band/.style={->, line width=0.9pt, draw=black!45},
  col/.style={rounded corners=2.5mm, fill=black!3, draw=black!30, line width=0.4pt},
]
\draw[col] (0.00,0.85) rectangle (5.10,-9.05);
\draw[col] (5.50,0.85) rectangle (10.45,-9.05);
\draw[col] (10.65,0.85) rectangle (16.30,-9.05);
\node[thm] (T1) at (2.425,0)  {Theorem~\ref{thm:intro-sk} (SK)\\ vertical boundary,\\ continuous transition};
\node[thm] (T2) at (7.975,0)  {Theorem~\ref{thm:intro-p} ($p$-spin)\\ reentrance near\\ the triple point};
\node[thm] (T3) at (13.525,0) {Theorem~\ref{thm:intro-lat} (lattice)\\ reentrant Gibbs\\ states on $\Z^2[H]$};
\node[ours]  (K2) at (1.2,-2.0)  {Parisi second\\ derivative $=$\\ susceptibility\rmloc{Thm~\ref{thm:sk}(b)}};
\node[known] (K0) at (3.65,-2.0) {Dey--Kang\\ criterion and\\ the envelope\rmloc{\cite{DeyKang2026,ItoiSakamoto2023}; Thm~\ref{thm:sk}(a)}};
\node[known] (K3) at (1.2,-4.0)  {Lopatto's\\ accumulation\\ theorem\rmloc{\cite{Lopatto2026}; Thms~\ref{thm:sk}(d), \ref{thm:accum}}};
\node[ours]  (K6) at (3.65,-4.0) {Continuity at\\ the boundary\\ line $J_0=1$\rmloc{Thm~\ref{thm:skline}}};
\node[ours]  (K4) at (1.2,-6.0)  {Alternative\\ proof by\\ convexity\rmloc{Sec.~\ref{sec:accum}; App.~\ref{app:mf-accum}}};
\node[redo]  (K7) at (3.65,-6.0) {Chen's\\ concentration,\\ proof repaired\rmloc{\cite{Chen2014}; Lem.~\ref{lem:concentration}}};
\node[cert]  (K5) at (1.2,-8.0)  {Polynomial\\ inequality\\ (V1), exact\\ arithmetic\rmloc{App.~\ref{sec:accum-cert}}};
\node[redo]  (K1) at (3.65,-8.0) {Dey--Kang's\\ Jensen bound,\\ just past $M$\rmloc{\cite{DeyKang2026}; Lem.~\mbox{\ref{lem:jensen}--\ref{lem:cw}}}};
\node[redo]  (P1)  at (6.75,-2.0) {Warm point:\\ weak recovery,\\ re-proved\rmloc{\cite{Chen2019}; Thm~\ref{thm:nm-main}(i)}};
\node[ours]  (P34) at (9.2,-2.0)  {Cold point:\\ second-order\\ comparison\rmloc{adapts \cite{AronowLopatto2026}; \S\S\mbox{\ref{sec:nm-sg}--\ref{sec:nm-fm}}}};
\node[ours]  (P7)  at (6.75,-4.0) {Reduction to\\ three finite\\ predicates\rmloc{Thm~\ref{thm:reduction}}};
\node[ours]  (P5)  at (9.2,-4.0)  {$C_{\rm SG}>C_{\rm FM}$:\\ positive kernel,\\ $p\ge26$ bounds\rmloc{\S\ref{sec:nm-nondeg}}};
\node[cert]  (P8)  at (6.75,-6.0) {Explicit\\ instances\\ $p=3$, $p=4$\rmloc{Thms~\mbox{\ref{thm:p3}--\ref{thm:p4}}; App.~\ref{app:cert34}}};
\node[cert]  (P6)  at (9.2,-6.0)  {Signs at $M$,\\ $p\le25$; bound\\ constants,\\ $p\ge26$\rmloc{App.~\ref{app:nm-cert}}};
\node[known] (PN)  at (6.75,-8.0) {Nishimori's\\ replica\\ prediction\rmloc{\cite{Nishimori2026}}};
\node[redo]  (PT)  at (9.2,-8.0)  {Triple point:\\ Zhou's claim,\\ proved here\rmloc{\cite{Zhou2024}; Prop.~\ref{prop:nm-triple}}};
\node[redo]  (L3) at (12.3,-2.0)  {Quenched\\ Peierls bound:\\ non-uniqueness\rmloc{\cite{HM1982}; Props~\mbox{\ref{prop:lat-CP}--\ref{prop:lat-NU}}}};
\node[known] (L4) at (14.75,-2.0) {Uniqueness:\\ Newman's\\ domination\rmloc{\cite{Newman1994,Harris1960}; Lem.~\ref{lem:lat-newman}}};
\node[ours,wide] (L1) at (13.525,-4.0) {Backbone reduction:\\ iid couplings, explicit law\\ not of the form $\beta J$\rmloc{Lem.~\ref{lem:lat-BR}, \ref{lem:lat-decimation}; Prop.~\ref{prop:lat-law}}};
\node[ours]  (L5) at (12.3,-6.0)  {Order at $\beta_N$:\\ two-defect\\ Peierls bound\rmloc{uses \cite{Kitatani1994}; Thm~\ref{thm:lat-A}}};
\node[known] (L2) at (14.75,-6.0) {Classical\\ decorated-bond\\ mechanism\rmloc{\cite{Syozi1968,Miyazima1968}; \S\ref{sec:lat-scope}}};
\node[cert]  (L6) at (14.75,-8.0) {Certified\\ inputs, two\\ separate codes\rmloc{App.~\ref{app:lat-cert}}};
\draw[use] (K2) -- (K2 |- T1.south);
\draw[use] (K0) -- (K0 |- T1.south);
\draw[use] (K3.east) -- (2.425,-4.0) -- (2.425,-4.0 |- T1.south);
\draw[rel] (K4) -- (K3);
\draw[use] (K5) -- (K4);
\draw[use] (K7) -- (K6);
\draw[use] (K6.east) -- (4.95,-4.0) -- (4.95,0) -- (T1.east);
\draw[use] (K1.west) -- (2.425,-8.0) -- (2.425,-5.0) -- (2.95,-5.0) -- (2.95,-5.0 |- K6.south);
\draw[use] (P1)  -- (P1 |- T2.south);
\draw[use] (P34) -- (P34 |- T2.south);
\draw[use] (P7)  -- (P1);
\draw[use] (P7)  -- (P8);
\draw[use] (P5)  -- (P34);
\draw[use] (P6)  -- (P5);
\draw[use] (PT)  -- (P6);
\draw[rel] (PN.east) -- (7.975,-8.0) -- (7.975,-8.0 |- T2.south);
\draw[use] (L3) -- (L3 |- T3.south);
\draw[use] (L4) -- (L4 |- T3.south);
\draw[use] (L1.north -| L3) -- (L3);
\draw[use] (L1.north -| L4) -- (L4);
\draw[use] (L1.south -| L5) -- (L5);
\draw[rel] (L2) -- (L2 |- L1.south);
\draw[use] (L5.west) -- (10.9,-6.0) -- (10.9,0) -- (T3.west);
\draw[use] (L6.east) -- (16.1,-8.0) -- (16.1,0) -- (T3.east);
\node[tool, text width=8.2cm] (S2) at (5.5,-9.75)
  {Parisi functional (Sec.~\ref{sec:prelim}); magnetization bins; concentration};
\node[tool, text width=15.4cm] (S1) at (8.15,-10.55)
  {Gauge symmetry~\cite{Nishimori1981}: Nishimori's identities and inequality; the Nishimori line as a planted model};
\draw[band] (2.425,-9.75 |- S2.north)  -- (2.425,-9.05);
\draw[band] (7.975,-9.75 |- S2.north)  -- (7.975,-9.05);
\draw[band] (0.7,-10.55 |- S1.north)    -- (0.7,-9.05);
\draw[band] (10.25,-10.55 |- S1.north)  -- (10.25,-9.05);
\draw[band] (13.525,-10.55 |- S1.north) -- (13.525,-9.05);
\node[sw, thmlook]   (g1) at (0.45,-11.35) {}; \node[lg] at (g1.east) {main theorem};
\node[sw, ourslook, rounded corners=0.6mm]  (g2) at (3.35,-11.35) {}; \node[lg] at (g2.east) {our ingredient};
\node[sw, toollook, minimum width=0.8cm]  (g3) at (6.45,-11.35) {}; \node[lg] at (g3.east) {shared technique};
\node[sw, certlook]  (g6) at (9.85,-11.35) {}; \node[lg] at (g6.east) {computer-assisted certificate};
\node[sw, knownlook] (g4) at (0.45,-11.85) {}; \node[lg] at (g4.east) {existing result};
\node[sw, redolook]  (g5) at (3.35,-11.85) {}; \node[lg] at (g5.east) {existing result, adapted: extended, repaired or (re-)proved};
\draw[use] (0.2,-12.35) -- (0.7,-12.35) node[lg] {used in proof};
\draw[rel] (3.1,-12.35) -- (3.6,-12.35) node[lg] {other relation};
\draw[band] (6.05,-12.35) -- (6.85,-12.35) node[lg] {shared technique used in the column};
\end{tikzpicture}%
\end{NoHyper}
\caption{How the main theorems are proved. Each column leads to one theorem; in the middle column the
reduction, with certified predicates, also gives the explicit instances for $p=3$ and $p=4$
(Theorems~\ref{thm:p3} and~\ref{thm:p4}), which are separate from Theorem~\ref{thm:intro-p}. A solid arrow means
that its source is used in the proof of its target, and the grey arrows from the two bottom bands mark the
shared techniques that each column uses; only the main dependencies are drawn. Dashed arrows mark
other relations: an alternative proof of Lopatto's theorem (Section~\ref{sec:accum}), the replica
prediction whose two-point form Theorem~\ref{thm:intro-p} proves, and the classical mechanism that
the gadget realizes with quenched iid couplings. Box styles give provenance; grey labels give
sources and where each item is proved. Section~\ref{sec:intro-ideas} explains each column, and
Section~\ref{sec:intro-prior} lists the existing results that we re-prove, extend or repair.}
\label{fig:roadmap}
\end{figure}

On cooling at the fixed disorder strength $p_0$, the Gibbs state is therefore unique, then
non-unique in a window that contains the Nishimori temperature, then unique again. On
$[\beta_-,\beta_+]$ the non-uniqueness is ferromagnetic in the following sense: there are a plus
and a minus state that give the spin at the origin of $\Z^2$ the expectations $\pm Y$, for a random
$Y$ with $\E Y>0$ (Proposition~\ref{prop:lat-NU}). The same
certified inputs show that $\E\gibbs{M_L^2}\to0$ and $\E\gibbs{Q_L^2}\to0$ as $L\to\infty$ for
$\beta\ge\beta_{\rm cold}$, where $Q_L$ is the overlap of two independent replicas
(Corollary~\ref{rem:lat-finite}). So the cold phase has neither ferromagnetic nor Edwards--Anderson
order: it is paramagnetic. Nothing is claimed in the gaps $(\beta_{\rm hot},\beta_-)$ and
$(\beta_+,\beta_{\rm cold})$.\footnote{Each criterion used is certified to fail at a point within $10^{-4}$
of the end of its certified range (an absolute distance in $\beta$, about $3\times10^{-5}\beta_N$),
so the gaps reflect the methods.}
The proof is
computer-assisted: finitely many inequalities about an explicit law are certified in interval
arithmetic, each by at least two implementations that share no code.

\paragraph{Proxies.} The mean-field models and the decorated lattice are proxies for the
nearest-neighbour $\pm J$ model on $\Z^2$, which motivated the question. Together they show that
verticality below $M$ is not universal: the SK boundary is vertical, the
$p$-spin boundaries are reentrant near the triple point, and so is the boundary on one
two-dimensional periodic lattice at one disorder strength. The two mean-field outcomes differ for
the reasons given in Sections~\ref{sec:intro-prior} and~\ref{sec:intro-ideas}: in the SK model the magnetization appears
continuously, at a boundary fixed at $J_0=1$ by Toulouse's marginality identity, whereas near the
triple point of the $p$-spin models the transition is discontinuous in the replica picture and the
bending is decided at second order by the sign of $C_{\rm SG}-C_{\rm FM}$. The proxies say nothing
about the nearest-neighbour lattices $\Z^d$. The lattice proof uses two features that $\Z^2$ lacks: an
exact decimation to independent backbone couplings (the effective couplings between neighbouring
vertices of $\Z^2$ once the gadget interiors are summed out) whose law is not of the form
$\beta J$, and a frustration-capped effective coupling (Section~\ref{sec:intro-ideas}), which is
what makes uniqueness at low temperature possible.

\subsection{Relation to earlier work}\label{sec:intro-prior}

\paragraph{The SK model.} In the SK model with $J=1$, write the bias as a Curie--Weiss term $\gamma Nm^2/2$, where
$m=N^{-1}\sum_i\sigma_i$ and $\gamma=\beta J_0$. Throughout, a free energy is a limit of
$N^{-1}\E\log Z$ (no factor $-T$), so of two competing free energies the larger one wins, and the
Parisi formula is an infimum; in particular $F^{\rm SK}(\beta,h)$ is the limit of
$N^{-1}\E\log\sum_\sigma\exp\big(\beta N^{-1/2}\sum_{i<j}g_{ij}\sigma_i\sigma_j+h\sum_i\sigma_i\big)$ with independent standard Gaussians $g_{ij}$. Let
$\varphi_\beta(h)=F^{\rm SK}(\beta,h)-F^{\rm SK}(\beta,0)$ be the free energy per spin that the SK
model without bias gains in a field $h$, the field term being $h\sum_i\sigma_i$. By Chen's
variational formula~\cite{Chen2014} (Section~\ref{sec:intro-ideas}), the magnetization appears
exactly when this gain exceeds the cost of producing the field through the Curie--Weiss term,
that is, if and only if $\varphi_\beta(h)>h^2/(2\gamma)$ for some
$h>0$~\cite[Remark~1.3]{DeyKang2026}. The replica prediction $J_0=1$ below $M$ therefore follows
from two facts~\cite[Remark~1.3]{DeyKang2026}: the envelope bound $\varphi_\beta(h)\le h^2/(2\beta)$
for all $h$, and Toulouse's identity that the zero-field susceptibility of the spin-glass phase
equals exactly $1/\beta$. Here and below, the zero-field susceptibility is the curvature
$\lim_{h\to0}2\varphi_\beta(h)/h^2$ of the limiting gain.\footnote{Because the field enters without a
factor $\beta$, this is the usual susceptibility, $\chi=\beta(1-\int_0^1q(x)\,\dd x)$ in Parisi's
notation, divided by $\beta$; so Toulouse's identity $\chi=1/J$~\cite{Toulouse1980} (see
also~\cite[\S3.3.3]{NishimoriBook}) reads $1/\beta$ here. In finite volume the zero-field susceptibility equals $1$ at every
temperature~\cite[Remark~1.3]{DeyKang2026}, because at zero field gauge symmetry gives
$\E\gibbs{\sigma_i\sigma_j}=0$ for $i\ne j$. The limit $N\to\infty$ must be taken before $h\to0$,
and below $T=1$ the two limits do not commute.}
The identity is a marginality property of the
Parisi solution and depends on the Parisi measure near overlap $0$. Gauge symmetry supplies the
envelope but not the identity.

Parts of Theorem~\ref{thm:intro-sk} are known. For $T\ge1$ the location of the boundary is due
to Dey and Kang~\cite[Proposition~1.1(i)--(ii)]{DeyKang2026}. For $J_0\le1$ the unmagnetized half
follows from Nishimori's correlation inequality~\cite{Nishimori1981}, which Itoi and Sakamoto
state for the SK model~\cite[Theorem~2.2 and \S3]{ItoiSakamoto2023}, applied at the Nishimori
point of the same bias. That point has Curie--Weiss parameter $\gamma=J_0^2\le1$, so $0$ is the
only maximizer in Chen's variational formula there~\cite[Proposition~1.1(i)]{DeyKang2026}, and
Chen's concentration theorem~\cite[Proposition~1]{Chen2014} gives $\E\gibbs{m^2}\to0$ there. For
$J_0=1$ this point is the multicritical point, so these results also give $\E\gibbs{m^2}\to0$ on
the boundary line $J_0=1$. The
envelope bound is a short consequence of these results: by the criterion above, it says exactly that no bias $J_0<1$ magnetizes the model at
inverse temperature $\beta$, which the unmagnetized half gives. We record it, with a direct proof,
as Theorem~\ref{thm:sk}(a). Below $T=1$, Dey and Kang's equivalence therefore reduces the boundary
$J_0=1$ to Toulouse's identity. Lopatto proved that for every $\beta>1$ the origin is an
accumulation point of the support of the zero-field Parisi measure, and deduced that the solution
of the Parisi PDE for that measure has second derivative $1/\beta$ at the
origin~\cite[Proposition~5.1 and Lemma~6.1]{Lopatto2026}. For $\beta$ near $1$ the accumulation
property had been proved earlier by Zhou, whose theorem identifies the support as an interval
$[0,\upsilon_\beta]$~\cite[Theorem~1]{Zhou2026}. Lopatto does not discuss the susceptibility or the
ferromagnetic boundary. What we add is the following.
\begin{itemize}
\item A proof that the zero-field susceptibility equals that second derivative
(Theorem~\ref{thm:sk}(b)), by a short argument from uniqueness and continuity of Parisi
minimizers. Dey and Kang describe this susceptibility, without proof, as an integral against the
zero-field Parisi measure~\cite[Remark~1.3]{DeyKang2026}. With Lopatto's result the
identification gives Toulouse's identity, hence the magnetized half for $T<1$ and the vertical
boundary below $M$.
\item On the boundary line $J_0=1$ itself: at every temperature $T<1$, $0$ is the only maximizer
in Chen's variational formula, whereas the comparison with the multicritical point gives only
$\E\gibbs{m^2}\to0$ (Remark~\ref{rem:right-nbhd}); the strict envelope bound
$\varphi_\beta(h)<h^2/(2\beta)$ for every $\beta>0$ and $h\ne0$, with an explicit margin (at least
$h^6/(1536\beta^5)$ for $0<|h|\le\beta$); and the continuity bound (Section~\ref{sec:skline}). The
proofs use neither the Parisi formula nor the accumulation theorem. With both, the first two
statements also follow for $T<1$ from the second-derivative identity of Theorem~\ref{thm:sk}(d)
(Remark~\ref{rem:envelope-route}).
\item A short alternative proof of the accumulation property (Section~\ref{sec:accum}). It uses
nothing from Sections~3--5 of~\cite{Lopatto2026}, which contain Lopatto's proof of it. Its last step, in which the optimality
conditions at the two ends of a gap at the origin contradict a shape property of $\Gamma$ on the
gap (here $\Gamma(q)$ is the mean square of the spatial derivative of the Parisi PDE solution along
the optimal diffusion; Section~\ref{sec:intro-ideas}), is of the type Lopatto uses. The shape property is different, convexity of
$q\mapsto\Gamma(q)/q$ in place of Lopatto's one-crossing property of $\Gamma''$, and it is proved
by a different method (Section~\ref{sec:intro-ideas}). The proof is computer-assisted through one
polynomial inequality that is verified in exact rational arithmetic.
\end{itemize}
We are not aware of earlier work that proves the identification of the susceptibility with this
second derivative, or with any other quantity of the zero-field Parisi solution; that proves, with
it and Lopatto's theorem, the vertical boundary for all $T<1$; that shows that $0$ is the only
maximizer on the boundary line for $T<1$, or contains the strict envelope bound or the continuity
bound of Section~\ref{sec:skline}; or that contains the argument for accumulation in
Section~\ref{sec:accum}. The two halves of
Theorem~\ref{thm:intro-sk} rest on different inputs. The magnetized half uses the Parisi formula
and, for $T<1$, the accumulation property. The unmagnetized half, including the boundary line and
the continuity bound, uses neither the Parisi formula nor the accumulation theorem
(Remark~\ref{rem:parisi-free}).

\paragraph{Nishimori's verticality argument.} Theorem~\ref{thm:intro-lat} bears on Nishimori's verticality argument, whose conclusion does
not hold on every two-dimensional periodic lattice: on $\Z^2[H]$ at $\Pp(J_e=-1)=p_0$, strictly on
the ordered side of the Nishimori-line threshold, the ferromagnetic phase does not extend to low
temperature, so the boundary there is reentrant, a shape allowed by the gauge
inequality~\cite{Nishimori1981} (see~\cite[Fig.~4.3]{NishimoriBook}), rather than vertical. The gauge
identities on which that argument rests hold on $\Z^2[H]$, as on any graph. At least one of two
heuristic steps of the argument therefore fails there, and we do not determine which: that the free energy on the Nishimori line, and with it the frustration
entropy, is singular at the multicritical point, or that such a singularity, being a property of
the disorder alone, fixes the ferromagnetic boundary at every lower temperature. Kitatani's
verticality result assumes a condition that, as Nishimori, Ohzeki and Okuyama observed, reentrance
would violate~\cite[\S III.F]{NOO2025} (Section~\ref{sec:lat-scope}).

\paragraph{The reentrance results.} We are not aware
of an earlier rigorous proof of reentrance in the $p$-spin models, or of an earlier rigorous
(computer-assisted) proof that, for the Ising model with quenched iid $\pm J$ couplings on a
planar, $\Z^2$-periodic graph of maximum degree four, ferromagnetic long-range order present at
the Nishimori temperature is absent at all sufficiently low temperatures at the same disorder
strength. The geometric conditions separate this statement from the Bethe-lattice and
hierarchical-lattice results cited above; the quenched iid couplings separate it from the exact
reentrance on decorated and periodic square lattices with deterministic or annealed couplings
(Section~\ref{sec:intro-ideas}).

\paragraph{Existing results re-proved, extended or repaired.} Several cited results are re-proved,
extended, repaired or replaced here; each is stated where it enters (see also
Figure~\ref{fig:roadmap}). We include a proof of Chen's concentration
theorem~\cite[Proposition~1]{Chen2014}, because the printed proof applies the exponential bound on
the complement of the event where it holds (Lemma~\ref{lem:concentration}). Dey and Kang's Jensen
argument~\cite[\S3.1]{DeyKang2026} is carried slightly past $\gamma=1$ (Lemmas~\ref{lem:jensen}
and~\ref{lem:cw}). Appendix~\ref{app:mf-parisi} records a short proof that $0$ lies in the support
of the zero-field Parisi measure (the last statement of \textup{(P6)} in Proposition~\ref{prop:facts}), because the argument
of~\cite{AC2015a} relies on a bound that at isolated support points is supplied
by~\cite[Corollary~3.10]{JT2017}. The same appendix runs Panchenko's argument~\cite[\S2]{Panchenko2014} with a
deterministic field. Zhou stated that the tangency system of the triple point (Section~\ref{sec:nm-triple}) has a unique
solution~\cite[\S2.2]{Zhou2024}, and Proposition~\ref{prop:nm-triple} gives a proof. Zhou's
intermediate-value argument for a one-step Parisi measure~\cite[Theorem~1(2) and \S4.2]{Zhou2024}
needs repairs as printed; Proposition~\ref{prop:nm-SG} recovers the one-step structure together
with the expansion. Lemma~\ref{lem:nm-LS} adapts the persistence argument of Aronow and
Lopatto~\cite[Proposition~2.7]{AronowLopatto2026} to a point where it does not apply directly.
Theorem~\ref{thm:nm-main} proves the two-point statement corresponding to Nishimori's replica
prediction~\cite{Nishimori2026}: Proposition~\ref{prop:nm-SG} and Lemma~\ref{lem:nm-closed}
identify his second-order coefficients, and Theorem~\ref{thm:nm-LP} turns his large-$p$ asymptotics
into inequalities. The warm point re-proves, in substance, Chen's weak-recovery theorem~\cite{Chen2019}, for
$p$-subsets with an explicit threshold (Theorem~\ref{thm:nm-main}(i)), and Horiguchi and Morita's
quenched Peierls bound~\cite{HM1982} is written per site with one Chernoff factor per contour
(Proposition~\ref{prop:lat-CP}).

\subsection{How the proofs go}\label{sec:intro-ideas}

Figure~\ref{fig:roadmap} on page~\pageref{fig:roadmap} shows the structure of the three proofs; the
paragraphs below give, for each column, the obstruction (for the SK model, see also
Section~\ref{sec:intro-prior}) and the step that overcomes it. The columns
share two sets of tools: Nishimori's gauge identities, which enter through the gauge transfer
(Lemma~\ref{lem:gauge-transfer}), the Nishimori identity on the Nishimori line
(Lemmas~\ref{lem:gn} and~\ref{lem:nm-NLid}) and the planted posterior (Theorem~\ref{thm:lat-A});
and, for the mean-field models, the facts about the Parisi functional of Section~\ref{sec:prelim}
together with magnetization bins (partition functions restricted to an interval of magnetizations)
and Gaussian concentration (Lemmas~\ref{lem:concentration}
and~\ref{lem:bins}).

\paragraph{The SK boundary as a susceptibility.} With the Curie--Weiss term $\gamma Nm^2/2$ and
the gain $\varphi_\beta$ of Section~\ref{sec:intro-prior} (see also Section~\ref{sec:sk}),
Chen's variational formula~\cite{Chen2014} expresses the free energy as the maximum over
$\mu\in[-1,1]$ of $F^{\rm SK}(\beta,\gamma\mu)-\gamma\mu^2/2$, where $F^{\rm SK}(\beta,h)$ is the
free energy of the SK model in a field $h$. Let $\Om(\beta,\gamma)$ be the set of maximizers; the
model is magnetized when $0\notin\Om(\beta,\gamma)$. By convexity of the free energy in $\gamma$,
this forces $\liminf_N\E\gibbs{m^2}>0$, while $\Om(\beta,\gamma)=\{0\}$ gives $\E\gibbs{m^2}\to0$ by
Chen's concentration theorem~\cite[Proposition~1]{Chen2014}. Dey and Kang made precise that the onset of
magnetization is at $\gamma_c(\beta)=\inf_{h>0}h^2/(2\varphi_\beta(h))$~\cite[Remark~1.3]{DeyKang2026}.
The argument has two parts.
\begin{itemize}
\item The Nishimori inequality gives the envelope $\varphi_\beta(h)\le h^2/(2\beta)$ for all $h$
(Theorem~\ref{thm:sk}(a)).
\item The curvature $\lim_{h\to0}2\varphi_\beta(h)/h^2$ exists and equals $\partial_x^2u(0,0)$,
the second derivative at the origin of the solution $u$ of the Parisi PDE~\eqref{eq:pde}, with
terminal condition $u(1,x)=\lc x$, for the \emph{zero-field} Parisi measure. The proof uses
uniqueness and continuity of Parisi minimizers.
\end{itemize}
As Lopatto shows~\cite[Lemma~6.1]{Lopatto2026}, using the first-order optimality conditions of
Jagannath and Tobasco~\cite{JT2017}, $\beta^2\bigl(\partial_x^2u(0,0)\bigr)^2=1$ as soon as the
origin is an accumulation point of the support of the Parisi measure; since $\partial_x^2u\ge0$,
this means $\partial_x^2u(0,0)=1/\beta$. Our alternative proof of
the accumulation property argues by contradiction. If the support had a gap at the origin, the
optimality conditions at the two ends of the gap would be incompatible with convexity on the gap
of $q\mapsto\Gamma(q)/q$, where $\Gamma(q)$ is the mean square of $\partial_xu$ along the optimal
diffusion. This convexity follows from the concavity of $\partial_xu$ on $[0,\infty)$ and an
explicit polynomial Lyapunov function. The zero-field susceptibility of the spin-glass phase is
then exactly $1/\beta$, which destabilizes the non-magnetized maximizer precisely when
$\gamma>\beta$, i.e.\ $J_0>1$.

\paragraph{The boundary line.} On the line $\gamma=\beta$ itself the verticality proof compares
$(\beta,\beta)$ with the Nishimori point of bias $J_0=1$, which is the multicritical point. There
$0$ is still the only maximizer~\cite[Proposition~1.1(i)]{DeyKang2026}, so Chen's concentration
theorem and the gauge inequality give $\E\gibbs{m^2}_{\beta,\beta}\to0$, but not
$\Om(\beta,\beta)=\{0\}$. By convexity in $\gamma$, a maximizer $\mu_*\ne0$ at $\gamma=\beta$
shows up only in the right derivative of the free energy, so bounds are needed at Nishimori
points with $J_0$ slightly above $1$ (Remark~\ref{rem:right-nbhd}). Dey and Kang's Jensen
bound $\varphi_\beta(h)\le\lc h$~\cite[\S3.1]{DeyKang2026} is still informative there: at the
Nishimori point of bias $1+\delta$ it confines the maximizers to $\mu^2\le y(\delta)$, where
$y(\delta)\sim12\delta$. Chen's concentration theorem~\cite[Proposition~1]{Chen2014} turns this into
a bound on Gibbs averages, and Nishimori's correlation inequality, in the averaged form of Itoi
and Sakamoto~\cite{ItoiSakamoto2023}, carries it down the vertical line to every temperature
(Figure~\ref{fig:transfer}). This gives $\Om(\beta,\beta)=\{0\}$ and the continuity bound of
Theorem~\ref{thm:intro-sk}, with $\max_{\Om(\beta,\gamma)}|\mu|\le(12(\gamma/\beta-1))^{1/4}$ for
$\beta\le\gamma\le21\beta/20$. The exponent $1/4$ has two sources: the Jensen bound gives
$\mu^2=O(\delta)$ at the Nishimori point, and the transfer down the vertical line loses a square
root.\footnote{On the Nishimori line the model is the Rademacher spiked Wigner model, whose free energy
is known~\cite{KoradaMacris2009,DAM2016}. It gives the sharper value
$\lim_N\E\gibbs{m^2}=O(\delta^2)$ at the Nishimori point of bias $1+\delta$, and after the same
transfer $\max_{\Om(\beta,\gamma)}|\mu|=O((\gamma/\beta-1)^{1/2})$ (Remark~\ref{rem:spiked}).}

\paragraph{Reentrance near the triple point.} In the replica picture~\cite{GNS2001,Nishimori2026}, the transition from the spin glass to the
ferromagnet near the triple point of the $p$-spin models is discontinuous: the magnetization
jumps. The two free energies have the same value and the same temperature derivative at the
triple point, namely those of the paramagnet: the one-step spin-glass solution branches off the
paramagnetic one there, and identities on the Nishimori line give the ferromagnetic solution the
paramagnetic value and derivative at $M$ (Lemma~\ref{lem:nm-FMloc}(a),(b) and Section~\ref{sec:nm-route}).
Fix $j_0=j_{0M}+\delta$ and $\beta=\beta_1+b$. Free
energies are limits of $N^{-1}\E\log Z$ (Section~\ref{sec:intro-prior}), so the larger one wins. The proof compares two of them.
The zero-field Parisi free energy bounds the full free energy from below, by Jensen's inequality
for the Gibbs measure without bias, under which the bias term has nonnegative mean up to $O(1)$;
for odd $p$ its leading part has mean zero, since reversing all spins and the centred couplings preserves the joint law
(Lemma~\ref{lem:lower}). The value of the
replica-symmetric Guerra bound at its ferromagnetic stationary point bounds from above the free
energy of every magnetization bin near the triple-point magnetization. We expand both relative to the
paramagnetic free energy $\log2+\beta^2/4$. Relative to it, both vanish at $M$ together with their
$\beta$-derivatives, and only the ferromagnet gains from the bias, at first order in $\delta$. Their
second-order terms in $b$ are $-A_{\rm SG}b^2/2$ and $-A_{\rm FM}b^2/2$, with Nishimori's
coefficients~\cite{Nishimori2026}. Since $\log2+\beta^2/4$ has second $\beta$-derivative $\frac12$,
the specific heats of the two phases at $M$ are $C_X=\beta_1^2(\frac12-A_X)$, and
$\beta_1^2(A_{\rm FM}-A_{\rm SG})=C_{\rm SG}-C_{\rm FM}$. When this difference is positive, the spin
glass wins at $b=c\,\delta^{1/2}$ by a margin of order $\delta$. Two obstructions have to be
removed.
\begin{itemize}
\item \emph{A second-order lower bound.} A proof of reentrance needs a lower bound on the spin-glass free
energy that is exact to second order at the triple point, whereas trial measures in the Parisi
formula give only upper bounds.
Convexity of the free energy in $\beta$, used alone,
gives the lower bound with $A_{\rm SG}$ replaced by $\frac12$ (it shows only $C_{\rm SG}\ge0$), and
the comparison would then need $C_{\rm FM}<0$, which is false at least for $3\le p\le25$. We show
instead, by a persistence argument of the kind Aronow and Lopatto use for nondegenerate
finite-step Parisi measures~\cite[Proposition~2.7]{AronowLopatto2026}, that every stationary
two-atom measure $x\delta_0+(1-x)\delta_q$ with $0<x<1$ and $(\beta,x,q)$ close to
$(\beta_1,1,q_1)$, where $q_1$ is the spin-glass overlap at $M$, is the Parisi measure
(Lemma~\ref{lem:nm-LS}). The one-step stationarity equations have a
nonsingular Jacobian at $M$, so the one-step branch is analytic and yields the exact expansion
(Proposition~\ref{prop:nm-SG}).\footnote{For $p$ larger than an absolute constant that is not made explicit,
Talagrand proved a Gibbs-measure counterpart just above $\beta_1$: the overlap of two replicas is
asymptotically close to $0$ or to $\pm1$, and the Gibbs measure splits into
lumps~\cite[Theorems~1.3--1.5]{Talagrand2000}. Zhou argued, in an unrefereed preprint, that the
Parisi measure is one-step just above $\beta_1$~\cite[Theorem~1(2)]{Zhou2024}, by an
intermediate-value argument that gives no continuity of the one-step parameters in $\beta$.
Neither result gives the second-order expansion of the free energy.}
\item \emph{The sign for every $p$.} The difference $C_{\rm SG}-C_{\rm FM}$ is of order
$2^{-p}p^{-1/2}$, while $A_{\rm FM}$ and $A_{\rm SG}$ are each exponentially close to $\frac12$, so
its sign cannot be read off numerically for all $p$. An exact identity writes its leading part
as a Gaussian integral with a positive kernel (Proposition~\ref{prop:nm-J}). The remaining terms
include negative ones, which the positive leading part must dominate. Interval arithmetic
certifies this, and hence $C_{\rm SG}>C_{\rm FM}$, for $3\le p\le25$. For $p\ge26$ explicit
two-sided bounds, which turn Nishimori's large-$p$ asymptotics into inequalities, bound the
remainder by the leading part (Theorem~\ref{thm:nm-LP}). Together they cover every integer
$p\ge3$.
\end{itemize}
At the cold point, magnetizations away from the triple-point value are excluded by Guerra-type
bounds, whose values at $M$ are controlled by a strict sign condition on the replica-symmetric
potential (Proposition~\ref{prop:nm-triple}), and small magnetizations by the bound
$\partial_x^2u\le1$ on solutions of the Parisi PDE. At the warm point, the order is, in substance,
W.-K.~Chen's weak-recovery theorem~\cite[proofs of Lemma~4, p.~13, and of Theorem~1(ii),
p.~14]{Chen2019} (Section~\ref{sec:intro-results}). We prove it directly, with the explicit
threshold $m_1$. A first-moment bound shows that the magnetizations $|m|\le m_1$ (for odd $p$, all
$m\le m_1$) carry free energy at most the paramagnetic value, which is $\log2+j_0^2$ at $\beta=2j_0$, while a
replica-symmetric interpolation lower bound on the Nishimori line puts the full free energy
strictly above that value (Theorem~\ref{thm:reduction}(i) and Section~\ref{sec:nm-proof}). This step uses neither the Parisi
formula nor information on Parisi measures.

The explicit instances for $p=3$ and $p=4$ come from a reduction of reentrance at one bias and two
temperatures to three finite predicates (Theorem~\ref{thm:reduction}). At the warm point,
predicate (W) states that a replica-symmetric interpolation lower bound on the free energy on the
Nishimori line exceeds a first-moment upper bound on the free energy of the excluded
magnetization bins. At the cold point, predicate (F) bounds the zero-field Parisi free energy
from below: by convexity of the Parisi functional in the measure, the Frank--Wolfe duality gap at
an explicit one-step measure turns its value there into a lower bound on the infimum. For this we
derive closed forms of the directional derivatives (Appendix~\ref{app:D}).
Predicate (L) bounds the free energy of each magnetization bin from above, below the value
from (F), by Guerra-type bounds with replica-symmetric and one-step trials in a field.\footnote{The
margins in these free-energy inequalities, per spin, are about $1.6\times10^{-4}$ ($p=3$) and
$2.3\times10^{-4}$ ($p=4$) at the warm point and $5.7\times10^{-4}$ and $1.3\times10^{-3}$ at the
cold point, so the numerical evaluation must be rigorous, not merely accurate.}
Every predicate is certified in interval arithmetic and
certified again by an independent implementation (Section~\ref{sec:certificates} and
Appendix~\ref{app:cert34}).

\paragraph{The decorated lattice.} On lattices, the two standard rigorous criteria for order and for uniqueness that we use see
the couplings through one-edge functionals. The quenched Peierls bound of Horiguchi and Morita~\cite{HM1982} proves
order once $\E e^{-2s\beta J_e}$ is small enough for some $s\in[0,1]$, and Newman's comparison
with Bernoulli percolation~\cite{Newman1994,Newman1997} proves uniqueness of the Gibbs state once
$\E(1-e^{-2\beta|J_e|})$ lies below the critical probability. For couplings $\beta J_e$ with $J_e$
of fixed law both criteria are monotone in $\beta$. They certify order only on a low-temperature
half-line and uniqueness only on a high-temperature interval, so they cannot detect reentrance.
Summing out the interior spins of each gadget is exact and
turns the model on $\Z^2[H]$ into an Ising model on $\Z^2$ whose couplings, the \emph{backbone
couplings}, are again independent and identically distributed, at every temperature
(Lemma~\ref{lem:lat-BR}). Their law is explicit (Proposition~\ref{prop:lat-law}) and not of the
form $\beta J$. At the Nishimori temperature each long path is a weak bond whose sign is negative
with probability about $0.42$. Every unit is then strongly coupled through its direct edge, and the backbone
coupling is reliable, i.e.\ large and positive with high probability. On cooling the paths
freeze. Consider a unit whose two paths have the same sign (the product of the couplings along
the path) and whose branch, the two paths together with the two connector edges, has the sign
opposite to the direct edge. At zero temperature both relative orientations of the ends of the
unit then cost exactly one broken edge: the direct edge, or one of the two connector edges. The
first has one ground state and the second two, so the coupling of the unit tends to
$\frac12\log2$. At $p_0$ a unit is in this frustrated state with probability about $0.17$, and in
a chain of ten units some unit is likely to be, so the backbone couplings do not become reliable
at low temperature. Quantitatively, as $\beta\to\infty$ the quantity $1-e^{-2|K|}$ for the backbone
coupling $K$ tends to $1$ on a chain with no frustrated unit, to $\frac12$ on a chain with one,
and to at most $\frac15$ on a chain with two or more. Its mean tends to $0.374\ldots$, below
$\frac12$, and bond percolation on $\Z^2$ has no infinite cluster at parameter $\frac12$ or below~\cite{Harris1960}. This is the
classical decorated-bond mechanism. It gives exact
paramagnet--ferromagnet--paramagnet reentrance on deterministic decorated
lattices~\cite{Syozi1968,Miyazima1968}, on the square lattice with a periodic pattern of
competing couplings~\cite{KMS1985,KMS1986}, and with annealed bond
randomness~\cite{KasaiMiyazimaSyozi1969,KasaiSyozi1973}. Here the couplings are quenched and iid
on every microscopic edge. The criteria described above apply at each fixed $\beta$ to the
decimated couplings, which are no longer monotone in $\beta$. Horiguchi and Morita's quenched
Peierls bound gives non-uniqueness on a window around $\beta_N$ (Proposition~\ref{prop:lat-NU}),
and Newman's domination by Bernoulli percolation, with Harris's theorem~\cite{Harris1960}, gives
uniqueness at high and at low temperature (Corollary~\ref{cor:lat-UH}). Order of all microscopic
spins at the Nishimori temperature follows from gauge identities, Kitatani's
inequality~\cite{Kitatani1994}, in the form stated in~\cite[(3)]{Kitatani2009}, which makes
the averaged correlations on the Nishimori line monotone under adding edges, and a Peierls bound
with two defect edges
(Theorem~\ref{thm:lat-A}). The statement for every disorder strength up to $1/1000$ follows from a
channel-degradation argument: along the Nishimori line, the one-gadget contour weight used by
the Peierls bound is nondecreasing in the disorder strength, so a bound certified at $1/1000$
holds for all smaller disorder strengths (Lemma~\ref{lem:lat-degradation}). A
further gauge inequality shows that, for every gadget, the Chernoff weight
$\min_{s\in[0,1]}\E e^{-2sK}$ of the Peierls bound, with $K$ the decimated coupling, is smallest at
the Nishimori temperature (Proposition~\ref{prop:lat-floor}). The contour side
therefore sees no further than the Nishimori line, and the reentrance comes from the uniqueness
side.

\paragraph{Computer-assisted steps.} The proofs of the theorems use computer assistance at four
places, each documented in an appendix and in the companion repository (Appendix~\ref{app:repro}).
Sections~\ref{sec:sk} and~\ref{sec:skline} contain no computer-assisted statement.
Theorem~\ref{thm:intro-sk} uses the accumulation theorem, which is due to
Lopatto~\cite{Lopatto2026}; our alternative proof of it verifies one polynomial inequality, property
(V1) of its polynomial Lyapunov function, in exact rational arithmetic
(Appendix~\ref{sec:accum-cert}). For Theorem~\ref{thm:intro-p}, the signs
at the triple point (Theorem~\ref{thm:nm-nondeg}) are certified in interval arithmetic for $3\le p\le25$, and the constants of the
explicit bounds for $p\ge26$ are certified as well (Appendices~\ref{app:nm-LP}
and~\ref{app:nm-cert}). The predicates behind Theorems~\ref{thm:p3} and~\ref{thm:p4} are certified
in interval arithmetic and again by an independent implementation (Appendix~\ref{app:cert34}). The
inputs of Theorem~\ref{thm:intro-lat} are certified in interval arithmetic, each by at least two
implementations that share no code (Appendix~\ref{app:lat-cert}).\footnote{Some statements also
have a second, computer-assisted proof (the single crossing of Lemma~\ref{lem:nm-U} for integer
$p$, Appendix~\ref{app:nm-U}; the conditions (N1)--(N3) of Section~\ref{sec:nm-fm} for $7\le
p\le25$, Appendix~\ref{app:nm-cert}). Remark~\ref{rem:nm-chen} reports an interval-arithmetic
check of a lemma of Zhou that no proof here uses, and two floating-point evaluations in the main
text, one quoted for orientation only and one suggesting that a set is an interval, are marked as
such.}

\subsection{Organization}\label{sec:intro-organization}

Sections~\ref{sec:prelim}--\ref{sec:accum} treat the SK model: Section~\ref{sec:prelim} collects the
facts about the Parisi functional, Section~\ref{sec:sk} proves Theorem~\ref{thm:sk},
Section~\ref{sec:skline} treats the boundary line $J_0=1$ and the continuity of the transition, and
Section~\ref{sec:accum} gives the alternative proof of the accumulation property.
Sections~\ref{sec:pspin}--\ref{sec:certificates} treat the $p$-spin models: Section~\ref{sec:pspin}
sets up the model and its free-energy bounds and reduces reentrance at one bias and two temperatures
to finite predicates, Section~\ref{sec:nearM} proves Theorem~\ref{thm:nm-main}, and
Section~\ref{sec:certificates} states the explicit instances for $p=3$ and $p=4$.
Section~\ref{sec:lattice} treats the decorated lattice and does not depend on
Sections~\ref{sec:prelim}--\ref{sec:certificates}. Section~\ref{sec:discussion} discusses open
problems. Each of these sections gives its statements, the short proofs that carry the main ideas,
and the deduction of its main theorem; longer proofs, routine estimates and the certificates are in
the appendices. Appendix~\ref{app:mf} contains the deferred proofs for
Sections~\ref{sec:prelim}--\ref{sec:pspin}, the certificate of Section~\ref{sec:accum} and the
certificates for $p=3$ and $p=4$; Appendix~\ref{app:largep} the proofs and certificates for
Section~\ref{sec:nearM}; Appendix~\ref{app:lattice} the proofs and certified inputs for
Section~\ref{sec:lattice}; and Appendix~\ref{app:repro} a list of the computer-assisted steps and
information on reproducing them.

\section{The Parisi functional}\label{sec:prelim}

All mean-field models in this paper are Ising models on $\{\pm1\}^N$ whose Hamiltonian is a centred
Gaussian process plus a deterministic function of the magnetization
$m(\sigma)=N^{-1}\sum_i\sigma_i$. For the centred part we use the Parisi formula. Let
$\xi(q)=\sum_{p\ge2}a_p^2q^p$ with finitely many nonzero coefficients, and let
$H^0_N$ be the centred Gaussian process with covariance
$\E H^0_N(\sigma)H^0_N(\tau)=N\xi(R(\sigma,\tau))+O(1)$, where $R=N^{-1}\sum_i\sigma_i\tau_i$
and the $O(1)$ is uniform. For $h\in\R$ set
\[
F^0(h)=\limsup_{N\to\infty}\frac1N\E\log\sum_{\sigma}e^{H^0_N(\sigma)+hNm(\sigma)} .
\]
For a probability measure $\mu$ on $[0,1]$ write $\alpha_\mu(s)=\mu([0,s])$. Let
$u_\mu:[0,1]\times\R\to\R$ solve
\begin{equation}\label{eq:pde}
\partial_s u+\frac{\xi''(s)}2\bigl(\partial_x^2u+\alpha_\mu(s)(\partial_xu)^2\bigr)=0,
\qquad u(1,x)=\lc x ,
\end{equation}
and define the Parisi functional
\begin{equation}\label{eq:parisi}
\Par(\mu,h)=\log2+u_\mu(0,h)-\frac12\int_0^1\xi''(s)\,s\,\alpha_\mu(s)\,\dd s .
\end{equation}

\begin{proposition}[Known facts]\label{prop:facts}\leavevmode
\begin{enumerate}
\item[\textup{(P1)}] \emph{Parisi formula.} For every $h$ and every $\mu$,
$F^0(h)\le\Par(\mu,h)$: this is Guerra's bound~\cite{Guerra2003}, which for odd $p$ uses
Talagrand's positivity principle~\cite{Panchenko2007}; see~\cite[\S2]{Panchenko2014} and
Appendix~\ref{app:mf-parisi}. The limit exists and
$F^0(h)=\min_\mu\Par(\mu,h)$ at $h=0$~\cite{Talagrand2006,Panchenko2014}, and for every $h$
when $\xi$ is even~\cite{Talagrand2006}. The $O(1)$ in the covariance changes
$N^{-1}\E\log Z$ by $O(1/N)$, by Gaussian interpolation.
\item[\textup{(P2)}] \emph{Convexity and uniqueness.} $\mu\mapsto\Par(\mu,h)$ is strictly convex
along linear interpolation of measures, so the minimizer is unique~\cite[Theorem~2,
Corollary~1]{AC2015b}.
\item[\textup{(P3)}] \emph{Regularity.} $u_\mu(s,\cdot)$ is even and smooth in $x$, with
$|\partial_xu_\mu|\le1$ and $0\le\partial_x^2u_\mu\le1$. If $\mu_n\Rightarrow\mu$, then
$\partial_x^ju_{\mu_n}\to\partial_x^ju_\mu$ uniformly on $[0,1]\times\R$ for every
$j\ge0$~\cite[Propositions~1--2]{AC2015a}.
\item[\textup{(P4)}] \emph{Optimal diffusion.} Let $X$ solve
$\dd X_s=\xi''(s)\alpha_\mu(s)\partial_xu_\mu(s,X_s)\dd s+\sqrt{\xi''(s)}\,\dd W_s$ with
$X_0=h$, and set $\Gamma_\mu(s)=\E[\partial_xu_\mu(s,X_s)^2]$. Then
$\Gamma_\mu'(s)=\xi''(s)\E[\partial_x^2u_\mu(s,X_s)^2]$~\cite[Lemma~8.4]{JT2017}, and for $h=0$,
$\Gamma_\mu$ is $C^1$~\cite[Proposition~3]{AC2015a} (see Appendix~\ref{app:mf-parisi}).
\item[\textup{(P5)}] \emph{First variation.} Put
$G_\mu(q)=\frac12\int_q^1\xi''(t)(\Gamma_\mu(t)-t)\,\dd t$. Then
$\Par(\tilde\mu,h)-\Par(\mu,h)=\int G_\mu\,\dd(\tilde\mu-\mu)+O(d(\tilde\mu,\mu)^2)$, where
$d(\tilde\mu,\mu)=\int_0^1|\alpha_{\tilde\mu}-\alpha_\mu|$ (see the proof
of~\cite[Lemma~3.2]{JT2017}).
See also~\cite{Chen2017}.
\item[\textup{(P6)}] \emph{Optimality.} $\mu$ minimizes $\Par(\cdot,h)$ if and only if
$\mu(\operatorname{argmin}G_\mu)=1$. For a minimizer, every $q\in\operatorname{supp}\mu$ satisfies
$\Gamma_\mu(q)=q$ and $\xi''(q)\E[\partial_x^2u_\mu(q,X_q)^2]\le1$~\cite[Proposition~1.1,
Corollaries~3.6 and~3.10]{JT2017}. At $h=0$ one has
$0\in\operatorname{supp}\mu$~\cite[Theorem~1]{AC2015a}.
\end{enumerate}
\end{proposition}

Two items of Proposition~\ref{prop:facts} combine sources: \textup{(P1)} with a deterministic
field, and \textup{(P4)} in the form along $X$; Appendix~\ref{app:mf-parisi} gives the details.
The last statement of \textup{(P6)} has a short proof from the rest of \textup{(P6)}, which we
record because the argument in~\cite{AC2015a} relies on the bound
$\xi''\E[(\partial_x^2u_\mu)^2]\le1$ on the support, and at isolated support points that bound
is supplied by~\cite[Corollary~3.10]{JT2017}; the proof is in Appendix~\ref{app:mf-parisi}.

We use two consequences repeatedly. The first is the bound $\partial_x^2u\le1$ in
\textup{(P3)}. For measures with finitely many atoms it follows from the explicit
Cole--Hopf form of~\eqref{eq:pde}: on an interval where $\alpha_\mu\equiv a>0$ and the
Gaussian variance increases by $v$,
\begin{equation}\label{eq:colehopf}
u(s_-,x)=\frac1a\log\E\exp\bigl(a\,u(s_+,x+\sqrt v\,g)\bigr),
\qquad g\sim N(0,1),
\end{equation}
and on an interval with $\alpha_\mu\equiv0$ the step is the heat semigroup. Writing
$\gibbs{\cdot}$ for the tilted average in~\eqref{eq:colehopf} and $\phi=u(s_+,\cdot)$, one gets
$\partial_x^2u=\gibbs{\phi''}+a\operatorname{Var}(\phi')$. If $\phi''\le1-\phi'^2$, then
$\partial_x^2u\le1-a\gibbs{\phi'}^2-(1-a)\gibbs{\phi'^2}\le1-(\partial_xu)^2$.
The terminal condition $\lc$ satisfies $\phi''=1-\phi'^2$, so the invariant
$\partial_x^2u\le1-(\partial_xu)^2$ propagates, and general $\mu$ follows by \textup{(P3)}.

The second consequence is the \emph{universal quadratic bound} on field responses. For every
$\mu$ and $h$,
\begin{equation}\label{eq:quadratic}
\Par(\mu,h)-\Par(\mu,0)=u_\mu(0,h)-u_\mu(0,0)
=\int_0^h(h-y)\,\partial_x^2u_\mu(0,y)\,\dd y\le\frac{h^2}2,
\end{equation}
because $u_\mu(0,\cdot)$ is even. Applied to the zero-field minimizer, \eqref{eq:quadratic}
gives $F^0(h)\le F^0(0)+h^2/2$.

\section{The Sherrington--Kirkpatrick model with a Curie--Weiss bias}\label{sec:sk}

\subsection{Setting}

Following Dey and Kang~\cite{DeyKang2026}, let $g_{ij}$, $i<j$, be independent standard
Gaussians and consider the Gibbs measure proportional to $e^{H_N(\sigma)}$ with
\[
H_N(\sigma)=\frac{\beta}{\sqrt N}\sum_{i<j}g_{ij}\sigma_i\sigma_j+\frac{\gamma N}2m(\sigma)^2 .
\]
Expanding $\frac{\gamma N}2m^2=\frac\gamma N\sum_{i<j}\sigma_i\sigma_j+\frac\gamma2$ shows that
this is the SK model at inverse temperature $\beta$ with couplings
$J_{ij}\sim N(J_0/N,1/N)$ and $J_0=\gamma/\beta$. The constant $\gamma/2$ does not affect the
Gibbs measure. In the physical variables $(J_0,T)=(\gamma/\beta,1/\beta)$ the Nishimori line is
$T=1/J_0$, i.e.\ $\gamma=\beta^2$.

Write $F_N(\beta,\gamma)=N^{-1}\E\log Z_N$ and let $F^{\rm SK}(\beta,h)$ be the limiting free
energy of the SK model with field term $hNm(\sigma)$. This is the setting of
Section~\ref{sec:prelim} with $\xi(q)=\beta^2q^2/2$. Chen~\cite[Theorem~1]{Chen2014} proved
\begin{equation}\label{eq:chen}
F(\beta,\gamma):=\lim_{N\to\infty}F_N(\beta,\gamma)=\max_{\mu\in[-1,1]}\Bigl\{F^{\rm SK}(\beta,\gamma\mu)-\frac{\gamma\mu^2}2\Bigr\}.
\end{equation}
Chen states this for the Hamiltonian $\beta'N^{-1/2}\sum_{i,j}g_{ij}\sigma_i\sigma_j+\gamma Nm^2/2$
summed over all ordered pairs. With $\beta'=\beta/\sqrt2$ this is our model plus the diagonal
term $\beta'N^{-1/2}\sum_ig_{ii}$, which does not depend on $\sigma$ and has mean zero, so
$F_N$ is the same for every $N$~\cite[Remark~2.2]{DeyKang2026}. We write $\Om(\beta,\gamma)$ for the set of maximizers.
In~\eqref{eq:chen} and throughout Sections~\ref{sec:sk} and~\ref{sec:skline} and
Appendices~\ref{app:mf-sk}--\ref{sec:spiked}, the letter $\mu$
without subscript is a real variable, a candidate magnetization; Parisi measures carry a
subscript, as in $\mu_\beta$, or are denoted~$\nu$. Put
$\varphi_\beta(h)=F^{\rm SK}(\beta,h)-F^{\rm SK}(\beta,0)$. Since $F_N(\beta,\cdot)$ is convex
with $\partial_\gamma F_N=\E\gibbs{m^2}/2$, the limit $F(\beta,\cdot)$ is convex. Moreover
$0\notin\Om(\beta,\gamma)$ exactly when $F(\beta,\gamma)>F^{\rm SK}(\beta,0)$. We call
$(\beta,\gamma)$ \emph{magnetized} if $0\notin\Om(\beta,\gamma)$. Since $F(\beta,\cdot)$ is
nondecreasing, convex and at least $F^{\rm SK}(\beta,0)$, the magnetized set at fixed $\beta$
is $\{\gamma>\gamma_c(\beta)\}$ with
$\gamma_c(\beta):=\inf\{\gamma:0\notin\Om(\beta,\gamma)\}$.

Let $\mu_\beta$ be the zero-field Parisi measure and
$\theta_\beta=\partial_x^2u_{\mu_\beta}(0,0)$. By \textup{(P3)}, $0\le\theta_\beta\le1$. For
$\beta\le1$ one has $\mu_\beta=\delta_0$, hence
$u_{\mu_\beta}(s,x)=\frac{\beta^2}2(1-s)+\lc x$ and $\theta_\beta=1$ (Appendix~\ref{app:mf-sk}).

\begin{theorem}[The SK model with a Curie--Weiss bias]\label{thm:sk}
Let $\beta>0$ and $\gamma>0$.
\begin{enumerate}
\item[\textup{(a)}] \emph{Envelope.} If $\gamma<\max(1,\beta)$, then $\Om(\beta,\gamma)=\{0\}$
and $\E\gibbs{m^2}\to0$. Moreover $\varphi_\beta(h)\le h^2/(2\max(1,\beta))$ for all $h$.
\item[\textup{(b)}] \emph{Curvature at zero field.}
$\lim_{h\to0}2\varphi_\beta(h)/h^2=\theta_\beta$.
\item[\textup{(c)}] \emph{Magnetization.} If $\gamma>1/\theta_\beta$, then
$0\notin\Om(\beta,\gamma)$ and $\liminf_N\E\gibbs{m^2}>0$.
\item[\textup{(d)}] \emph{Marginality.} If $\beta>1$, then $\theta_\beta=1/\beta$.
\end{enumerate}
Consequently $\theta_\beta=\min(1,1/\beta)$ and $\gamma_c(\beta)=\max(1,\beta)$ for every
$\beta>0$. The magnetized region in physical variables is $\{J_0>\max(1,T)\}$, where
$\liminf_N\E\gibbs{m^2}>0$ by~(c), while $\E\gibbs{m^2}\to0$ for $J_0<\max(1,T)$ by~(a). Below
the multicritical point $(1,1)$ the boundary is the vertical line $J_0=1$.
\end{theorem}

Part~(d) rests on the fact that $0$ is an accumulation point of $\operatorname{supp}\mu_\beta$ for
every $\beta>1$. This is due to Lopatto~\cite[Proposition~5.1]{Lopatto2026}, and
Section~\ref{sec:accum} gives a short alternative proof (Theorem~\ref{thm:accum}). For
$\beta\in(1,1+\sqrt2\,\eta]$ with $\eta>0$ from~\cite[Theorem~1]{Zhou2026}, it also follows from
Zhou's theorem, which is stated for $\xi(q)=\beta^2q^2$ with $\beta\in(1/\sqrt2,1/\sqrt2+\eta]$; our
normalization has $\xi(q)=\beta^2q^2/2$. On the critical line $\gamma=\max(1,\beta)$ itself, part~(a) and
continuity give $0\in\Om$. For $\beta\le1$, i.e.\ on the segment $\gamma=1$, Dey and Kang show
$\Om(\beta,1)=\{0\}$~\cite[Proposition~1.1(i)]{DeyKang2026}. For $\gamma=\beta>1$, the gauge
comparison in the proof of part~(a) (Appendix~\ref{app:mf-sk}), made with the Nishimori point $(1,1)$, where
$\Om(1,1)=\{0\}$ by~\cite[Proposition~1.1(i)]{DeyKang2026}, together with Chen's concentration
theorem~\cite[Proposition~1]{Chen2014} (Lemma~\ref{lem:concentration}), already gives
$\E\gibbs{m^2}\to0$. Theorem~\ref{thm:skline} in Section~\ref{sec:skline} shows moreover
$\Om(\beta,\beta)=\{0\}$, and that the order parameter vanishes continuously as
$\gamma\downarrow\beta$.

The case $\gamma<1$ of part~(a) is contained in~\cite[Proposition~1.1(i)]{DeyKang2026}, and for
$\beta\le1$ part~(c) and $\gamma_c(\beta)=1$ recover~\cite[Proposition~1.1(ii) and
Remark~1.3]{DeyKang2026}. For $\beta>1$, part~(a) is a short consequence of known results:
Nishimori's inequality in the form of~\cite[Theorem~2.2 and \S3]{ItoiSakamoto2023}, applied with
the Nishimori point unmagnetized by~\cite[Proposition~1.1(i)]{DeyKang2026}, gives
$\E\gibbs{m^2}\to0$ and hence $0\in\Om(\beta,\gamma)$ for $\gamma<\beta$, and Dey and Kang's
formula $\gamma_c(\beta)=\inf_{h>0}h^2/(2\varphi_\beta(h))$~\cite[Remark~1.3]{DeyKang2026} turns
this into the envelope. We include short proofs. The case $\gamma<1$ of~(a) uses only the
Parisi-PDE bound~\eqref{eq:quadratic}; part~(c) uses the curvature formula~(b) and
$\theta_\beta=1$ for $\beta\le1$.

\subsection{The envelope}

For $\gamma<1$, part~(a) follows from the quadratic bound~\eqref{eq:quadratic}. For $\beta>1$ and
$1\le\gamma<\beta$, it compares $(\beta,\gamma)$ with the Nishimori point $(J_0,J_0^2)$ of the same
bias $J_0=\gamma/\beta<1$, through Nishimori's gauge identity and the Cauchy--Schwarz step
of~\cite[eq.~(26)]{ItoiSakamoto2023} (compare Lemma~\ref{lem:gauge-transfer}); integrating
$\partial_\gamma F_N=\E\gibbs{m^2}/2$ then gives the envelope. The proof is in
Appendix~\ref{app:mf-sk}.

\subsection{Proof of the curvature formula}

\emph{Upper bound.} Using $\mu_\beta$ as a trial at field $h$ in \textup{(P1)},
\[
\varphi_\beta(h)\le u_{\mu_\beta}(0,h)-u_{\mu_\beta}(0,0)=\int_0^h(h-y)\,\partial_x^2u_{\mu_\beta}(0,y)\,\dd y,
\]
and continuity of $\partial_x^2u_{\mu_\beta}(0,\cdot)$ gives
$\limsup_{h\to0}2\varphi_\beta(h)/h^2\le\theta_\beta$.

\emph{Lower bound.} For each $h$ let $\nu_h$ minimize $\Par(\cdot,h)$. A minimizer exists
because $\Par(\cdot,h)$ is continuous on the compact set of probability measures, by
\textup{(P3)} and continuity of $\nu\mapsto\int_0^1s\,\alpha_\nu(s)\,\dd s$. Since
$\Par(\mu_\beta,0)\le\Par(\nu_h,0)$,
\[
\varphi_\beta(h)\ge\Par(\nu_h,h)-\Par(\nu_h,0)
=\int_0^h(h-y)\,\partial_x^2u_{\nu_h}(0,y)\,\dd y
\ge\frac{h^2}2\min_{|y|\le h}\partial_x^2u_{\nu_h}(0,y).
\]
Let $\nu$ be a weak limit of $\nu_{h_n}$ with $h_n\to0$. By joint continuity of $\Par$ and
$|\partial_hF^{\rm SK}|\le1$, we get $\Par(\nu,0)=\lim F^{\rm SK}(\beta,h_n)=F^{\rm SK}(\beta,0)$.
So $\nu$ is a zero-field minimizer, and $\nu=\mu_\beta$ by \textup{(P2)}. Hence
$\nu_h\Rightarrow\mu_\beta$, and by \textup{(P3)} $\partial_x^2u_{\nu_h}(0,\cdot)\to\partial_x^2u_{\mu_\beta}(0,\cdot)$
uniformly. Therefore $\liminf_{h\to0}2\varphi_\beta(h)/h^2\ge\theta_\beta$.\qed

\subsection{Proof of the magnetization statement}

Suppose $\gamma>1/\theta_\beta$. By part~(b), some $h\in(0,\gamma]$ satisfies
$\varphi_\beta(h)>h^2/(2\gamma)$. With $\mu=h/\gamma$,
\[
F^{\rm SK}(\beta,\gamma\mu)-\frac{\gamma\mu^2}2=F^{\rm SK}(\beta,0)+\varphi_\beta(h)-\frac{h^2}{2\gamma}>F^{\rm SK}(\beta,0),
\]
so $0\notin\Om(\beta,\gamma)$ and $F(\beta,\gamma)>F^{\rm SK}(\beta,0)$. Let
$\gamma_0=\max(1,\beta)$. Parts~(a) and~(b) give
$\theta_\beta=\lim_{h\to0}2\varphi_\beta(h)/h^2\le1/\gamma_0$, so $\gamma>\gamma_0$. By part~(a) and continuity, $F(\beta,\cdot)$ equals
$F^{\rm SK}(\beta,0)$ on $[0,\gamma_0]$, and it is convex. For convex $F_N\to F$,
\[
\liminf_N\frac{\E\gibbs{m^2}}2=\liminf_N\partial_\gamma F_N(\beta,\gamma)\ge\partial^-_\gamma F(\beta,\gamma)
\ge\frac{F(\beta,\gamma)-F(\beta,\gamma_0)}{\gamma-\gamma_0}>0.\qed
\]

\subsection{Proof of marginality}

Let $\beta>1$ and write $\Gamma=\Gamma_{\mu_\beta}$ with $X_0=0$. At zero field $u(0,\cdot)$ is
even, so $\Gamma(0)=0$. By \textup{(P4)}, $\Gamma'(0)=\beta^2\theta_\beta^2$. The optimality
conditions \textup{(P6)} give $\Gamma(q)=q$ on $\operatorname{supp}\mu_\beta$. By
Theorem~\ref{thm:accum} there are $q_n\in\operatorname{supp}\mu_\beta\setminus\{0\}$ with
$q_n\downarrow0$, and then
\[
\Gamma'(0)=\lim_n\frac{\Gamma(q_n)-\Gamma(0)}{q_n}=1 .
\]
Hence $\theta_\beta=1/\beta$, because $\theta_\beta\ge0$. This is the argument
of~\cite[Lemma~6.1]{Lopatto2026}. For $\beta$ near $1$, the sequence $q_n$ is also supplied by
Zhou's inclusion $[0,\upsilon_\beta]\subset\operatorname{supp}\mu_\beta$ for some
$\upsilon_\beta>0$~\cite[Theorem~1]{Zhou2026}.\qed

For $\beta>1$, part~(d), Lemma~\ref{lem:uxxx} and the Parisi formula also give the envelope in
strict form, without gauge symmetry (Remark~\ref{rem:envelope-route} in Appendix~\ref{app:mf-accum}).

\section{The boundary line \texorpdfstring{$\gamma=\beta$}{gamma = beta}}\label{sec:skline}

Theorem~\ref{thm:sk} locates the ferromagnetic boundary at $\gamma=\max(1,\beta)$ but is stated
off the boundary. On the segment $\gamma=1$, $\beta\le1$, Dey and Kang show
$\Om(\beta,1)=\{0\}$~\cite[Proposition~1.1(i)]{DeyKang2026}. On the vertical part $\gamma=\beta>1$,
part~(a) of Theorem~\ref{thm:sk} and continuity give only $0\in\Om(\beta,\beta)$. Whether other
maximizers exist there decides whether the order parameter, which is positive for $\gamma>\beta$
by Theorem~\ref{thm:sk}(c),(d), vanishes continuously at the boundary. This section shows that
it does: $\Om(\beta,\beta)=\{0\}$ and $\E\gibbs{m^2}_{\beta,\beta}\to0$ for every $\beta>0$, with
explicit bounds as $\gamma\downarrow\beta$.

\paragraph{The obstruction.} For $1\le\gamma<\beta$ the proof of Theorem~\ref{thm:sk}(a)
compares $(\beta,\gamma)$ with the Nishimori point $(J_0,J_0^2)$ of the same bias
$J_0=\gamma/\beta$, and it needs $J_0<1$, where that point lies in the paramagnetic phase. At
$\gamma=\beta$ the Nishimori point is the multicritical point $(1,1)$. There
$\Om(1,1)=\{0\}$~\cite[Proposition~1.1(i)]{DeyKang2026}, so Lemmas~\ref{lem:concentration}
and~\ref{lem:gauge-transfer} already give $\E\gibbs{m^2}_{\beta,\beta}\to0$. This does not give
$\Om(\beta,\beta)=\{0\}$. By convexity, the
Gibbs average at $\gamma=\beta$ controls only the left derivative of $F(\beta,\cdot)$ at $\beta$,
while a maximizer $\mu_*\ne0$ at $\gamma=\beta$ shows up in the right derivative
(Remark~\ref{rem:right-nbhd}). What is needed is a bound that vanishes on a right neighbourhood
$\gamma\downarrow\beta$, that is, at Nishimori points with $J_0$ slightly above $1$.

\paragraph{The result.} Such a bound comes from Dey and Kang's Jensen
argument~\cite[\S3.1]{DeyKang2026}, which they apply for $\gamma\le1$. Carried slightly past
$\gamma=1$, it confines the maximizers at the Nishimori point $(J_0,J_0^2)$ to
$\mu^2\le y(J_0-1)$, where $y(\delta)\sim12\delta$ as $\delta\downarrow0$. Chen's concentration
theorem~\cite[Proposition~1]{Chen2014} turns this into a bound on Gibbs averages, and the gauge
transfer carries it down the vertical line of fixed bias to $(\beta,\gamma)$
(Figure~\ref{fig:transfer}). Each ingredient is known; the point is that the Jensen bound is
still informative just past $\gamma=1$, which is where the transfer needs it. The argument does
not use the Parisi formula.

For $b>0$ let $\bar x(b)=0$ if $b\le1$, and for $b>1$ let $\bar x(b)$ be the unique positive root
of $\lc x=x^2/(2b^2)$ (Lemma~\ref{lem:y}). Put
\[
y(\delta)=\Bigl(\frac{\bar x(1+\delta)}{(1+\delta)^2}\Bigr)^2\quad(\delta>-1),\qquad
Y(\delta)=\min\{1,y(\delta)\}.
\]
By Lemma~\ref{lem:y}, $y=0$ on $(-1,0]$, $y$ is continuous, $y(\delta)/\delta\to12$ as
$\delta\downarrow0$, $y(\delta)\le12\delta$ for $0\le\delta\le1/20$, and $y(\delta)<1$ if and
only if $\delta<\delta_c:=\sqrt{2\log\tau}-1=0.10397\ldots$, where $\tau$ is the real root of
$u^3=u^2+u+1$.

\begin{theorem}[The vertical boundary line]\label{thm:skline}
Let $\beta>0$, $\gamma>0$ and $\delta=\gamma/\beta-1$.
\begin{enumerate}
\item[\textup{(a)}] $\limsup_N\E\gibbs{m^2}_{\beta,\gamma}\le Y(\delta)^{1/2}$.
\item[\textup{(b)}] $\mu^2\le Y(\delta)^{1/2}$ for every $\mu\in\Om(\beta,\gamma)$.
\item[\textup{(c)}] If $\varepsilon>0$ and $\varepsilon^2>Y(\delta)^{1/2}$, then there is $K=K(\beta,\gamma,\varepsilon)$
such that $\E\gibbs{\ind\{|m|\ge\varepsilon\}}_{\beta,\gamma}\le Ke^{-N/K}$ for every $N$.
\end{enumerate}
In particular, for every $\beta>0$ and $0<\gamma\le\beta$ one has $\Om(\beta,\gamma)=\{0\}$ and
$\E\gibbs{m^2}_{\beta,\gamma}\to0$; this includes the vertical boundary line $\gamma=\beta>1$.
For $0\le\delta\le1/20$ the bounds read $\limsup_N\E\gibbs{m^2}_{\beta,\gamma}\le(12\delta)^{1/2}$
and $\max_{\Om(\beta,\gamma)}|\mu|\le(12\delta)^{1/4}$.
\end{theorem}

The constant $K$ in~(c) is not explicit. The bounds are nontrivial only for $\delta<\delta_c$.

\needspace{6\baselineskip}
\emph{Proof of Theorem~\ref{thm:skline}, given the lemmas of
Section~\ref{sec:skline-ingredients}, whose longer proofs are in Appendix~\ref{app:mf-skline}.} Put
$J_0=\gamma/\beta=1+\delta$.

(a) By Lemma~\ref{lem:cw} at the Nishimori point $(J_0,J_0^2)$ and $|\mu|\le1$, Jensen's bound
confines the maximizers there: $\max_{\Om(J_0,J_0^2)}\mu^2\le Y(\delta)$. By Chen's concentration
theorem (Lemma~\ref{lem:concentration}), $\limsup_N\E\gibbs{m^2}_{J_0,J_0^2}\le Y(\delta)$. By the
gauge transfer (Lemma~\ref{lem:gauge-transfer}) with $b=\beta$,
$\limsup_N\E\gibbs{m^2}_{\beta,\gamma}\le Y(\delta)^{1/2}$.

(b) Let $\mu_*\in\Om(\beta,\gamma)$ and $\gamma'>\gamma$. By convexity in $\gamma$ and a
comparison at fixed field (Lemma~\ref{lem:fixed-field}(c)), together with part~(a) at $\gamma'$,
$\mu_*^2\le\frac{\gamma'}\gamma Y(\gamma'/\beta-1)^{1/2}$. Let $\gamma'\downarrow\gamma$ and use
the continuity of $Y$ (Lemma~\ref{lem:y}(b)).

(c) Apply Lemma~\ref{lem:concentration} with $\varepsilon>Y(\delta)^{1/4}\ge\max_\Om|\mu|$.

For $\gamma\le\beta$ one has $\delta\le0$ and $Y(\delta)=0$, which gives the statement for
$0<\gamma\le\beta$. The quantitative forms follow from Lemma~\ref{lem:y}(c).\qed

\medskip
The hypotheses enter as follows. The Nishimori property enters only in
Lemma~\ref{lem:gauge-transfer}: the law $N(J_0/N,1/N)$ of the couplings has Nishimori inverse
temperature $J_0$. Lemma~\ref{lem:jensen} needs only that the SK couplings are centred Gaussians,
so that their law is gauge invariant; Lemma~\ref{lem:cw} adds Chen's formula and the trial
$\mu=0$. Lemma~\ref{lem:concentration} needs $\gamma\ge0$, Gaussian couplings and the existence
of $F^{\rm SK}$ at every field. The value of $\gamma/\beta$ decides where the Nishimori point
lies: for $\gamma\le\beta$ it has $J_0^2\le1$, where $\bar x=0$, and for $\gamma$ slightly above
$\beta$ it has $J_0$ slightly above $1$, where $\bar x$ is small. No restriction $\beta>1$ is
needed; it matters only in Corollary~\ref{cor:continuity}, where the magnetized side comes from
Theorem~\ref{thm:sk}(c),(d).

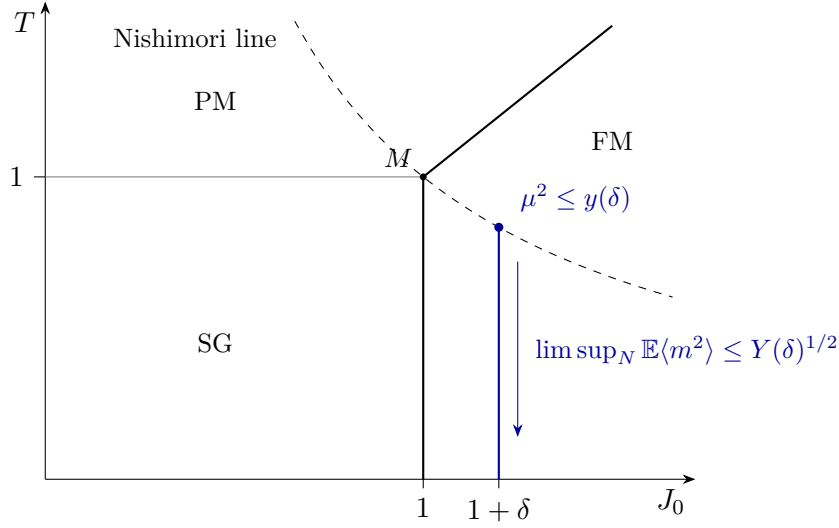
\begin{figure}[t]
\centering
\begin{tikzpicture}[x=5cm,y=4cm,>=Stealth]
\draw[->] (0,0) -- (1.72,0) node[below left]{$J_0$};
\draw[->] (0,0) -- (0,1.58) node[below left]{$T$};
\draw (1,0) -- (1,-0.03) node[below]{$1$};
\draw (1.2,0) -- (1.2,-0.03) node[below]{$1+\delta$};
\draw (0,1) -- (-0.03,1) node[left]{$1$};
\draw[gray] (0,1) -- (1,1);
\draw[thick] (1,1) -- (1.5,1.5);
\draw[thick] (1,0) -- (1,1);
\draw[dashed,domain=0.66:1.66,smooth,variable=\x] plot ({\x},{1/\x});
\node[anchor=east] at (0.64,1.46) {\small Nishimori line};
\fill (1,1) circle (1.3pt) node[above left]{\small $M$};
\draw[blue!60!black,thick] (1.2,0) -- (1.2,0.8333);
\fill[blue!60!black] (1.2,0.8333) circle (1.7pt);
\node[blue!60!black,anchor=west] at (1.23,0.93) {\small $\mu^2\le y(\delta)$};
\draw[->,blue!60!black] (1.25,0.72) -- (1.25,0.14);
\node[blue!60!black,anchor=west] at (1.27,0.43) {\small $\limsup_N\E\gibbs{m^2}\le Y(\delta)^{1/2}$};
\node at (0.45,1.25) {\small PM};
\node at (0.45,0.45) {\small SG};
\node at (1.5,1.12) {\small FM};
\end{tikzpicture}
\caption{The mechanism of Theorem~\ref{thm:skline} in the variables
$(J_0,T)=(\gamma/\beta,1/\beta)$. Thick lines bound the magnetized region $\{J_0>\max(1,T)\}$ of
Theorem~\ref{thm:sk}, the thin line $T=1$ separates the paramagnetic (PM) and spin-glass (SG)
regions, $M$ is the multicritical point, and the dashed curve is the Nishimori line $T=1/J_0$.
At fixed bias $J_0=1+\delta$ the Nishimori point (dot) lies just past $M$. There Jensen's bound
confines the maximizers to $\mu^2\le y(\delta)\approx12\delta$
(Lemma~\ref{lem:cw}), and the gauge transfer (Lemma~\ref{lem:gauge-transfer}) carries the bound
down the vertical line to every temperature, as
$\limsup_N\E\gibbs{m^2}\le Y(\delta)^{1/2}$. Letting $\delta\downarrow0$ and using convexity in
$\gamma$ gives $\Om=\{0\}$ on $J_0=1$. The value of $\delta$ is exaggerated.}
\label{fig:transfer}
\end{figure}

\subsection{Ingredients}\label{sec:skline-ingredients}

This subsection and Appendix~\ref{app:mf-skline} prove the lemmas used above: the Jensen bound
and its consequence for the maximizers, the gauge transfer, Chen's concentration theorem, the
fixed-field comparison, and the elementary properties of $y$.

Write $F^{\rm SK}_N(\beta,h)=N^{-1}\E\log\sum_\sigma\exp\bigl(\beta N^{-1/2}\sum_{i<j}g_{ij}\sigma_i\sigma_j+hNm(\sigma)\bigr)$,
and let $Z^{\rm SK}_N(\beta,h)$ be the partition function in this expression. Then
$F^{\rm SK}(\beta,h)=\lim_NF^{\rm SK}_N(\beta,h)$, and the limit exists for every $h$ by Guerra and
Toninelli's superadditivity argument (see~\cite[eq.~(2.1)]{Chen2014}). Since $|\partial_hF^{\rm SK}_N|=|\E\gibbs{m}|\le1$, the function
$F^{\rm SK}(\beta,\cdot)$ is $1$-Lipschitz, and it is even by spin flip. Recall from
Section~\ref{sec:sk} that $\partial_\gamma F_N=\E\gibbs{m^2}/2$ and that $F_N(\beta,\cdot)$ is
convex, so $\gamma\mapsto\E\gibbs{m^2}_{\beta,\gamma}$ is nondecreasing. The trial $\mu=0$
in~\eqref{eq:chen} gives $F(\beta,\gamma)\ge F^{\rm SK}(\beta,0)$.

\begin{lemma}[Jensen field bound]\label{lem:jensen}
For every $\beta\ge0$, $h\in\R$ and $N$, $F^{\rm SK}_N(\beta,h)-F^{\rm SK}_N(\beta,0)\le\lc h$.
Hence $\varphi_\beta(h)\le\lc h$.
\end{lemma}

\begin{proof}
This is the argument of~\cite[\S3.1]{DeyKang2026}. The left side equals
$N^{-1}\E\log\gibbs{e^{hNm}}_{\beta,0}$, which is at most
$N^{-1}\log\E\gibbs{e^{hNm}}_{\beta,0}$ by Jensen's inequality. For $\tau\in\{\pm1\}^N$ the map
$g_{ij}\mapsto g_{ij}\tau_i\tau_j$ preserves the law of the centred couplings, so the averaged
zero-field Gibbs measure $\E\gibbs{\ind_{\{\sigma\}}}_{\beta,0}$ is invariant under
$\sigma\mapsto\sigma\tau$ and hence uniform. Thus
$\E\gibbs{e^{hNm}}_{\beta,0}=2^{-N}\sum_\sigma e^{h\sum_i\sigma_i}=\cosh^Nh$.
\end{proof}

\begin{lemma}[Curie--Weiss comparison at the maximizers]\label{lem:cw}
For every $\beta\ge0$, $\gamma>0$ and $\mu\in\Om(\beta,\gamma)$, one has
$\lc(\gamma\mu)\ge\gamma\mu^2/2$, and hence $\gamma|\mu|\le\bar x(\sqrt\gamma)$. At a Nishimori
point $(b,b^2)$ this reads $\mu^2\le y(b-1)$; in particular $\Om(b,b^2)=\{0\}$ for $b\le1$.
\end{lemma}

\begin{proof}
By $F(\beta,\gamma)\ge F^{\rm SK}(\beta,0)$ and Lemma~\ref{lem:jensen},
\[
0\le F(\beta,\gamma)-F^{\rm SK}(\beta,0)=\varphi_\beta(\gamma\mu)-\frac{\gamma\mu^2}2
\le\lc(\gamma\mu)-\frac{\gamma\mu^2}2 .
\]
With $x=\gamma|\mu|>0$ this says $\lc(x)/x^2\ge1/(2\gamma)$. Since $\lc(x)/x^2$ decreases strictly
from $1/2$ to $0$ on $(0,\infty)$ (Lemma~\ref{lem:y}(a)), this fails for every $x>0$ if
$\gamma\le1$, and for $\gamma>1$ it holds if and only if $x\le\bar x(\sqrt\gamma)$. At $(b,b^2)$
we have $\sqrt\gamma=b$, so $\mu^2\le(\bar x(b)/b^2)^2=y(b-1)$.
\end{proof}

For $\gamma\le1$ this is~\cite[Proposition~1.1(i)]{DeyKang2026}, with the same proof. For
$\gamma>1$ the bound is informative only for $\gamma$ near $1$; the next lemma moves every
$(\beta,\gamma)$ with $\gamma/\beta$ near $1$ to such a point.

\begin{lemma}[Gauge transfer]\label{lem:gauge-transfer}
For every $b>0$, $J_0>0$ and $N$, with $\gamma=bJ_0$,
\[
\E\gibbs{m^2}_{b,\gamma}\le\bigl(\E\gibbs{m^2}_{J_0,J_0^2}\bigr)^{1/2}.
\]
\end{lemma}

This is Nishimori's correlation inequality~\cite{Nishimori1981},~\cite[\S4.6.2]{NishimoriBook}, in
the averaged form of~\cite[Theorem~2.2 and eq.~(26)]{ItoiSakamoto2023}. It is the inequality used
in the proof of Theorem~\ref{thm:sk}(a) (Appendix~\ref{app:mf-sk}). There it carries an extra term $1/N$ and is applied only
for $J_0<1$; neither the term nor the restriction is needed.

\begin{proof}
Write $J_{ij}=g_{ij}/\sqrt N+J_0/N$, so that $\gibbs{\cdot}_{b,bJ_0}$ is proportional to
$\exp(b\sum_{i<j}J_{ij}\sigma_i\sigma_j)$, and $\gibbs{\cdot}_{J_0,J_0^2}$ is the same measure with
$b$ replaced by $J_0$. The law of $J=(J_{ij})_{i<j}$ has density $\rho(J)e^{J_0\sum_{i<j}J_{ij}}$
with $\rho(J)=\prod_{i<j}c_Ne^{-NJ_{ij}^2/2}$. For $\tau\in\{\pm1\}^N$ the map
$J\mapsto J^\tau$, $J^\tau_{ij}=J_{ij}\tau_i\tau_j$, preserves Lebesgue measure and $\rho$, so
averaging over $\tau$ gives, for every integrable $f$,
\[
\E f(J)=\int\rho(J)\,2^{-N}\sum_\tau e^{J_0\sum_{i<j}J_{ij}\tau_i\tau_j}f(J^\tau)\,\dd J .
\]
Apply this to $f=\gibbs{\sigma_i\sigma_j}_{b,bJ_0}$, which satisfies $f(J^\tau)=\tau_i\tau_jf(J)$,
and to the gauge-invariant $f=\gibbs{\sigma_i\sigma_j}_{J_0,J_0^2}\gibbs{\sigma_i\sigma_j}_{b,bJ_0}$.
Since $2^{-N}\sum_\tau\tau_i\tau_je^{J_0\sum J_{kl}\tau_k\tau_l}$ equals
$2^{-N}\sum_\tau e^{J_0\sum J_{kl}\tau_k\tau_l}$ times $\gibbs{\sigma_i\sigma_j}_{J_0,J_0^2}$, both
integrals coincide:
\[
\E\gibbs{\sigma_i\sigma_j}_{b,bJ_0}
=\E\bigl[\gibbs{\sigma_i\sigma_j}_{J_0,J_0^2}\gibbs{\sigma_i\sigma_j}_{b,bJ_0}\bigr]
\qquad\text{for every }b>0.
\]
At $b=J_0$ this gives
$a_{ij}:=\E\gibbs{\sigma_i\sigma_j}_{J_0,J_0^2}=\E\gibbs{\sigma_i\sigma_j}_{J_0,J_0^2}^2\ge0$, and
Cauchy--Schwarz gives $\E\gibbs{\sigma_i\sigma_j}_{b,bJ_0}\le a_{ij}^{1/2}$, also for $i=j$, where
$a_{ii}=1$. Averaging over all $N^2$ pairs and using concavity of the square root,
$\E\gibbs{m^2}_{b,bJ_0}\le N^{-2}\sum_{i,j}a_{ij}^{1/2}\le(N^{-2}\sum_{i,j}a_{ij})^{1/2}
=(\E\gibbs{m^2}_{J_0,J_0^2})^{1/2}$.
\end{proof}

\begin{lemma}[Concentration on the maximizers]\label{lem:concentration}
Let $\beta\ge0$ and $\gamma\ge0$. For every $U\subset[-1,1]$ at positive distance from
$\Om(\beta,\gamma)$ there is $K=K(\beta,\gamma,U)$ such that
$\E\gibbs{\ind\{m\in U\}}_{\beta,\gamma}\le Ke^{-N/K}$ for every $N$. In particular, for every
$\varepsilon>\max_{\Om(\beta,\gamma)}|\mu|$ there is $K$ with
$\E\gibbs{\ind\{|m|\ge\varepsilon\}}_{\beta,\gamma}\le Ke^{-N/K}$ for every $N$, and
$\limsup_N\E\gibbs{m^2}_{\beta,\gamma}\le\max_{\Om(\beta,\gamma)}\mu^2$.
\end{lemma}

This is Chen's Proposition~1~\cite{Chen2014}, stated there for open $U$ and a possibly random
external field. We include the proof, which does not use openness, because the printed
proof~\cite[p.~12]{Chen2014} applies the exponential bound on the complement of the event where
it holds. With the threshold used there, twice the limiting gap, that event is moreover the rare
one, so Step~3 of the proof redefines the good event rather than only exchanging it with its
complement. The idea is to split $Z_N$ by the value of $m$, which compares the Gibbs weight of
$\{m\in U\}$ with $N+1$ SK partition functions at the fields $\gamma\mu$; their free energies
converge uniformly to the function maximized in~\eqref{eq:chen}, and Gaussian concentration
controls all $N+1$ fluctuations at once.
The proof is in Appendix~\ref{app:mf-skline}.

\begin{lemma}[Convexity and a fixed-field comparison]\label{lem:fixed-field}
Fix $\beta>0$.
\begin{enumerate}
\item[\textup{(a)}] For $0\le\gamma<\gamma'$,
$0\le F(\beta,\gamma')-F(\beta,\gamma)\le\frac{\gamma'-\gamma}2\limsup_N\E\gibbs{m^2}_{\beta,\gamma'}$.
\item[\textup{(b)}] For $\gamma>0$, $\mu_*\in\Om(\beta,\gamma)$ and $\gamma'>\gamma$,
$F(\beta,\gamma')-F(\beta,\gamma)\ge\frac{\gamma\mu_*^2(\gamma'-\gamma)}{2\gamma'}$.
\item[\textup{(c)}] Hence $\mu_*^2\le\frac{\gamma'}\gamma\limsup_N\E\gibbs{m^2}_{\beta,\gamma'}$ for
every $\gamma>0$, $\mu_*\in\Om(\beta,\gamma)$ and $\gamma'>\gamma$.
\end{enumerate}
\end{lemma}

No derivative of the limit $F$ is taken, and none of $F^{\rm SK}$ in $h$: the argument uses only
the monotonicity of $\partial_\gamma F_N$ at finite $N$ and a comparison at fixed field
(Appendix~\ref{app:mf-skline}).

\subsection{Consequences}

\begin{corollary}[Free energy near the line]\label{cor:free-energy-line}
Fix $\beta>0$.
\begin{enumerate}
\item[\textup{(a)}] For $0\le\gamma<\gamma'$,
$0\le F(\beta,\gamma')-F(\beta,\gamma)\le\frac{\gamma'-\gamma}2Y(\gamma'/\beta-1)^{1/2}$.
\item[\textup{(b)}] $F(\beta,\gamma)=F^{\rm SK}(\beta,0)$ for all $\gamma\in[0,\beta]$.
\item[\textup{(c)}] $F(\beta,\cdot)$ is differentiable at $\gamma=\beta$ with derivative $0$. For
$\gamma\ge\beta$,
$F(\beta,\gamma)-F^{\rm SK}(\beta,0)\le\frac{\gamma-\beta}2Y(\gamma/\beta-1)^{1/2}$, which is at
most $\frac{\gamma-\beta}2(12(\gamma/\beta-1))^{1/2}=O((\gamma-\beta)^{3/2})$ for
$\gamma\le21\beta/20$.
\end{enumerate}
\end{corollary}

\begin{proof}
(a) is Lemma~\ref{lem:fixed-field}(a) with Theorem~\ref{thm:skline}(a) at $\gamma'$. For (b),
take $\gamma'\le\beta$ in (a), where $Y=0$, and use $F(\beta,0)=F^{\rm SK}(\beta,0)$. For (c), the
right difference quotient at $\beta$ is at most $\frac12Y(\gamma'/\beta-1)^{1/2}\to0$ as
$\gamma'\downarrow\beta$, and the left derivative is $0$ by (b).
\end{proof}

\begin{corollary}[Continuity across the vertical boundary]\label{cor:continuity}
Let $\beta>1$. For every $\gamma>0$,
\[
\sup_{\mu\in\Om(\beta,\gamma)}|\mu|\le Y\bigl((\gamma/\beta-1)_+\bigr)^{1/4},\qquad
\limsup_N\E\gibbs{m^2}_{\beta,\gamma}\le Y\bigl((\gamma/\beta-1)_+\bigr)^{1/2},
\]
and both bounds tend to $0$ as $\gamma\to\beta$.
\end{corollary}

\begin{proof}
For $\gamma<\beta$, $\Om=\{0\}$ and $\E\gibbs{m^2}\to0$ by Theorem~\ref{thm:skline}. For
$\gamma\ge\beta$, use Theorem~\ref{thm:skline}(a),(b). The limit holds because $Y$ is continuous
with $Y(0)=0$.
\end{proof}

By Theorem~\ref{thm:sk}(c),(d), $0\notin\Om(\beta,\gamma)$ and
$\liminf_N\E\gibbs{m^2}_{\beta,\gamma}>0$ for $\gamma>\beta>1$. So the magnetization vanishes
continuously at $\gamma=\beta$, i.e.\ at $J_0=1$. Quantitatively, with
$\delta=\gamma/\beta-1\in[0,1/20]$, $\max_\Om|\mu|\le(12\delta)^{1/4}$. The exponent $1/4$ is an
upper bound only: it combines the bound $y(\delta)\sim12\delta$ at the Nishimori point with the
square root in Lemma~\ref{lem:gauge-transfer}. On the Nishimori line itself no transfer is needed,
and Lemma~\ref{lem:cw} gives exponent $1/2$ there unconditionally. Off the line,
Remark~\ref{rem:spiked} feeds into the same transfer the exact value
$\lim_N\E\gibbs{m^2}=O(\delta^2)$ at the Nishimori point of bias $1+\delta$, which follows from
the known Nishimori-line free energy~\cite{KoradaMacris2009,DAM2016}, and improves the exponent
to $1/2$.

For $0<\beta\le1$ the order parameter also vanishes as $\gamma\to1$, without a rate, so with
Corollary~\ref{cor:continuity} it vanishes continuously along the whole boundary
$\gamma=\max(1,\beta)$ (Remark~\ref{rem:beta-le-one}).

The next consequence concerns the envelope that Dey and Kang identify, together with Toulouse's
identity, as equivalent to $\gamma_c(\beta)=\beta$~\cite[Remark~1.3]{DeyKang2026}.
Theorem~\ref{thm:sk}(a) proves the non-strict bound $\varphi_\beta(h)\le h^2/(2\max(1,\beta))$.
For $\beta\le1$ the strict bound $\varphi_\beta(h)<h^2/(2\beta)$, $h\ne0$, follows from
Lemma~\ref{lem:jensen}, since $\lc h<h^2/2\le h^2/(2\beta)$. For every $\beta>0$,
by~\eqref{eq:chen} and Corollary~\ref{cor:free-energy-line}(b), $\Om(\beta,\beta)=\{0\}$ holds
if and only if $\varphi_\beta(h)<h^2/(2\beta)$ for $0<|h|\le\beta$. The following corollary
quantifies this.

\begin{corollary}[Strict envelope]\label{cor:envelope}
For every $\beta>0$ and $h\ne0$, $\varphi_\beta(h)<h^2/(2\beta)$. Quantitatively,
\[
\frac{h^2}{2\beta}-\varphi_\beta(h)\ge
\begin{cases}
\dfrac{h^6}{1536\,\beta^5}, & 0<|h|\le\beta,\\[2ex]
\dfrac{(|h|-\beta)^2}{2\beta}+\dfrac{\beta}{1536}, & |h|>\beta.
\end{cases}
\]
\end{corollary}

The proof takes the trial $\mu=h/\gamma$ in~\eqref{eq:chen} at $\gamma=\beta(1+\delta)$ with
$\delta=h^4/(192\beta^4)$ for $0<|h|\le\beta$, and reduces $|h|>\beta$ to $h=\beta$ by the
$1$-Lipschitz property (Appendix~\ref{app:mf-skline}).

\begin{remark}[Theorem~\ref{thm:sk}(a) without the Parisi formula]\label{rem:parisi-free}
For every $\beta>0$ and $0<\gamma\le\max(1,\beta)$ one has $\Om(\beta,\gamma)=\{0\}$ and
$\E\gibbs{m^2}_{\beta,\gamma}\to0$, and $\varphi_\beta(h)\le h^2/(2\max(1,\beta))$ for all $h$.
For $\gamma\le1$ this follows from Lemmas~\ref{lem:cw} and~\ref{lem:concentration}, and for
$\gamma\le\beta$ from Theorem~\ref{thm:skline}. For the envelope, use Lemma~\ref{lem:jensen} if
$\beta\le1$; if $\beta\ge1$, Corollary~\ref{cor:free-energy-line}(b) at $\gamma=\beta$ gives
$\varphi_\beta(\beta\mu)\le\beta\mu^2/2$ for $|\mu|\le1$, and the Lipschitz step in the proof of
Theorem~\ref{thm:sk}(a) (Appendix~\ref{app:mf-sk}) extends this to all $h$. This proves part~(a) of Theorem~\ref{thm:sk} on
the closed region $\gamma\le\max(1,\beta)$ without the quadratic bound~\eqref{eq:quadratic}, and
hence without the Parisi formula. Neither proof of part~(a) uses Theorem~\ref{thm:accum}, which
enters only through part~(d).
\end{remark}

\section{Accumulation at the origin in the SK model}\label{sec:accum}

Throughout this section $\xi(q)=\beta^2q^2/2$ with $\beta>1$, the field is $0$, and
$\mu=\mu_\beta$ is the zero-field Parisi measure. We write $u=u_\mu$, $\alpha=\alpha_\mu$,
$\Gamma=\Gamma_\mu$, $G=G_\mu$ and $X$ for the objects of Section~\ref{sec:prelim}, and
$\theta=\theta_\beta=\partial_x^2u(0,0)$.

\begin{theorem}[Lopatto]\label{thm:accum}
For every $\beta>1$, $0$ is an accumulation point of $\operatorname{supp}\mu_\beta$.
\end{theorem}

Theorem~\ref{thm:accum} is due to Lopatto. It first appears in the second version
of~\cite{Lopatto2026} (Theorem~1.1 there: $[0,\varepsilon_\beta]\subset\operatorname{supp}\mu_\beta$);
in v3 it is~\cite[Proposition~5.1]{Lopatto2026}, and Theorem~1.1 there
identifies $\operatorname{supp}\mu_\beta=[0,q_\beta]$. An AI-written proof that the support has no
gaps~\cite[Proposition~2.12]{ChatGPT2026} relies on the slope inequalities of the second version
of~\cite{Lopatto2026} (Proposition~3.6 and~(3.39) there).
We give a short alternative proof of
Theorem~\ref{thm:accum} that uses nothing from~\cite[\S\S3--5]{Lopatto2026}. Its inputs are
\textup{(P3)}, \textup{(P4)}, \textup{(P6)}, the sign of $\partial_x^3u$
(Lemma~\ref{lem:uxxx}), and one exact polynomial inequality checked by computer
(Appendix~\ref{sec:accum-cert}).

The strategy is as follows. If $0$ were isolated in the support, the Parisi PDE on the first gap
$[0,q_*)$ would be a Cole--Hopf flow, and the optimality conditions at $0$ and $q_*$ are
incompatible with convexity of $q\mapsto\Gamma(q)/q$ on the gap (Lemma~\ref{lem:endgame}).
Lopatto instead shows $\Gamma(q_*)<\beta^2q_*\theta^2$ from a one-crossing property of $\Gamma''$
on the gap~\cite[Proposition~4.1 and Corollary~4.2]{Lopatto2026}. Gaussian integration by parts
reduces convexity to the sign of an explicit kernel built from the profile $u(s,\cdot)$
(Lemmas~\ref{lem:kernel} and~\ref{lem:omega}). A change of variables turns the kernel into an
integral along a planar trajectory whose dynamics do not depend on the gap
(Lemma~\ref{lem:reduction}), and a polynomial Lyapunov function shows that this integral is
nonnegative (Lemma~\ref{lem:accumulator}; Appendix~\ref{app:mf-accum}). Shape properties of $\Gamma(s)/s$ on a first gap also
appear in Zhou's analysis of the pure $p$-spin model~\cite[Propositions~3.1 and 3.2]{Zhou2026b}:
for the one-step candidates considered there, every critical point of $\Gamma(s)/s$ on the gap is a
strict local maximum.%
\footnote{Precisely, for a one-step candidate $x\delta_0+(1-x)\delta_q$ (Zhou's $m$ is our $x$) satisfying the
stationarity conditions and $\Gamma'(q)<1$, a tangency $\Gamma(s)=s\Gamma'(s)$ with
$s\in(0,q)$ forces $\Gamma''(s)<0$, so that every critical point of $\Gamma(s)/s$ on $(0,q)$ is a
strict local maximum.}
In the SK model the opposite shape, convexity, excludes the gap. We
have not found the convexity statement in the literature.

\subsection{The first gap}

Suppose that $0$ is not an accumulation point of $\operatorname{supp}\mu$. By \textup{(P6)}
and Appendix~\ref{app:mf-parisi}, $0\in\operatorname{supp}\mu$. Moreover $\mu\ne\delta_0$: for $\delta_0$ one has $\theta=1$
(Section~\ref{sec:sk}), whereas \textup{(P6)} at $q=0$ gives $\beta^2\theta^2\le1$ (this
is~\cite[Lemma~2.3]{Lopatto2026}; Theorem~\ref{thm:sk}(a)--(b) also give $\theta\le1/\beta$).
Hence $\operatorname{supp}\mu\setminus\{0\}$ is nonempty and closed, and
\[
q_*:=\min\bigl(\operatorname{supp}\mu\setminus\{0\}\bigr)>0,\qquad
\lambda:=\mu(\{0\})\in(0,1).
\]
Here $\lambda>0$ because $0$ is an isolated point of the support, and $\lambda<1$ because
$\mu\ne\delta_0$. Thus $\alpha\equiv\lambda$ on $[0,q_*)$. Put $t=\beta^2q_*$,
$\psi=u(q_*,\cdot)$ and $U(r,x)=u(r/\beta^2,x)$ for $r\in[0,t]$. By~\eqref{eq:colehopf} (derived
for finitely atomic measures; for general $\mu$ approximate by finitely atomic measures that agree
with $\mu$ on $[0,q_*)$ and use \textup{(P3)}, or use uniqueness for the constant-coefficient
equation on $[0,q_*]$),
\begin{equation}\label{eq:gap-ch}
U(r,x)=\frac1\lambda\log\E\exp\bigl(\lambda\,\psi(x+\sqrt{t-r}\,g)\bigr),\qquad 0\le r<t,
\end{equation}
so $\partial_rU+\frac12\bigl(\partial_x^2U+\lambda(\partial_xU)^2\bigr)=0$ on $[0,t)$. The process
$Y_r=X_{r/\beta^2}$ solves $\dd Y_r=\lambda\,\partial_xU(r,Y_r)\,\dd r+\dd B_r$ with $Y_0=0$ and a
standard Brownian motion $B$. Set $q(r)=\Gamma(r/\beta^2)$. For $0<r<t$ write $V=U(r,\cdot)$,
with primes denoting $x$-derivatives, and
\[
\langle f\rangle_r=\int_\R f\,p_r\,\dd x,\qquad
p_r(x)=\frac{\exp\bigl(\lambda V(x)-x^2/(2r)\bigr)}{\int_\R\exp\bigl(\lambda V(y)-y^2/(2r)\bigr)\dd y}.
\]

\begin{lemma}\label{lem:gap}
\textup{(i)} For $0<r<t$, $Y_r$ has density $p_r$.
\textup{(ii)} $q\in C^1[0,t]\cap C^2[0,t)$, and for $0<r<t$
\begin{equation}\label{eq:gap-moments}
q(r)=\langle V'^2\rangle_r,\qquad q'(r)=\langle V''^2\rangle_r,\qquad
q''(r)=\langle V'''^2-2\lambda V''^3\rangle_r .
\end{equation}
Moreover $q(0)=0$, $q'(0)=\theta^2$ and $q''(0)=-2\lambda\theta^3$ (right derivatives).
\textup{(iii)} $q(t)=t/\beta^2$, $\beta^2\theta^2\le1$, $q'(t)\le1/\beta^2$ and
$\int_0^tq(r)\,\dd r=tq(t)/2$.
\end{lemma}

The proof (Appendix~\ref{app:mf-accum}) combines Girsanov's theorem on the gap, the formulas
of~\cite[Lemma~8.4]{JT2017} and the optimality conditions \textup{(P6)}.

\begin{lemma}[Endgame]\label{lem:endgame}
Let $\kappa(r)=q(r)/r$ for $0<r\le t$. If $\kappa$ is convex on $(0,t)$, the conclusions of
Lemma~\ref{lem:gap} are contradictory.
\end{lemma}

\begin{proof}
By Lemma~\ref{lem:gap}(ii), $\kappa\in C^1(0,t]\cap C^2(0,t)$, $\kappa(0+)=\theta^2$, and
$\kappa'(r)=(rq'(r)-q(r))/r^2\to q''(0)/2=-\lambda\theta^3$ as $r\downarrow0$. By (iii),
$\kappa(t)=1/\beta^2\ge\theta^2=\kappa(0+)$. We give two ways to conclude.

\emph{Via $q'(t)\le1/\beta^2$.} Then $\kappa'(t)=(tq'(t)-q(t))/t^2\le0$. Since $\kappa'$ is
nondecreasing on $(0,t)$ and continuous on $(0,t]$, $\kappa'\le0$ on $(0,t]$, so
$\kappa(t)\le\kappa(0+)$. Hence $\kappa$ is constant, $\theta=1/\beta$, and
$0=\kappa'(0+)=-\lambda/\beta^3$, which is false.

\emph{Via $\int_0^tq=tq(t)/2$.} Then $\int_0^tr\bigl(\kappa(r)-\kappa(t)\bigr)\dd r=\int_0^tq-tq(t)/2=0$.
By convexity and $\kappa(0+)\le\kappa(t)$, $\kappa(r)\le(1-r/t)\kappa(0+)+(r/t)\kappa(t)\le\kappa(t)$ on
$(0,t)$. So $\kappa\equiv\kappa(t)$, hence $\theta=1/\beta$, and again
$0=\kappa'(0+)=-\lambda/\beta^3$.
\end{proof}

The second route uses the same optimality conditions as the endgame
of~\cite[Proposition~5.1]{Lopatto2026}. The first, which uses the
stability condition at both ends of the gap, is of the type used
in~\cite[Proposition~2.12]{ChatGPT2026} and, for gaps away from the origin,
in~\cite[Proposition~9.2]{Lopatto2026}. Since
\[
r^3\kappa''(r)=\mathcal F(r):=r^2q''(r)-2rq'(r)+2q(r),
\]
it remains to show $\mathcal F\ge0$ on $(0,t)$.

\subsection{Convexity on the gap}\label{sec:accum-convex}

\begin{lemma}\label{lem:uxxx}
For every probability measure $\nu$ on $[0,1]$, every $s\in[0,1]$ and every $x\ge0$,
${\partial_x^3u_\nu(s,x)\le0}$.
\end{lemma}

The statement is~\cite[(7.6)]{Lopatto2026}, where it is derived from the slope-coordinate
invariants of~\cite[\S3]{Lopatto2026}. A maximum-principle proof of the same sign, written for the
(regularized) zero-temperature Parisi PDE, appears in~\cite[(3.32)--(3.33)]{ChatGPT2026} and
in~\cite[Lemma~3.4 and Appendix~B.3]{Chen2026b}; we include the short argument at positive
temperature, in probabilistic form, for completeness.
The proof, by a Feynman--Kac representation of $\partial_x^3u_\nu$ on each interval where
$\alpha_\nu$ is constant and downward induction over the intervals, is in
Appendix~\ref{app:mf-accum}.

For $V=u_\nu(s,\cdot)$ and $\lambda>0$, Lemma~\ref{lem:uxxx} and \textup{(P3)} make $\lambda V'$
concave on $[0,\infty)$, vanishing at $0$ and nondecreasing; this is the only structural input of
Lemma~\ref{lem:accumulator}, on which the proof of the next proposition rests. That proof follows
the chain described at the beginning of this section (Appendix~\ref{app:mf-accum}). The polynomial
Lyapunov function used there has three properties (V1)--(V3); (V2) and (V3) are elementary, and
(V1) is the computer-assisted step: an inequality for an explicit polynomial on a rectangle,
verified in exact rational arithmetic (Appendix~\ref{sec:accum-cert}).

\begin{proposition}[Convexity on the gap]\label{prop:accum-convex}
Suppose that $0$ is not an accumulation point of $\operatorname{supp}\mu$, and use the notation of
Lemma~\ref{lem:gap}, with $\mathcal F$ as defined after Lemma~\ref{lem:endgame}. Then
$\mathcal F(r)\ge0$ for every $r\in(0,t)$; equivalently, the function
$\kappa$ of Lemma~\ref{lem:endgame} is convex on $(0,t)$.
\end{proposition}

\subsection{Proof of Theorem~\ref{thm:accum}}

Suppose that $0$ is not an accumulation point of $\operatorname{supp}\mu$, and use the notation of
Lemma~\ref{lem:gap}. By Proposition~\ref{prop:accum-convex}, $\kappa''=\mathcal F/r^3\ge0$ on
$(0,t)$, so $\kappa$ is convex, contradicting Lemma~\ref{lem:endgame}.\qed

\section{Reentrance in \texorpdfstring{$p$}{p}-spin glasses: a reduction to finite predicates}
\label{sec:pspin}

This section sets up the model and the free-energy bounds on which both $p$-spin results rest.
Section~\ref{sec:nearM} uses them to prove reentrance near the triple point for every integer
$p\ge3$ (Theorem~\ref{thm:nm-main}), with constants that are not explicit. Here we also reduce
reentrance at a given bias and two given temperatures to three finite predicates
(Theorem~\ref{thm:reduction}). Section~\ref{sec:certificates} and Appendix~\ref{app:cert34}
certify them for $p=3$ and $p=4$
at biases just larger than the triple-point bias and at cold temperatures far below the triple
point. This gives explicit
quantitative instances of the same two-point phenomenon (Theorems~\ref{thm:p3}
and~\ref{thm:p4}).

\subsection{Model}

Fix $p\ge3$ and $j_0>0$. For each $p$-subset $I=\{i_1<\dots<i_p\}$ of $\{1,\dots,N\}$ let
$\sigma_I=\sigma_{i_1}\cdots\sigma_{i_p}$. Let the $J_I$ be independent Gaussians with
\[
\E J_I=\frac{j_0\,p!}{N^{p-1}},\qquad \operatorname{Var}J_I=\frac{p!}{2N^{p-1}},
\]
which is the convention of~\cite{NishimoriWong1999,Nishimori2026} with $J=1$;~\cite{GNS2001} use an
equivalent deterministic bias with $J_0=pj_0$. The
Gibbs measure is proportional to $e^{\beta H_N}$ with $H_N(\sigma)=\sum_IJ_I\sigma_I$; write
$Z_N(\beta,j_0)$ for its partition function. The density of $J_I$ is an even function times
$e^{2j_0J_I}$, so the Nishimori line is $\beta=2j_0$.

Split $\beta H_N=\beta H^0_N+\beta j_0\,\frac{p!}{N^{p-1}}\sum_I\sigma_I$, where $H^0_N$ is
centred. Newton's identities express $\sum_I\sigma_I$ through power sums. They give
\begin{equation}\label{eq:bias}
\frac{p!}{N^{p-1}}\sum_I\sigma_I=Nm(\sigma)^p+r_{N,p}(m(\sigma)),\qquad
\sup_{N,\,|m|\le1}|r_{N,p}(m)|\le C_p .
\end{equation}
For example, $r_{N,3}(m)=-(3-2/N)m$ and $r_{N,4}(m)=-(6-8/N)m^2+3/N-6/N^2$. The covariance
of $\beta H^0_N$ is $N\xi(R)+O(1)$ with $\xi(q)=\beta^2q^p/2$. Let
$F^0(\beta,h)$ be the free energy of $\beta H^0_N$ with field term $hNm$, as in
Section~\ref{sec:prelim}. It is even in $h$: $\sigma\mapsto-\sigma$ maps $H^0_N$ to
$(-1)^pH^0_N$, which has the same law.

\subsection{Magnetization bins}

For a set $\mathcal M\subset[-1,1]$ write $Z_{N,\mathcal M}=\sum_{\sigma:m(\sigma)\in\mathcal M}e^{\beta H_N(\sigma)}$.

\begin{lemma}[Bins]\label{lem:bins}
Let $0\le a<b\le1$, $h\ge0$ and let $\mathcal M$ be $[a,b]$ or $[-b,-a]$. Then
\[
\limsup_{N}\frac1N\E\log Z_{N,\mathcal M}\le\beta j_0\max_{m\in\mathcal M}m^p+F^0(\beta,h)-ha .
\]
Suppose $\mathcal M\subset[-1,1]$ is covered by finitely many such bins $\mathcal M_k$ with bounds $B_k$, and
$\max_kB_k<\liminf_N N^{-1}\E\log Z_N$. Then $\E\gibbs{\ind\{m\in\mathcal M\}}\to0$.
\end{lemma}

The proof tilts by the field and uses Gaussian concentration~\cite[eq.~(1.22)]{Ledoux1999}
(Appendix~\ref{app:mf-pspin}).

\begin{lemma}[Quadratic cost of magnetization]\label{lem:universal}
For all $h$, $F^0(\beta,h)\le F^0(\beta,0)+h^2/2$. Hence for $0\le a<b\le1$ and
$\mathcal M\in\{[a,b],[-b,-a]\}$,
\[
\limsup_N\frac1N\E\log Z_{N,\mathcal M}\le F^0(\beta,0)+\beta j_0\max_{m\in\mathcal M}m^p-\frac{a^2}2 .
\]
\end{lemma}

\begin{proof}
Let $\mu_0$ be the zero-field Parisi minimizer. By \textup{(P1)},
$F^0(\beta,h)\le\Par(\mu_0,h)$ and $\Par(\mu_0,0)=F^0(\beta,0)$, so~\eqref{eq:quadratic} gives
the first claim. Then use Lemma~\ref{lem:bins} with $h=a$ (for $a=0$, with $h=0$).
\end{proof}

\begin{lemma}[A lower bound on the free energy]\label{lem:lower}
For every $\beta$, $\liminf_N N^{-1}\E\log Z_N(\beta,j_0)\ge F^0(\beta,0)$.
\end{lemma}

\begin{proof}
By~\eqref{eq:bias} and Jensen's inequality for the Gibbs measure of $\beta H^0_N$,
\[
\frac1N\E\log Z_N\ge\frac1N\E\log Z^0_N+\beta j_0\E\gibbs{m^p}_0-\frac{\beta j_0C_p}N .
\]
For even $p$, $\E\gibbs{m^p}_0\ge0$. For odd $p$, the map
$(\sigma,H^0_N)\mapsto(-\sigma,-H^0_N)$ preserves the joint law and reverses the sign of
$m^p$, so $\E\gibbs{m^p}_0=0$.
\end{proof}

\subsection{The Nishimori point}\label{sec:nishimori-point}

In this subsection $\beta=2j_0$, and $\xi(q)=\beta^2q^p/2=2j_0^2q^p$ is the covariance
function of $\beta H^0_N$. Then $\beta j_0=\xi(1)$, so by~\eqref{eq:bias} the bias part of
$\beta H_N$ equals $N\xi(m)+\xi(1)r_{N,p}(m)$. Put $\theta(q)=q\xi'(q)-\xi(q)=(p-1)\xi(q)$ and,
with $g\sim N(0,1)$,
\[
\psi(r)=\E\lc\bigl(r+\sqrt r\,g\bigr)\ \ (r\ge0),\qquad
\Phi(q)=\psi(\xi'(q))-\frac{\xi'(q)+\theta(q)}2\ \ (0\le q\le1),
\]
\[
\mathcal I(m)=\frac{1+m}2\log(1+m)+\frac{1-m}2\log(1-m)\ \ (|m|\le1).
\]
Thus $\Phi(0)=0$, and $\mathcal I$ is the rate function of the magnetization of a uniform
configuration. Let $\Delta_q(x)=\xi(x)-\xi(q)-\xi'(q)(x-q)$ and let $Q_p$ be the set of $q\in[0,1]$
with $\Delta_q\ge0$ on $[-1,1]$, i.e.\ at which the tangent line of $\xi$ stays below $\xi$ on
$[-1,1]$. For even $p$, $\xi$ is convex and $Q_p=[0,1]$. For odd $p$, $\Delta_q$ is convex on
$[0,1]$ with $\Delta_q(q)=\Delta_q'(q)=0$, and concave on $[-1,0]$ with $\Delta_q(0)=\theta(q)\ge0$. Hence
$q\in Q_p$ if and only if $\Delta_q(-1)=\xi(1)(pq^{p-1}+(p-1)q^p-1)\ge0$. For $p=3$ one has
$\Delta_q(x)=\xi(1)(x-q)^2(x+2q)$ and $Q_3=[1/2,1]$.

The next lemma is a Guerra-type interpolation on the Nishimori line. Interpolation bounds of
this kind are standard for the equivalent planted (spiked-tensor)
model~\cite{KoradaMacris2009}, and the limit of the free energy is identified with the
maximum of the replica-symmetric potential $\log2+j_0^2+\Phi$
in~\cite[Theorem~1]{LMLKZ2017} and~\cite[Theorem~2]{BarbierMacris2019}. We need only the
lower bound, at finite $N$. For $q\in Q_p$ it requires no perturbation of the Hamiltonian.

\begin{lemma}[Replica-symmetric lower bound on the Nishimori line]\label{lem:gn}
Let $\beta=2j_0$. For every $N\ge p$ and every $q\in Q_p$,
\[
\frac1N\E\log Z_N(2j_0,j_0)\ge\log2+j_0^2+\Phi(q)-\frac{j_0^2}N\Bigl(\frac{p(p-1)}2+C_p\Bigr).
\]
\end{lemma}

Put $a=\xi(1)p!/N^{p-1}$, so that $\beta J_I$ has mean $a$ and variance $a$. The proof
(Appendix~\ref{app:mf-pspin}) interpolates between the model and a decoupled one with fields
$\rho=\xi'(q)$. Along the interpolation the Nishimori identity
$\E\gibbs{\sigma_A}_t^2=\E\gibbs{\sigma_A}_t$ holds by Bayes' formula, and the derivative of the
interpolating free energy is at least
$\frac12(\xi(1)-\xi'(q)-\theta(q))+\frac12\E\gibbs{\Delta_q(m)}_t$ up to $O(1/N)$; the second term
is nonnegative for $q\in Q_p$.\footnote{At $q=0$ (allowed for even $p$) the lemma reduces, up to $O(1/N)$, to the annealed paramagnetic
value $\log2+\beta^2/4=\log2+j_0^2$, which also follows from gauge averaging.}

Configurations of small magnetization are controlled by a first
moment.

\begin{lemma}[First moment]\label{lem:annealed}
Let $\beta=2j_0$. For every $N\ge p$ and every set $\mathcal M\subset[-1,1]$,
\[
\E Z_{N,\mathcal M}(2j_0,j_0)\le(N+1)\exp\Bigl(N\Bigl[\log2+j_0^2+\sup_{m\in\mathcal M}\bigl(\xi(m)-\mathcal I(m)\bigr)\Bigr]+\xi(1)C_p\Bigr).
\]
\end{lemma}

The proof counts configurations by magnetization (Appendix~\ref{app:mf-pspin}).

\begin{lemma}[Endpoint reduction]\label{lem:endpoint}
Let $\beta=2j_0$ and $0<m_1\le(1-2/p)^{1/2}$. Then
\[
\max_{0\le m\le m_1}\bigl(\xi(m)-\mathcal I(m)\bigr)=\max\bigl\{0,\ \xi(m_1)-\mathcal I(m_1)\bigr\}.
\]
Moreover $\xi-\mathcal I\le0$ on $[-1,0]$ if $p$ is odd, and $\xi-\mathcal I$ is even if $p$ is
even.
\end{lemma}

The proof is in Appendix~\ref{app:mf-pspin}.

Predicate (W) below requires $\Phi(q_0)>0$, so the replica-symmetric potential must exceed its
value at $q=0$ somewhere in $Q_p$. By Remark~\ref{rem:nm-chen}, $\max_{[0,1]}\Phi>0$ exactly
when $j_0$ exceeds the triple-point bias $j_{0M}=\beta_1/2$ of Section~\ref{sec:nearM}, which is
the information-theoretic threshold of the planted model.\footnote{Here $j_{0M}\approx0.767595$ for $p=3$
and $j_{0M}\approx0.810526$ for $p=4$, in agreement with~\cite[Table~I]{Nishimori2026}.}

\subsection{A convexity lower bound on the zero-field free energy}

Upper bounds on $F^0$ come from any trial measure. A lower bound needs control of the
minimizer. We use convexity instead of structural information about the Parisi measure.

\begin{proposition}[Frank--Wolfe bound]\label{prop:fw}
For every probability measure $\mu^*$ on $[0,1]$,
\[
F^0(\beta,0)\ge\Par(\mu^*,0)+\min_{u\in[0,1]}D_{\mu^*}(u),
\qquad
D_{\mu^*}(u)=\frac{\dd}{\dd\varepsilon}\Par\bigl((1-\varepsilon)\mu^*+\varepsilon\delta_u,0\bigr)\Big|_{\varepsilon=0^+}.
\]
\end{proposition}

\begin{proof}
Let $\mu_0$ be the zero-field minimizer and $\mu_\varepsilon=(1-\varepsilon)\mu^*+\varepsilon\mu_0$.
By convexity \textup{(P2)},
$\Par(\mu_0)-\Par(\mu^*)\ge\varepsilon^{-1}(\Par(\mu_\varepsilon)-\Par(\mu^*))$. By the first
variation \textup{(P5)}, the right side tends to $\int G_{\mu^*}\,\dd(\mu_0-\mu^*)$ as
$\varepsilon\to0$. The same formula with $\mu_0$ replaced by $\delta_u$ gives
$D_{\mu^*}(u)=G_{\mu^*}(u)-\int G_{\mu^*}\,\dd\mu^*$. Hence
$\int G_{\mu^*}\,\dd(\mu_0-\mu^*)=\int D_{\mu^*}\,\dd\mu_0\ge\min_uD_{\mu^*}(u)$.
\end{proof}

For a two-atom trial $\mu^*=x\delta_0+(1-x)\delta_q$ with $0<x\le1$ and $0<q<1$, the
function $D_{\mu^*}$ has closed forms in terms of one- and two-dimensional Gaussian integrals.
Put $\theta(s)=s\xi'(s)-\xi(s)$, let $Y_s\sim N(0,\xi'(s))$, and set
\[
A=\E\cosh^x Y_q,\qquad K=\frac{\E[\cosh^xY_q\,\lc Y_q]}{A},\qquad
C=-\frac{\log A}{x^2}+\frac Kx-\frac{\theta(q)}2 .
\]
Here $C=\partial_x\Par(x\delta_0+(1-x)\delta_q)$, and
$\Par(\mu^*,0)=\log2+\frac1x\log A+\frac12[\xi(1)-\xi'(q)+(1-x)\theta(q)]$.

\begin{lemma}[Directional derivatives]\label{lem:D}\leavevmode
\begin{enumerate}
\item[\textup{(i)}] For $0\le u\le q$, let
\[
B(y)=\E\cosh^x\bigl(y+\sqrt{\xi'(q)-\xi'(u)}\,g\bigr).
\]
Then
\[
D(u)=\frac{\log A}x+\frac{1-x}xK-\frac{\E[B(Y_u)\log B(Y_u)]}{x^2A}+\frac{\theta(u)-(1-x)\theta(q)}2 .
\]
\item[\textup{(ii)}] For $q\le u\le1$, let $s^2=\xi'(u)-\xi'(q)$ and $Y_u=Y_q+s\,g$ with $g$
independent. Then
\[
D(u)=\frac{\log A}x+\frac{s^2}2-\frac{\E[\cosh^{x-1}Y_q\,\cosh Y_u\,\lc Y_u]}{Ae^{s^2/2}}
+\frac{x\theta(q)+\theta(u)-\theta(q)}2 .
\]
\item[\textup{(iii)}] $D$ is continuous, $D(0)=(1-x)C$ and $D(q)=-xC$. Moreover
$D'(u)=-\frac12\xi''(u)(\Gamma(u)-u)$, where $\Gamma=\Gamma_{\mu^*}$ is given by
\[
\Gamma(u)=\frac{\E[B_1(Y_u)^2/B(Y_u)]}{A}\ \ (u\le q),\qquad
\Gamma(u)=\frac{\E[\cosh^{x-1}Y_q\,\cosh Y_u\tanh^2Y_u]}{Ae^{s^2/2}}\ \ (u\ge q),
\]
with $B_1(y)=\E[\cosh^x\tanh](y+\sqrt{\xi'(q)-\xi'(u)}\,g)$.
\item[\textup{(iv)}] $\Gamma$ is nondecreasing, $\Gamma(0)=0$ and $0\le\Gamma'\le\xi''$. In
particular $\Gamma\le\xi'$, so $D$ is nondecreasing on $[0,u_0]$, where $\xi'(u_0)=u_0$,
i.e.\ $u_0=(2/(p\beta^2))^{1/(p-2)}$.
\end{enumerate}
\end{lemma}

The derivation is in Appendix~\ref{app:D}. Items~(iii) and~(iv) turn
Proposition~\ref{prop:fw} into a finite computation (Appendix~\ref{app:cert34}).

\subsection{The reduction}

\begin{theorem}\label{thm:reduction}
Let $p\ge3$, $j_0>0$, $\beta_{\rm w}=2j_0$, $\beta_{\rm cold}>0$ and $L_F\in\R$. Consider the
following predicates.
\begin{enumerate}
\item[\textup{(W)}] There are $q_0\in Q_p$ and $m_1\in(0,(1-2/p)^{1/2}]$ such that
$\Phi(q_0)>\max\{0,\ \xi(m_1)-\mathcal I(m_1)\}$, computed with $\xi(q)=\beta_{\rm w}^2q^p/2$.
\item[\textup{(F)}] $L_F\le F^0(\beta_{\rm cold},0)$.
\item[\textup{(L)}] There is $m_{\rm lo}\in(0,1]$ such that
$g(m):=\beta_{\rm cold}j_0m^p-m^2/2<0$ on $(0,m_{\rm lo}]$. There are also finitely many intervals
$[a_i,b_i]\subset[0,1]$ covering $[m_{\rm lo},1]$, with trials $\tau_i$ and fields $h_i\ge0$ such that
$U_i(m):=\beta_{\rm cold}j_0m^p+\Par(\tau_i,h_i)-h_im<L_F$ at $m\in\{a_i,b_i\}$. Here $\Par$ uses
$\xi(q)=\beta_{\rm cold}^2q^p/2$.
\end{enumerate}
Let $\mathcal M_1=[-1,m_1]$ if $p$ is odd and $\mathcal M_1=[-m_1,m_1]$ if $p$ is even. Then:
\begin{enumerate}
\item[\textup{(i)}] if \textup{(W)} holds, then at $(\beta_{\rm w},j_0)$ there is $\eta_{\rm w}>0$ with
$\E\gibbs{\ind\{m\in\mathcal M_1\}}\le e^{-\eta_{\rm w}N}$ for all large $N$; in particular
$\E\gibbs{\ind\{m>m_1\}}\to1$ if $p$ is odd, and $\E\gibbs{\ind\{|m|>m_1\}}\to1$ if $p$ is even;
\item[\textup{(ii)}] if \textup{(F)} and \textup{(L)} hold, then at $(\beta_{\rm cold},j_0)$,
$\E\gibbs{|m|}\to0$.
\end{enumerate}
\end{theorem}

\begin{proof}
(i) Let $s=\sup_{m\in\mathcal M_1}(\xi(m)-\mathcal I(m))$. By Lemma~\ref{lem:endpoint},
$s=\max\{0,\xi(m_1)-\mathcal I(m_1)\}$, so $\eta:=\Phi(q_0)-s>0$ by (W). Let
$X_N=N^{-1}\log Z_N(\beta_{\rm w},j_0)$. By Lemma~\ref{lem:gn} with $q=q_0$,
$\E X_N\ge\log2+j_0^2+\Phi(q_0)-C/N$ for a constant $C$. By Lemma~\ref{lem:annealed} and
Markov's inequality, the event
$A_N=\{Z_{N,\mathcal M_1}\ge\exp(N(\log2+j_0^2+s+\eta/3))\}$ has probability at most
$(N+1)e^{\xi(1)C_p-N\eta/3}$. With $a=\xi(1)p!/N^{p-1}$ as above, the gradient of
$X_N$ in the standardized couplings has squared norm $N^{-2}a\sum_I\gibbs{\sigma_I}^2\le\xi(1)/N$,
so $X_N$ is $(\xi(1)/N)^{1/2}$-Lipschitz in them. Gaussian concentration for Lipschitz
functions~\cite[eq.~(1.22)]{Ledoux1999}, applied to $-X_N$, therefore bounds the probability of
$B_N=\{X_N\le\E X_N-\eta/3\}$ by $e^{-N\eta^2/(18\xi(1))}$. Outside $A_N\cup B_N$,
\[
\gibbs{\ind\{m\in\mathcal M_1\}}=\frac{Z_{N,\mathcal M_1}}{Z_N}\le\exp\Bigl(N\bigl(s+\tfrac\eta3\bigr)-N\bigl(\Phi(q_0)-\tfrac\eta3\bigr)+C\Bigr)=e^{C-N\eta/3} .
\]
Since $\gibbs{\ind\{m\in\mathcal M_1\}}\le1$ always,
$\E\gibbs{\ind\{m\in\mathcal M_1\}}\le e^{C-N\eta/3}+(N+1)e^{\xi(1)C_p-N\eta/3}+e^{-N\eta^2/(18\xi(1))}$.

(ii) By Lemma~\ref{lem:lower} and (F), $\liminf_NN^{-1}\E\log Z_N(\beta_{\rm cold},j_0)\ge L_F$;
by (L), Lemma~\ref{lem:universal} and Lemma~\ref{lem:bins}, for every $\varepsilon>0$ every fine bin
of magnetizations with $|m|\ge\varepsilon$ lies below $L_F$, and Lemma~\ref{lem:bins} gives
$\E\gibbs{\ind\{|m|\ge\varepsilon\}}\to0$ (Appendix~\ref{app:mf-pspin}).
\end{proof}

Part~(i) uses Lemmas~\ref{lem:gn}, \ref{lem:annealed} and~\ref{lem:endpoint} and Gaussian
concentration only. It does not use the Parisi formula, Lemma~\ref{lem:universal}, or any
information on Parisi measures.

\section{Reentrance near the triple point for every \texorpdfstring{$p\ge3$}{p>=3}}
\label{sec:nearM}

Section~\ref{sec:pspin} reduces reentrance at one bias and two temperatures to finitely many
numerical predicates, and Section~\ref{sec:certificates} and
Appendix~\ref{app:cert34} verify them for $p=3$ and $p=4$.
Nishimori's replica computation predicts more~\cite{Nishimori2026}. For every $p>2$ the
ferromagnet--spin-glass boundary leaves the triple point $M$ tangent to the vertical and bends
toward larger bias. Its curvature is proportional to $C_{\rm SG}-C_{\rm FM}$, the difference
of the specific heats of the one-step spin-glass phase and the replica-symmetric ferromagnet at
$M$. Nishimori checked the sign of this difference numerically for integer $3\le p\le12$ and in
a scan over real $p$ from near $2$ to about $12$, and asymptotically as $p\to\infty$ and
$p\to2^+$, and he notes that a proof uniform in $p$ remains
open~\cite[p.~19]{Nishimori2026}. In this section we prove the corresponding two-point statement
for every integer $p\ge3$: at each bias slightly larger than that of the triple point, the model is
magnetized at its Nishimori temperature and unmagnetized at an explicit lower temperature
(Theorem~\ref{thm:nm-main}).

Two obstructions must be overcome. Near $M$ the ferromagnetic and spin-glass free energies
agree to first order, so the comparison is decided by second-order coefficients $A_{\rm FM}$
and $A_{\rm SG}$. On the spin-glass side we therefore need a \emph{lower} bound on the Parisi
free energy that is exact to second order just below the temperature at which replica symmetry
breaks. Trial measures give only upper bounds, and convexity alone loses the second-order term
(Section~\ref{sec:nm-route}). We show instead that every stationary one-step measure near the
triple point is the Parisi measure (Lemma~\ref{lem:nm-LS}); this gives an analytic one-step
branch and the exact expansion (Proposition~\ref{prop:nm-SG}). The second obstruction is that
$\kappa_p=\beta_1^2(A_{\rm FM}-A_{\rm SG})$ is of order $2^{-p}p^{-1/2}$, while $A_{\rm FM}$
and $A_{\rm SG}$ are each exponentially close to $\frac12$, so the sign cannot be read off
numerically for all $p$. An exact identity writes the leading part of $\kappa_p$ as an integral with a
nonnegative kernel (Proposition~\ref{prop:nm-J}). Together with interval arithmetic for
$p\le25$, it gives $\kappa_p>0$ for every integer $p\ge3$.

\subsection{The triple point}\label{sec:nm-triple}

Throughout this section $p\ge3$ is an integer. For $\beta>0$ put $\xi_\beta(q)=\beta^2q^p/2$
and $\theta_\beta(q)=q\xi_\beta'(q)-\xi_\beta(q)=(p-1)\xi_\beta(q)$, and, with $\psi$ as in
Section~\ref{sec:nishimori-point} and $g\sim N(0,1)$,
\begin{equation}\label{eq:nm-Phi}
\Phi_\beta(q)=\psi\bigl(\xi_\beta'(q)\bigr)-\frac{\xi_\beta'(q)+\theta_\beta(q)}2,\qquad
\Gamma_\beta(q)=\E\tanh\bigl(\xi_\beta'(q)+\sqrt{\xi_\beta'(q)}\,g\bigr)\qquad(0\le q\le1).
\end{equation}
For $\beta=2j_0$, $\Phi_\beta$ is the replica-symmetric potential $\Phi$ of
Lemma~\ref{lem:gn}. The following Gaussian identities are used throughout.

\begin{lemma}[Nishimori-line identities]\label{lem:nm-NLid}
Let $r>0$ and $h\sim N(r,r)$.
\begin{enumerate}
\item[\textup{(a)}] For $f\in C^2$ with $f,f',f''$ of at most exponential growth,
$\frac{\dd}{\dd r}\E f(h)=\E[f'(h)+\frac12f''(h)]$.
\item[\textup{(b)}] If $F$ is even and $\E|F(h)|<\infty$, then
$\E[F(h)\tanh h]=\E[F(h)\tanh^2h]$. In particular $\E\tanh^{2k-1}h=\E\tanh^{2k}h$ for
$k\ge1$.
\item[\textup{(c)}] For even $F$, $\E[F(\sqrt r\,g)\cosh(\sqrt r\,g)]\,e^{-r/2}=\E F(h)$.
\end{enumerate}
Consequently $\mu(r):=\E\tanh h$ satisfies $\mu(r)=\E\tanh^2h$,
$\mu'(r)=\E\sech^4h\in(0,1)$, $\mu(0^+)=0$ and $\mu(r)<\min(r,1)$, and
$\psi'(r)=(1+\mu(r))/2$.
\end{lemma}

The proof is in Appendix~\ref{app:nm-proofs}.

In the notation of Lemma~\ref{lem:nm-NLid}, $\Gamma_\beta=\mu\circ\xi_\beta'$, and $\Gamma_\beta=\Gamma_{\delta_0}$ is the function
$\Gamma$ of Lemma~\ref{lem:D} at $x=1$. Since $\theta_\beta'(q)=q\xi_\beta''(q)$,
\begin{equation}\label{eq:nm-Phiprime}
\Phi_\beta(0)=0,\qquad \Phi_\beta'(q)=\frac{\xi_\beta''(q)}2\bigl(\Gamma_\beta(q)-q\bigr).
\end{equation}
The triple point is where $\Phi_\beta$ touches zero at a positive overlap:
\begin{equation}\label{eq:nm-system}
\Phi_\beta(q)=0,\qquad \Phi_\beta'(q)=0,\qquad 0<q<1 .
\end{equation}
This is Zhou's tangency system for the end of the replica-symmetric phase of the pure
$p$-spin model~\cite[(2)]{Zhou2024}, written in our normalization (Zhou's inverse temperature is
$\beta/\sqrt2$).

\begin{proposition}[The triple point]\label{prop:nm-triple}
For every integer $p\ge3$, the system~\eqref{eq:nm-system} has exactly one solution
$(\beta_1,q_1)$ with $\beta_1>0$. At this solution:
\begin{enumerate}
\item[\textup{(M1)}] $\Phi_{\beta_1}(q)<0$ for every $q\in(0,1]\setminus\{q_1\}$;
\item[\textup{(M2)}] $\lambda:=1-\xi_{\beta_1}''(q_1)\,\E\sech^4h_c>0$, where
$\Lambda_c:=\xi_{\beta_1}'(q_1)$ and $h_c\sim N(\Lambda_c,\Lambda_c)$.
\end{enumerate}
Moreover $q_1=\mu(\Lambda_c)$, and $\Lambda_c$ is the unique zero on $(0,\infty)$ of
\[
D_p(\Lambda)=\psi(\Lambda)-\frac\Lambda2-\frac{p-1}{2p}\,\Lambda\,\mu(\Lambda),
\]
with $D_p<0$ on $(0,\Lambda_c)$, $D_p>0$ on $(\Lambda_c,\infty)$ and $\Lambda_c<2p\log2$.
\end{proposition}

The proof (Appendix~\ref{app:nm-U}) rests on one property of the Nishimori-line magnetization
$\mu$. It is independent of $p$.

\begin{lemma}[Single crossing]\label{lem:nm-U}
The elasticity $e(s):=s\mu'(s)/\mu(s)$ is a continuous, strictly decreasing bijection of
$(0,\infty)$ onto $(0,1)$.
\end{lemma}

Solutions of~\eqref{eq:nm-system} correspond to zeros of $D_p$ through $\Lambda=\xi_\beta'(q)$
and $q=\mu(\Lambda)$, and $D_p'=\frac{\mu}{2p}\bigl(1-(p-1)e\bigr)$. By
Lemma~\ref{lem:nm-U}, $D_p$ first decreases and then increases, which gives existence and
uniqueness. The same lemma makes $s\mapsto\log\mu(s)-(\log s)/(p-1)$ unimodal, which gives the
sign pattern of $\Gamma_{\beta_1}-{\rm id}$ behind (M1) and (M2).

\begin{remark}[Relation to earlier work]\label{rem:nm-chen}
W.-K.~Chen showed that the set of $\beta$ at which the pure $p$-spin free energy equals its
annealed value is an interval $(0,\beta_p]$, and characterized it by a first-variation
criterion~\cite[Theorem~2]{Chen2019}. His auxiliary function $\rho_\beta$ equals $\Gamma_\beta$
by Lemma~\ref{lem:nm-NLid}(c), and his criterion reads: $\beta\le\beta_p$ if and only if
$\Phi_\beta\le0$ on $(0,1]$, and $\Phi_\beta<0$ on $(0,1]$ if $\beta<\beta_p$. With (M1) it
follows that $\beta_1=\beta_p$, since $\Phi_{\beta_1}\le0$ and $\Phi_{\beta_1}(q_1)=0$. So
$\beta_1$ is the inverse temperature at which replica symmetry breaks in the pure $p$-spin model, and,
by~\cite[Theorem~1]{Chen2019}, the distinguishability and weak-recovery threshold of the spiked
Rademacher tensor, which after a gauge transformation is our model on the Nishimori line, up to
the passage from ordered $p$-tuples to $p$-subsets (not used here); see
also~\cite{LMLKZ2017,BarbierMacris2019}. In particular $\max_{[0,1]}\Phi_\beta>0$ exactly when
$\beta>\beta_1$. Zhou characterized this inverse temperature by~\eqref{eq:nm-system} and stated that its
solution exists and is unique~\cite[\S2.2]{Zhou2024}; Proposition~\ref{prop:nm-triple} gives a
proof. For integer $p$, the unimodality used in the proof also follows from Zhou's convexity
lemma~\cite[Lemma~6]{Zhou2024} after the change of variables $s=\xi_\beta'(q)$. The final
numerical step of that lemma needs a finer bound than the printed one, whose value is
$-0.043$; interval arithmetic on a finer cover gives a positive margin of $0.04$. We do not use
it. At replica level, Gillin, Nishimori and Sherrington found that the Nishimori line passes
through the paramagnet--spin-glass--ferromagnet triple point~\cite[App.~C.2]{GNS2001}; the
corresponding identity here is Lemma~\ref{lem:nm-FMloc}(a): at $M$ the replica-symmetric
ferromagnetic critical value equals the paramagnetic free energy.
\end{remark}

Let $(\beta_1,q_1)$ be the solution of~\eqref{eq:nm-system}. On the Nishimori line $\beta=2j_0$
the triple point is
\[
M=(j_{0M},T_M)=\bigl(\beta_1/2,\,1/\beta_1\bigr),\qquad \mu:=q_1 ,
\]
where $\mu$ is the common value at $M$ of the spin-glass overlap and the ferromagnetic
magnetization. By Proposition~\ref{prop:nm-triple}, $q_1=\mu(\Lambda_c)$ is the value at
$\Lambda_c$ of the function $\mu$ of Lemma~\ref{lem:nm-NLid}; from here on $\mu$ without an
argument denotes this number. With $\Lambda_c$ and $h_c$ as in Proposition~\ref{prop:nm-triple}, so that
$\mu=\E\tanh h_c=\E\tanh^2h_c$, define the constants and the coefficients
\begin{equation}\label{eq:nm-const}
\begin{gathered}
\tau=\E\tanh^3h_c,\qquad V=\operatorname{Var}\bigl(\lc h_c\bigr),\qquad B=1-3\mu+2\tau,\\
\Omega=\xi_{\beta_1}''(\mu)=\frac{(p-1)\Lambda_c}\mu,\qquad
\varphi_{xx}=V-\frac{(p-1)\Lambda_c\mu}p ,
\end{gathered}
\end{equation}
\begin{equation}\label{eq:nm-A}
A_{\rm SG}=\frac{(\beta_1\mu^p/2)^2}{\varphi_{xx}},\qquad
A_{\rm FM}=\frac{\mu^p}2+\frac{\Lambda_c^2}{\beta_1^2}\cdot\frac{B}{1-\Omega B},\qquad
\kappa_p=\beta_1^2\,(A_{\rm FM}-A_{\rm SG}).
\end{equation}
By Lemma~\ref{lem:nm-NLid}(b), $\E\sech^4h_c=1-2\mu+\tau$, so $\lambda=1-\Omega(1-2\mu+\tau)$.
We show $\varphi_{xx}>0$ in Lemma~\ref{lem:nm-G} and $0<\Omega B<1$ in
Theorem~\ref{thm:nm-nondeg}, so~\eqref{eq:nm-A} is well defined. These are Nishimori's
second-order coefficients~\cite[(36), (42), (44)--(45)]{Nishimori2026}:
$\frac12-A_{\rm SG}$ and $\frac12-A_{\rm FM}$ are the second $\beta$-derivatives, at fixed
$j_0=j_{0M}$, of the one-step spin-glass and replica-symmetric ferromagnetic free energies at $M$
(Proposition~\ref{prop:nm-SG} and Lemma~\ref{lem:nm-closed}). Hence
$\kappa_p=C_{\rm SG}-C_{\rm FM}$, where $C_{\rm SG}=\beta_1^2(\frac12-A_{\rm SG})$ and
$C_{\rm FM}=\beta_1^2(\frac12-A_{\rm FM})$ are his specific heats of the two phases at $M$.\footnote{For $p=3$, $j_{0M}\approx0.767595$, $T_M\approx0.651385$, $\mu\approx0.813518$ and
$\kappa_3\approx0.0303916$; for $p=4$, $j_{0M}\approx0.810526$, $T_M\approx0.616883$,
$\mu\approx0.948088$ and $\kappa_4\approx0.0131494$ (Table~\ref{tab:nm-cert}).
Nishimori writes $\beta_c$ for our $\beta_1$.}

\subsection{Main result}\label{sec:nm-result}

\begin{theorem}[Reentrance near the triple point]\label{thm:nm-main}
Let $p\ge3$ be an integer, and let $j_{0M}$, $\mu$, $A_{\rm SG}$ and $A_{\rm FM}$ be as in
Section~\ref{sec:nm-triple}. Then $A_{\rm FM}>A_{\rm SG}$. Put
$m_1=\min\{(1-2/p)^{1/2},\,4^{-1/(p-2)}\}$, and let $\mathcal M_1=[-1,m_1]$ if $p$ is odd and
$\mathcal M_1=[-m_1,m_1]$ if $p$ is even. There is $\bar\delta=\bar\delta(p)>0$ such that for every
$j_0\in(j_{0M},j_{0M}+\bar\delta)$, with $\delta=j_0-j_{0M}$ and
\[
b(\delta)=\Bigl(\frac{4\beta_1\mu^p\,\delta}{A_{\rm FM}-A_{\rm SG}}\Bigr)^{1/2},
\]
the following hold.
\begin{enumerate}
\item[\textup{(i)}] \emph{Order at the Nishimori point.} At $T_{\rm w}=1/(2j_0)$ there is
$\eta_{\rm w}=\eta_{\rm w}(\delta)>0$ with $\E\gibbs{\ind\{m\in\mathcal M_1\}}\le e^{-\eta_{\rm w}N}$
for all large $N$. In particular $\E\gibbs{\ind\{m>m_1\}}\to1$ if $p$ is odd, and
$\E\gibbs{\ind\{|m|>m_1\}}\to1$ if $p$ is even.
\item[\textup{(ii)}] \emph{No order at a lower temperature.} At
$T_{\rm cold}=1/(\beta_1+b(\delta))$, $\E\gibbs{|m|}\to0$.
\item[\textup{(iii)}] $T_{\rm cold}<T_{\rm w}$.
\end{enumerate}
\end{theorem}

Parts~(i) and~(ii) concern one disorder law at two temperatures, so cooling from $T_{\rm w}$ to
$T_{\rm cold}$ destroys ferromagnetic order. In inverse temperature the warm point is
$\beta_1+2\delta$ and the cold point is $\beta_1+b(\delta)$ with $b(\delta)\asymp\delta^{1/2}$
(Figure~\ref{fig:nm-local}). The theorem is a two-point statement at each fixed bias: it
constructs no boundary curve, and $\bar\delta$ is neither explicit nor uniform in $p$. It
concerns integer $p$ only. Nishimori formulates the problem for real $p>2$, including the limit
$p\to2^+$; that question is not resolved here.\footnote{Nishimori's boundary is
$j_0^*(\beta)-j_{0M}=K_b(\beta-\beta_1)^2+O(|\beta-\beta_1|^3)$ with
$K_b=(A_{\rm FM}-A_{\rm SG})/(2\beta_1\mu^p)$~\cite[(51)]{Nishimori2026}. At bias
$j_{0M}+\delta$ it is crossed at $\beta_1+b_{\rm N}$ with $b_{\rm N}=(\delta/K_b)^{1/2}$, and
$b(\delta)=\sqrt2\,b_{\rm N}$. At $\beta_1+b(\delta)$ his boundary lies at
$j_{0M}+2\delta+O(\delta^{3/2})>j_0$, so the unmagnetized point lies on the spin-glass side of
the predicted boundary. The factor $\sqrt2$ yields a free-energy margin of order $\delta$ in the
proof (step~(e) in Section~\ref{sec:nm-proof}).}

The inputs that depend on $p$ are $A_{\rm FM}>A_{\rm SG}$ and two further nondegeneracies
(Theorem~\ref{thm:nm-nondeg}). For $3\le p\le25$ they are verified in interval arithmetic, and
for $p\ge26$ they follow from explicit bounds whose constants are also certified in interval
arithmetic. Theorem~\ref{thm:nm-main} is therefore a computer-assisted theorem.

\paragraph{The warm point.} Part~(i) without the explicit threshold $m_1$ is, in substance,
W.-K.~Chen's weak-recovery theorem for the spiked Rademacher tensor~\cite[proofs of Lemma~4,
p.~13, and of Theorem~1(ii), p.~14]{Chen2019}, read through the gauge map, since
$2j_0>\beta_1=\beta_p$ (Remark~\ref{rem:nm-chen}). The positivity of the limiting overlap for
$\beta>\beta_p$ is established in the proof, not the statement, of Theorem~1(ii) there. Chen's
Hamiltonian sums over ordered $p$-tuples, and we do not carry
out the passage to $p$-subsets. Instead we prove~(i) directly, with the explicit threshold $m_1$, from the
Nishimori-line lemmas of Section~\ref{sec:nishimori-point}.

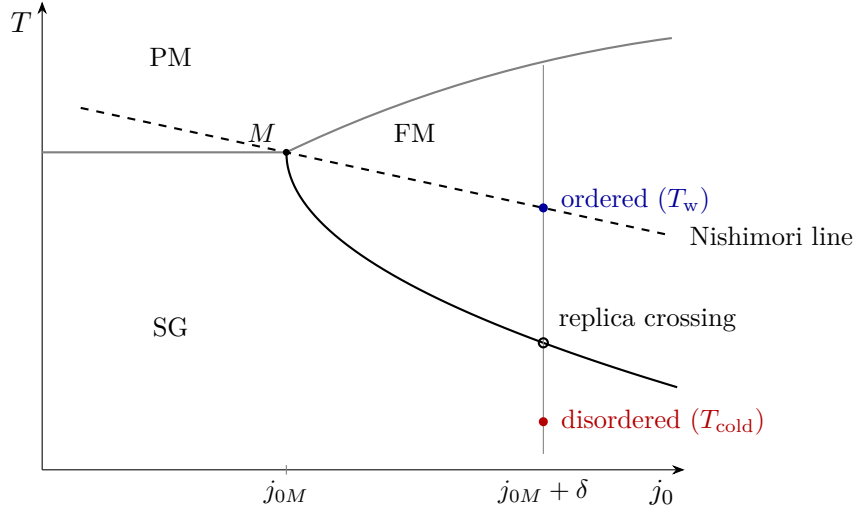
\begin{figure}[t]
\centering
\begin{tikzpicture}[x=3.4cm,y=2.1cm,>=Stealth]
\draw[->] (-0.95,-2.0) -- (1.55,-2.0) node[below left]{$j_0$};
\draw[->] (-0.95,-2.0) -- (-0.95,0.95) node[below left]{$T$};
\draw[gray,thick] (-0.95,0) -- (0,0);
\draw[gray,thick] (0,0) .. controls (0.45,0.35) and (0.95,0.6) .. (1.5,0.72);
\draw[thick,domain=0:1.48,smooth,variable=\y] plot ({0.6944*\y*\y},{-\y});
\draw[thick,dashed] (-0.8,0.28) -- (1.5,-0.525);
\node[anchor=west] at (1.53,-0.525) {\small Nishimori line};
\draw[gray] (1,-1.9) -- (1,0.55);
\node[below] at (1,-2.0) {\small $j_{0M}+\delta$};
\draw[gray] (0,-2.03) -- (0,-1.97);
\node[below] at (0,-2.0) {\small $j_{0M}$};
\fill (0,0) circle (1.3pt) node[above left]{\small $M$};
\fill[blue!60!black] (1,-0.35) circle (1.7pt);
\node[blue!60!black,anchor=west] at (1.03,-0.30) {\small ordered ($T_{\rm w}$)};
\draw[thick] (1,-1.2) circle (1.7pt);
\node[anchor=south west] at (1.02,-1.19) {\small replica crossing};
\fill[red!70!black] (1,-1.697) circle (1.7pt);
\node[red!70!black,anchor=west] at (1.03,-1.697) {\small disordered ($T_{\rm cold}$)};
\node at (-0.45,0.6) {\small PM};
\node at (-0.45,-1.1) {\small SG};
\node at (0.5,0.12) {\small FM};
\end{tikzpicture}
\caption{Schematic neighbourhood of the triple point $M$ in the $(j_0,T)$ plane, not to scale.
Dashed: the Nishimori line. Black curve below $M$: Nishimori's replica prediction for the
ferromagnet--spin-glass boundary, $j_0-j_{0M}=K_b(\beta-\beta_1)^2+O(|\beta-\beta_1|^3)$,
tangent to the vertical at $M$; grey: the other replica phase boundaries. At the bias
$j_{0M}+\delta$, Theorem~\ref{thm:nm-main} proves order at the Nishimori point (blue) and its
absence at inverse temperature $\beta_1+b(\delta)$ (red). The open circle marks the replica
crossing, at inverse temperature $\beta_1+b_{\rm N}$, and $b(\delta)=\sqrt2\,b_{\rm N}$. Only the two filled points are
proved; the phase boundaries drawn are replica predictions.}
\label{fig:nm-local}
\end{figure}

\subsection{Mechanism and route}\label{sec:nm-route}

Fix $j_0=j_{0M}+\delta$ and $\beta=\beta_1+b$ with $b>0$. The proof compares two free energies
at $(\beta,j_0)$:
\begin{align*}
F^0(\beta,0)&=\varphi_{\rm PM}(\beta)-\tfrac{A_{\rm SG}}2\,b^2+O(b^3),\\
\hat\varphi_{\rm FM}(\beta,j_0)&=\varphi_{\rm PM}(\beta)-\tfrac{A_{\rm FM}}2\,b^2
+\beta_1\mu^p\delta+O(|b|^3+|b|\delta+\delta^2),
\end{align*}
where $\varphi_{\rm PM}(\beta)=\log2+\beta^2/4$. The first is the zero-field free energy,
which bounds the full free energy from below (Lemma~\ref{lem:lower}); its expansion is
Proposition~\ref{prop:nm-SG}. The second is the value of the replica-symmetric ferromagnetic
critical point (Lemma~\ref{lem:nm-FMloc}), which bounds the free energy of every magnetization
bin near $\mu$ from above. Both equal $\varphi_{\rm PM}(\beta_1)$ at $M$ and have the same
$\beta$-derivative $\beta_1/2$ there; only the ferromagnet gains from the bias. Their
difference is
\[
\tfrac12(A_{\rm FM}-A_{\rm SG})\,b^2-\beta_1\mu^p\delta+O(|b|^3+|b|\delta+\delta^2),
\]
which is positive and of order $\delta$ at $b=b(\delta)$ once $A_{\rm FM}>A_{\rm SG}$.
Magnetizations away from $\mu$ are handled by Guerra-type bounds whose values at $M$ are
$\varphi_{\rm PM}(\beta_1)+\Phi_{\beta_1}(m)<\varphi_{\rm PM}(\beta_1)$ by (M1)
(Lemma~\ref{lem:nm-diag}), and small magnetizations by Lemma~\ref{lem:universal}.

\paragraph{The obstruction.} The spin-glass side needs a lower bound on $F^0(\beta,0)$ with the
exact coefficient $A_{\rm SG}$. Convexity of $\beta\mapsto F^0(\beta,0)$ and its tangent at
$\beta_1$ give only
$F^0(\beta,0)\ge\varphi_{\rm PM}(\beta)-b^2/4$~\cite[proof of Lemma~3.1]{Talagrand2000}, i.e.\
$A_{\rm SG}$ replaced
by $\frac12$. The comparison would then require $A_{\rm FM}>\frac12$, i.e.\ $C_{\rm FM}<0$.
This fails: $C_{\rm FM}=\beta_1^2(\frac12-A_{\rm FM})\approx0.198$ for $p=3$, and interval
arithmetic gives $C_{\rm FM}>0$ for every $p=3,\dots,25$. The lower bound must therefore resolve
the spin-glass specific heat to within $\kappa_p$. Talagrand studied the present model at
$j_0=0$, with the same normalization; his critical value $\beta_p$ is our $\beta_1$
(Remark~\ref{rem:nm-chen}). His results~\cite[Theorems~1.3--1.5 and~1.8]{Talagrand2000} are statements about the Gibbs measure
for large $p$, made before the Parisi formula was proved, and they give no second-order expansion
of the free energy.\footnote{For $p$ larger than an absolute constant $L$ that is not made explicit
and $\beta\le\beta_1+1/L$, he proved that the overlap of two replicas is asymptotically either
close to $0$ or within $1/p$ of $\pm1$, and that the Gibbs measure splits into lumps, which come in pairs related by the global spin flip,
with vanishing overlaps between lumps of different pairs~\cite[Theorems~1.3--1.5]{Talagrand2000};
under a hypothesis on the weights of the lumps, which a note added in proof reports as proved, the
lumps are pure states for $\beta\le\beta_1+1/(pL)$~\cite[Theorem~1.8 and p.~360]{Talagrand2000}.} Zhou argued in an unrefereed preprint that the
Parisi measure is one-step just above $\beta_1$~\cite[Theorem~1(2) and \S4.2]{Zhou2024}, by an
intermediate-value argument. As printed, that argument needs repairs: one intermediate claim is
false as stated, and uniformity near $0$ and near $q$ is not argued. It also gives no continuity of the one-step parameters in $\beta$ and no expansion of the
free energy; Proposition~\ref{prop:nm-SG} recovers the one-step structure together with the
expansion.

\paragraph{Proof architecture.} The proof has five parts.
\begin{enumerate}
\item \emph{Structure at $M$} (Proposition~\ref{prop:nm-triple}): uniqueness of the triple
point, the strict sign condition (M1) and the nondegeneracy (M2), from Lemma~\ref{lem:nm-U}.
\item \emph{Spin-glass lower bound} (Section~\ref{sec:nm-sg}). By (M1)--(M2), the
Frank--Wolfe certificate of $\delta_0$ at $M$ is strict away from its two forced zeros, and it
persists along stationary one-step measures near $M$ (Lemma~\ref{lem:nm-LS}). The one-step
stationarity equations are analytic across $x=1$ with a nonsingular Jacobian, so the implicit
function theorem gives a one-step branch (Lemma~\ref{lem:nm-branch}); its value is $F^0$, and
an envelope computation gives the coefficient (Proposition~\ref{prop:nm-SG}).
\item \emph{Ferromagnetic upper bounds} (Section~\ref{sec:nm-fm}): a Nishimori-line identity
for bins away from $\mu$ (Lemma~\ref{lem:nm-diag}) and the replica-symmetric critical value for
bins near $\mu$ (Lemmas~\ref{lem:nm-FMloc} and~\ref{lem:nm-closed}).
\item \emph{Assembly} (Section~\ref{sec:nm-proof}), with the constants fixed in an order that
does not depend on the bin width.
\item \emph{Nondegeneracy for every $p$} (Section~\ref{sec:nm-nondeg}): $\kappa_p>0$ and two
further conditions, certified for $p\le25$ and proved by explicit bounds for $p\ge26$.
\end{enumerate}
The ingredients we have not found in the literature are Lemma~\ref{lem:nm-LS}, which adapts the
persistence argument of~\cite[Proposition~2.7]{AronowLopatto2026} to a point where it does not
apply directly (see the discussion after the lemma), Proposition~\ref{prop:nm-SG}, the local
ferromagnetic analysis and the positive-kernel identity of Proposition~\ref{prop:nm-J}. The
Parisi formula, the first-variation criterion and Guerra's bounds enter through
Proposition~\ref{prop:facts}.

\subsection{The spin-glass side}\label{sec:nm-sg}

We use the two-atom notation of Section~\ref{sec:pspin}. For $0<x\le1$, $0<q<1$ and $\beta>0$
write $\Par(x,q;\beta)=\Par(x\delta_0+(1-x)\delta_q,0)$ at inverse temperature $\beta$,
$C(x,q;\beta)=\partial_x\Par$ as in Lemma~\ref{lem:D}, and
\[
T(x,q;\beta)=\frac{\E[\cosh^xY_q\,\tanh^2Y_q]}{A},\qquad Y_q\sim N\bigl(0,\xi_\beta'(q)\bigr),
\]
which is $\Gamma_{\mu^*}(q)$ in Lemma~\ref{lem:D}(iii). Differentiating the formula for
$\Par(\mu^*,0)$ before Lemma~\ref{lem:D}, with $\frac{\dd}{\dd r}\E f(\sqrt r g)=\frac12\E
f''(\sqrt rg)$ and $(\cosh^x)''=x\cosh^x(1-(1-x)\tanh^2)$, gives
\begin{equation}\label{eq:nm-stat}
\partial_q\Par(x,q;\beta)=-(1-x)\,\frac{\xi_\beta''(q)}2\,\bigl(T(x,q;\beta)-q\bigr).
\end{equation}
So for $0<x<1$ the measure $x\delta_0+(1-x)\delta_q$ is stationary for $\Par$ exactly when
\begin{equation}\label{eq:nm-15}
C(x,q;\beta)=0,\qquad T(x,q;\beta)=q .
\end{equation}
At $x=1$ the measure is $\delta_0$ for every $q$, $\Par(1,q;\beta)=\varphi_{\rm PM}(\beta)$, and
Lemma~\ref{lem:nm-NLid}(c) gives $C(1,q;\beta)=\Phi_\beta(q)$ and $T(1,q;\beta)=\Gamma_\beta(q)$.
Put
\[
\varphi_{xx}^{(r)}=\partial_x C(1,q;\beta)\quad\text{at }\xi'_\beta(q)=r,\qquad
\varphi_{x\beta}=\partial_\beta C(1,q_1;\beta_1).
\]
Since $\partial_\beta\Phi_\beta(q)=\beta^{-1}(\xi_\beta'(q)\Gamma_\beta(q)-\theta_\beta(q))$ by
Lemma~\ref{lem:nm-NLid}, and $\Gamma_{\beta_1}(q_1)=q_1$,
$\varphi_{x\beta}=\xi_{\beta_1}(\mu)/\beta_1=\beta_1\mu^p/2>0$.

\begin{lemma}[Curvature in the one-step parameter]\label{lem:nm-G}
At $x=1$, $\partial_xC(1,q;\beta)$ depends on $(q,\beta)$ only through $r=\xi_\beta'(q)$. With
$\ell=\lc$ and $Q$ the law $N(r,r)$,
\[
\varphi_{xx}^{(r)}=\operatorname{Var}_Q\ell-2\E_Q\ell+r,\qquad
\frac{\dd}{\dd r}\varphi_{xx}^{(r)}=\operatorname{Cov}_Q(\ell,\tanh^2)>0 .
\]
Hence $\varphi_{xx}^{(r)}>0$ for $r>0$. At $M$, $\varphi_{xx}^{(\Lambda_c)}=\varphi_{xx}$ of
\eqref{eq:nm-const}.
\end{lemma}

The proof, by Chebyshev's association inequality, is in Appendix~\ref{app:nm-proofs}.

At $(x,q,\beta)=(1,q_1,\beta_1)$ the Frank--Wolfe derivative of $\delta_0$ is
$D_{\delta_0}=-\Phi_{\beta_1}$, by Lemma~\ref{lem:D}(iii) at $x=1$ and~\eqref{eq:nm-Phiprime}. By
(M1) it is strictly positive on $(0,1]\setminus\{q_1\}$. Its zero at $0$ has order $p$, and its
zero at $q_1$ is quadratic and nondegenerate, since
$\Phi_{\beta_1}''(q_1)=-\frac12\xi_{\beta_1}''(q_1)\lambda<0$ by (M2). The next lemma shows that
this strict certificate persists along stationary one-step measures near $M$. Near $0$ we use a
sign argument, near $q$ the derivative, and in the bulk continuity.

\begin{lemma}[Local sufficiency]\label{lem:nm-LS}
There is $\rho>0$ with the following property. If $|\beta-\beta_1|+|x-1|+|q-q_1|<\rho$,
$0<x<1$, and $(x,q)$ solves~\eqref{eq:nm-15} at $\beta$, then
$F^0(\beta,0)=\Par(x,q;\beta)$ and $x\delta_0+(1-x)\delta_q$ is the Parisi measure at $\beta$
and zero field.
\end{lemma}

\begin{proof}
\emph{Setup.} Let $\Gamma=\Gamma_{\mu^*}$ for $\mu^*=x\delta_0+(1-x)\delta_q$ at $\beta$, and put
$\tilde f(u)=\int_q^u\frac12\xi_\beta''(s)(\Gamma(s)-s)\,\dd s$. By Lemma~\ref{lem:D}(iii),
$D_{\mu^*}=-xC-\tilde f$. Under~\eqref{eq:nm-15}, $C=0$ and $\Gamma(q)=T=q$, so
$D_{\mu^*}=-\tilde f$ and $\tilde f(0)=\tilde f(q)=0$. By Proposition~\ref{prop:fw} and
\textup{(P1)},
$\Par(x,q;\beta)\ge F^0(\beta,0)\ge\Par(x,q;\beta)+\min_u(-\tilde f(u))$, so it suffices to show
$\tilde f\le0$ on $[0,1]$; uniqueness of the minimizer is \textup{(P2)}.

Showing $\tilde f\le0$ is done in Appendix~\ref{app:nm-proofs}, in the three regions described
before the lemma.
\end{proof}

Lemma~\ref{lem:nm-LS} perturbs a strict first-order certificate, as in the persistence argument
of Aronow and Lopatto for nondegenerate finite-step Parisi measures~\cite[Proposition~2.7]%
{AronowLopatto2026}, which also treats the origin, the support points and the remainder
separately. Their hypotheses do not hold here: their class consists of mixed even covariances
with a $q^4$ term at zero field, and they require the first-variation function to vanish exactly
on the support. At $M$, $D_{\delta_0}=-\Phi_{\beta_1}$ also vanishes at $q_1\notin\operatorname{supp}\delta_0$.
Lemma~\ref{lem:nm-LS} works instead along stationary two-atom measures with $x<1$ near $x=1$.

\begin{lemma}[The one-step branch]\label{lem:nm-branch}
The map $\mathcal E(x,q,\beta)=\bigl(C(x,q;\beta),\,T(x,q;\beta)-q\bigr)$ extends real-analytically to a neighbourhood of
$(1,q_1,\beta_1)$ in $\R^3$. There are $\varepsilon_1>0$, a neighbourhood $\mathcal N$ of
$(1,q_1)$ and a real-analytic map $\beta\mapsto(x(\beta),q(\beta))$ on
$(\beta_1-\varepsilon_1,\beta_1+\varepsilon_1)$ such that $(x(\beta),q(\beta))$ is the only zero of
$\mathcal E(\cdot,\cdot,\beta)$ in $\mathcal N$. It satisfies $(x(\beta_1),q(\beta_1))=(1,q_1)$ and
$x'(\beta_1)=-\varphi_{x\beta}/\varphi_{xx}<0$; in particular $0<x(\beta)<1$ for
$\beta\in(\beta_1,\beta_1+\varepsilon_1)$ after shrinking $\varepsilon_1$.
\end{lemma}

\begin{proof}
The map $\mathcal E$ extends real-analytically across $x=1$
(Appendix~\ref{app:nm-proofs}). At $(1,q_1,\beta_1)$: $C=\Phi_{\beta_1}(q_1)=0$ and $T-q_1=0$ by~\eqref{eq:nm-system};
$\partial_xC=\varphi_{xx}>0$ (Lemma~\ref{lem:nm-G});
$\partial_qC=\Phi_{\beta_1}'(q_1)=0$; $\partial_q(T-q)=\Gamma_{\beta_1}'(q_1)-1=-\lambda<0$
by (M2); and $\partial_\beta C=\varphi_{x\beta}$. The Jacobian in $(x,q)$ is lower triangular
with determinant $-\varphi_{xx}\lambda\ne0$, and the implicit function theorem gives the branch,
with $\varphi_{xx}x'(\beta_1)+\varphi_{x\beta}=0$.
\end{proof}

\begin{proposition}[The spin-glass free energy to second order]\label{prop:nm-SG}
Let $\hat\varphi_{\rm SG}(\beta)=\Par(x(\beta),q(\beta);\beta)$ on the branch of
Lemma~\ref{lem:nm-branch}. After shrinking $\varepsilon_1$, for $b\in[0,\varepsilon_1)$,
\[
F^0(\beta_1+b,0)=\hat\varphi_{\rm SG}(\beta_1+b),
\]
and for $b\in(0,\varepsilon_1)$ the Parisi measure at $\beta=\beta_1+b$ is
$x(\beta)\delta_0+(1-x(\beta))\delta_{q(\beta)}$ with $0<x(\beta)<1$. The function
$\hat\varphi_{\rm SG}$ is real-analytic, with $\hat\varphi_{\rm SG}(\beta_1)=\varphi_{\rm PM}(\beta_1)$,
$\hat\varphi_{\rm SG}'(\beta_1)=\beta_1/2$ and
$\hat\varphi_{\rm SG}''(\beta_1)=\frac12-\varphi_{x\beta}^2/\varphi_{xx}=\frac12-A_{\rm SG}$.
Hence, uniformly for $b\in[0,\varepsilon_1)$,
\[
F^0(\beta_1+b,0)=\varphi_{\rm PM}(\beta_1+b)-\frac{A_{\rm SG}}2\,b^2+O(b^3).
\]
\end{proposition}

\begin{proof}
Shrink $\varepsilon_1$ so that the branch lies in the $\rho$-neighbourhood of
Lemma~\ref{lem:nm-LS}. For $b\in(0,\varepsilon_1)$ the branch solves~\eqref{eq:nm-15} with
$0<x<1$, so Lemma~\ref{lem:nm-LS} gives the first two claims. For $b=0$,
$D_{\delta_0}=-\Phi_{\beta_1}\ge0$ by (M1), so Proposition~\ref{prop:fw} gives
$F^0(\beta_1,0)\ge\Par(\delta_0,0)=\varphi_{\rm PM}(\beta_1)$, and \textup{(P1)} gives the
reverse inequality. On the branch $\partial_x\Par=C=0$ and $\partial_q\Par=0$
by~\eqref{eq:nm-stat}, so $\hat\varphi_{\rm SG}'=\partial_\beta\Par$ along the branch. Since
$\Par(1,q;\beta)=\varphi_{\rm PM}(\beta)$ for every $q$, at $(1,q_1,\beta_1)$ we have
$\partial_\beta\Par=\beta_1/2$, $\partial_\beta^2\Par=\frac12$ and $\partial_q\partial_\beta\Par=0$,
while $\partial_x\partial_\beta\Par=\varphi_{x\beta}$. Therefore
$\hat\varphi_{\rm SG}''(\beta_1)=\frac12+\varphi_{x\beta}\,x'(\beta_1)=\frac12-\varphi_{x\beta}^2/\varphi_{xx}$.
Shrinking $\varepsilon_1$ once more, so that $\hat\varphi_{\rm SG}'''$ is bounded on
$[\beta_1,\beta_1+\varepsilon_1)$, Taylor's formula for the analytic function
$\hat\varphi_{\rm SG}$, compared with the quadratic $\varphi_{\rm PM}(\beta)=\log2+\beta^2/4$, gives
the expansion, uniformly for $b\in[0,\varepsilon_1)$.
\end{proof}

\subsection{The ferromagnetic side}\label{sec:nm-fm}

For $h\ge0$ and $q\in[0,1]$ the replica-symmetric Parisi functional is
\[
\Par(\delta_q,h)=\log2+\E\lc\bigl(h+\sqrt{\xi_\beta'(q)}\,g\bigr)
+\frac12\bigl(\xi_\beta(1)-\xi_\beta(q)-(1-q)\xi_\beta'(q)\bigr),
\]
and we put $U_{\beta,j_0}(m;h,q)=\beta j_0m^p+\Par(\delta_q,h)-hm$. By Lemma~\ref{lem:bins}
with the trial $\delta_q$ and \textup{(P1)}, for every bin $[m',m'']\subset[0,1]$,
\begin{equation}\label{eq:nm-binU}
\limsup_N\frac1N\E\log Z_{N,[m',m'']}(\beta,j_0)\le U_{\beta,j_0}(m'';h,q)+h\,(m''-m') .
\end{equation}
For odd $p$ and $q\notin Q_p$, \textup{(P1)} at $h\ne0$ uses the positivity principle, with the
deterministic field included as in Appendix~\ref{app:mf-parisi}. For
$q\in Q_p$ it is Guerra's replica-symmetric interpolation~\cite{Guerra2003}: with $\Delta_q$ as in
Section~\ref{sec:nishimori-point} for $\xi_\beta$,
$N^{-1}\E\log\sum_\sigma e^{\beta H^0_N+hNm}=\Par(\delta_q,h)-\frac12\int_0^1
\E\gibbs{\Delta_q(R_{12})}_t\,\dd t+O(1/N)$, and $\Delta_q\ge0$ on $[-1,1]$.

\begin{lemma}[Nishimori-line diagonal]\label{lem:nm-diag}
Let $\beta=2j_0$. For $m\in[0,1]$,
$U_{\beta,j_0}\bigl(m;\xi_\beta'(m),m\bigr)=\log2+j_0^2+\Phi_\beta(m)$.
\end{lemma}

\begin{proof}
On the Nishimori line $\beta j_0m^p=\xi_\beta(m)$ and $\xi_\beta(1)/2=j_0^2$. At $h=\xi_\beta'(m)$
and $q=m$ the Gaussian term of $\Par(\delta_m,h)$ is $\psi(\xi_\beta'(m))$, so
$U=\log2+j_0^2+\psi(\xi_\beta'(m))-\frac12\xi_\beta'(m)+[\xi_\beta(m)+\frac12m\xi_\beta'(m)
-\frac12\xi_\beta(m)-m\xi_\beta'(m)]$, and the bracket is $-\theta_\beta(m)/2$.
\end{proof}

Let $S$ be the Hessian of $(h,q)\mapsto\Par(\delta_q,h)$ at $(\Lambda_c,\mu)$ for
$\beta=\beta_1$. By Lemma~\ref{lem:nm-NLid},
\begin{equation}\label{eq:nm-S}
S_{hh}=1-\mu,\qquad S_{hq}=-\Omega(\mu-\tau),\qquad
S_{qq}=\frac\Omega2\bigl(1-\Omega(1-4\mu+3\tau)\bigr),
\end{equation}
and we put $\tilde B''(\mu)=\Omega-S_{qq}/\det S$ when $\det S\ne0$. The per-$p$ conditions used
below are
\begin{enumerate}
\item[(N1)] $\kappa_p>0$, i.e.\ $A_{\rm FM}>A_{\rm SG}$;
\item[(N2)] $\det S\ne0$ and $\tilde B''(\mu)<0$;
\item[(N3)] if $p$ is odd, $p\mu^{p-1}+(p-1)\mu^p>1$, i.e.\ $\mu$ lies in the interior of $Q_p$
(Section~\ref{sec:nishimori-point}).
\end{enumerate}
They hold for every integer $p\ge3$ (Theorem~\ref{thm:nm-nondeg}).

\begin{lemma}[Ferromagnetic bins near $\mu$]\label{lem:nm-FMloc}
Assume \textup{(N2)}. Then \textup{(a)} and \textup{(b)} hold, and \textup{(c)} holds if in
addition \textup{(N3)} holds when $p$ is odd.
\begin{enumerate}
\item[\textup{(a)}] $(\mu,\Lambda_c,\mu)$ is a critical point of $U_{\beta_1,j_{0M}}$, and
$U_{\beta_1,j_{0M}}(\mu;\Lambda_c,\mu)=\varphi_{\rm PM}(\beta_1)$.
\item[\textup{(b)}] There is a neighbourhood $\mathcal V$ of $(\beta_1,j_{0M})$ on which
$U_{\beta,j_0}$ has a unique critical point $z^*(\beta,j_0)$ near $(\mu,\Lambda_c,\mu)$, depending
smoothly on $(\beta,j_0)$. Its value $\hat\varphi_{\rm FM}(\beta,j_0)$ is $C^\infty$ on
$\mathcal V$, and at $(\beta_1,j_{0M})$ one has $\partial_\beta\hat\varphi_{\rm FM}=\beta_1/2$
and $\partial_{j_0}\hat\varphi_{\rm FM}=\beta_1\mu^p$.
\item[\textup{(c)}] There is $r_0>0$ such that, after shrinking $\mathcal V$, for
$(\beta,j_0)\in\mathcal V$ and every bin $[m',m'']\subset[\mu-r_0,\mu+r_0]$,
\[
\limsup_N\frac1N\E\log Z_{N,[m',m'']}(\beta,j_0)\le\hat\varphi_{\rm FM}(\beta,j_0)+2\Lambda_c(m''-m').
\]
\end{enumerate}
\end{lemma}

\begin{proof}
(a) The first derivatives are $\partial_mU=p\beta j_0m^{p-1}-h$,
$\partial_hU=\E\tanh(h+\sqrt{\xi_\beta'(q)}g)-m$ and, by the heat identity,
$\partial_qU=\frac12\xi_\beta''(q)\bigl(q-\E\tanh^2(h+\sqrt{\xi_\beta'(q)}g)\bigr)$. At $M$ and
$(m,h,q)=(\mu,\Lambda_c,\mu)$ they vanish: $p\beta_1j_{0M}\mu^{p-1}=\xi_{\beta_1}'(\mu)=\Lambda_c$,
$\E\tanh h_c=\mu$ by~\eqref{eq:nm-system}, and $\E\tanh^2h_c=\mu$ by
Lemma~\ref{lem:nm-NLid}(b). The value is $\varphi_{\rm PM}(\beta_1)+\Phi_{\beta_1}(\mu)
=\varphi_{\rm PM}(\beta_1)$ by Lemma~\ref{lem:nm-diag}.

(b) and (c) follow from the implicit function theorem at this critical point, which is nondegenerate
by \textup{(N2)}, and the envelope theorem; see Appendix~\ref{app:nm-proofs}.
\end{proof}

\begin{lemma}[Closed forms at the triple point]\label{lem:nm-closed}
At $M$,
\[
\det S=\frac\Omega2\bigl(\lambda B+2(\mu-\tau)\bigr),\qquad
\tilde B''(\mu)=-\frac{\lambda\,(1-\Omega B)}{\lambda B+2(\mu-\tau)} .
\]
Since $\lambda>0$ by (M2), $\mu-\tau=\E[\tanh^2h_c\sech^2h_c]>0$ and $B>0$
(Proposition~\ref{prop:nm-J}), condition \textup{(N2)} is equivalent to $\Omega B<1$. Under
\textup{(N2)}, $A_{\rm FM}$ of~\eqref{eq:nm-A} equals
$\frac12-\partial_\beta^2\hat\varphi_{\rm FM}(\beta_1,j_{0M})$, the derivative taken at fixed
$j_0=j_{0M}$.
\end{lemma}

The proof, a computation with the Hessian in $(m,h,q)$, is in Appendix~\ref{app:nm-closed}.
Nishimori obtained the same coefficient at replica level~\cite[(44)--(45),
App.~A4]{Nishimori2026}, in the parametrization $h=p\beta j_0m^{p-1}$.

\subsection{Proof of Theorem~\ref{thm:nm-main}}\label{sec:nm-proof}

By Theorem~\ref{thm:nm-nondeg}, (N1)--(N3) hold; in particular $A_{\rm FM}>A_{\rm SG}$.
Throughout, $j_0=j_{0M}+\delta$.

\paragraph{The warm point.} Let $\beta_{\rm w}=2j_0=\beta_1+2\delta$. As computed before
Lemma~\ref{lem:nm-G}, $\partial_\beta\Phi_\beta(\mu)|_{\beta=\beta_1}=\beta_1\mu^p/2>0$, and
$\Phi_{\beta_1}(\mu)=0$, so $\Phi_{\beta_{\rm w}}(\mu)>0$ for small $\delta>0$. Next,
$\lc y\ge|y|-\log2$ and $\E|h|>\E h=r$ give $\psi(r)>r-\log2$, hence
$\Phi_\beta(1)>\beta^2/4-\log2$. By (M1), $\Phi_{\beta_1}(1)<0$, so
$\xi_{\beta_1}(1)=\beta_1^2/2<2\log2<2$, and $\xi_{\beta_{\rm w}}(1)<2$ for small $\delta$.
Since $m_1^{p-2}\le\frac14$ and $\mathcal I(m)\ge m^2/2$,
\[
\xi_{\beta_{\rm w}}(m_1)=\xi_{\beta_{\rm w}}(1)\,m_1^{p-2}m_1^2\le\frac{m_1^2}2\le\mathcal I(m_1).
\]
So predicate (W) of Theorem~\ref{thm:reduction} holds at $j_0$ with $q_0=\mu$, which lies in
$Q_p$ by (N3) for odd $p$ (and $Q_p=[0,1]$ for even $p$), and with $m_1$ as in
Theorem~\ref{thm:nm-main}. Theorem~\ref{thm:reduction}(i) gives part~(i). This part uses no
Parisi input.

\paragraph{The cold point: order of constants.} Part~(ii) requires
$\E\gibbs{\ind\{|m|\ge\varepsilon\}}\to0$ for every $\varepsilon>0$ at one fixed cold point.
We fix the constants in the following order; none of them depends on $\varepsilon$.
\begin{enumerate}
\item Put $m_{\rm lo}=\frac12\beta_1^{-2/(p-2)}$. Then $m_{\rm lo}<q_1=\mu$. Indeed,
$m_{\rm lo}\le u_0(\beta_1)=(2/(p\beta_1^2))^{1/(p-2)}$ because $2^{-(p-2)}\le2/p$. Also the
function $s\mapsto\mu(s)$ of Lemma~\ref{lem:nm-NLid} satisfies $\mu(s)<s$ for $s>0$, so
$q_1=\mu(\Lambda_c)<\Lambda_c=\xi_{\beta_1}'(q_1)$,
i.e.\ $q_1>u_0(\beta_1)$. Fix $r\in(0,r_0]$ with $r<q_1-m_{\rm lo}$, where $r_0$ is as in
Lemma~\ref{lem:nm-FMloc}(c).
\item Put $\eta_r=-\max\{\Phi_{\beta_1}(m):m\in[m_{\rm lo},\mu-r]\cup[\mu+r,1]\}$. Then
$\eta_r>0$ by (M1).
\item Choose $\bar\delta>0$ such that for every $\delta\in(0,\bar\delta)$, with $b=b(\delta)$ and
$\beta=\beta_1+b$:
\begin{itemize}
\item $(\beta,j_0)\in\mathcal V$ (Lemma~\ref{lem:nm-FMloc}) and $b<\varepsilon_1$
(Proposition~\ref{prop:nm-SG});
\item $U_{\beta,j_0}(m;\xi_{\beta_1}'(m),m)\le\varphi_{\rm PM}(\beta_1)-\eta_r/2$ for
$m\in[m_{\rm lo},\mu-r]\cup[\mu+r,1]$, and $F^0(\beta,0)\ge\varphi_{\rm PM}(\beta_1)-\eta_r/4$;
\item the $O(b^3)$ term of Proposition~\ref{prop:nm-SG} is at most $\eta b^2$, where
$\eta=(A_{\rm FM}-A_{\rm SG})/8$;
\item the remainder $O(|b|^3+|b|\delta+\delta^2)$ in~\eqref{eq:nm-FMexp} below is less than
$\beta_1\mu^p\delta/4$;
\item $\beta j_0m_{\rm lo}^{p-2}<\frac12$ (at $M$ this quantity is $2^{-(p-1)}$);
\item $\xi_{2j_0}(1)<2$ and $\Phi_{2j_0}(\mu)>0$ (the warm point);
\item $\delta<\beta_1\mu^p/(A_{\rm FM}-A_{\rm SG})$.
\end{itemize}
Each condition holds for small $\delta$, because $b(\delta)\asymp\delta^{1/2}\to0$ and the
quantities involved are continuous in $(\beta,j_0)$; the second uses the joint continuity of
$(\beta,j_0,m)\mapsto U_{\beta,j_0}(m;\xi_{\beta_1}'(m),m)$ with the trials fixed at $\beta_1$,
Lemma~\ref{lem:nm-diag} at $M$, and continuity of the convex function $F^0(\cdot,0)$.
\item Fix $j_0\in(j_{0M},j_{0M}+\bar\delta)$; this fixes $\delta$, $b$ and the cold point. Now let
$\varepsilon\in(0,m_{\rm lo})$ be arbitrary. It enters only through the bins.
\end{enumerate}
The regions $[\varepsilon,m_{\rm lo}]$, $[m_{\rm lo},\mu-r]\cup[\mu+r,1]$ and
$[\mu-r,\mu+r]$ cover $[\varepsilon,1]$.

\paragraph{The cold point: the bounds.} Let $\beta=\beta_1+b(\delta)$.
\begin{enumerate}
\item[(a)] \emph{Lower bound.} By Lemma~\ref{lem:lower} and Proposition~\ref{prop:nm-SG},
$\liminf_NN^{-1}\E\log Z_N(\beta,j_0)\ge F^0(\beta,0)\ge\varphi_{\rm PM}(\beta)
-(\frac12A_{\rm SG}+\eta)b^2$.
\item[(b)] \emph{Small magnetization.} On bins $[m',m'']\subset[\varepsilon,m_{\rm lo}]$,
Lemma~\ref{lem:universal} gives $F^0(\beta,0)+g(m'')+(m''-m')$ with
$g(m)=\beta j_0m^p-m^2/2\le-m^2(\frac12-\beta j_0m_{\rm lo}^{p-2})$, which is below
$F^0(\beta,0)$ for bins finer than $\min_{[\varepsilon,m_{\rm lo}]}|g|$. For odd $p$, the
same lemma bounds every bin in $[-1,-\varepsilon]$ by $F^0(\beta,0)-\varepsilon^2/2$; for even
$p$ the bounds depend only on $|m|$, so negative bins are handled as positive ones.
\item[(c)] \emph{Intermediate magnetization.} On bins in $[m_{\rm lo},\mu-r]\cup[\mu+r,1]$ use
\eqref{eq:nm-binU} with the Nishimori-line diagonal trials $(h,q)=(\xi_{\beta_1}'(m''),m'')$. At
$M$ the bound is $\varphi_{\rm PM}(\beta_1)+\Phi_{\beta_1}(m'')$ up to the bin slack
(Lemma~\ref{lem:nm-diag}). By the second condition of the order of constants, fine bins are
bounded by $\varphi_{\rm PM}(\beta_1)-\eta_r/3<F^0(\beta,0)$.
\item[(d)] \emph{Near $\mu$.} By Lemma~\ref{lem:nm-FMloc}(c), bins in $[\mu-r,\mu+r]$ are
bounded by $\hat\varphi_{\rm FM}(\beta,j_0)$ plus the slack $2\Lambda_c(m''-m')$. By
Lemma~\ref{lem:nm-FMloc}(b), Lemma~\ref{lem:nm-closed} and Taylor's formula at $M$,
\begin{equation}\label{eq:nm-FMexp}
\hat\varphi_{\rm FM}(\beta,j_{0M}+\delta)=\varphi_{\rm PM}(\beta)-\frac{A_{\rm FM}}2b^2
+\beta_1\mu^p\delta+O(|b|^3+|b|\delta+\delta^2).
\end{equation}
\item[(e)] \emph{Comparison.} Since $b^2=4\beta_1\mu^p\delta/(A_{\rm FM}-A_{\rm SG})$, the lower
bound (a) minus the right side of~\eqref{eq:nm-FMexp} is at least
\[
\Bigl(\frac{A_{\rm FM}-A_{\rm SG}}2-\eta\Bigr)b^2-\beta_1\mu^p\delta-\frac{\beta_1\mu^p\delta}4
=\frac{A_{\rm FM}-A_{\rm SG}}8\,b^2-\frac{\beta_1\mu^p\delta}4=\frac{\beta_1\mu^p\delta}4>0 .
\]
This is where (N1) enters.
\end{enumerate}
With bins of fixed width finer than all these margins, Lemma~\ref{lem:bins} gives
$\E\gibbs{\ind\{|m|\ge\varepsilon\}}\to0$. As $\varepsilon$ is arbitrary, $\E\gibbs{|m|}\to0$,
which is part~(ii). Part~(iii) holds because $b(\delta)>2\delta$ exactly when
$\delta<\beta_1\mu^p/(A_{\rm FM}-A_{\rm SG})$, the last condition on $\bar\delta$.\qed

For odd $p$, step~(b) uses \textup{(P1)} with a field for the zero-field Parisi measure
(Lemma~\ref{lem:universal}), and step~(c) uses it for trial overlaps below $\min Q_p$; for
$p=3$, $Q_3=[\frac12,1]$ while $m_{\rm lo}\approx0.21$. Both rely on the positivity principle.
Step~(d) and the warm point do not. The size of $\bar\delta$ is
not computed.\footnote{For orientation only, a floating-point evaluation at $p=3$ of the step-(e) gap,
computed from the free energies along $b(\delta)$ without the second-order expansions, is
positive on a grid of values $\delta\le10^{-4}$; at $\delta=10^{-4}$, $b(\delta)\approx0.16$.}

\subsection{Nondegeneracy for every \texorpdfstring{$p$}{p}}\label{sec:nm-nondeg}

\begin{theorem}[Nondegeneracy at the triple point]\label{thm:nm-nondeg}
For every integer $p\ge3$, conditions \textup{(N1)}--\textup{(N3)} hold, and $\lambda>0$,
$0<\Omega B<1$ and $\varphi_{xx}>0$. For $3\le p\le25$ the conditions hold with the certified
margins of Table~\ref{tab:nm-cert}. For $p\ge26$ they follow from
Theorem~\ref{thm:nm-LP}.
\end{theorem}

\begin{proof}
$\lambda>0$ is (M2) and $\varphi_{xx}>0$ is Lemma~\ref{lem:nm-G}. By
Lemma~\ref{lem:nm-closed}, (N2) is equivalent to $\Omega B<1$, and $B>0$ by
Proposition~\ref{prop:nm-J}. For $3\le p\le25$, the zero $\Lambda_c$ of $D_p$ is enclosed in
Arb ball arithmetic, and $\kappa_p$, $\lambda$, $\varphi_{xx}$, $\det S$, $\tilde B''(\mu)$ and
$p\mu^{p-1}+(p-1)\mu^p-1$ are evaluated at the enclosure with the signs of
Table~\ref{tab:nm-cert} (Appendix~\ref{app:nm-cert}). For $p\ge26$ use
Theorem~\ref{thm:nm-LP}(b)--(d).
\end{proof}

For $p\ge26$ no quadrature is needed. Everything is expressed through Gaussian averages at the
single point $\Lambda=\Lambda_c\approx2p\log2$. There $1-\mu$, $B$ and $\kappa_p$ are governed by
the values of $h\sim N(\Lambda,\Lambda)$ of order one, which have probability of order
$\nu(\Lambda_c)\approx1/N_p$, where
\[
\nu(\Lambda)=\frac{e^{-\Lambda/2}}{\sqrt{2\pi\Lambda}},\qquad N_p=2^p(4\pi p\log2)^{1/2}.
\]
We also put $L(h)=\log(1+e^{-2h})$, so that $\lc h=h+L(h)-\log2$.
Writing $T_1=\Lambda_c\mu/p$ and $X=V-\Lambda_c\mu$, so that $\varphi_{xx}=T_1+X$, the
definitions~\eqref{eq:nm-A} give $\beta_1^2A_{\rm SG}=T_1^2/(T_1+X)$ and
$\beta_1^2A_{\rm FM}=T_1+\Lambda_c^2B/(1-\Omega B)$, hence
\begin{equation}\label{eq:nm-A1}
\kappa_p=\frac{T_1X}{T_1+X}+\frac{\Lambda_c^2B}{1-\Omega B}
=\kappa_0-\frac{X^2}{\varphi_{xx}}+\frac{\Lambda_c^2B\,\Omega B}{1-\Omega B},
\qquad \kappa_0=X+\Lambda_c^2B .
\end{equation}
Here $X$ and $\Lambda_c^2B$ are each of order $\Lambda_c\nu(\Lambda_c)$, with opposite signs, while
$\kappa_p$ is of order $\nu(\Lambda_c)$. The next identity performs the cancellation exactly.

\begin{proposition}[A positive kernel]\label{prop:nm-J}
Let $\Lambda>0$, $h\sim N(\Lambda,\Lambda)$, and let $\mu,\tau,V,B$ be defined by the
formulas~\eqref{eq:nm-const} with $h$ in place of $h_c$. Then
\[
V-\Lambda\mu+\Lambda^2B=\nu(\Lambda)\int_\R e^{-h^2/(2\Lambda)}J(h)\,\dd h-\bigl(\E L(h)\bigr)^2,
\]
where
\[
J(h)=\tfrac12\bigl(e^hL(h)^2+e^{-h}L(-h)^2\bigr)-h^2\sech h .
\]
The function $J$ is even, and for $h\ge0$
\[
J(h)=h^2e^{-2h}\sech h+2he^{-h}L(h)+\cosh(h)\,L(h)^2>0 .
\]
Moreover $B=(\nu(\Lambda)/\Lambda)\int e^{-h^2/(2\Lambda)}h\tanh h\sech h\,\dd h>0$,
$0<\E L(h)\le\pi\nu(\Lambda)$, and
\[
c_0:=\int_\R J=4\pi\log2-\frac{\pi^3}4=0.95877519\ldots
\]
\end{proposition}

The proof, by two Gaussian integrations by parts, is in Appendix~\ref{app:nm-J}. Nishimori
tracks the same two cancellations by asymptotic expansion and obtains the leading
term~\cite[(A118)--(A126)]{Nishimori2026}, and $c_0$ is his coefficient: he derives
$C_{\rm SG}-C_{\rm FM}\simeq(4\pi\log2-\pi^3/4)\,\nu(\Lambda_c)$ with a relative error
$O(\log\Lambda_c/\Lambda_c)$~\cite[(56), (A98)]{Nishimori2026}, and notes that he derived no
bound on that error and no explicit threshold for large $p$~\cite[p.~40]{Nishimori2026}. The
kernel identity turns his asymptotics into bounds with explicit constants.

\begin{theorem}[Large $p$]\label{thm:nm-LP}
For every integer $p\ge26$:
\begin{enumerate}
\item[\textup{(a)}] $0<2p\log2-\Lambda_c\le4.97\times10^{-6}$;
\item[\textup{(b)}] $\lambda\ge1-2.4\times10^{-6}$ and $0<\Omega B\le1.35\times10^{-7}$;
\item[\textup{(c)}] $p\mu^{p-1}+(p-1)\mu^p-1\ge(2p-2)(1-1.35\times10^{-7})$;
\item[\textup{(d)}] $0.94616\le N_p\kappa_p\le0.96105$;
\item[\textup{(e)}] $\varphi_{xx}\ge1.38629$ and $\beta_1^2/2<2\log2$.
\end{enumerate}
With the Arb enclosures of $c_0$ and $c_2$ of Appendix~\ref{app:nm-cert}, \textup{(d)} sharpens
to $0.94844\le N_p\kappa_p\le0.95880$. Moreover
\[
\frac{\kappa_p}{\nu(\Lambda_c)}=c_0-\frac{c_2}{2\Lambda_c}+\vartheta_p,\qquad
-R_-(\Lambda_c)\le\vartheta_p\le\frac{c_4}{8\Lambda_c^2}+R_+(\Lambda_c),
\]
where $c_k=\int_\R h^kJ(h)\,\dd h$ (so that $c_0$ is as above), $c_2=0.72233912\ldots$ and
$c_4=2.02594193\ldots$, and
$R_\pm(\Lambda)=O(\Lambda^{3/2}e^{-\Lambda/2})$ are explicit and at most $1.5\times10^{-5}$ for
$\Lambda\ge35$ (Appendix~\ref{app:nm-LP}). In particular $\vartheta_p=O(p^{-2})$.
\end{theorem}

\begin{remark}[Scope and relation to other work]\label{rem:nm-scope}
Theorems~\ref{thm:nm-main} and~\ref{thm:nm-nondeg} settle the sign of $\kappa_p$ and the
two-point reentrance for every integer $p\ge3$. Nishimori's open problem~\cite[p.~19]{Nishimori2026}
concerns the sign of $C_{\rm SG}-C_{\rm FM}$ for all real $p>2$, including $p\in(2,3)$ and the limit
$p\to2^+$, and it is posed for the curvature of a boundary curve. We treat neither non-integer $p$
nor the curve, and $\bar\delta$ is not explicit, so that problem is not resolved here. Lemma~\ref{lem:nm-U} holds at
every level in $(0,1)$, so Proposition~\ref{prop:nm-triple} extends to real $p>2$; we do not use
this. H.-B.~Chen and Issa recently proved, in a preprint, a Parisi-type formula and a
large-deviation principle for the magnetization of vector spin glasses with a general Mattis
interaction whose spin-glass part satisfies a convexity
assumption~\cite[Theorems~1.2 and~1.4]{ChenIssa2026}. As far as we can see, these results apply to
the present model and reduce the question to a variational problem in the magnetization; our
argument does not use them and can be read as an explicit analysis of that problem near the
triple point. We are not aware of an earlier rigorous proof of reentrance in
this model.
\end{remark}

\section{Explicit instances for \texorpdfstring{$p=3$}{p=3} and \texorpdfstring{$p=4$}{p=4}}
\label{sec:certificates}

Certifying the predicates of Theorem~\ref{thm:reduction} for $p=3$ and $p=4$, at biases just
larger than the triple-point bias and at cold temperatures far below the triple point, gives the
following explicit instances.

\begin{theorem}[$p=3$]\label{thm:p3}
Let $p=3$ and $j_0=3841/5000$. At the Nishimori point $\beta_{\rm w}=2j_0=3841/2500$
($T_{\rm w}=1/(2j_0)\approx0.65087$), $\E\gibbs{\ind\{m\le11/25\}}\to0$ exponentially fast; in particular
$\E\gibbs{\ind\{m>11/25\}}\to1$. At $\beta_{\rm cold}=5/2$ ($T=0.4$), $\E\gibbs{|m|}\to0$.
\end{theorem}

\begin{theorem}[$p=4$]\label{thm:p4}
Let $p=4$ and $j_0=8107/10000$. At the Nishimori point $\beta_{\rm w}=2j_0=8107/5000$
($T_{\rm w}=1/(2j_0)\approx0.61675$), $\E\gibbs{\ind\{|m|\le16/25\}}\to0$ exponentially fast; in
particular $\E\gibbs{\ind\{|m|>16/25\}}\to1$. At $\beta_{\rm cold}=10/3$ ($T=0.3$), $\E\gibbs{|m|}\to0$.
\end{theorem}

Both theorems follow from Theorem~\ref{thm:reduction} once the predicates (W), (F) and (L)
are certified; Table~\ref{tab:cert} (Appendix~\ref{app:cert34}) summarizes the certified values. The triple point of
Section~\ref{sec:nearM} has $j_{0M}=\beta_1/2\approx0.767595$ and $T_M\approx0.6514$ for $p=3$,
and $j_{0M}\approx0.810526$ and $T_M\approx0.6169$ for $p=4$, in agreement with the replica
values of~\cite[Table~I]{Nishimori2026}. The cold temperatures lie far below $T_M$.
Theorem~\ref{thm:nm-main} does not decide these points, since its $\bar\delta$ is not explicit and
its cold temperature is $1/(\beta_1+b(\delta))$. Theorems~\ref{thm:p3} and~\ref{thm:p4} exhibit
the same two-point phenomenon with explicit constants. The reentrance there is small, as
predicted: for $p=3$ at $T=0.4$, Nishimori's replica computation places the boundary about
$1.1\times10^{-3}$ above the triple-point bias~\cite[Table~VI]{Nishimori2026}, while our bias
$j_0=0.7682$ exceeds the triple-point bias by only $6.1\times10^{-4}$. The certified point thus
lies in the region that his computation predicts to be non-ferromagnetic at $T=0.4$. He labels
his values below $T\approx0.435$ provisional, because the ferromagnetic branch used there is
replicon unstable. For $p=4$ our bias exceeds the triple-point bias by $1.7\times10^{-4}$. Nishimori
does not tabulate his $p=4$ boundary at $T=0.3$; at $T=0.70\,T_M\approx0.432$ it lies
$3.1\times10^{-4}$ above the triple-point bias~\cite[Table~VII]{Nishimori2026}.

Every predicate is certified in interval arithmetic and certified again by an independent
implementation; Appendix~\ref{app:cert34} describes what is computed and how.

\section{Reentrant Gibbs states on a decorated square lattice}\label{sec:lattice}

The theorems of the previous sections concern fully connected models. This section proves
reentrance at the level of infinite-volume Gibbs states for one planar, $\Z^2$-periodic lattice with
independent $\pm J$ couplings. Throughout the section $p\in(0,\frac12)$ denotes the probability of an
antiferromagnetic bond (not the order of an interaction), $\vartheta:=1-2p$, and
\[
\beta_N(p):=\tfrac12\log\frac{1-p}{p}
\]
is the Nishimori inverse temperature, so that $\tanh\beta_N(p)=\vartheta$.

\paragraph{The obstruction on plain lattices.} The two standard rigorous criteria for order and for
uniqueness in random-bond Ising models that we use see the couplings through one-edge functionals. The quenched Peierls bound of
Horiguchi and Morita~\cite{HM1982} yields a positive plus-boundary magnetization once
$\E e^{-2s\beta J}$ is small enough for some $s\in[0,1]$. Newman's comparison with Bernoulli
percolation~\cite{Newman1994,Newman1997} yields a unique Gibbs state once $\E(1-e^{-2\beta|J|})$ lies
below the critical probability. For couplings of the form $\beta J$ with $J$ of fixed law,
$\min_{s\in[0,1]}\E e^{-2s\beta J}=\min_{t\in[0,\beta]}\E e^{-2tJ}$ is nonincreasing in $\beta$ and
$\E(1-e^{-2\beta|J|})$ is nondecreasing. On a lattice with independent couplings these criteria
therefore certify order only on a low-temperature half-line and uniqueness only on a
high-temperature interval. They cannot detect reentrance. Proposition~\ref{prop:lat-floor} below
sharpens the first half: for $\pm J$ couplings, a contour criterion of this Chernoff--Peierls type
holds at some temperature only if it holds at the Nishimori temperature.

\paragraph{The idea.} Replace each bond of $\Z^2$ by a fixed finite gadget. Summing out the gadget
interiors is exact and yields an Ising model on $\Z^2$ whose couplings are again independent and
identically distributed, at every temperature. Their common law is not of the form $\beta J$: for
the gadget used here the effective coupling first strengthens and then weakens on cooling, because
frustrated units cap it at low temperature. Both criteria then apply at each fixed $\beta$ to the
decimated couplings. They give uniqueness on both sides of a window of non-uniqueness that contains
the Nishimori temperature.

\subsection{The model and the result}\label{sec:lat-result}

\paragraph{The gadget.} A \emph{two-terminal graph} is a finite multigraph with two distinct
distinguished vertices $\pi_0,\pi_1$, its poles; the other vertices are \emph{interior}. The
\emph{series} composition $B_1\cdot B_2$ identifies $\pi_1$ of $B_1$ with $\pi_0$ of $B_2$ and has
poles $\pi_0$ of $B_1$ and $\pi_1$ of $B_2$. The \emph{parallel} composition $B_1\parallel B_2$
identifies the two $\pi_0$ and the two $\pi_1$. Graphs built from single edges by these operations are
\emph{series--parallel}. Let $P_b$ be a path of $b$ edges. For $b_0\ge2$ and $k\ge1$ define the unit
and the gadget
\[
U:=e_d\parallel\bigl(e_1\cdot(P_{b_0}\parallel P_{b_0})\cdot e_2\bigr),\qquad
H_4(b_0,k):=e_a\cdot U^{(1)}\cdot U^{(2)}\cdots U^{(k)}\cdot e_b ,
\]
where $e_a,e_b,e_d,e_1,e_2$ are single edges and the $U^{(i)}$ are copies of $U$
(Figure~\ref{fig:lat-gadget}). The gadget has $k(2b_0+3)+2$ edges and $I=2kb_0+k+1$ interior
vertices. Its poles have degree one and every other vertex has degree at most four, which the
subscript in $H_4$ records. We use the design
\[
H:=H_4(1000,10),\qquad 20032\text{ edges},\qquad I=20011 .
\]

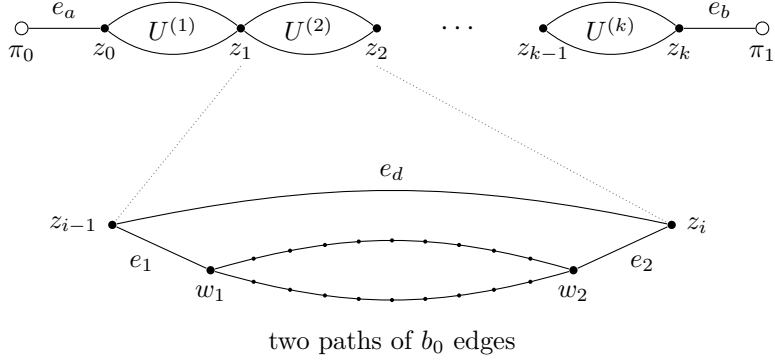
\begin{figure}[t]
\centering
\begin{tikzpicture}[x=1cm,y=1cm,
  v/.style={circle,fill,inner sep=1.1pt},
  s/.style={circle,fill,inner sep=0.55pt},
  pole/.style={circle,draw,fill=white,inner sep=1.7pt}]
\node[pole,label=below:{\small$\pi_0$}] (p0) at (0,2.6) {};
\node[v,label=below:{\small$z_0$}] (z0) at (1.1,2.6) {};
\node[v,label=below:{\small$z_1$}] (z1) at (2.9,2.6) {};
\node[v,label=below:{\small$z_2$}] (z2) at (4.7,2.6) {};
\node[v,label=below:{\small$z_{k-1}$}] (zk1) at (6.9,2.6) {};
\node[v,label=below:{\small$z_k$}] (zk) at (8.7,2.6) {};
\node[pole,label=below:{\small$\pi_1$}] (p1) at (9.8,2.6) {};
\draw (p0) -- node[above]{\small$e_a$} (z0);
\draw (zk) -- node[above]{\small$e_b$} (p1);
\foreach \a/\b/\l in {z0/z1/{U^{(1)}},z1/z2/{U^{(2)}},zk1/zk/{U^{(k)}}}{
  \draw (\a) to[bend left=38] (\b);
  \draw (\a) to[bend right=38] (\b);
  \node at ($(\a)!0.5!(\b)$) {\small$\l$};
}
\node at (5.8,2.6) {$\cdots$};
\node[v,label=left:{\small$z_{i-1}$}] (a) at (1.2,0) {};
\node[v,label=right:{\small$z_i$}] (o) at (8.6,0) {};
\node[v,label=below:{\small$w_1$}] (w1) at (2.5,-0.6) {};
\node[v,label=below:{\small$w_2$}] (w2) at (7.3,-0.6) {};
\draw (a) to[bend left=12] node[above]{\small$e_d$} (o);
\draw (a) -- node[below left=-1pt]{\small$e_1$} (w1);
\draw (w2) -- node[below right=-1pt]{\small$e_2$} (o);
\draw (w1) to[bend left=16]
  node[s,pos=0.1]{} node[s,pos=0.2]{} node[s,pos=0.3]{} node[s,pos=0.4]{} node[s,pos=0.5]{}
  node[s,pos=0.6]{} node[s,pos=0.7]{} node[s,pos=0.8]{} node[s,pos=0.9]{} (w2);
\draw (w1) to[bend right=16]
  node[s,pos=0.1]{} node[s,pos=0.2]{} node[s,pos=0.3]{} node[s,pos=0.4]{} node[s,pos=0.5]{}
  node[s,pos=0.6]{} node[s,pos=0.7]{} node[s,pos=0.8]{} node[s,pos=0.9]{} (w2);
\node at (4.9,-1.55) {\small two paths of $b_0$ edges};
\draw[densely dotted,gray] (2.9,2.1) -- (a);
\draw[densely dotted,gray] (4.7,2.1) -- (o);
\end{tikzpicture}
\caption{The gadget $H_4(b_0,k)$ (top): a series chain of $k$ units between two terminal edges. One
unit $U^{(i)}$ (bottom): a direct edge $e_d$ in parallel with a branch made of a connector $e_1$, two
parallel paths of $b_0$ edges, and a connector $e_2$. The design has $b_0=1000$ and $k=10$. Every
bond of $\Z^2$ is replaced by a copy of the gadget, with $\pi_0,\pi_1$ at its endpoints.}
\label{fig:lat-gadget}
\end{figure}

\paragraph{The lattice.} Let $\mathbf e_1,\mathbf e_2$ be the standard unit vectors of $\Z^2$.
Orient each edge $g=\{x,x+\mathbf e_j\}$ of $\Z^2$ from $x$ to $x+\mathbf e_j$ and replace it
by a copy $H^g$ of $H$ with $\pi_0\mapsto x$ and $\pi_1\mapsto x+\mathbf e_j$; interiors of distinct copies are
disjoint. The resulting graph $\Z^2[H]$ is planar, because a series--parallel graph embeds in a disc
with both poles on the boundary. It is $\Z^2$-periodic, and its maximum degree is four, since each
vertex of $\Z^2$ (a \emph{backbone vertex}) meets four gadgets through one edge each. The same
construction applied to the free box $\Lambda_L=\{1,\dots,L\}^2$ and to the torus
$T_L=(\Z/L\Z)^2$, $L\ge3$, gives finite graphs $\Lambda_L[H]$ and $T_L[H]$.

\paragraph{Disorder and Gibbs measures.} The couplings $J_e\in\{\pm1\}$ of the microscopic edges are
independent with $\Pp(J_e=-1)=p$; write $\Pp_p$ and $\E$ for their law and expectation. On a finite
graph with vertex set $V$ the Gibbs measure is $\mu_{J,\beta}(\sigma)\propto
\exp(\beta\sum_{e=\{u,v\}}J_e\sigma_u\sigma_v)$ with expectation $\gibbs{\cdot}_\beta$ (free boundary
condition in the graph sense), and
\[
M_L:=\frac1{|V|}\sum_{w\in V}\sigma_w
\]
averages all microscopic spins of $\Lambda_L[H]$ or $T_L[H]$. On the infinite graph, $\mathcal
G(J,\beta)$ denotes the set of Gibbs measures in the sense of Dobrushin, Lanford and Ruelle: the
probability measures on $\{\pm1\}^{V(\Z^2[H])}$ under which, for every finite set $\Delta$ of
vertices, the conditional law of $\sigma_\Delta$ given the spins off $\Delta$ is proportional to
$\exp(\beta\sum_{e\cap\Delta\neq\emptyset}J_e\sigma_u\sigma_v)$. This set is nonempty.

\begin{theorem}[Reentrant Gibbs states on $\Z^2{[H]}$]\label{thm:lat-main}
Let $H=H_4(1000,10)$ and $p_0=9/10000$, with Nishimori inverse temperature
$\beta_N:=\beta_N(p_0)=\frac12\log(9991/9)=3.50610\ldots$. Put
\[
\beta_{\rm hot}:=\frac{30847}{20000},\qquad \beta_-:=\frac{24577}{12500},\qquad
\beta_+:=\frac{98231}{20000},\qquad \beta_{\rm cold}:=\frac{349277}{50000},
\]
that is $1.54235$, $1.96616$, $4.91155$ and $6.98554$, or $0.43990\ldots\beta_N$,
$0.56078\ldots\beta_N$, $1.40085\ldots\beta_N$ and $1.99239\ldots\beta_N$. Then
$\beta_{\rm hot}<\beta_-<\beta_N<\beta_+<\beta_{\rm cold}$, and for the iid $\pm J$ model on $\Z^2[H]$ with
$\Pp(J_e=-1)=p_0$ the following hold.
\begin{enumerate}
\item[\textup{(a)}] \emph{Uniqueness.} For each fixed $\beta\in(0,\beta_{\rm hot}]\cup[\beta_{\rm cold},\infty)$,
almost surely $|\mathcal G(J,\beta)|=1$; the unique Gibbs measure is invariant under the global spin
flip.
\item[\textup{(b)}] \emph{Non-uniqueness.} For each fixed $\beta\in[\beta_-,\beta_+]$, almost surely
$|\mathcal G(J,\beta)|\ge2$.
\item[\textup{(c)}] \emph{Deterministic form.} The set
$\mathcal N:=\{\beta>0:\Pp(|\mathcal G(J,\beta)|\ge2)=1\}$ satisfies
$[\beta_-,\beta_+]\subset\mathcal N\subset(\beta_{\rm hot},\beta_{\rm cold})$, and
$\Pp(|\mathcal G(J,\beta)|=1)=1$ for every $\beta\notin\mathcal N$.
\item[\textup{(d)}] \emph{Order at the Nishimori temperature.} On $\Lambda_L[H]$ and on $T_L[H]$,
$\liminf_{L\to\infty}\E\gibbs{M_L^2}_{\beta_N}\ge0.08$.
\item[\textup{(e)}] \emph{The Nishimori line.} For every $p\in(0,1/1000]$,
$\liminf_{L\to\infty}\E\gibbs{M_L^2}_{\beta_N(p)}>0$ on $\Lambda_L[H]$ and on $T_L[H]$. For every
$p\in(0,3/2500]$, $\Pp_p$-almost surely $|\mathcal G(J,\beta_N(p))|\ge2$.
\end{enumerate}
The proof is computer-assisted: finitely many inequalities about the explicit law of the decimated
coupling (Proposition~\ref{prop:lat-law}) are certified in interval arithmetic
(Appendix~\ref{app:lat-cert}).
\end{theorem}

On cooling at the fixed disorder strength $p_0$, the Gibbs state is therefore unique, then
non-unique, then unique again, and the non-unique window contains the Nishimori temperature. Nothing is
claimed for $\beta\in(\beta_{\rm hot},\beta_-)$ or $\beta\in(\beta_+,\beta_{\rm cold})$. The almost-sure
statements hold for each fixed $\beta$, and the exceptional null set may depend on $\beta$; part~(c)
is a statement about a deterministic set and needs no simultaneity in $\beta$.

Part~(e) places $p_0$ strictly inside the ordered part of the Nishimori line. Put
\[
p_N^{M}:=\sup\bigl\{p:\liminf_{L}\E\gibbs{M_L^2}_{\beta_N(p)}>0\text{ on }T_L[H]\bigr\},\qquad
p_N^{\rm DLR}:=\sup\bigl\{p:\Pp_p\text{-a.s. }|\mathcal G(J,\beta_N(p))|\ge2\bigr\}.
\]
Then $p_0<1/1000\le p_N^M$ and $p_0<3/2500\le p_N^{\rm DLR}$. The locations of $p_N^M$ and
$p_N^{\rm DLR}$ are not determined.

The graph is not the nearest-neighbour square lattice: each bond of $\Z^2$ is replaced by a fixed
two-terminal series--parallel gadget with $20032$ edges, built from paths of length $1000$.

\begin{corollary}[Finite volume]\label{rem:lat-finite}
The same certified inputs give the finite-volume counterpart of~(a): for every
$\beta\in(0,\beta_{\rm hot}]\cup[\beta_{\rm cold},\infty)$, $\E\gibbs{M_L^2}_\beta\to0$ and
$\E\gibbs{Q_L^2}_\beta\to0$ on $\Lambda_L[H]$ and on $T_L[H]$, where $Q_L$ is the overlap of two
independent replicas. So at low temperature there is neither ferromagnetic nor Edwards--Anderson order
in this sense (Appendix~\ref{app:lat-unique}).
\end{corollary}

\subsection{The decimated coupling}\label{sec:lat-coupling}

For a two-terminal graph $B$ with real couplings $a_e$, summing over the interior spins gives
$Z_B(s_0,s_1)=A_Be^{K_Bs_0s_1}$ for the pole values $s_0,s_1$, with $A_B>0$ and an \emph{effective
coupling} $K_B\in\R$ (Lemma~\ref{lem:lat-reduction}). Series composition multiplies $\tanh K$ and
parallel composition adds $K$. For the gadget with microscopic couplings $\beta J$ write
$K_H(J,\beta)$, and $K_H(\beta)$ for the random variable when $J$ is iid. The whole argument rests on
the following functionals of its law:
\[
v_H(\beta):=\E e^{-K_H(\beta)},\qquad
F_H(\beta,s):=\E e^{-2sK_H(\beta)},
\]
\[
w_H(\beta):=\min_{s\in[0,1]}F_H(\beta,s),\qquad
\bar p_H(\beta):=\E\bigl[1-e^{-2|K_H(\beta)|}\bigr].
\]
Taking $s=0$ and $s=\frac12$ shows $w_H\le\min(1,v_H)$.

\begin{proposition}[Law of the decimated coupling]\label{prop:lat-law}
Let $H=H_4(b_0,k)$ and $\beta>0$. Put $\theta:=\tanh\beta$, $\chi:=\operatorname{artanh}(\theta^{b_0})$,
$\tau:=\theta^2\tanh(2\chi)$ and $\chi_2:=\operatorname{artanh}\tau$, so that $0\le\chi_2<\beta$. Put
$q:=(1-\vartheta^{b_0})/2$ (recall $\vartheta=1-2p$; the two letters agree, $\theta=\vartheta$, only
at $\beta=\beta_N(p)$),
\[
\alpha_1:=2q(1-q),\qquad \alpha_3:=(1-q)^2\tfrac{1-\vartheta^3}2+q^2\tfrac{1+\vartheta^3}2,\qquad
\alpha_2:=1-\alpha_1-\alpha_3,
\]
and $y_1:=\theta$, $y_2:=\tanh(\beta+\chi_2)$, $y_3:=\tanh(\beta-\chi_2)$.
\begin{enumerate}
\item[\textup{(i)}] For each unit, $|\tanh K_U|=y_c$ with probability $\alpha_c$, $c=1,2,3$, and the sign of
$K_U$ is the sign of the coupling of its direct edge $e_d$.
\item[\textup{(ii)}] $|\tanh K_H|=\theta^2\prod_{i=1}^k|\tanh K_{U^{(i)}}|$ with independent factors.
Hence, given the class counts $n=(n_1,n_2,n_3)$, which are multinomial with parameters $k$ and
$(\alpha_1,\alpha_2,\alpha_3)$, one has $|\tanh K_H|=x_n:=\theta^2y_1^{n_1}y_2^{n_2}y_3^{n_3}$ and
$\Pp(K_H>0\mid n)=(1+A_n)/2$ with $A_n:=\vartheta^2\rho_1^{n_1}\rho_2^{n_2}\rho_3^{n_3}$, where
$\rho_c$ is the mean of the direct-edge coupling of a unit in class $c$. The $\alpha_c$ and $\rho_c$
depend on $p$ and $b_0$ only, not on $\beta$.
\item[\textup{(iii)}] As $\beta\to\infty$, $\beta-\chi_2\to\frac12\log2$, so $y_3\to\frac13$, while
$y_1,y_2\to1$.
\end{enumerate}
\end{proposition}

For $H_4(1000,10)$ the law of $K_H(\beta)$ is thus a mixture of $66$ magnitudes and $132$ signed atoms,
each an explicit function of $\beta$. The three classes are: the two long paths disagree (class~1,
$|K_U|=\beta$); they agree and the branch reinforces the direct edge (class~2, $|K_U|=\beta+\chi_2$);
they agree and the branch opposes it (class~3, the \emph{cancel} event, $|K_U|=\beta-\chi_2$). At
$p_0$ one has $q\approx0.4175$ and $(\alpha_1,\alpha_2,\alpha_3)\approx(0.4864,0.3389,0.1747)$.

The mechanism is visible in the classes. At the Nishimori temperature a path of $1000$ edges is a weak
and unreliable bond: its sign is negative with probability $q\approx0.42$, and $\theta^{b_0}\approx0.165$,
so $\chi_2\approx0.33$ is small compared with $\beta_N$. Every unit is then strongly coupled, and the
backbone coupling is a series chain of twelve strong bonds whose main defect is the sign disorder of the
two terminal edges and the ten direct edges; its Bhattacharyya parameter is
$v_H(\beta_N)=0.2052950\ldots$. On cooling the paths freeze, $\chi_2$ catches up with $\beta$, and on
the cancel event the unit coupling tends to $\frac12\log2$. There the two relative orientations of
the poles have the same ground-state energy: each costs one broken edge, the direct edge $e_d$ or
one of the connectors $e_1,e_2$. Their ground-state degeneracies are therefore in the ratio $2:1$
in favour of the orientation preferred by the direct edge. Since $\tanh(\frac12\log2)=\frac13$
and each of the ten units cancels with probability $\alpha_3$, the backbone couplings do not become
reliable at low temperature: as $\beta\to\infty$, $\E\tanh|K_H(\beta)|\to(1-2\alpha_3/3)^{10}\approx0.2898$ and
$\bar p_H(\beta)\to0.37434\ldots<\frac12$. This is the classical decorated-bond mechanism
(Section~\ref{sec:lat-scope}); what is specific here is that the couplings are quenched and iid on
every microscopic edge.

\subsection{Proof architecture}\label{sec:lat-proof}

The proof has five steps. Complete proofs of the steps are in Appendix~\ref{app:lattice}; the proof
of Theorem~\ref{thm:lat-main} from them is at the end of this subsection.

\paragraph{Step 1: exact decimation.} Summing out the interior of each gadget turns the model on
$\Z^2[H]$ into the Ising model on $\Z^2$ with couplings $K_g:=K_H(J|_{H^g},\beta)$. Distinct gadgets
have disjoint edge sets, so the $K_g$ are iid with the law of $K_H(\beta)$ at every $\beta$, on or off
the Nishimori line. This holds in finite volume and in infinite volume:

\begin{lemma}[Backbone reduction]\label{lem:lat-BR}
Fix $J$ and $\beta$. Let $\mathcal G^{\rm bb}(K)$ be the set of Gibbs measures of the Ising model on
$\Z^2$ with couplings $K_g=K_g(J,\beta)$. Taking the marginal on the backbone spins is a bijection
$\mathcal G(J,\beta)\to\mathcal G^{\rm bb}(K)$. Its inverse attaches to a backbone configuration
independent gadget interiors with their conditional Gibbs laws. The event $\{|\mathcal G(J,\beta)|=1\}$
is measurable, and $\Pp(|\mathcal G(J,\beta)|\ge2)\in\{0,1\}$.
\end{lemma}

Both halves of Theorem~\ref{thm:lat-main}(a)--(c) are thus statements about a planar Ising model with
iid couplings of both signs whose law is given by Proposition~\ref{prop:lat-law}.

\paragraph{Step 2: non-uniqueness from a quenched Peierls bound.} Let $\nu^+_{\Lambda,K}$ be the Ising
measure on a finite $\Lambda\subset\Z^2$ with couplings $K_g$ on the edges meeting $\Lambda$ and all
spins outside $\Lambda$ equal to $+1$. Define the Peierls functions
\[
\mathsf P(w):=\frac{9w^4}{(1-9w^2)^2},\qquad \mathsf Q(w):=\frac{9w^3}{1-9w^2},\qquad
\mathsf h(w):=4\mathsf P(w)+4\mathsf Q(w)\qquad(0<w<\tfrac13),
\]
and let $w_+:=(9+3\sqrt2)^{-1/2}=0.27479\ldots$ and $v_P=0.22240\ldots$ be the roots of $2\mathsf P=1$
and of $\mathsf h=1$ in $(0,\frac13)$. All three functions increase on $(0,\frac13)$.

\begin{proposition}[Plus-state bound]\label{prop:lat-CP}
Let $(K_g)_{g\in E(\Z^2)}$ be independent real random variables, and let $s\in[0,1]$ and
$w\in(0,\frac13)$ satisfy $\E e^{-2sK_g}\le w$ for every $g$. Then for every finite $\Lambda\subset\Z^2$
with connected complement and every $x\in\Lambda$,
\[
\E\,\nu^+_{\Lambda,K}(\sigma_x)\ge1-2\mathsf P(w).
\]
\end{proposition}

This is the quenched Peierls bound of Horiguchi and Morita~\cite[\S2, (2.12)--(2.19),
pp.~3553--3554]{HM1982}, which holds for independent nearest-neighbour couplings on the square lattice
of arbitrary law and either sign. Their convergence criterion
$\min_{0\le t\le2\beta}\gibbs{e^{-tJ}}_c<\frac13$ is $3w<1$ in the present notation ($K=\beta J$,
$t=2s\beta$). We write the bound per site, uniformly in $\Lambda$, and bound each contour weight by a
single Chernoff factor: flipping the spins inside a contour $\partial A$ shows that it is present with
probability at most $\min(1,e^{-2\sum_{\partial A}K})\le e^{-2s\sum_{\partial A}K}$ for every
$s\in[0,1]$, whatever the signs. Independence then turns each contour of length $\ell$ into $w^\ell$.
Exact decimation is what makes the bound applicable: the $K_g$ are independent with a common law at
every $\beta$, and taking $s$ optimal gives $w=w_H(\beta)$.\footnote{Earlier, for a small concentration
of random antiferromagnetic bonds, ferromagnetic order at low temperature had been proved by
Peierls--Griffiths arguments on the square lattice by Avron, Roepstorff and Schulman~\cite{ARS1981}
(see~\cite[p.~3551]{HM1982}), and on hypercubic lattices of any dimension by
Roepstorff~\cite{Roepstorff1981}.}

\begin{proposition}[Non-uniqueness]\label{prop:lat-NU}
Fix $\beta>0$, and suppose there are $s\in[0,1]$ and $W$ with $F_H(\beta,s)\le W$, $9W^2<1$ and
$2\mathsf P(W)<1$; equivalently, $w_H(\beta)<w_+$. Then there is a measurable function $Y$ of the
disorder with values in $[-1,1]$ and $\E Y\ge1-2\mathsf P(W)>0$ such that for every $J$ there are $\mu^\pm_J\in\mathcal G(J,\beta)$ with
$\mu^\pm_J(\sigma_0)=\pm Y(J)$ at the backbone origin. Almost surely $|\mathcal G(J,\beta)|\ge2$.
\end{proposition}

Here $Y$ is the limit superior of the plus-state magnetizations of growing boxes; the bound on $\E Y$
survives by reverse Fatou. The flip gives the partner $\mu^-_J$, and ergodicity of the disorder under
backbone translations upgrades $\Pp(Y>0)>0$ to almost-sure non-uniqueness. This is the standard
flip-and-ergodicity step~\cite[Corollary~2.8, p.~45]{Bovier2001}; compare~\cite[p.~417,
Remark~1]{CCF1985}.

\paragraph{Step 3: uniqueness from Newman's domination.} For couplings of arbitrary sign, Newman
bounded the influence of the boundary condition on a finite set by a Bernoulli connection probability:

\begin{lemma}[Newman~\cite{Newman1994,Newman1997}]\label{lem:lat-newman}
Let $\Lambda\subset\Z^2$ be finite, let $K_g\in\R$ for the edges $g$ meeting $\Lambda$, and let
$\nu^\eta_{\Lambda,K}$ be the Ising measure on $\Lambda$ with boundary condition $\eta$ on the outer
vertex boundary $\partial\Lambda$. For $A\subset\Lambda$ and $\sigma_A:=\prod_{x\in A}\sigma_x$,
\[
\bigl|\nu^\eta_{\Lambda,K}(\sigma_A)-\nu^{\eta'}_{\Lambda,K}(\sigma_A)\bigr|
\le2\,\Phi_{p(K)}(A\leftrightarrow\partial\Lambda),\qquad p_g(K):=1-e^{-2|K_g|},
\]
where $\Phi_{p}$ is independent bond percolation on the edges meeting $\Lambda$ with $\Pp(g\text{ open})=p_g$.
\end{lemma}

Appendix~\ref{app:lat-unique} gives the exact locators in~\cite{Newman1994,Newman1997} and the
translation to our normalization.

\begin{corollary}[Uniqueness]\label{cor:lat-UH}
If $\bar p_H(\beta)\le\frac12$, then almost surely $|\mathcal G(J,\beta)|=1$, and the unique Gibbs
measure is flip-invariant.
\end{corollary}

Averaging over the independent $K_g$ turns $\Phi_{p(K)}$ into Bernoulli percolation with density
$\bar p_H(\beta)$, and Harris's theorem that bond percolation on $\Z^2$ has no infinite cluster at
density $\frac12$~\cite{Harris1960} makes the bound vanish as $\Lambda\uparrow\Z^2$. For
$\bar p_H(\beta)<\frac12$ this is Newman's uniqueness criterion~\cite[Theorems~3.10--3.11,
pp.~39--41]{Newman1997} applied to the decimated couplings.\footnote{Uniqueness at high temperature is
almost free: the terminal edges alone give $|\tanh K_H|\le\tanh^2\beta$, hence $\bar p_H(\beta)\le\frac12$
for $\beta\le\frac12\log(2+\sqrt3)=0.658\ldots$ at every $p$.} The substance is uniqueness at low
temperature. It is possible here because $|K_H(J,\beta)|$ is not monotone in $\beta$ on disorders with
cancel units (Proposition~\ref{prop:lat-law}(iii)).

\paragraph{Step 4: order at the Nishimori temperature.} Part~(d) concerns the magnetization of all
microscopic spins, including the interior ones.

\begin{theorem}[Order at the Nishimori temperature]\label{thm:lat-A}
Let $H$ be a two-terminal series--parallel gadget with $I$ interior vertices and let $p\in(0,\frac12)$.
Let $H^{(i)}$ be $H$ with the edges at $\pi_{1-i}$ deleted, and for an interior vertex $u$ put
$\kappa_H(u):=\max_{i\in\{0,1\}}\E\gibbs{\sigma_u\sigma_{\pi_i}}_{H^{(i)}}$, the average being at
$\beta_N(p)$. Let
$\bar\kappa_H$ be the mean of $\kappa_H(u)$ over the interior vertices and
$\hat\kappa_H:=(1+2I\bar\kappa_H)/(1+2I)$. If $v_H(\beta_N(p))<v_P$, then on $\Lambda_L[H]$ and on
$T_L[H]$
\[
\liminf_{L\to\infty}\E\gibbs{M_L^2}_{\beta_N(p)}\ge\hat\kappa_H^2\,\bigl(1-\mathsf h(v_H(\beta_N(p)))\bigr)>0 .
\]
Moreover $\kappa_H(u)\ge0$ for every $u$, so $\hat\kappa_H\ge1/(1+2I)$.
\end{theorem}

On the Nishimori line the Gibbs measure is the posterior of a planted configuration, and gauge
symmetry makes every $\E\gibbs{\sigma_u\sigma_v}$ nonnegative~\cite{Nishimori1981} and monotone under
adding edges (Kitatani's inequality~\cite{Kitatani1994}, in the form stated
in~\cite[(3)]{Kitatani2009}). For two interior vertices $u,v$ in far
apart gadgets, deleting the edges of each gadget at one pole isolates $u$ and $v$ behind cut vertices,
and $\E\gibbs{\sigma_u\sigma_v}$ is at least $\kappa_H(u)\kappa_H(v)$ times a backbone two-point
function in which the two gadgets carry zero coupling. A Peierls bound with two such defect edges
controls that two-point function once $\mathsf h(v_H)<1$; this is where the threshold $v_P$ enters. For
$H_4(1000,10)$ at $p_0$, chains of whole units and of two-path cycles give $\bar\kappa_H\ge0.49146$,
and $\hat\kappa_H^2(1-\mathsf h(v_H(\beta_N)))\ge0.0802$ (Lemma~\ref{lem:lat-kappa}).

\paragraph{Step 5: the Nishimori line.}

\begin{lemma}[Monotonicity along the Nishimori line]\label{lem:lat-degradation}
For every two-terminal graph $H$, the map $p\mapsto v_H(\beta_N(p))$ is nondecreasing on $(0,\frac12)$.
\end{lemma}

The coupling $K_H(\beta_N(p))$ is half the log-likelihood ratio of a binary-input channel that observes
the product of the two pole spins of a uniformly random configuration through the edge signs, each
flipped with probability $p$; $v_H(\beta_N(p))$ is its Bhattacharyya parameter. Raising $p$ degrades
the channel, and a degradation cannot decrease the Bhattacharyya parameter (Cauchy--Schwarz;
see~\cite{RichardsonUrbanke2008} for this language). By the lemma, bounds on $v_H$ certified at
$p=1/1000$ and at $p=3/2500$ hold for all smaller $p$, which gives part~(e) through
Theorem~\ref{thm:lat-A} and Proposition~\ref{prop:lat-NU} (with $s=\frac12$).

\paragraph{The finite inputs.} Steps 2--5 reduce Theorem~\ref{thm:lat-main} to explicit inequalities
about the finite law of Proposition~\ref{prop:lat-law}: $w_H(\beta)<w_+$ on $[\beta_-,\beta_+]$ (C2, C3),
$\bar p_H(\beta)<\frac12$ on $(0,\beta_{\rm hot}]$ (C1) and on $[\beta_{\rm cold},\infty)$ (C4), $v_H(\beta_N)<v_P$ (C5) together
with the chain bound on $\bar\kappa_H$ (C6), $v_H(\beta_N(1/1000))<v_P$ (C7), and $v_H(\beta_N(3/2500))<w_+$ (C8).
Appendix~\ref{app:lat-cert} lists them with their certified values and describes the certificates. Each
input was certified by at least two implementations that share no code, and each criterion is
certified to fail at a point within $10^{-4}$ (an absolute distance in $\beta$) outside its certified
range (Remark~\ref{rem:lat-sharp}). The gaps
$(\beta_{\rm hot},\beta_-)$ and $(\beta_+,\beta_{\rm cold})$ therefore reflect the methods.

\begin{proof}[Proof of Theorem~\ref{thm:lat-main}]
Here (C1)--(C8) are the certified inputs of Table~\ref{tab:lat-inputs}, and $W_\star$, $V_0$, $V_1$,
$W_N$ the certified constants recorded there and in Appendix~\ref{app:lat-cert}.
\par\emph{Ordering.} $\beta_{\rm hot}<\beta_-$ and $\beta_+<\beta_{\rm cold}$ are rational comparisons, and since
$e^{2\beta_N}=9991/9$,
\[
e^{2\beta_-}<e^4<55<\tfrac{9991}9<7600<e^9<e^{2\beta_+},
\]
because $2\beta_-=3.93232<4$ and $2\beta_+=9.8231>9$, with $2.71828<e<2.71829$.
\par\emph{(a)} By (C1) and (C4), $\bar p_H(\beta)<\frac12$ on $(0,\beta_{\rm hot}]\cup[\beta_{\rm cold},\infty)$; apply
Corollary~\ref{cor:lat-UH} at each such $\beta$.
\par\emph{(b)} For $\beta\in[\beta_-,\beta_+]$, the cell of (C2) containing $\beta$ gives $s_j\in[0,1]$ with
$F_H(\beta,s_j)\le W_\star$, and (C3) gives $9W_\star^2<1$ and $2\mathsf P(W_\star)<1$; apply
Proposition~\ref{prop:lat-NU}. At $\beta=\beta_N$ one may instead take $s=\frac12$ and $W=V_0$ from (C5), which
gives $\E Y\ge0.917$.
\par\emph{(c)} By the zero--one law of Lemma~\ref{lem:lat-BR}, for each $\beta$ either $\beta\in\mathcal N$ or
$\Pp(|\mathcal G(J,\beta)|=1)=1$. Part (b) gives $[\beta_-,\beta_+]\subset\mathcal N$ and part (a) gives
$\mathcal N\cap\bigl((0,\beta_{\rm hot}]\cup[\beta_{\rm cold},\infty)\bigr)=\emptyset$.
\par\emph{(d)} By (C5), $v_H(\beta_N)\le V_0<v_P$, so Theorem~\ref{thm:lat-A} applies, and Lemma~\ref{lem:lat-kappa}
evaluates its constant: $\hat\kappa_H^2(1-\mathsf h(v_H(\beta_N)))\ge\hat\kappa_H^2(1-\mathsf h(V_0))\ge0.0802$.
\par\emph{(e)} Let $p\le1/1000$. By Lemma~\ref{lem:lat-degradation} and input (C7),
$v_H(\beta_N(p))\le V_1<v_P$, so Theorem~\ref{thm:lat-A} applies with
$\hat\kappa_H\ge1/(1+2I)$ and gives $\liminf_L\E\gibbs{M_L^2}_{\beta_N(p)}\ge(1-\mathsf h(V_1))/(1+2I)^2>0$,
where $1-\mathsf h(V_1)\ge0.135491$. Let $p\le3/2500$. By Lemma~\ref{lem:lat-degradation} and input (C8),
$F_H(\beta_N(p),\frac12)=v_H(\beta_N(p))\le W_N$ with $9W_N^2<1$ and $2\mathsf P(W_N)<1$, so
Proposition~\ref{prop:lat-NU} applies at $\beta=\beta_N(p)$ with $s=\frac12$, and $\E Y\ge1-2\mathsf P(W_N)
\ge0.770991$. Since $v_P<w_+$, the Gibbs-state criterion reaches further along the Nishimori line than the
magnetization criterion.
\end{proof}

\subsection{The Nishimori floor of the contour weight}\label{sec:lat-floor}

The set of temperatures certified by the contour criterion is bounded here (see the end of this
subsection), while on a lattice with couplings $\beta J$ it is a half-line. The reason is that the
Chernoff weight $w_H$ is not monotone in $\beta$.
Where it can be small is nevertheless fixed by gauge symmetry, for every gadget.

\begin{proposition}[Nishimori floor]\label{prop:lat-floor}
Let $H$ be a two-terminal graph and $p\in(0,\frac12)$, with $J$ iid and $\Pp(J_e=-1)=p$. For all real
$\beta$ and $s$,
\[
F_H(\beta,s)=\E_pe^{-2sK_H(J,\beta)}\ \ge\ v_H(\beta_N(p))=F_H(\beta_N(p),\tfrac12).
\]
Consequently $w_H(\beta)\ge w_H(\beta_N(p))=v_H(\beta_N(p))$ for every $\beta>0$, and for every $c$ the
set $\{\beta>0:w_H(\beta)<c\}$ is nonempty if and only if it contains $\beta_N(p)$.
\end{proposition}

\begin{proof}
Write $\gamma:=\beta_N(p)$, $V$ and $E$ for the vertex and edge sets of $H$, and $Z_H(a,b)(\beta J)$ for
the sum over interior spins with pole values $(a,b)$, so that $Z_H(a,b)=A_He^{K_Hab}$. For
$\tau\in\{\pm1\}^V$ let $(J^\tau)_e:=\tau_u\tau_vJ_e$. Substituting $\sigma_v\mapsto\tau_v\sigma_v$ at
every vertex gives $K_H(J^\tau,\beta)=\tau_{\pi_0}\tau_{\pi_1}K_H(J,\beta)$ and
$A_H(J^\tau,\beta)=A_H(J,\beta)$ (Lemma~\ref{lem:lat-reduction}). Since $J\mapsto J^\tau$ is a
bijection, $F_H(\beta,s)=\sum_J\Pp_p(J^\tau)e^{-2s\tau_{\pi_0}\tau_{\pi_1}K_H(J,\beta)}$ for every
$\tau$. Average over all $2^{|V|}$ gauges and use $\Pp_p(J)=\prod_ee^{\gamma J_e}/(2\cosh\gamma)$. For
fixed $J$ and $x\in\{\pm1\}$,
\[
\sum_{\tau:\,\tau_{\pi_0}\tau_{\pi_1}=x}e^{\gamma\sum_eJ_e\tau_u\tau_v}=2Z_H(1,x)(\gamma J)
=2A_H(J,\gamma)e^{xK_H(J,\gamma)},
\]
where the factor $2$ pairs $\tau$ with $-\tau$. With $C_H:=2^{1-|V|}(2\cosh\gamma)^{-|E|}$ this gives
\begin{equation}\label{eq:lat-floor}
F_H(\beta,s)=C_H\sum_JA_H(J,\gamma)\,2\cosh\bigl(K_H(J,\gamma)-2sK_H(J,\beta)\bigr).
\end{equation}
Since $\cosh\ge1$, $F_H(\beta,s)\ge2C_H\sum_JA_H(J,\gamma)$, and at $(\beta,s)=(\gamma,\frac12)$ every
$\cosh$ in~\eqref{eq:lat-floor} equals $1$. So the lower bound is $F_H(\gamma,\frac12)=v_H(\gamma)$.
Minimizing over $s\in[0,1]$ gives $w_H(\beta)\ge v_H(\gamma)\ge w_H(\gamma)$, and the last claim follows.
\end{proof}

This is a gauge inequality in the spirit of Nishimori's~\cite{Nishimori1981}: a gauge average followed
by $\cosh\ge1$. Its consequence for contour methods is structural. At fixed $p$ every Chernoff--Peierls
criterion (Horiguchi and Morita's convergence criterion $3w_H<1$, the plus-state criterion $w_H<w_+$,
and the two-point criterion $\mathsf h(v_H)<1$ used in Theorem~\ref{thm:lat-A}; indeed every
criterion of the form $w_H(\beta)<c$ or $v_H(\beta)<c$, since $v_H(\beta)=F_H(\beta,\frac12)$) holds
at some temperature if and only if it holds at the Nishimori temperature, and by Lemma~\ref{lem:lat-degradation} the set of $p$ at which it holds anywhere is an
initial segment of $(0,\frac12)$. Such methods therefore certify order off the Nishimori line only at
disorder strengths where the Nishimori point itself is certified; they see no further than the
Nishimori line. On a lattice with couplings $\beta J$ the certified set is a low-temperature half-line
through $\beta_N(p)$. On $\Z^2[H]$, at $p_0$, the set $\{\beta:w_H(\beta)<w_+\}$ contains $\beta_N$,
where $w_H(\beta_N)\le0.20530$, and it is bounded: it lies in $(\beta_{\rm hot},\beta_{\rm cold})$,
since at each $\beta$ outside that interval Theorem~\ref{thm:lat-main}(a), through
Corollary~\ref{cor:lat-UH}, gives almost-sure uniqueness, which excludes $w_H(\beta)<w_+$ by Proposition~\ref{prop:lat-NU}. Its connected
component containing $\beta_N$ lies strictly between the points $0.56075\beta_N$ and
$1.40088\beta_N$ of Remark~\ref{rem:lat-sharp}, where $w_H>w_+$.\footnote{Floating-point evaluations
suggest that the set is an interval.} The reentrance itself, however, comes from the uniqueness side, through
the non-monotone law of $|K_H(\beta)|$; a Chernoff--Peierls contour argument alone cannot produce it.

\subsection{Scope, and relation to earlier work}\label{sec:lat-scope}

\paragraph{What the lattice result says.} We are not aware of an earlier rigorous (computer-assisted)
proof that, for the Ising model with quenched iid $\pm J$ couplings on a planar, $\Z^2$-periodic graph of
maximum degree $4$, ferromagnetic long-range order present at the Nishimori temperature is absent at all
sufficiently low temperatures at the same disorder strength. Specifically, on $\Z^2[H]$, $H=H_4(1000,10)$, with
$\Pp(J=-1)=9/10000$, the Gibbs state is almost surely non-unique at the Nishimori temperature and
$\liminf_L\E\gibbs{M_L^2}\ge0.08$ there, in free and periodic boxes, while for each fixed
$\beta\ge1.9924\beta_N$ the Gibbs state is almost surely unique and $\E\gibbs{M_L^2}\to0$. We make no
claim for $\beta\in(0.4399\beta_N,0.5608\beta_N)$ or $\beta\in(1.4008\beta_N,1.9924\beta_N)$.

The mechanism is the classical decorated-bond one: at high temperature long decorations are
entropically weak and the direct bond dominates, whereas at low temperature frustrated decorations cap
the effective coupling. Exact paramagnet--ferromagnet--paramagnet reentrance on deterministic decorated
square lattices goes back to Syozi and Miyazima~\cite{Syozi1968,Miyazima1968}, and Kitatani, Miyashita
and Suzuki~\cite{KMS1985,KMS1986} obtained it from the exact free-fermion critical conditions of the
square lattice itself with a periodic pattern of competing nearest-neighbour couplings, relating it to
ground-state degeneracy; see also Miyashita~\cite{Miyashita2010} and Miyashita and
Vincent~\cite{MiyashitaVincent2001}. Exact reentrance with annealed bond randomness was found by Kasai,
Miyazima and Syozi~\cite{KasaiMiyazimaSyozi1969} and Kasai and Syozi~\cite{KasaiSyozi1973}.

For quenched random-bond models, Wolff and Zittartz~\cite{WolffZittartz1985,WolffZittartz1986} argued,
assuming that the critical temperature depends continuously on the couplings, that the Ising model on
$\Z^d$ with couplings $J$ and $aJ$, $-1<a<0$, has a backbending ferromagnetic boundary; their argument
does not cover the $\pm J$ case $a=-1$, which their qualitative phase diagram draws without
backbending, and does not involve the Nishimori line. Numerical evidence for reentrance in the
nearest-neighbour two-dimensional $\pm J$ model is reported by Nobre~\cite{Nobre2001} and Thomas and
Katzgraber~\cite{ThomasKatzgraber2011}, and on hierarchical lattices by Hinczewski and
Berker~\cite{HinczewskiBerker2005} and G\"uven et al.~\cite{Guven2008}; Carlson et
al.~\cite{CCCST1990} proved a reentrant local bifurcation of the magnetization recursion on the
Bethe lattice with fixed boundary conditions. Relative to these works, what we add, as far as we
are aware, is a rigorous statement for quenched iid $\pm J$ couplings, anchored at the Nishimori
temperature. The probabilistic tools are known: Horiguchi and Morita's quenched Peierls bound,
Newman's random-cluster domination with Harris's theorem, and the flip-and-ergodicity step. Their
combination through the exact decimation, and the certified inputs, are our contribution.
A related use of decimation followed by percolation is
Okuyama and Ohzeki's proof that spin-glass order is absent on certain Migdal--Kadanoff hierarchical
lattices~\cite{OkuyamaOhzeki2026}; there a single decimation step produces zero couplings with a
temperature-independent probability, and the conclusion holds at all temperatures.

\paragraph{Nishimori's verticality argument.} Section~\ref{sec:intro-prior} explains how
Theorem~\ref{thm:lat-main} bears on Nishimori's verticality argument~\cite{Nishimori1986}: its
conclusion does not hold on every two-dimensional periodic lattice, and at least one of its two
heuristic steps fails on $\Z^2[H]$; we do not determine which. Kitatani~\cite{Kitatani1992} proved the
verticality of the ferromagnetic--spin-glass boundary ``assuming an appropriate condition'', which
concerns a modified $\pm J$ model with correlated disorder~\cite[\S4.8]{NishimoriBook}. Nishimori,
Ohzeki and Okuyama observed that reentrance would violate this condition~\cite[\S III.F]{NOO2025};
we do not discuss it further.

\paragraph{A proxy, not the nearest-neighbour lattice.} The fully connected models of
Sections~\ref{sec:sk}--\ref{sec:certificates} and the decorated lattice are proxies for the
nearest-neighbour $\pm J$ model on $\Z^2$. Together they show that verticality below the
multicritical point is not universal: the Sherrington--Kirkpatrick boundary is vertical, the $p$-spin boundaries
with $p\ge3$ are reentrant near the triple point, and so is the boundary on one two-dimensional
periodic lattice at one disorder strength. They say nothing about the
nearest-neighbour lattices $\Z^d$. Two features of the present argument do not transfer to $\Z^2$. The
first is the exact decimation to iid backbone couplings, whose law is not of the form $\beta J$. The
second is the frustration-capped effective coupling of the gadget, which is what makes low-temperature
uniqueness possible. Chernoff--Peierls contour criteria, moreover, are anchored at the Nishimori
temperature for every gadget (Proposition~\ref{prop:lat-floor}), so they cannot produce reentrance by
themselves; on $\Z^2$ the
low-temperature side would need a different argument (Section~\ref{sec:discussion}).
\section{Discussion}\label{sec:discussion}

\paragraph{Verticality is model-dependent.} Gauge symmetry fixes the side on which the
ferromagnetic boundary can depart from the vertical line through the multicritical point. The
argument of the gauge transfer (Lemma~\ref{lem:gauge-transfer}) applies verbatim to
$p$-body couplings; \cite[\S3]{ItoiSakamoto2023} state it for two-body interactions, and the proof is
the same. Magnetic order at $(\beta,j_0)$ therefore forces
magnetic order at the Nishimori point $(2j_0,j_0)$. A non-vertical boundary must bend so that
cooling destroys order, which is what Theorem~\ref{thm:reduction} exhibits. Gauge symmetry does
not decide whether a departure happens. In the SK model
the boundary is vertical at every temperature below the multicritical point, and the order
parameter vanishes continuously across it (Theorem~\ref{thm:skline}). Verticality holds because
the zero-field spin-glass phase is marginally stable: its susceptibility equals $1/\beta$
(Theorem~\ref{thm:sk}(b),(d)), a consequence of the absence of a gap at the origin in the support of
the Parisi measure (Theorem~\ref{thm:accum}). For $p\ge3$ the transition from the spin-glass
phase to the ferromagnet is discontinuous in the replica picture~\cite{GNS2001,Nishimori2026}.
The competing free-energy branches then have different temperature dependence, and the boundary
bends. This is the mechanism of Nishimori's replica computation~\cite{Nishimori2026}, where the
curvature of the boundary is proportional to the difference of the specific heats of the two
phases at the triple point; Theorem~\ref{thm:nm-main} proves the corresponding two-point
statement near the triple point for every integer $p\ge3$. On the decorated lattice of
Section~\ref{sec:lattice} the mechanism is different again, and
classical~\cite{Syozi1968,Miyazima1968,KMS1986}: the effective backbone coupling first
strengthens and then weakens on cooling, because frustrated units cap it at low temperature.

\paragraph{Temperature chaos.} Nishimori, Ohzeki and Okuyama argue that, if a spin-glass phase
exists at finite temperature, a reentrant ferromagnet--spin-glass boundary forces temperature
chaos in that phase~\cite{NOO2025}. Stated precisely~\cite[\S II.A]{Nishimori2026}, reentrance
implies that either no finite-temperature spin-glass phase exists, or the two-temperature
overlap distribution is $\delta(x)$ for at least one pair of distinct temperatures in the
spin-glass phase. Nishimori applies this to the $p$-spin model~\cite{Nishimori2026}.
Theorems~\ref{thm:nm-main}, \ref{thm:p3} and~\ref{thm:p4} supply the reentrance premise
rigorously, in two-point form. We have not verified the general implication in the form needed
here, so we draw no conclusion about temperature chaos. In the SK model temperature chaos is
expected at the replica level~\cite{Nishimori2026}, while the boundary is vertical
(Theorem~\ref{thm:sk}); at that level the SK model is an example of temperature chaos without
reentrance.

\paragraph{Open problems.}
\begin{itemize}
\item \emph{The nearest-neighbour square lattice.} For the $\pm J$ model on $\Z^2$ with
$\Pp(J_e=-1)=p$ (here, as in Section~\ref{sec:lattice}, $p$ is the probability of a negative bond),
decide whether there is a $p$ at which the Gibbs state is almost surely
non-unique at the Nishimori temperature $\beta_N(p)$ but almost surely unique at some
$\beta>\beta_N(p)$ (reentrance), or whether non-uniqueness at $\beta_N(p)$ always persists at
every $\beta\ge\beta_N(p)$ (a vertical boundary). The obstruction is the one of
Section~\ref{sec:lat-floor}. Take $H$ a single edge, so that $\Z^2[H]=\Z^2$;
Lemma~\ref{lem:lat-BR} and Proposition~\ref{prop:lat-NU}, with their proofs, hold for every finite
two-terminal graph $H$. The Chernoff weight of $\beta J$ is nonincreasing in $\beta$ and equals
$2\sqrt{p(1-p)}$ for every $\beta\ge\beta_N(p)/2$, where the optimal $s$ has $2s\beta=\beta_N(p)$.
So Proposition~\ref{prop:lat-NU} gives almost-sure non-uniqueness on the nearest-neighbour lattice
at every $\beta\ge\beta_N(p)/2$ once $2\sqrt{p(1-p)}<w_+$, which holds for $p\le0.0192$; at such
$p$ the ferromagnetic phase extends to arbitrarily low temperature. With $H$ a single edge, the
contour criterion of Proposition~\ref{prop:lat-NU} gives non-uniqueness only when
$2\sqrt{p(1-p)}<w_+$, i.e.\ $p<0.019248\ldots$; by Proposition~\ref{prop:lat-floor} a criterion of
this Chernoff--Peierls type holds at some temperature only if it holds at the Nishimori temperature.
Newman's criterion gives uniqueness on $\Z^2$ only for $\beta\le\frac12\log2$, at every $p$,
because $|J_e|=1$. Near the multicritical point of the nearest-neighbour model, where reentrance is
debated, neither criterion applies, and the shape of the boundary there remains open. A proof of
reentrance would need an argument for
uniqueness at low temperature that is not anchored at the Nishimori temperature; ours rests on
an exact decimation to independent couplings with a frustration-capped law, which $\Z^2$ does not
have. A proof of verticality would need order at every low temperature at disorder strengths
where contour arguments do not certify the Nishimori point.
\item \emph{Real $p\in(2,3)$.} Nishimori poses the sign of $C_{\rm SG}-C_{\rm FM}$ at the triple
point as an open problem for every real $p>2$, including the limit
$p\to2^+$~\cite[p.~19]{Nishimori2026}. The triple-point system~\eqref{eq:nm-system} has a unique
solution for every real $p>2$, since the argument of Lemma~\ref{lem:nm-U} applies at every level
in $(0,1)$ (Remark~\ref{rem:nm-scope}); so the constants~\eqref{eq:nm-const} are defined there. Decide
whether $0<\Omega B<1$ and $A_{\rm FM}>A_{\rm SG}$ for every $p\in(2,3)$, and determine the
behaviour of $\kappa_p$ as $p\to2^+$. Our certificates treat each integer $p\le25$ separately, and
the bounds of Theorem~\ref{thm:nm-LP} hold for integer $p\ge26$; neither covers this range.
\item \emph{The curvature of the boundary.} Theorem~\ref{thm:nm-main} places, at each bias
$j_{0M}+\delta$ with $0<\delta<\bar\delta$, an ordered and a disordered temperature. Nishimori's
replica boundary near the triple point is the curve
$j_0-j_{0M}=K_b(\beta-\beta_1)^2+O(|\beta-\beta_1|^3)$ with
$K_b=(A_{\rm FM}-A_{\rm SG})/(2\beta_1\mu^p)$~\cite[(51)]{Nishimori2026}. Prove that for small
$\delta>0$ there is $\beta^*(\delta)$ such that at bias $j_0=j_{0M}+\delta$ one has
$\liminf_N\E\gibbs{|m|}>0$ for $2j_0\le\beta<\beta^*(\delta)$ and $\E\gibbs{|m|}\to0$ for
$\beta>\beta^*(\delta)$, with $\beta^*(\delta)-\beta_1\sim(\delta/K_b)^{1/2}$, as the replica
picture predicts. Our cold point lies at
$\beta_1+\sqrt2\,(\delta/K_b)^{1/2}$, a factor $\sqrt2$ beyond the replica crossing, which buys a
free-energy margin of order $\delta$; $\bar\delta$ is neither explicit nor uniform in $p$.
On the decorated lattice the corresponding questions are whether the set $\mathcal N$ of
Theorem~\ref{thm:lat-main}(c) is an interval and where in the two gaps left by that theorem its
endpoints lie.
\item \emph{The rate at the vertical boundary.} For $\beta>1$ and $\gamma\downarrow\beta$,
Theorem~\ref{thm:skline} gives $\max_{\Om(\beta,\gamma)}|\mu|\le(12(\gamma/\beta-1))^{1/4}$, and the
exact Nishimori-line magnetization improves the exponent to $1/2$ (Remark~\ref{rem:spiked}).
Determine the true exponent. The constant in the exponential concentration of
Theorem~\ref{thm:skline}(c) is not explicit.
\item \emph{Accumulation without computer assistance.} Prove Lemma~\ref{lem:accumulator} without
the computer-assisted inequality (V1). Decide whether a shape property of $\Gamma(q)/q$ of the
kind used in Section~\ref{sec:accum} also excludes gaps of the support away from the origin;
this would give another route to Lopatto's theorem that $\operatorname{supp}\mu_\beta$ is an
interval~\cite{Lopatto2026}.
\end{itemize}

\appendix

\section{Proofs and certificates for Sections~\ref{sec:prelim}--\ref{sec:pspin} and~\ref{sec:certificates}}\label{app:mf}

This appendix contains the proofs deferred from Sections~\ref{sec:prelim}--\ref{sec:pspin} and the
certificates of Sections~\ref{sec:accum} and~\ref{sec:certificates}: facts about the Parisi
functional (Appendix~\ref{app:mf-parisi}), the SK model
(Appendices~\ref{app:mf-sk}--\ref{sec:spiked}), accumulation at the origin with the proof of
Proposition~\ref{prop:accum-convex} and its certificate (Appendix~\ref{app:mf-accum}), the
$p$-spin free-energy bounds (Appendix~\ref{app:mf-pspin}), the directional derivatives
(Appendix~\ref{app:D}) and the certificates for $p=3$ and $p=4$ (Appendix~\ref{app:cert34}).

\subsection{The Parisi functional}\label{app:mf-parisi}

Two items of Proposition~\ref{prop:facts} combine sources. In \textup{(P1)}, Panchenko writes the argument of~\cite[\S2]{Panchenko2014}
without a deterministic field~\cite[p.~3]{Panchenko2014}. With the field included it runs unchanged: the
field term does not depend on the interpolation parameter, so it enters only through the Gibbs
measure, and the positivity bound~\cite[(2.9)]{Panchenko2014} is uniform over all measures on
$\{\pm1\}^N$. In \textup{(P4)}, Auffinger and Chen define $\Gamma_\mu$ through a
Girsanov-weighted expectation; the change of measure of~\cite[Lemma~8.5 and
Corollary~8.7]{JT2017} identifies it with the expectation along $X$.

\paragraph{The last statement of (P6): $0\in\operatorname{supp}\mu$ at $h=0$.}\label{app:P6}
Let $h=0$ and suppose
$q_m=\min\operatorname{supp}\mu>0$. On $[0,q_m)$ one has $\alpha_\mu=0$, so $X$ has no drift and
$\partial_x^2u_\mu(s,X_s)$ is a martingale. Hence $\Gamma_\mu'=\xi''\E[(\partial_x^2u_\mu)^2]$ is
nondecreasing there. It is strictly increasing on $(0,q_m)$: either $\xi''$ is strictly
increasing ($a_p\ne0$ for some $p\ge3$), or $\xi''$ is constant and
$\frac{\dd}{\dd s}\E[(\partial_x^2u_\mu)^2]=\xi''\E[(\partial_x^3u_\mu(s,X_s))^2]>0$ for
$0<s<q_m$, because $X_s$ has a positive density and $\partial_x^3u_\mu(s,\cdot)\not\equiv0$
(otherwise $\partial_xu_\mu(s,\cdot)$ would be linear with slope
$\partial_x^2u_\mu\ge C/\cosh^2x>0$~\cite[Proposition~2, (16)]{AC2015a}, contradicting
$|\partial_xu_\mu|\le1$). Thus $\Gamma_\mu$ is strictly convex on $[0,q_m]$ with
$\Gamma_\mu(0)=0$ (by evenness) and $\Gamma_\mu(q_m)=q_m$, so $\Gamma_\mu(s)<s$ on $(0,q_m)$.
Then $G_\mu'=-\frac12\xi''(\Gamma_\mu-s)>0$ on $(0,q_m)$, and $G_\mu(0)<G_\mu(q_m)=\min G_\mu$, a
contradiction.

\subsection{The SK model with a Curie--Weiss bias}\label{app:mf-sk}

As in Sections~\ref{sec:sk} and~\ref{sec:skline}, $\mu$ without subscript is a real variable.

\paragraph{The zero-field Parisi measure for $\beta\le1$.}
For $\beta\le1$ one has $\mu_\beta=\delta_0$, hence
$u_{\mu_\beta}(s,x)=\frac{\beta^2}2(1-s)+\lc x$ and $\theta_\beta=1$. Indeed, for the measure
$\delta_0$ this $u$ solves~\eqref{eq:pde}, $\Gamma_{\delta_0}(0)=\tanh^2(0)=0$, and \textup{(P4)} gives
$\Gamma_{\delta_0}'=\beta^2\E\operatorname{sech}^4(X_s)\le1$. So $\Gamma_{\delta_0}(s)\le s$,
$G_{\delta_0}$ is nondecreasing, and $0\in\operatorname{argmin}G_{\delta_0}$. By \textup{(P6)} and
\textup{(P2)}, $\delta_0$ is the unique minimizer.

\paragraph{Proof of Theorem~\ref{thm:sk}(a).}
\emph{The case $\gamma<1$.} By~\eqref{eq:quadratic} applied to $\mu_\beta$,
$\varphi_\beta(h)\le h^2/2$. Hence, for $\mu\ne0$,
\[
F^{\rm SK}(\beta,\gamma\mu)-\frac{\gamma\mu^2}2-F^{\rm SK}(\beta,0)
\le\frac{\gamma\mu^2}2(\gamma-1)<0,
\]
so $\Om(\beta,\gamma)=\{0\}$ and $F(\beta,\cdot)=F^{\rm SK}(\beta,0)$ on $[0,1)$. The
functions $F_N(\beta,\cdot)$ are convex and converge pointwise to a function that is constant
on $[0,1)$. At an interior point $\gamma\in(0,1)$ this gives
$\limsup_N\partial_\gamma F_N(\beta,\gamma)\le\partial^+_\gamma F(\beta,\gamma)=0$, so
$\E\gibbs{m^2}\to0$.

\emph{The case $\beta>1$ and $1\le\gamma<\beta$.} Put $J_0=\gamma/\beta<1$. The coefficient of
$\sigma_i\sigma_j$ in $H_N$ is $\beta J_{ij}$ with $J_{ij}\sim N(J_0/N,1/N)$. The density of
$J_{ij}$ is proportional to $e^{-NJ^2/2}e^{J_0J}$, an even function times $e^{\beta_NJ}$ with
$\beta_N=J_0$. Nishimori's gauge identity~\cite{Nishimori1981} (for SK and mixed $p$-spin
models, see~\cite[Lemma~2.1 and \S3]{ItoiSakamoto2023}) states that
\[
\E\gibbs{\sigma_i\sigma_j}_\beta=\E\bigl[\gibbs{\sigma_i\sigma_j}_{\beta_N}\gibbs{\sigma_i\sigma_j}_\beta\bigr],
\qquad
a_{ij}:=\E\gibbs{\sigma_i\sigma_j}_{\beta_N}=\E\gibbs{\sigma_i\sigma_j}_{\beta_N}^2\ge0 ,
\]
where $\gibbs{\cdot}_{\beta_N}$ is the Gibbs measure with the same disorder law at inverse
temperature $\beta_N$. By Cauchy--Schwarz, as in~\cite[eq.~(26)]{ItoiSakamoto2023},
\[
\E\gibbs{m^2}_{\beta,\gamma}\le\frac1N+\frac1{N^2}\sum_{i\ne j}a_{ij}^{1/2}
\le\frac1N+\Bigl(\E\gibbs{m^2}_{J_0,J_0^2}\Bigr)^{1/2},
\]
where $(J_0,J_0^2)$ is the Nishimori point of the same bias. Since $J_0^2<1$, the first case
gives $\E\gibbs{m^2}_{J_0,J_0^2}\to0$, and hence
$\E\gibbs{m^2}_{\beta,\gamma}\to0$.

\emph{The bound on $\varphi_\beta$ and $\Om=\{0\}$.} Integrating
$\partial_\gamma F_N=\E\gibbs{m^2}/2$ and using bounded convergence gives
$F(\beta,\gamma)=F^{\rm SK}(\beta,0)$ for all $\gamma<\max(1,\beta)$. By~\eqref{eq:chen}, this
means $\varphi_\beta(\gamma\mu)\le\gamma\mu^2/2$ for $|\mu|\le1$. The same inequality holds for
$|\mu|>1$. Indeed, $|\partial_hF^{\rm SK}|\le1$ gives
$\varphi_\beta(\gamma\mu)\le\varphi_\beta(\gamma)+\gamma(|\mu|-1)$, and $\varphi_\beta(\gamma)\le\gamma/2$, so
$\varphi_\beta(\gamma\mu)-\gamma\mu^2/2\le-\gamma(|\mu|-1)^2/2\le0$. Letting $\gamma\uparrow\max(1,\beta)$
yields $\varphi_\beta(h)\le h^2/(2\max(1,\beta))$ for all $h$. For $\gamma<\max(1,\beta)$ and
$\mu\ne0$ it then follows that
$\varphi_\beta(\gamma\mu)-\gamma\mu^2/2\le\frac{\gamma\mu^2}2(\gamma/\max(1,\beta)-1)<0$, so
$\Om=\{0\}$.\qed

\subsection{The boundary line}\label{app:mf-skline}

As in Sections~\ref{sec:sk} and~\ref{sec:skline}, $\mu$ without subscript is a real variable.
This subsection proves Lemmas~\ref{lem:concentration} and~\ref{lem:fixed-field}, Lemma~\ref{lem:y}
on the function $y$ of Section~\ref{sec:skline}, and Corollary~\ref{cor:envelope};
Remarks~\ref{rem:right-nbhd} and~\ref{rem:beta-le-one} supplement Section~\ref{sec:skline}.

\begin{proof}[Proof of Lemma~\ref{lem:concentration}]
If $\gamma=0$, then $\Om=[-1,1]$ and $U=\emptyset$; so let $\gamma>0$ and $U\ne\emptyset$. For
$\mu\in[-1,1]$ put
\[
W_\mu=\frac1N\log Z^{\rm SK}_N(\beta,\gamma\mu)-\frac{\gamma\mu^2}2,\qquad
\bar W_\mu=\E W_\mu,\qquad
F_\mu=F^{\rm SK}(\beta,\gamma\mu)-\frac{\gamma\mu^2}2,
\]
and $\Delta_\mu=W_\mu-\bar W_\mu$. Let $\Theta_N=\{-1,-1+\frac2N,\dots,1\}$ be the set of values of
$m$. Thus $\Om=\operatorname{argmax}_{[-1,1]}F_\mu$.

\emph{Step 1: splitting by the value of $m$.} If $\Theta_N\cap U=\emptyset$ there is nothing to
prove. If $m(\sigma)=\mu$, then $\gamma Nm^2/2=\gamma\mu\sum_i\sigma_i-\gamma N\mu^2/2$, so
$Z_N\gibbs{\ind\{m\in U\}}\le\sum_{\mu\in\Theta_N\cap U}e^{NW_\mu}$. Also $m^2\ge2\mu m-\mu^2$ and
$\gamma\ge0$ give $Z_N\ge e^{NW_\mu}$ for every $\mu$. Hence
\[
\gibbs{\ind\{m\in U\}}\le(N+1)\exp N\Bigl(\max_{\mu\in\Theta_N\cap U}W_\mu-\max_{\mu\in\Theta_N}W_\mu\Bigr).
\]

\emph{Step 2: a deterministic gap.} The functions $h\mapsto F^{\rm SK}_N(\beta,h)$ are
$1$-Lipschitz and converge pointwise, hence uniformly on $[-\gamma,\gamma]$. So
$\bar W_\mu\to F_\mu$ uniformly on $[-1,1]$, and $\mu\mapsto F_\mu$ is $2\gamma$-Lipschitz. The
closure of $U$ is compact and disjoint from $\Om$, so
$\eta:=\max_{[-1,1]}F_\mu-\sup_{\mu\in U}F_\mu>0$. Choose $N_1$ such that
$\sup_{[-1,1]}|\bar W_\mu-F_\mu|\le\eta/16$ and $2\gamma/N\le\eta/16$ for $N\ge N_1$. Every point
of $[-1,1]$ is within $1/N$ of $\Theta_N$, so
$\max_{\Theta_N}\bar W_\mu\ge\max_{[-1,1]}F_\mu-\eta/8$ and
$\max_{\Theta_N\cap U}\bar W_\mu\le\sup_UF_\mu+\eta/16$. Hence
$\max_{\Theta_N}\bar W_\mu-\max_{\Theta_N\cap U}\bar W_\mu\ge13\eta/16\ge3\eta/4$.

\emph{Step 3: a good event.} On $E_N=\{\max_{\mu\in\Theta_N}|\Delta_\mu|<\eta/8\}$, Step~2 gives
$\max_{\Theta_N\cap U}W_\mu-\max_{\Theta_N}W_\mu\le-3\eta/4+\eta/4=-\eta/2$, so
$\gibbs{\ind\{m\in U\}}\le(N+1)e^{-N\eta/2}$. Off $E_N$, $\gibbs{\ind\{m\in U\}}\le1$.

\emph{Step 4: concentration.} Write $\gibbs{\cdot}^{\rm SK}_{\beta,h}$ for the Gibbs average of
the SK model with field $h$, i.e.\ the measure in $Z^{\rm SK}_N(\beta,h)$. Since
$\partial_{g_{ij}}\log Z^{\rm SK}_N(\beta,h)=\beta N^{-1/2}\gibbs{\sigma_i\sigma_j}^{\rm SK}_{\beta,h}$,
for every field $h$ the map $(g_{ij})_{i<j}\mapsto N^{-1}\log Z^{\rm SK}_N(\beta,h)$ is
$\beta/\sqrt{2N}$-Lipschitz. For $\beta>0$, Gaussian concentration for Lipschitz
functions~\cite[eq.~(1.22)]{Ledoux1999}, applied to this map and to its negative, gives
$\Pp(|\Delta_\mu|\ge t)\le2e^{-Nt^2/\beta^2}$ for every $\mu$ and every $t>0$, and a union bound
over the $N+1$ fields $\gamma\mu$, $\mu\in\Theta_N$, gives
$\Pp(\max_{\Theta_N}|\Delta_\mu|\ge t)\le2(N+1)e^{-Nt^2/\beta^2}$. With $t=\eta/8$,
$\Pp(E_N^c)\le2(N+1)e^{-N\eta^2/(64\beta^2)}$. For $\beta=0$, $\Delta\equiv0$ and $E_N^c=\emptyset$.

Therefore
$\E\gibbs{\ind\{m\in U\}}\le(N+1)e^{-N\eta/2}+2(N+1)e^{-N\eta^2/(64\beta^2)}$ for $N\ge N_1$,
without the second term if $\beta=0$.
The prefactors and the finitely many $N<N_1$ are absorbed into $K$, which is not explicit because
$N_1$ is not. For the particular case take $U=\{x\in[-1,1]:|x|\ge\varepsilon\}$, which is at
distance $\varepsilon-\max_\Om|\mu|>0$ from $\Om$. Finally,
$\E\gibbs{m^2}\le\varepsilon^2+\E\gibbs{\ind\{|m|\ge\varepsilon\}}$ because $m^2\le1$; let
$N\to\infty$ and then $\varepsilon\downarrow\max_\Om|\mu|$.
\end{proof}

\begin{proof}[Proof of Lemma~\ref{lem:fixed-field}]
(a) Since $\E\gibbs{m^2}_{\beta,s}$ is nondecreasing in $s$,
\[
F_N(\beta,\gamma')-F_N(\beta,\gamma)=\frac12\int_\gamma^{\gamma'}\E\gibbs{m^2}_{\beta,s}\,\dd s
\le\frac{\gamma'-\gamma}2\E\gibbs{m^2}_{\beta,\gamma'};
\]
let $N\to\infty$, using~\eqref{eq:chen}
for the existence of the limits. (b) Let $h_*=\gamma\mu_*$; then $\mu'=h_*/\gamma'\in[-1,1]$.
Formula~\eqref{eq:chen} at $\gamma'$ with the trial $\mu'$, and at $\gamma$ with the maximizer
$\mu_*$, gives
\[
F(\beta,\gamma')-F(\beta,\gamma)\ge\Bigl[F^{\rm SK}(\beta,h_*)-\frac{h_*^2}{2\gamma'}\Bigr]
-\Bigl[F^{\rm SK}(\beta,h_*)-\frac{h_*^2}{2\gamma}\Bigr]=\frac{\gamma\mu_*^2(\gamma'-\gamma)}{2\gamma'} .
\]
(c) Combine (a) and (b).
\end{proof}

\begin{remark}[Why a right neighbourhood is needed]\label{rem:right-nbhd}
Lemma~\ref{lem:fixed-field}(c) uses the Gibbs bound at $\gamma'>\gamma$. The bound at $\gamma$
itself would not suffice. For convex $F_N\to F$ one has only
\[
\partial^-F(\gamma)\le\liminf_N\partial_\gamma F_N(\gamma)\le\limsup_N\partial_\gamma F_N(\gamma)\le\partial^+F(\gamma),
\]
so $\E\gibbs{m^2}_{\beta,\beta}\to0$ yields only the trivial $\partial^-_\gamma F(\beta,\beta)\le0$,
whereas by Lemma~\ref{lem:fixed-field}(b) a maximizer $\mu_*$ at $\gamma=\beta$ forces
$\partial^+_\gamma F(\beta,\beta)\ge\mu_*^2/2$. The inequality
$\partial^+F(\gamma)\le\liminf_N\partial_\gamma F_N(\gamma)$ fails in general: take $F_N$ convex,
vanishing on $[0,\beta+1/N]$, with slope $1$ after $\beta+2/N$.
\end{remark}

The functions $\bar x$, $y$, $Y$ and the constants $\delta_c$, $\tau$ in the next lemma and its proof
are those defined before Theorem~\ref{thm:skline}.

\needspace{5\baselineskip}
\begin{lemma}[The function $y$]\label{lem:y}\leavevmode
\begin{enumerate}
\item[\textup{(a)}] $r(x)=\lc(x)/x^2$ is strictly decreasing on $(0,\infty)$, with $r(0+)=1/2$ and
$r(\infty)=0$. Hence $\bar x(b)=r^{-1}(1/(2b^2))$ is well defined for $b>1$.
\item[\textup{(b)}] $\bar x$ is continuous on $(0,\infty)$ and $\bar x(b)\to0$ as $b\downarrow1$.
Hence $y$ and $Y$ are continuous on $(-1,\infty)$.
\item[\textup{(c)}] $y(\delta)/\delta\to12$ as $\delta\downarrow0$, and $y(\delta)\le12\delta$ for
$0\le\delta\le1/20$.
\item[\textup{(d)}] $y(\delta)<1$ if and only if $\delta<\delta_c$.
\end{enumerate}
\end{lemma}

\begin{proof}
(a) One has $r'=-k/x^3$ with $k(x)=2\lc x-x\tanh x$, $k(0)=0$ and
$k'(x)=\sech^2x\,(\sinh x\cosh x-x)>0$ for $x>0$.

(b) $r^{-1}$ is continuous on $(0,1/2)$ and $r^{-1}(s)\to0$ as $s\uparrow1/2$.

(c) For $x\ge0$, $\tanh x\ge x-x^3/3$, since $\tanh x\le x$ gives
$(\tanh x-x+x^3/3)'=x^2-\tanh^2x\ge0$. For $0\le x\le\sqrt3$, $\tanh x\le x-x^3/3+2x^5/15$: the
difference $g(x)=x-x^3/3+2x^5/15-\tanh x$ has $g(0)=0$ and, since $\tanh x\ge x-x^3/3\ge0$ there,
$g'(x)=\tanh^2x-x^2+2x^4/3\ge(x-x^3/3)^2-x^2+2x^4/3=x^6/9\ge0$. Integrating gives
\[
\frac{x^2}2-\frac{x^4}{12}\le\lc x\quad(x\ge0),\qquad
\lc x\le\frac{x^2}2-\frac{x^4}{12}+\frac{x^6}{45}\quad(0\le x\le\sqrt3).
\]
Let $b=1+\delta$, so $r(\bar x)=1/(2b^2)$. The lower bound gives $\bar x^2\ge6(1-b^{-2})$, and the
upper bound gives $\bar x^2\le6(1-b^{-2})/(1-4\bar x^2/15)$ once $\bar x\le\sqrt3$. Since
$\bar x\to0$ and $6(1-b^{-2})/(b^4\delta)\to12$, the limit follows. For the inequality let $0<\delta\le1/20$ (at $\delta=0$ both sides vanish) and put
$X=b^2\sqrt{12\delta}$. Then $X\le0.854<\sqrt3$, so
$r(X)\le\frac12-\frac{X^2}{12}+\frac{X^4}{45}$. Substituting $X^2=12\delta b^4$ and multiplying by
$2b^2$, the inequality $\frac12-\frac{X^2}{12}+\frac{X^4}{45}\le\frac1{2b^2}$ becomes
$\delta^2R(\delta)\ge0$ with
\[
R(\delta)=11+30\delta+40\delta^2+30\delta^3+12\delta^4+2\delta^5-\tfrac{32}5(1+\delta)^{10} .
\]
The terms $30\delta,\dots,2\delta^5$ are nonnegative and
$(1+\delta)^{10}\le e^{10\delta}\le e^{1/2}<\frac{33}{20}$, so
$R(\delta)\ge11-\frac{32}5\cdot\frac{33}{20}=\frac{11}{25}>0$. Hence $r(X)\le1/(2b^2)=r(\bar x)$, and
$\bar x\le X$ by~(a), i.e.\ $y(\delta)\le12\delta$.

(d) $y(\delta)<1$ if and only if $\bar x(b)<b^2$, i.e.\ $r(b^2)<1/(2b^2)$, i.e.\ $\lc s<s/2$ at
$s=b^2$. The function $\lc s-s/2$ is convex, vanishes at $0$ and has slope $-1/2$ there, so it is
negative exactly on $(0,s_c)$, where $s_c>0$ solves $\cosh s=e^{s/2}$. With $u=e^{s/2}$ this is
$u^4-2u^3+1=(u-1)(u^3-u^2-u-1)=0$, so $s_c=2\log\tau$ and $\delta_c=\sqrt{s_c}-1$.
\end{proof}

\begin{proof}[Proof of Corollary~\ref{cor:envelope}]
For $\gamma\ge\max(\beta,|h|)$, the trial $\mu=h/\gamma$ in~\eqref{eq:chen} gives
$\varphi_\beta(h)\le\frac{h^2}{2\gamma}+F(\beta,\gamma)-F^{\rm SK}(\beta,0)$. Set
$\gamma=\beta(1+\delta)$ with $0\le\delta\le1/20$. By Corollary~\ref{cor:free-energy-line}(c) and
$Y(\delta)^{1/2}\le(12\delta)^{1/2}$,
\[
\frac{h^2}{2\beta}-\varphi_\beta(h)\ge\frac\delta2\Bigl[\frac{h^2}{\beta(1+\delta)}-\beta\sqrt{12\delta}\Bigr].
\]
Let $0<|h|\le\beta$ and choose $\delta=h^4/(192\beta^4)$, i.e.\ $\sqrt{12\delta}=h^2/(4\beta^2)$.
Then $\delta\le1/192$ and $\gamma\ge\beta\ge|h|$. The bracket is at least
$\frac{h^2}{2\beta}-\frac{h^2}{4\beta}=\frac{h^2}{4\beta}$, which gives $h^6/(1536\beta^5)$. For
$|h|>\beta$, evenness and the $1$-Lipschitz property give
$\varphi_\beta(h)\le\varphi_\beta(\beta)+(|h|-\beta)$, and the case $h=\beta$ gives the claim.
\end{proof}

\begin{remark}[Continuity at $\gamma=1$ for $\beta\le1$]\label{rem:beta-le-one}
Let $0<\beta\le1$. Then $\sup_{\Om(\beta,\gamma)}|\mu|\to0$ and
$\limsup_N\E\gibbs{m^2}_{\beta,\gamma}\to0$ as $\gamma\to1$. Indeed,
$(\mu,\gamma)\mapsto F^{\rm SK}(\beta,\gamma\mu)-\gamma\mu^2/2$ is jointly continuous, so if
$\gamma_n\to1$, $\mu_n\in\Om(\beta,\gamma_n)$ and $\mu_n\to\bar\mu$, then passing to the limit in
the maximality of $\mu_n$ gives $\bar\mu\in\Om(\beta,1)=\{0\}$ (Lemma~\ref{lem:cw}
or~\cite[Proposition~1.1(i)]{DeyKang2026}). Lemma~\ref{lem:concentration} then gives
$\limsup_N\E\gibbs{m^2}_{\beta,\gamma}\le\max_{\Om(\beta,\gamma)}\mu^2\to0$. This gives no rate
in $\gamma-1$. Since $(\beta,\gamma)$ is magnetized for $\beta\le1<\gamma$ by
Theorem~\ref{thm:sk}(c), with Corollary~\ref{cor:continuity} the order parameter vanishes
continuously along the whole boundary $\gamma=\max(1,\beta)$.
\end{remark}

\subsection{The Nishimori line as a spiked matrix model}\label{sec:spiked}

As in Sections~\ref{sec:sk} and~\ref{sec:skline}, $\mu$ without subscript is a real variable.
Remark~\ref{rem:spiked} below records the limiting magnetization on the Nishimori line, which
follows from the known Nishimori-line free energy, and the resulting improvements of the exponent
in Theorem~\ref{thm:skline} and of the margin in Corollary~\ref{cor:envelope}.

On the Nishimori line the magnetization is known exactly, through a rank-one estimation
problem. For $\rho\ge0$ let $Z\sim N(0,1)$, $w=\rho+\sqrt\rho\,Z$,
$\iota(\rho)=\rho-\E\lc w$, and for $\lambda>0$
\[
\Psi(\rho,\lambda)=\frac\lambda4+\frac{\rho^2}{4\lambda}-\frac\rho2+\iota(\rho).
\]
Let $q^*(\lambda)$ be the largest solution in $[0,1]$ of
$q=\E\tanh(\lambda q+\sqrt{\lambda q}\,Z)$. Then $q^*(\lambda)=0$ for $\lambda\le1$,
$q^*(\lambda)\in(0,1)$ for $\lambda>1$, $q^*$ is continuous, and
$q^*(\lambda)\le4(\lambda-1)/\lambda^2$ for $1<\lambda\le3/2$. Moreover
$\rho\mapsto\Psi(\rho,\lambda)$ has the unique minimizer $\lambda q^*(\lambda)\in[0,\lambda]$ on
$[0,\infty)$. These facts, except the bound on $q^*$, follow from the strict concavity of
$G(\rho):=\E\tanh^2w$, which has slope $1$ at $\rho=0$, and from the Nishimori identity
$\E\tanh w=\E\tanh^2w$; cf.~\cite[\S3.2, Lemma~3.3 and Appendix~B.2]{DAM2016}. For the bound
on $q^*$: Stein's lemma and the Nishimori identity give $G'(\rho)=\E\sech^4w\le\E\sech^2w=1-G(\rho)$,
hence $G(\rho)\le1-e^{-\rho}$. The value $\rho^*=\lambda q^*$ satisfies $\rho^*=\lambda G(\rho^*)$, so for
$\lambda>1$, where $\rho^*>0$, one has
$\rho^*\le\lambda(1-e^{-\rho^*})\le\lambda(\rho^*-\rho^{*2}/2+\rho^{*3}/6)$. Since
$\rho^*\le\lambda\le3/2$, $\rho^{*2}/6\le\rho^*/4$, so $1\le\lambda(1-\rho^*/4)$, i.e.\
$\rho^*\le4(\lambda-1)/\lambda$; divide by $\lambda$.

\begin{remark}[Exact magnetization on the Nishimori line]\label{rem:spiked}
Let $b>0$ and $\lambda=b^2$. Let $X$ be uniform on $\{\pm1\}^N$, let $Z_{ij}$, $i<j$, be
independent standard Gaussians independent of $X$, and let
$Y_{ij}=\sqrt{\lambda/N}\,X_iX_j+Z_{ij}$: the Rademacher spiked Wigner, or
$\mathbb Z_2$-synchronization, model. The variables $g_{ij}=Z_{ij}X_iX_j$ are independent
standard Gaussians independent of $X$. Write $H_N(\cdot\,;\beta,\gamma)$ for the Hamiltonian of
Section~\ref{sec:sk} with disorder $g$. With $\tau=x\circ X$ (entrywise product),
$\sqrt{\lambda/N}\sum_{i<j}Y_{ij}x_ix_j=H_N(\tau;b,b^2)-b^2/2$. So if $x$ is drawn from the
posterior law of $X$ given $Y$, then, given $(g,X)$, $\tau$ has the law $\gibbs{\cdot}_{b,b^2}$.
Consequently, for every $N$,
\[
F_N(b,b^2)=\log2+\frac\lambda2-\frac{I(X;Y)}N,\qquad
\E\gibbs{m^2}_{b,b^2}=\frac1N+\Bigl(1-\frac1N\Bigr)\bigl(1-\mathrm{mmse}_N(\lambda)\bigr),
\]
where $\mathrm{mmse}_N(\lambda)=\E(X_1X_2-\E[X_1X_2\mid Y])^2$. The limit of the free
energy on the Nishimori line is the case $p=2$, $h=0$, $J=\lambda$ of Korada and
Macris~\cite[Theorem~3 and Appendix~A]{KoradaMacris2009}: their model at $p=2$ and $h=0$ has
couplings $N(J/N,J/N)$ at inverse temperature $1$ and no field, which at $J=b^2$ is the
Hamiltonian $H_N(\cdot\,;b,b^2)-b^2/2$. The limit
of the mutual information is also Deshpande, Abbe and Montanari's
Theorem~4.3~\cite{DAM2016}, whose Gaussian model additionally observes the diagonal entries,
which carry no information about $X$. Since the minimizer of $\Psi(\cdot,\lambda)$ lies in
$[0,\lambda]$, both results read, in our notation (Deshpande, Abbe and Montanari write $\gamma$
for $\rho$ and $I$ for $\iota$),
\begin{equation}\label{eq:spiked-fe}
\lim_NF_N(b,b^2)=\log2+\frac\lambda2-\min_{\rho\ge0}\Psi(\rho,\lambda).
\end{equation}
From~\eqref{eq:spiked-fe} one obtains
\begin{equation}\label{eq:spiked-m}
\lim_N\E\gibbs{m^2}_{b,b^2}=q^*(\lambda)^2 .
\end{equation}
Indeed, $\Phi_N(\lambda)=F_N(\sqrt\lambda,\lambda)$ satisfies
$\Phi_N'(\lambda)=\frac14(1+\E\gibbs{m^2}_{\sqrt\lambda,\lambda})$, which by the two identities above
is the I-MMSE relation $\frac{\dd}{\dd\lambda}I(X;Y)=\frac{N-1}4\,\mathrm{mmse}_N(\lambda)$~\cite[(150)--(151)
with $\varepsilon=0$]{DAM2016}, and
$\Phi_N$ is convex because $\mathrm{mmse}_N$ is nonincreasing in $\lambda$ (the observation at a
smaller $\lambda$ is a degraded version of the one at a larger $\lambda$). Since the minimizer
$\lambda q^*(\lambda)$ of $\Psi(\cdot,\lambda)$ is unique, the limit in~\eqref{eq:spiked-fe} is
differentiable in $\lambda$ with derivative $\frac14(1+q^*(\lambda)^2)$, by the envelope theorem,
and derivatives of convex functions converge where the limit is differentiable. For $b\le1$, \eqref{eq:spiked-m} holds
without~\eqref{eq:spiked-fe}: $\Om(b,b^2)=\{0\}$ by Lemma~\ref{lem:cw}, and
Lemma~\ref{lem:concentration} applies.

The free-energy identity~\eqref{eq:spiked-fe} is thus known~\cite{KoradaMacris2009,DAM2016}; see also~\cite[Theorem~1]{BarbierMacris2019}.
Through the second display above, \eqref{eq:spiked-m} is the statement that
$\mathrm{mmse}_N(\lambda)\to1-q^*(\lambda)^2$. Deshpande, Abbe and Montanari state this limit for
the two-group stochastic block model~\cite[Theorem~1.4]{DAM2016}. For symmetric rank-one estimation,
Lelarge and Miolane prove the limit of the matrix minimum mean-square error for a general prior at
every $\lambda$ at which the limiting mutual information is
differentiable~\cite[Theorem~1]{LelargeMiolane2019}, and, for a prior with finite support, at every
$\lambda$ at which their replica-symmetric potential has a unique
maximizer~\cite[\S3.2 and Corollary~17]{LelargeMiolane2019} (see also~\cite{BDMKLZ2016}).
With the prior uniform on $\{\pm1\}$ their channel~\cite[(2)]{LelargeMiolane2019} is the present
one, their matrix minimum mean-square error is $\mathrm{mmse}_N(\lambda)$, and their potential is
$\lambda/4-\Psi(\lambda q,\lambda)$, whose maximizer $q^*(\lambda)$ is unique for every $\lambda>0$.
Their result therefore gives $\mathrm{mmse}_N(\lambda)\to1-q^*(\lambda)^2$ for every $\lambda>0$, as
they note for this prior~\cite[p.~6]{LelargeMiolane2019}. We include the short
derivation of~\eqref{eq:spiked-m} only for completeness and claim no novelty for either identity.

Used at the Nishimori point in place of Lemmas~\ref{lem:cw} and~\ref{lem:concentration} in the
proof of Theorem~\ref{thm:skline}, \eqref{eq:spiked-m} improves the exponent. Let $\gamma\ge\beta$ and $\delta=\gamma/\beta-1$. By
Lemma~\ref{lem:gauge-transfer}, $\limsup_N\E\gibbs{m^2}_{\beta,\gamma}\le q^*((1+\delta)^2)$, and
as in Theorem~\ref{thm:skline}(b), using the continuity of $q^*$,
$\max_{\Om(\beta,\gamma)}\mu^2\le q^*((1+\delta)^2)$. For $0\le\delta\le\sqrt{3/2}-1$ this is at
most $4(2\delta+\delta^2)/(1+\delta)^4\le8\delta$, so $\max_\Om|\mu|\le(8\delta)^{1/2}$.
(Numerically, $q^*((1+\delta)^2)\approx2\delta$ for small $\delta$.) The margin of
Corollary~\ref{cor:envelope} improves as well. Let $0<|h|\le\beta$, $\delta=h^2/(18\beta^2)\le1/18<\sqrt{3/2}-1$
and $\gamma=\beta(1+\delta)$. By Lemma~\ref{lem:fixed-field}(a),
Corollary~\ref{cor:free-energy-line}(b) and the bound just proved,
$F(\beta,\gamma)-F^{\rm SK}(\beta,0)\le\frac{\beta\delta}2\,q^*((1+\delta)^2)\le4\beta\delta^2$. The trial
$\mu=h/\gamma$ in~\eqref{eq:chen}, as in the proof of Corollary~\ref{cor:envelope}, and
$1/(1+\delta)\ge1-\delta$ give
\[
\frac{h^2}{2\beta}-\varphi_\beta(h)\ge\frac\delta2\Bigl[\frac{h^2(1-\delta)}\beta-8\beta\delta\Bigr]
\ge\frac{c(1-9c)}2\,\frac{h^4}{\beta^3}=\frac{h^4}{72\beta^3},
\]
with $c=1/18$, where the middle step writes $\delta=ch^2/\beta^2$ and uses $h^2\le\beta^2$. For
$|h|>\beta$ the Lipschitz step of Corollary~\ref{cor:envelope} with the case $h=\beta$ gives
$(|h|-\beta)^2/(2\beta)+\beta/72$.
These improvements rely on the cited identity~\eqref{eq:spiked-fe}. The true exponent of
$\max_{\Om(\beta,\gamma)}|\mu|$ as $\gamma\downarrow\beta$ is not known.
\end{remark}

\subsection{Accumulation at the origin}\label{app:mf-accum}

As in Section~\ref{sec:accum}, $\xi(q)=\beta^2q^2/2$ with $\beta>1$, the field is $0$, and
$\mu=\mu_\beta$ is the zero-field Parisi measure. We write $u=u_\mu$, $\alpha=\alpha_\mu$,
$\Gamma=\Gamma_\mu$, $G=G_\mu$ and $X$ for the objects of Section~\ref{sec:prelim}, and
$\theta=\theta_\beta=\partial_x^2u(0,0)$. In the proofs of Lemma~\ref{lem:gap} and Proposition~\ref{prop:accum-convex} and in
Lemma~\ref{lem:kernel}, $q_*$, $\lambda$, $t$, $\psi$, $U$, $Y$, $q(r)$, $p_r$,
$\langle\cdot\rangle_r$, $\kappa$ and $\mathcal F$ are as in Section~\ref{sec:accum}, under its
assumption that $0$ is not an accumulation point of $\operatorname{supp}\mu$. This subsection proves Lemma~\ref{lem:gap}, Lemma~\ref{lem:uxxx} and
Proposition~\ref{prop:accum-convex}, which Section~\ref{sec:accum} combines into
Theorem~\ref{thm:accum}, and certifies (V1) (Appendix~\ref{sec:accum-cert}).

\subsubsection{The first gap}

\begin{proof}[Proof of Lemma~\ref{lem:gap}]
(i) By~\eqref{eq:gap-ch}, $e^{\lambda U}$ solves $\partial_rf+\frac12\partial_x^2f=0$ on $[0,t)$,
and it grows at most like $e^{\lambda|x|}$ because $|\partial_xU|\le1$ by \textup{(P3)}. For a
Brownian motion $W$ started at $0$, $e^{\lambda U(r,W_r)-\lambda U(0,0)}$ is therefore a positive
martingale with mean $1$ on $[0,t)$. By Girsanov's theorem, under the tilted law $W$ solves the
equation of $Y$, which is unique in law because $\partial_xU$ is Lipschitz. So the law of $Y_r$
is $e^{\lambda U(r,x)-\lambda U(0,0)}$ times the $N(0,r)$ density, which is $p_r$.

(ii) By \textup{(P4)}, $\Gamma\in C^1[0,1]$, so $q\in C^1[0,t]$ and
$q'(r)=\E[\partial_x^2U(r,Y_r)^2]$. For $r\le t-\varepsilon$, \eqref{eq:gap-ch} and $|\psi'|\le1$
show that $U$ is smooth with bounded $x$-derivatives of every positive order. Differentiating the
equation for $U$ and applying It\^o's formula,
$\dd\,\partial_x^2U(r,Y_r)=-\lambda\,\partial_x^2U(r,Y_r)^2\,\dd r+\partial_x^3U(r,Y_r)\,\dd B_r$,
so $\frac{\dd}{\dd r}\E[\partial_x^2U(r,Y_r)^2]=\E[\partial_x^3U(r,Y_r)^2-2\lambda\,\partial_x^2U(r,Y_r)^3]$,
which is continuous on $[0,t)$. These are the formulas of~\cite[Lemma~8.4]{JT2017} on the gap,
where $\alpha\equiv\lambda$, in the time variable $r=\beta^2s$. In the Girsanov-weighted form
of~\cite[Proposition~3]{AC2015a} they are~(20)--(22) there with $\xi'''=0$ and $\alpha=\lambda$; the
two forms are related by the change of measure of~\cite[Lemma~8.5 and Corollary~8.7]{JT2017}.
With (i) they give~\eqref{eq:gap-moments}. At $r=0$ we have $Y_0=0$, and $\partial_xU(0,0)=\partial_x^3U(0,0)=0$
by evenness. (Equivalently, $\partial_x^2U+\lambda(\partial_xU)^2$ evaluated along $Y$ is a
martingale on the gap, as follows from the representation in the proof
of~\cite[Lemma~16]{JT2016}.)

(iii) By \textup{(P6)}, $\Gamma(q)=q$ and $\xi''(q)\E[\partial_x^2u(q,X_q)^2]\le1$ for
$q\in\{0,q_*\}\subset\operatorname{supp}\mu$; see~\cite[Corollary~3.10]{JT2017}, which includes
the case $q=0$.
At $q_*$ these read $q(t)=t/\beta^2$ and $q'(t)=\E[\partial_x^2u(q_*,X_{q_*})^2]\le1/\beta^2$; at $0$
the second reads $\beta^2\theta^2\le1$. Since $G$ is continuous and $\mu(\operatorname{argmin}G)=1$,
$\operatorname{supp}\mu\subset\operatorname{argmin}G$, so $G(0)=G(q_*)$, i.e.\
$\int_0^{q_*}(\Gamma(s)-s)\,\dd s=0$. Substituting $r=\beta^2s$ gives
$\int_0^tq=t^2/(2\beta^2)=tq(t)/2$.
\end{proof}

\subsubsection{A kernel for the curvature}\label{sec:kernel}

Let $\lambda>0$ and let $V:\R\to\R$ be even and smooth, with $V''\ge0$ and $V',V'',V'''$ bounded.
Put $K_V=V'''^2-2\lambda V''^3$ and
\begin{equation}\label{eq:Pi}
\Pi_V(x)=x\,e^{-\lambda V(x)}\Bigl[\int_0^xy\int_0^ye^{\lambda V(z)}K_V(z)\,\dd z\,\dd y
-2\int_0^xe^{\lambda V(z)}V''(z)^2\,\dd z\Bigr]+2V'(x)^2 ,
\end{equation}
an even function of $x$.

\begin{lemma}[Kernel]\label{lem:kernel}
For $0<r<t$ and $V=U(r,\cdot)$, $\mathcal F(r)=\langle\Pi_V\rangle_r$.
\end{lemma}

\begin{proof}
Since $p_r'=(\lambda V'-x/r)\,p_r$ and $p_r$ has Gaussian tails, every $C^1$ function $\phi$ such
that $\phi$ and $\phi'$ grow at most polynomially satisfies
$\langle x\phi\rangle_r=r\langle\phi'+\lambda V'\phi\rangle_r$. For a continuous function $f$ of
polynomial growth put $\phi_f(x)=e^{-\lambda V(x)}\int_0^xe^{\lambda V(z)}f(z)\,\dd z$. Then
$\phi_f'+\lambda V'\phi_f=f$, hence $r\langle f\rangle_r=\langle x\phi_f\rangle_r$. Because $V$ is
even and nondecreasing on $[0,\infty)$, $e^{\lambda(V(z)-V(x))}\le1$ for $|z|\le|x|$, so
$|\phi_f(x)|\le|\int_0^x|f||$. In particular $|\phi_f(x)|\le c|x|$ if $|f|\le c$, and
$|\phi_f(x)|\le c|x|^3/3$ if $|f(z)|\le cz^2$; also $|\phi_f'|\le|f|+\lambda\|V'\|_\infty|\phi_f|$. Apply this to
$f=V''^2$, to $f=K_V$, and to $f=W:=x\phi_{K_V}$, which satisfies
$|W(x)|\le\|K_V\|_\infty x^2$:
\[
r\langle V''^2\rangle_r=\langle x\phi_{V''^2}\rangle_r,\qquad
r^2\langle K_V\rangle_r=r\langle W\rangle_r=\langle x\phi_W\rangle_r .
\]
By~\eqref{eq:gap-moments},
$\mathcal F(r)=\langle x\phi_W-2x\phi_{V''^2}+2V'^2\rangle_r$, and
$x\phi_W-2x\phi_{V''^2}+2V'^2=\Pi_V$.
\end{proof}

\begin{lemma}[Integrated form]\label{lem:omega}
For $V$ as above let $a=xV''-V'$,
\begin{align*}
Q_V&=\tfrac12\bigl(x^2V'''-2a\bigr)^2+a^2\bigl(1+3\lambda\,x\,a-4\lambda x^2V''\bigr)+\lambda^2x^2V'^4,\\
\Omega_V(x)&=\int_0^xe^{\lambda V}K_V-\frac{2e^{\lambda V(x)}}{x^3}\bigl(a^2-\lambda xV'^3\bigr).
\end{align*}
Then, for $x>0$,
\[
\Pi_V(x)=x\,e^{-\lambda V(x)}\int_0^xy\,\Omega_V(y)\,\dd y,\qquad
\Omega_V(x)=2\int_0^xe^{\lambda V(z)}Q_V(z)\,z^{-4}\,\dd z .
\]
In particular $\Pi_V\ge0$ on $\R$ if $\Omega_V\ge0$ on $(0,\infty)$.
\end{lemma}

\begin{proof}
Let $\Psi=e^{\lambda V}\Pi_V/x$ on $(0,\infty)$. The product rule gives $\Psi'=x\Omega_V$ and
$\Omega_V'=2e^{\lambda V}Q_V/x^4$. Since $V$ is even and smooth, $V'=O(x)$, $a=O(x^3)$ and
$Q_V=O(x^6)$ as $x\downarrow0$, so $\Psi(0+)=\Omega_V(0+)=0$, and both formulas follow by
integration. The last claim uses evenness and $\Pi_V(0)=0$.
\end{proof}

\begin{lemma}[Reduction]\label{lem:reduction}
For $V$ as above write $x=e^\tau$ and regard
\[
L=\lambda V',\qquad \ell=xL,\qquad k=x\,(L-xL')
\]
as functions of $\tau\in\R$, with dots denoting $\dd/\dd\tau$. Then $\dot\ell=2\ell-k$,
\begin{equation}\label{eq:Qtilde}
\lambda^2x^2Q_V=\widetilde Q:=\tfrac12(\dot k-3k)^2+k^2(1-4\ell+k)+\ell^4 ,
\end{equation}
and $\omega:=\lambda^2x^5e^{-\lambda V}\Omega_V$ satisfies
\begin{equation}\label{eq:omega}
\omega(\tau)=\int_{-\infty}^{\tau}2\,\widetilde Q(\tau')
\exp\Bigl(\int_{\tau'}^{\tau}\bigl(5-\ell(\tau'')\bigr)\dd\tau''\Bigr)\dd\tau' .
\end{equation}
\end{lemma}

\begin{proof}
Since $\dd/\dd\tau=x\,\dd/\dd x$, $\dot\ell=xL+x^2L'=2\ell-k$ and $\dot k=k-x^3L''$. With
$\lambda xa=-k$, $\lambda x^3V'''=x^3L''=k-\dot k$, $\lambda x^2V''=x^2L'=\ell-k$ and
$\lambda xV'=\ell$, the three terms of $\lambda^2x^2Q_V$ become the three terms
of~\eqref{eq:Qtilde}. For~\eqref{eq:omega}, substitute $z=e^{\tau'}$ in the formula for
$\Omega_V$ of Lemma~\ref{lem:omega} and multiply by $\lambda^2x^5e^{-\lambda V(x)}$: the integrand
becomes
$2\widetilde Q(\tau')\,(x/z)^5e^{-\lambda(V(x)-V(z))}\dd\tau'$, and
$\log\bigl((x/z)^5e^{-\lambda(V(x)-V(z))}\bigr)={\int_{\tau'}^\tau(5-\ell)}$.
\end{proof}

The parameter $\lambda$ has dropped out: $\widetilde Q$ and~\eqref{eq:omega} depend only on the
trajectory $(\ell,k)$. We verified the identities of Lemmas~\ref{lem:omega}
and~\ref{lem:reduction} symbolically as well (Appendix~\ref{app:repro}).

\subsubsection{The sign of \texorpdfstring{$\partial_x^3u$}{the third derivative}}

\begin{proof}[Proof of Lemma~\ref{lem:uxxx}]
First let $\nu$ have finitely many atoms, so that $\alpha_\nu$ is constant, say $\alpha_j$, on
consecutive intervals $[s_j,s_{j+1})$ covering $[0,1)$. On $[s_j,s_{j+1}]$, $u=u_\nu$ is obtained
from $u(s_{j+1},\cdot)$ by~\eqref{eq:colehopf} or by the heat semigroup, so it is smooth in the
interior, and its $x$-derivatives are bounded and continuous up to $s=s_{j+1}$. Differentiating
\eqref{eq:pde} three times, $w=\partial_x^3u$ solves
\[
\partial_sw+\frac{\beta^2}2\bigl(\partial_x^2w+2\alpha_j\,\partial_xu\,\partial_xw
+6\alpha_j\,\partial_x^2u\,w\bigr)=0 .
\]
Fix $s_0\in[s_j,s_{j+1})$ and $x_0>0$, let
$\dd Z_s=\beta^2\alpha_j\,\partial_xu(s,Z_s)\,\dd s+\beta\,\dd W_s$ with $Z_{s_0}=x_0$, and let $T_0$ be the
hitting time of $0$. For $T<s_{j+1}$, It\^o's formula shows that
$w(s,Z_s)\exp\bigl(3\beta^2\alpha_j\int_{s_0}^s\partial_x^2u(v,Z_v)\,\dd v\bigr)$, stopped at
$T\wedge T_0$, is a bounded martingale; it is bounded because $0\le\partial_x^2u\le1$. Since $w(s,\cdot)$ is odd,
$w(s,0)=0$, and so
\[
w(s_0,x_0)=\E\Bigl[w(T,Z_T)\,e^{3\beta^2\alpha_j\int_{s_0}^T\partial_x^2u(v,Z_v)\,\dd v};\ T_0>T\Bigr].
\]
Letting $T\uparrow s_{j+1}$, we see that $w(s_{j+1},\cdot)\le0$ on $[0,\infty)$ implies
$w(s_0,\cdot)\le0$ on $[0,\infty)$. Since $w(1,x)=-2\tanh x\,\sech^2x\le0$ for $x\ge0$, downward
induction over the intervals gives $w\le0$ on $[0,1]\times[0,\infty)$. For general $\nu$, take
finitely atomic $\nu_n\Rightarrow\nu$ and use the uniform convergence
$\partial_x^3u_{\nu_n}\to\partial_x^3u_\nu$ of \textup{(P3)}.
\end{proof}

In the next remark, parts (a), (b) and (d) are those of Theorem~\ref{thm:sk}, and the trial bound is
the upper bound in the proof of part~(b) in Section~\ref{sec:sk}.

\begin{remark}[A second route to the envelope]\label{rem:envelope-route}
For $\beta>1$, Theorem~\ref{thm:sk}(d) and the Parisi formula also give the envelope of part~(a), in strict form,
without gauge symmetry. By evenness, the trial bound in the proof of part~(b) reads
$\varphi_\beta(h)\le\int_0^{|h|}(|h|-y)\,\partial_x^2u_{\mu_\beta}(0,y)\,\dd y$, and by
Lemma~\ref{lem:uxxx} and part~(d), $\partial_x^2u_{\mu_\beta}(0,y)\le\theta_\beta=1/\beta$ for
$y\ge0$. Hence $\varphi_\beta(h)\le h^2/(2\beta)$. The bound is strict for $h\ne0$, because
$\partial_x^3u_\nu(0,x_0)<0$ for every $\nu$ and every $x_0>0$. To see this, let $\nu$ have
finitely many atoms, put $w=\partial_x^3u_\nu$, let $Z$ solve
$\dd Z_s=\beta^2\alpha_\nu(s)\,\partial_xu_\nu(s,Z_s)\,\dd s+\beta\,\dd W_s$ on $[0,1]$ with
$Z_0=x_0$, and let $T_0$ be its hitting time of $0$. On each interval where $\alpha_\nu$ is
constant, the proof of Lemma~\ref{lem:uxxx} writes $w$ at the left end as the expectation of $w$
at the right end, times an exponential factor, on the event that $Z$ has not hit $0$. Combining
these representations over the intervals by the Markov property, and using that the exponential
factors are at least $1$ and that $w(1,\cdot)\le0$ on $[0,\infty)$, we get
$w(0,x_0)\le\E[w(1,Z_1);\,T_0>1]$. By \textup{(P3)} the drift takes values
in $[0,\beta^2]$ when $Z_s\ge0$ and in $[-\beta^2,\beta^2]$ always. Hence, with
$Y_s=x_0+\beta W_s$, one has $Z_s\ge Y_s$ for $s\le T_0$ and $Z_1\le Y_1+\beta^2$. On the event
$A=\{\min_{[0,1]}Y>0,\ x_0\le Y_1\le x_0+1\}$, which has positive probability, it follows that
$T_0>1$ and $x_0\le Z_1\le x_0+1+\beta^2$, so
\[
\partial_x^3u_\nu(0,x_0)\le-2\tanh(x_0)\sech^2(x_0+1+\beta^2)\,\Pp(A) .
\]
The right side does not depend on $\nu$, so by \textup{(P3)} the bound holds for every $\nu$.
Therefore $\partial_x^2u_{\mu_\beta}(0,y)<1/\beta$ for $y>0$, so $\varphi_\beta(h)<h^2/(2\beta)$
for $h\ne0$, and~\eqref{eq:chen} with $\gamma=\beta$ gives $\Om(\beta,\beta)=\{0\}$.
Section~\ref{sec:skline} proves both statements without the Parisi formula or the accumulation
theorem, with explicit margins (Corollary~\ref{cor:envelope}) and rates as $\gamma\downarrow\beta$
(Theorem~\ref{thm:skline}).
\end{remark}

\subsubsection{The accumulator}

For $V=u_\nu(s,\cdot)$ and $\lambda>0$, Lemma~\ref{lem:uxxx} and \textup{(P3)} give the only
structural input of Lemma~\ref{lem:accumulator} below: $L=\lambda V'$ is concave on
$[0,\infty)$ with $L(0)=0$ and $L'\ge0$, so $L(x)\ge xL'(x)\ge0$, that is,
\begin{equation}\label{eq:cone}
0\le k\le\ell .
\end{equation}

The proof of the next lemma uses the explicit polynomial
\begin{equation}\label{eq:cert-P}
1024\,P(\ell,\delta)=222-25\ell+\delta\,(617-343\ell)+\delta^2(368\ell-833),
\end{equation}
together with $\mathcal Y(\ell,k)=\ell^4P(\ell,k/\ell)$, a polynomial in $(\ell,k)$, and
\begin{equation}\label{eq:cert-B}
\mathcal B=2(5-\ell)\mathcal Y-2\mathcal Y_\ell\,(2\ell-k)-6k\,\mathcal Y_k-\mathcal Y_k^2
+2k^2(1-4\ell+k)+2\ell^4 ,
\end{equation}
where subscripts denote partial derivatives. Let $\ell_{\max}=26181/10000$, which exceeds
$\ell_2={(3+\sqrt5)/2}\approx2.618034$. The certificate has three properties:
\begin{itemize}
\item[(V1)] $\mathcal B(\ell,k)\ge0$ for $1/4\le\ell\le\ell_{\max}$ and $0\le k\le\ell$;
\item[(V2)] $P(\ell,\delta)\ge0$ for $3/8\le\ell\le\ell_{\max}$ and $0\le\delta\le1$;
\item[(V3)] $P(1/4,\delta)\le4/13$ for $0\le\delta\le1$.
\end{itemize}
Properties (V2) and (V3) are elementary. $P$ is affine in $\ell$ and $P(\ell,1)=6/1024$ for every
$\ell$. The quadratic $1024P(3/8,\delta)=\frac{1701}8+\frac{3907}8\delta-695\delta^2$ is concave, so
its minimum on $[0,1]$ is $\min(1701/8,6)>0$. The quadratic $1024P(\ell_{\max},\delta)$ is convex
with vertex at $\delta=2810083/2609216>1$, so its minimum on $[0,1]$ is its value $6$ at $\delta=1$.
Finally $\max_\delta1024P(1/4,\delta)=\max_\delta\bigl(\frac{863}4+\frac{2125}4\delta-741\delta^2\bigr)$
is attained at $\delta=2125/5928$, and
$\max_\delta P(1/4,\delta)=14747353/48562176\approx0.303680<4/13\approx0.307692$. Property (V1) is
the computer-assisted step (Appendix~\ref{sec:accum-cert}).

\begin{lemma}[Accumulator]\label{lem:accumulator}
Let $\ell,k\in C^1(\R)$ satisfy $0\le k\le\ell$ and $\dot\ell=2\ell-k$, define $\widetilde Q$
by~\eqref{eq:Qtilde}, and assume $\int_{-\infty}^0|\widetilde Q(\tau)|e^{-5\tau}\dd\tau<\infty$. Then
the function $\omega$ defined by~\eqref{eq:omega} is nonnegative on $\R$.
\end{lemma}

\begin{proof}
\emph{Preliminaries.} We have $\ell\le\dot\ell\le2\ell$, so $e^{-\tau}\ell$ is nondecreasing and
$e^{-2\tau}\ell$ is nonincreasing. If $\ell$ vanishes somewhere, these two facts force
$\ell\equiv0$, hence $k\equiv0$, $\widetilde Q\equiv0$ and $\omega\equiv0$. Otherwise $\ell>0$,
$\ell$ is strictly increasing, and $\ell(\tau)\to0$ as $\tau\to-\infty$. Since
$\exp\int_{\tau'}^\tau(5-\ell)\le e^{5(\tau-\tau')}$, the integral~\eqref{eq:omega} converges
absolutely, $\omega$ is $C^1$ with $\dot\omega=(5-\ell)\omega+2\widetilde Q$, and for
$\tau_1\le\tau$
\begin{equation}\label{eq:omega-restart}
\omega(\tau)=\omega(\tau_1)e^{\int_{\tau_1}^\tau(5-\ell)}
+\int_{\tau_1}^\tau2\widetilde Q(\tau')e^{\int_{\tau'}^\tau(5-\ell)}\dd\tau' .
\end{equation}

\emph{Step 1: $\ell\le3/8$.} If $\ell\le1/4$, then $1-4\ell+k\ge0$ and $\widetilde Q\ge\ell^4$. If
$1/4\le\ell\le3/8$, then $\min_{0\le k\le\ell}k^2(1-4\ell+k)=-(8\ell-2)^3/54$, attained at
$k=(8\ell-2)/3\le\ell$, and $\ell^4\ge(8\ell-2)^3/54$: for $\ell>1/4$, with $y=8\ell-2\in(0,1]$,
this reads $(y+2)^4/y^3\ge4096/54$, and $(y+2)^4/y^3$ is decreasing on $(0,1]$ with value $81$ at $y=1$. So
$\widetilde Q\ge0$ whenever $\ell\le3/8$. Since $\ell$ is increasing, \eqref{eq:omega} gives
$\omega(\tau)\ge0$ whenever $\ell(\tau)\le3/8$. Let $\tau_0$ be the time with $\ell(\tau_0)=1/4$,
if there is one. For $\tau'\le\tau_0$, $\widetilde Q(\tau')\ge\ell(\tau')^4\ge4^{-4}e^{-8(\tau_0-\tau')}$
and $5-\ell\ge19/4$, so
\[
\omega(\tau_0)\ge2\cdot4^{-4}\int_{-\infty}^{\tau_0}e^{-8(\tau_0-\tau')+\frac{19}4(\tau_0-\tau')}\dd\tau'
=2\cdot4^{-4}\cdot\frac4{13}=\frac1{416}.
\]

\emph{Step 2: $1/4\le\ell\le\ell_{\max}$.} Let $\eta=\omega-2\mathcal Y(\ell,k)$. Then
$\dot\eta=(5-\ell)\eta+2(5-\ell)\mathcal Y+2\widetilde Q-2\mathcal Y_\ell\dot\ell-2\mathcal Y_k\dot k$.
Completing the square in $\dot k$,
$(\dot k-3k)^2-2\mathcal Y_k\dot k\ge-6k\mathcal Y_k-\mathcal Y_k^2$, with equality at
$\dot k=3k+\mathcal Y_k$. Hence $\dot\eta\ge(5-\ell)\eta+\mathcal B(\ell,k)$, and by (V1)
$\dot\eta\ge(5-\ell)\eta$ while $1/4\le\ell\le\ell_{\max}$. By Step~1 and (V3),
$\eta(\tau_0)\ge\frac1{416}-2\cdot4^{-4}\cdot\frac4{13}=0$. Therefore $\eta\ge0$, i.e.\
$\omega\ge2\ell^4P(\ell,k/\ell)$, while $1/4\le\ell\le\ell_{\max}$, and by (V2) $\omega\ge0$ while
$3/8\le\ell\le\ell_{\max}$.

\emph{Step 3: $\ell\ge\ell_2$.} For $\ell\ge1/2$ the map $k\mapsto k^2(1-4\ell+k)$ is decreasing on
$[0,\ell]$, since its derivative is $k(2-8\ell+3k)\le k(2-5\ell)\le0$. Hence
$\widetilde Q\ge\ell^4+\ell^2(1-3\ell)=\ell^2(\ell^2-3\ell+1)\ge0$ for $\ell\ge\ell_2$. Let $\tau_2$ be
the time with $\ell(\tau_2)=\ell_2\in[3/8,\ell_{\max}]$. By Step~2, $\omega(\tau_2)\ge0$,
and~\eqref{eq:omega-restart} with $\tau_1=\tau_2$ gives $\omega\ge0$ for $\tau\ge\tau_2$.

Since $\ell$ is increasing, the three steps cover every $\tau\in\R$.
\end{proof}

\begin{proof}[Proof of Proposition~\ref{prop:accum-convex}]
Fix $r\in(0,t)$ and let $V=U(r,\cdot)=u(r/\beta^2,\cdot)$. By \textup{(P3)}
and~\eqref{eq:gap-ch}, $V$ is even and smooth, $V',V'',V'''$ are bounded, and $V''\ge0$. By
Lemma~\ref{lem:uxxx} and~\eqref{eq:cone}, the functions $\ell,k$ of Lemma~\ref{lem:reduction} satisfy
$0\le k\le\ell$; they are $C^1$ in $\tau$, and $\dot\ell=2\ell-k$. Moreover
$\widetilde Q=\lambda^2x^2Q_V=O(x^8)=O(e^{8\tau})$ as $\tau\to-\infty$, so
Lemma~\ref{lem:accumulator} applies and gives $\omega\ge0$. Thus $\Omega_V\ge0$ on $(0,\infty)$,
so $\Pi_V\ge0$ by Lemma~\ref{lem:omega}, and $\mathcal F(r)=\langle\Pi_V\rangle_r\ge0$ by
Lemma~\ref{lem:kernel}.
\end{proof}

\begin{remark}\label{rem:accum}
Three comments on the argument.
\begin{itemize}
\item[(i)] The argument shows more: for every $\lambda>0$ and every even smooth $V$ with bounded
derivatives, $V''\ge0$ on $\R$ and $V'''\le0$ on $[0,\infty)$, the kernel $\Pi_V$ is
nonnegative. The Cole--Hopf flow~\eqref{eq:gap-ch} preserves these properties (the sign of $V''$
by the formula $\partial_x^2u=\gibbs{\phi''}+a\operatorname{Var}(\phi')$ of
Section~\ref{sec:prelim}, the sign of $V'''$ by the proof of Lemma~\ref{lem:uxxx} on a single
interval), so $q(r)/r$ is convex along the flow from any such terminal profile $\psi$. The optimality conditions enter only through
Lemma~\ref{lem:endgame}.
\item[(ii)] The SK covariance enters twice: $\xi''\equiv\beta^2$ makes the time change linear, and it
gives $\kappa(t)=1/\beta^2\ge\theta^2$. Some such input is necessary: for mixed models with
$\xi''(0)<1$, the Parisi measure puts no mass on $(0,\hat q)$, where $\xi''(\hat q)=1$, so $0$ is
isolated in its support~\cite[Theorem~1]{AC2015a}.
\item[(iii)] Without the sign~\eqref{eq:cone} the kernel can be negative. Numerically, for
$\lambda=4$ and the even profile with $V''=0.05+0.9\,(e^{-(x-2.5)^2}+e^{-(x+2.5)^2})$, which has
$V'''>0$ at some $x>0$, both $\min\Pi_V$ and $\langle\Pi_V\rangle_r$ at $r=0.3$ are negative. (This
profile has unbounded $V'$, so it lies outside the standing assumptions of Appendix~\ref{sec:kernel};
it serves only as a numerical illustration.)
\end{itemize}
\end{remark}

\subsubsection{The certificate}\label{sec:accum-cert}

Property (V1) follows from the positivity of
$\widehat{\mathcal B}(\ell,\delta)=\ell^{-2}\mathcal B(\ell,\delta\ell)$ on the rectangle
$[1/4,\ell_{\max}]\times[0,1]$. From~\eqref{eq:cert-P} and~\eqref{eq:cert-B},
\begin{align*}
1024^2\,\mathcal B(\ell,k)={}&733184\ell^4-198656\ell^5-329489\ell^6+423262\ell^7-117649\ell^8\\
&+k\,\bigl(-3235840\ell^3+2695168\ell^4+2758308\ell^5-2051100\ell^6+504896\ell^7\bigr)\\
&+k^2\bigl(2097152-8388608\ell+12320768\ell^2-6379520\ell^3-3529220\ell^4\\
&\qquad+2452352\ell^5-541696\ell^6\bigr)
+k^3\bigl(2097152-3411968\ell+2260992\ell^2\bigr),
\end{align*}
and $\widehat{\mathcal B}$ is a polynomial of bidegree $(6,3)$ with $20$ monomials. We verified
(V1) in exact rational arithmetic by two methods.
\begin{itemize}
\item \emph{Bernstein subdivision.} A polynomial whose Bernstein coefficients on a box are all
positive is positive on that box. Six exact bisections of the rectangle (thirteen boxes examined)
suffice for $\widehat{\mathcal B}$.
\item \emph{Critical points.} On each of the four edges of the rectangle,
$\widehat{\mathcal B}$ has no zero, by Sturm sequences. The resultant
$\operatorname{Res}_\delta(\partial_\ell\widehat{\mathcal B},\partial_\delta\widehat{\mathcal B})$
has degree $23$ and four real roots in $[1/4,\ell_{\max}]$. Exact rational lower bounds on small
boxes around the interior critical points above these roots give
$\widehat{\mathcal B}\ge0.0358$ there.
\end{itemize}
The minimum of $\widehat{\mathcal B}$ on the rectangle is about $0.0320$, attained on the edge
$\ell=1/4$ near $\delta=0.106$. The same Bernstein check also confirms (V2), (V3) and the Step~1
inequality $54\ell^4>(8\ell-2)^3$ on $[1/4,3/8]$. The certificate was found by a convex search over
polynomial Lyapunov functions of bidegree $(1,2)$ in $(\ell,\delta)$, followed by rounding to the
denominator $1024$. A second certificate, of bidegree $(2,2)$ with nine coefficients and
denominator $4096$, also passes both exact checks. It plays no role in the proof.

\subsection{Free-energy bounds for the \texorpdfstring{$p$}{p}-spin models}\label{app:mf-pspin}

This subsection proves Lemmas~\ref{lem:bins}, \ref{lem:gn}, \ref{lem:annealed} and~\ref{lem:endpoint}
and part~(ii) of Theorem~\ref{thm:reduction}; the notation is that of Section~\ref{sec:pspin}.

\begin{proof}[Proof of Lemma~\ref{lem:bins}]
On $\{m\in[a,b]\}$ we have $e^{hN(m-a)}\ge1$, and on $\{m\in[-b,-a]\}$ we have
$e^{hN(-m-a)}\ge1$. By~\eqref{eq:bias}, $Z_{N,\mathcal M}$ is at most
$e^{\beta j_0(N\max_{\mathcal M}m^p+C_p)-hNa}\sum_\sigma e^{\beta H^0_N(\sigma)\pm hNm(\sigma)}$.
Taking $N^{-1}\E\log$ and using $F^0(\beta,-h)=F^0(\beta,h)$ proves the first claim. For the
second, $N^{-1}\log Z_{N,\mathcal M_k}$ and $N^{-1}\log Z_N$ are Lipschitz functions of the Gaussian
disorder with constant $O(N^{-1/2})$, so they concentrate at exponential rates around their
means~\cite[eq.~(1.22)]{Ledoux1999}. With probability $1-e^{-cN}$ we then have $Z_{N,\mathcal M}/Z_N\le e^{-\eta N}$ for some
$\eta>0$, while always $\gibbs{\ind\{m\in\mathcal M\}}\le1$.
\end{proof}

\begin{proof}[Proof of Lemma~\ref{lem:gn}]
Put $a=\xi(1)p!/N^{p-1}$. Then $\beta J_I$ has mean $a$ and variance $a$, so
$Z_N=\sum_\sigma\exp\sum_I(a+\sqrt a\,g_I)\sigma_I$ with independent standard Gaussians $g_I$.
Let $\rho=\xi'(q)$ and let $z_1,\dots,z_N$ be standard Gaussians independent of the $g_I$. For
$t\in[0,1]$ let $\gibbs\cdot_t$ be the Gibbs measure with Hamiltonian
\[
H_t(\sigma)=\sum_I\bigl(ta+\sqrt{ta}\,g_I\bigr)\sigma_I+\sum_{i=1}^N\bigl((1-t)\rho+\sqrt{(1-t)\rho}\,z_i\bigr)\sigma_i ,
\]
and let $\varphi(t)=N^{-1}\E\log\sum_\sigma2^{-N}e^{H_t(\sigma)}$. Then
$\varphi(1)=N^{-1}\E\log Z_N-\log2$ and $\varphi(0)=\psi(\rho)$.

We first check the Nishimori identity $\E\gibbs{\sigma_A}_t^2=\E\gibbs{\sigma_A}_t$ for every
$A\subset\{1,\dots,N\}$, where $\sigma_A=\prod_{i\in A}\sigma_i$. Let $\tau$ be uniform on
$\{\pm1\}^N$, independent of $(g,z)$, and consider the observations $Y_I=ta\tau_I+\sqrt{ta}\,g_I$ and
$y_i=(1-t)\rho\tau_i+\sqrt{(1-t)\rho}\,z_i$ (if $\rho=0$, omit $y$). Since $\tau_I^2=\tau_i^2=1$, the density of $(Y,y)$ given
$\tau$ is $\exp(\sum_IY_I\tau_I+\sum_iy_i\tau_i)$ times a function of $(Y,y)$ alone. By Bayes'
formula the conditional law of $\tau$ given $(Y,y)$ is therefore the Gibbs measure
$\gibbs\cdot_{Y,y}$ with weights $\exp(\sum_IY_I\sigma_I+\sum_iy_i\sigma_i)$, and
$\E\gibbs{\sigma_A}_{Y,y}^2=\E[\tau_A\gibbs{\sigma_A}_{Y,y}]$. The substitution
$\sigma_i\mapsto\tau_i\sigma_i$ turns $\gibbs\cdot_{Y,y}$ into $\gibbs\cdot_t$ with $g_I$ and $z_i$
replaced by $\tau_Ig_I$ and $\tau_iz_i$. These have the joint law of $(g,z)$ and are independent
of $\tau$. Hence $\tau_A\gibbs{\sigma_A}_{Y,y}$ and $\gibbs{\sigma_A}_{Y,y}^2$ have the laws of
$\gibbs{\sigma_A}_t$ and $\gibbs{\sigma_A}_t^2$, which proves the identity.

The function $\varphi$ is continuous on $[0,1]$. For $0<t<1$, differentiating under the
expectation, integrating by parts in $g_I$ and $z_i$, and using the identity with $A=I$ and
$A=\{i\}$ gives
\[
\varphi'(t)=\frac a{2N}\sum_I\bigl(1+\E\gibbs{\sigma_I}_t\bigr)-\frac\rho2\bigl(1+\E\gibbs{m}_t\bigr).
\]
By~\eqref{eq:bias}, $\frac aN\sum_I\sigma_I=\xi(m)+\xi(1)r_{N,p}(m)/N$, and
$\frac a{2N}\binom Np=j_0^2\prod_{k=1}^{p-1}(1-\frac kN)\ge j_0^2(1-\frac{p(p-1)}{2N})$. Since
$\xi(m)-\xi'(q)m=\Delta_q(m)-\theta(q)$ and $\xi(1)=2j_0^2$,
\[
\varphi'(t)\ge\frac12\bigl(\xi(1)-\xi'(q)-\theta(q)\bigr)+\frac12\E\gibbs{\Delta_q(m)}_t
-\frac{j_0^2}N\Bigl(\frac{p(p-1)}2+C_p\Bigr).
\]
For $q\in Q_p$ the middle term is nonnegative. Integrating over $t\in[0,1]$ and using
$\xi(1)/2=j_0^2$ proves the lemma.
\end{proof}

\begin{proof}[Proof of Lemma~\ref{lem:annealed}]
For fixed $\sigma$, independence of the $J_I$ gives
\[
\E e^{\beta H_N(\sigma)}=\exp\Bigl(\beta j_0\frac{p!}{N^{p-1}}\sum_I\sigma_I+\frac{\beta^2}2\cdot\frac{p!}{2N^{p-1}}\binom Np\Bigr).
\]
The first term in the exponent is $N\xi(m)+\xi(1)r_{N,p}(m)\le N\xi(m)+\xi(1)C_p$, and the
second is $Nj_0^2\prod_{k=1}^{p-1}(1-k/N)\le Nj_0^2$. The number of $\sigma$ with
$m(\sigma)=m=2k/N-1$ is $\binom Nk\le{(N/k)^k(N/(N-k))^{N-k}}=e^{N(\log2-\mathcal I(m))}$ (with
$0^0=1$), and $m$ takes $N+1$ values.
\end{proof}

\begin{proof}[Proof of Lemma~\ref{lem:endpoint}]
Let $f=\xi-\mathcal I$. Then $f(0)=f'(0)=0$ and $f''(m)=(k(m)-1)/(1-m^2)$ with
$k(m)=\xi''(m)(1-m^2)=p(p-1)\xi(1)m^{p-2}(1-m^2)$. Since
$k'(m)=p(p-1)\xi(1)m^{p-3}\bigl(p-2-pm^2\bigr)\ge0$ on $[0,m_1]$ and $k(0)=0$, there is
$s\in(0,m_1]$ with $f''<0$ on $[0,s)$ and $f''\ge0$ on $(s,m_1]$. So $f'<0$ on $(0,s]$ and $f'$
is nondecreasing on $[s,m_1]$. Hence $f'$ changes sign at most once on $(0,m_1]$, from negative
to positive, and the maximum of $f$ on $[0,m_1]$ is attained at $0$ or at $m_1$. For odd $p$ and
$m\le0$, $\xi(m)\le0\le\mathcal I(m)$.
\end{proof}

\begin{proof}[Proof of Theorem~\ref{thm:reduction}(ii)]
By Lemma~\ref{lem:lower} and (F), $\liminf_NN^{-1}\E\log Z_N(\beta_{\rm cold},j_0)\ge
F^0(\beta_{\rm cold},0)\ge L_F$. Fix $\varepsilon>0$.
\begin{itemize}
\item Bins in $[\varepsilon,m_{\rm lo}]$: Lemma~\ref{lem:universal} gives
$F^0(\beta_{\rm cold},0)+g(b)+(b-a)$, which is below $F^0(\beta_{\rm cold},0)$ once the bins are finer than
$\min_{[\varepsilon,m_{\rm lo}]}|g|$.
\item Bins in $[m_{\rm lo},1]$: subdivide each $[a_i,b_i]$ into sub-bins $[a,b]$. By
Lemma~\ref{lem:bins} with $\tau_i,h_i$, each sub-bin is bounded by
$U_i(b)+h_i(b-a)$. Since $U_i$ is convex on $[0,1]$, $U_i(b)\le\max(U_i(a_i),U_i(b_i))<L_F$, so fine
sub-bins stay below $L_F$.
\item Negative bins $[-b,-a]$ with $a\ge\varepsilon$: for odd $p$, Lemma~\ref{lem:universal}
gives $F^0(\beta_{\rm cold},0)-a^2/2$. For even $p$ they are handled like the positive bins, since the
bounds depend only on $|m|$.
\end{itemize}
Lemma~\ref{lem:bins} gives $\E\gibbs{\ind\{|m|\ge\varepsilon\}}\to0$ for every $\varepsilon>0$.
\end{proof}

\subsection{Proof of Lemma~\ref{lem:D}}\label{app:D}

Throughout, $\mu^*=x\delta_0+(1-x)\delta_q$ and the field is $0$. Here $u(s,y)$ denotes the
solution of the Parisi PDE, with time $s$ and spatial variable $y$ (the variable called $x$
in~\eqref{eq:pde} and \textup{(P3)}, since $x$ is now the weight of the atom at $0$); the scalar
$u\in[0,1]$ is the location of the perturbing atom $\delta_u$, and for $u\ge q$,
$s^2=\xi'(u)-\xi'(q)$ as in Lemma~\ref{lem:D}(ii). On $[q,1]$ the Parisi PDE
has $\alpha\equiv1$, so $u(s,y)=\lc y+\frac12(\xi'(1)-\xi'(s))$. On $[0,q)$ it has
$\alpha\equiv x$, so~\eqref{eq:colehopf} gives
$u(0,0)=\frac1x\log A+\frac12(\xi'(1)-\xi'(q))$. With
$\int_a^bs\xi''(s)\,\dd s=\theta(b)-\theta(a)$, this yields the formula for $\Par(\mu^*,0)$
stated before Lemma~\ref{lem:D} and $C=\partial_x\Par$.

\paragraph{Case $u\le q$.} The perturbed measure
$\mu_\varepsilon=(1-\varepsilon)\mu^*+\varepsilon\delta_u$ has distribution function
$(1-\varepsilon)x$ on $[0,u)$, $b_\varepsilon=x+\varepsilon(1-x)$ on $[u,q)$ and $1$ on $[q,1]$.
Two Cole--Hopf steps give
\[
u_\varepsilon(0,0)=\frac{1}{(1-\varepsilon)x}\log\E\exp\Bigl(\frac{(1-\varepsilon)x}{b_\varepsilon}
\log\E_g\cosh^{b_\varepsilon}\bigl(Y_u+\sqrt{\xi'(q)-\xi'(u)}\,g\bigr)\Bigr)+\frac12(\xi'(1)-\xi'(q)),
\]
where the outer expectation is over $Y_u\sim N(0,\xi'(u))$. The penalty term of
$\Par(\mu_\varepsilon)$ is
$-\frac12[(1-\varepsilon)x\theta(u)+b_\varepsilon(\theta(q)-\theta(u))+\theta(1)-\theta(q)]$.
Differentiate at $\varepsilon=0$, using $\E_{Y_u}B(Y_u)=A$ and the tower property
$\E_{Y_u}\E_g[\cosh^x\lc](\cdot)=AK$. This gives~(i).

\paragraph{Case $u\ge q$.} Now the distribution function is $(1-\varepsilon)x$ on $[0,q)$,
$1-\varepsilon$ on $[q,u)$ and $1$ on $[u,1]$. The middle step is
$u_\varepsilon(q,y)=\frac1{1-\varepsilon}\log\E\cosh^{1-\varepsilon}(y+sg)+\frac12(\xi'(1)-\xi'(u))$.
At $\varepsilon=0$ its derivative is
$\lc y+\frac{s^2}2-\E[\cosh(y+sg)\lc(y+sg)]/(\cosh(y)e^{s^2/2})$. Averaging against the tilt
$\cosh^x(Y_q)/A$ gives $K+\frac{s^2}2-\E[\cosh^{x-1}Y_q\cosh Y_u\lc Y_u]/(Ae^{s^2/2})$. The
outer step on $[0,q)$ also has parameter $(1-\varepsilon)x$; differentiating in that
parameter, with $u_\varepsilon(q,\cdot)$ held at $\varepsilon=0$, contributes
$\frac{\log A}x-K$. Adding these two terms and the derivative of the penalty,
$\frac12[x\theta(q)+\theta(u)-\theta(q)]$, gives~(ii).

\paragraph{Items (iii) and (iv).} Continuity at $q$ and the values $D(0)$, $D(q)$ follow by
substitution. For $u<q$ write $t=\xi'(u)$ and let $f_t=P_{\xi'(q)-t}\cosh^x$ be the backward
heat flow, so that $B=f_t$ and $B_1=f_t'/x$. The heat equation gives
$\frac{\dd}{\dd t}\E[f_t\log f_t(Y)]=\frac12\E[f_t'^2/f_t]$ with $Y\sim N(0,t)$. Hence
$\frac{\dd}{\dd u}\E[B\log B]=\frac{x^2}2\xi''(u)A\,\Gamma(u)$. Also
$\frac{\dd}{\dd u}\theta(u)=u\,\xi''(u)$. Together these give
$D'(u)=-\frac12\xi''(u)(\Gamma(u)-u)$. For $u>q$ the analogous computation uses
$\frac{\dd}{\dd s^2}\bigl[e^{-s^2/2}\E_g\cosh(y+sg)\lc(y+sg)\bigr]
=\frac12e^{-s^2/2}\E_g[\cosh(y+sg)(1+\tanh^2(y+sg))]$, and the same identity follows.

For monotonicity, below $q$ one computes
$\frac{\dd\Gamma}{\dd t}=\frac{1}{x^2A}\E\bigl[f_t\,(\partial_y^2\log f_t)^2\bigr]\ge0$, and
this is at most $1$ because $0\le\partial_y^2\log f_t/x=\partial_y^2u\le1$ by \textup{(P3)}.
Above $q$, $\Gamma=1-\E[\cosh^{x-1}Y_q\,\E_g\sech Y_u]/(Ae^{s^2/2})$ and
$\frac{\dd\Gamma}{\dd s^2}=\E[\cosh^{x-1}Y_q\,\E_g\sech^3Y_u]/(Ae^{s^2/2})\in[0,1]$.
Since $\dd t=\xi''(u)\,\dd u$, we get $0\le\Gamma'\le\xi''$. Finally, $\Gamma(0)=0$ because
$u(0,\cdot)$ is even.\qed

\subsection{Certificates for \texorpdfstring{$p=3$}{p=3} and \texorpdfstring{$p=4$}{p=4}}\label{app:cert34}

This subsection describes how the predicates (W), (F) and (L) of Theorem~\ref{thm:reduction} are
certified at the parameters of Theorems~\ref{thm:p3} and~\ref{thm:p4}; the notation is that of
Section~\ref{sec:pspin}, and Table~\ref{tab:cert} lists the certified values. We begin with the cell
bound used for (F).

Items (iii) and (iv) of Lemma~\ref{lem:D} turn
Proposition~\ref{prop:fw} into a finite computation. Evaluate $D$ and $\Gamma$ at grid points
$u_0<u_1<\dots<u_n=1$ that include $u=q$. Since $\Gamma$ is monotone, on each cell
$[u_k,u_{k+1}]$,
\begin{equation}\label{eq:cell}
\min_{[u_k,u_{k+1}]}D\ge\max\Bigl\{D(u_k)-\tfrac12\int_{u_k}^{u_{k+1}}\xi''(t)\bigl(\Gamma(u_{k+1})-t\bigr)^+\dd t,\
D(u_{k+1})-\tfrac12\int_{u_k}^{u_{k+1}}\xi''(t)\bigl(t-\Gamma(u_k)\bigr)^+\dd t\Bigr\},
\end{equation}
and $\min_{[0,u_0]}D=D(0)$.

\begin{table}[!htb]
\centering\small
\begin{tabular}{@{}lll@{}}
\toprule
& $p=3$ & $p=4$\\
\midrule
$j_0$, $\beta_{\rm w}=2j_0$, $\beta_{\rm cold}$ & $3841/5000$, $3841/2500$, $5/2$ & $8107/10000$, $8107/5000$, $10/3$\\
(W) $q_0$, $m_1$ & $\frac{814809}{10^6}$, $\frac{11}{25}$ & $\frac{948189}{10^6}$, $\frac{16}{25}$\\
(W) $\Phi(q_0)$ & $\ge5.0150255\times10^{-4}$ & $\ge2.2785439\times10^{-4}$\\
 & ($\ge5.0150261\times10^{-4}$) & ($\ge2.2785462\times10^{-4}$)\\
(W) $\xi(m_1)-\mathcal I(m_1)$ & $\le3.4561599\times10^{-4}$ & $\le-1.2223937\times10^{-3}$\\
(W) margin & $\ge1.5588657\times10^{-4}$ & $\ge2.2785439\times10^{-4}$\\
 & ($\ge1.5588662\times10^{-4}$) & ($\ge2.2785462\times10^{-4}$)\\
(L) $m_{\rm lo}$, $g(m_{\rm lo})$ & $1/4$, $-159/128000$ & $2/5$, $-2536/234375$\\
(L) cover of $[m_{\rm lo},1]$ & 75 intervals & 67 intervals\\
(L) $\max_iU_i$ at endpoints & $\le2.0451151094$ & $\le2.7532998672$\\
(F) trial $\mu^*=x\delta_0+(1-x)\delta_q$ & $x=\frac{554371}{10^6}$, $q=\frac{929223}{10^6}$ & $x=\frac{930529}{2\cdot10^6}$, $q=\frac{4930831}{5\cdot10^6}$\\
(F) $\Par(\mu^*,0)$ & $\ge2.04575395212$ & $\ge2.75458158138$\\
(F) certified $\min_uD(u)$ & $\ge-6.8\times10^{-5}$ ($\ge-6.6\times10^{-7}$) & $\ge-4.9\times10^{-6}$ ($\ge-1.5\times10^{-7}$)\\
(F) $L_F$ & $2.0456862122$ ($2.0457533010$) & $2.7545766959$ ($2.7545814375$)\\
cold margin $L_F-\max_iU_i$ & $5.7\times10^{-4}$ ($6.3\times10^{-4}$) & $1.27\times10^{-3}$ ($1.28\times10^{-3}$)\\
\bottomrule
\end{tabular}
\caption{Certified predicates. Numbers are rigorous enclosure endpoints, rounded outward in
the last displayed digit (lower bounds down, upper bounds up). Where two values are shown, the
first is the primary certificate and the parenthesized one is from the independent
implementation, whose bounds are tighter. The (W) margin is
$\Phi(q_0)-\max\{0,\xi(m_1)-\mathcal I(m_1)\}$. The (L) trials are replica-symmetric $\delta_q$ or two-atom
$\{0,q\}$ measures with a field, listed in the frozen trial files described in
Appendix~\ref{app:repro}.}
\label{tab:cert}
\end{table}

\subsubsection{What is computed}

\begin{itemize}
\item \emph{(W).} Here $\Phi(q_0)$ is the single one-dimensional Gaussian integral
\[
\psi(\xi'(q_0))=\E\lc\bigl(\xi'(q_0)+\sqrt{\xi'(q_0)}\,g\bigr)
\]
minus the rational number $(\xi'(q_0)+\theta(q_0))/2$, and $\xi(m_1)-\mathcal I(m_1)$ is a rational number minus a
combination of $\log(1+m_1)$ and $\log(1-m_1)$. The conditions $q_0\in Q_p$ (for $p=3$,
$q_0\ge1/2$) and $m_1^2\le1-2/p$ are exact rational checks. The value $q_0$ approximates the
maximizer of $\Phi$, and $m_1$ lies just below the root of $\xi(m)-\mathcal I(m)=\Phi(q_0)$,
which is about $0.44073$ for $p=3$ and $0.64232$ for $p=4$.
\item \emph{(L).} Each interval needs one integral, either $\E\lc(h+\sqrt{\xi'(q)}g)$ or
$\E\cosh^x(h+\sqrt{\xi'(q)}g)$. The bound $U_i$ is convex in $m\ge0$, so its endpoint values
suffice. Two-atom trials are used for $m\le0.57$ ($p=3$) and $m\le0.73$ ($p=4$), and
replica-symmetric trials above. The worst interval sits at the replica-symmetric ferromagnetic branch
($m\approx0.84$ for $p=3$, $m\approx0.96$ for $p=4$).
\item \emph{(F).} Here $\Par(\mu^*,0)$, $A$ and $K$ are one-dimensional integrals, while $D(u)$
and $\Gamma(u)$ at grid points are two-dimensional (Lemma~\ref{lem:D}). The minimum of $D$ over
$[0,1]$ is bounded through~\eqref{eq:cell} and the monotone region
$[0,u_0]$ ($u_0\ge8/75$ for $p=3$, $\ge21/100$ for $p=4$). At the optimum $D$ vanishes at $0$
and at $q$; its minimum is controlled by near-stationarity of $\mu^*$ and the grid spacing
near $q$.
\end{itemize}

\subsubsection{Rigorous arithmetic and independent verification}

The primary certificate uses IEEE double-precision interval arithmetic with outward rounding.
Elementary functions are widened by three units in the last place, which assumes the platform
\texttt{exp} and \texttt{log} err by less than two units. Integrals over the real line are
truncated with explicit Gaussian tail bounds. Each remaining interval is handled by composite
Simpson rules whose remainders are enclosed with interval Taylor arithmetic of order four.
The two-dimensional integrals in Lemma~\ref{lem:D}(i) are nested: the inner integral is
enclosed tightly at the outer nodes, and crude enclosures of its derivatives control the
outer remainder. For $p=4$, predicate (F) at the colder $\beta_{\rm cold}=10/3$ ($\Par(\mu^*,0)$,
$A$, $K$, $D$ and $\Gamma$) was certified instead with the trapezoid rule on $\R$, the
analytic-strip bound of~\cite[Theorem~5.1]{TW2014}, and the a priori bound
$\gamma_{n-1}\sum_i|x_i|$ on the rounding error of a floating-point sum of $n$ terms
$x_i$~\cite{Higham2002}, where $\gamma_n=nu/(1-nu)$ and $u$ is the unit roundoff.
For (W), the primary certificate encloses $\psi$ with the same Simpson rules and evaluates the
logarithms in exact rational arithmetic, through $\log y=2\operatorname{artanh}((y-1)/(y+1))$
with an explicit geometric bound on the tail of the series.

For (L) and (F), the second implementation was developed independently of the first and shares
no code with it. For (W), the two drivers share only the problem data and differ in arithmetic
and quadrature (Simpson rules with Taylor remainders and exact rational logarithms, versus
\texttt{mpmath}~\cite{mpmath} intervals with the trapezoid rule). The second implementation
rederives Lemma~\ref{lem:D}, including the monotonicity of $\Gamma$ from the heat-semigroup
representation, and differs from the primary implementation as follows:
\begin{itemize}
\item \texttt{mpmath} interval arithmetic with directed rounding at $80$--$140$ bits;
\item the trapezoid rule on $\R$ with the error bound of Trefethen and
Weideman~\cite[Theorem~5.1]{TW2014} for integrands analytic in a strip, with explicit
majorants on the strip;
\item Higham's forward error bound~\cite{Higham2002} for the floating-point convolutions in the
nested integrals.
\end{itemize}
It does not use the platform's transcendental functions for any certified value. For $p=3$ and
for the $p=4$ predicates (W) and (L), its quadrature also differs from the primary
Simpson/Taylor rules. For $p=4$ (F), both implementations use the trapezoid rule with strip
bounds and a priori summation bounds; there they differ in the evaluation of elementary
functions, in node lattices and strip widths, and in the organization of the nested integrals.
For $p=3$ every enclosure it produces lies inside the corresponding primary enclosure. For
$p=4$ its (W), (L), $\Par(\mu^*,0)$, $D(0)$ and $D(q)$ enclosures lie inside the primary ones;
at the $41$ shared grid points its $D$ and $\Gamma$ enclosures overlap the primary ones but are
not nested in general (all but one $D$ enclosure is nested, most $\Gamma$ enclosures are not). It certifies all three predicates for both values of $p$.

\section{Proofs and certificates for Section~\ref{sec:nearM}}
\label{app:largep}

This appendix proves Lemma~\ref{lem:nm-closed}, Lemma~\ref{lem:nm-U} with
Proposition~\ref{prop:nm-triple}, Proposition~\ref{prop:nm-J} and Theorem~\ref{thm:nm-LP}, and
describes the certificates behind Theorem~\ref{thm:nm-nondeg}; Appendix~\ref{app:nm-proofs} proves
Lemmas~\ref{lem:nm-NLid}, \ref{lem:nm-G} and~\ref{lem:nm-LS}, the analytic extension in
Lemma~\ref{lem:nm-branch} and Lemma~\ref{lem:nm-FMloc}(b),(c). In the Gaussian averages below,
$\Lambda>0$ and $h\sim N(\Lambda,\Lambda)$, except in Appendix~\ref{app:nm-closed}, where $h$ is the
field variable of $U$ and $h_c\sim N(\Lambda_c,\Lambda_c)$ plays this role. Throughout,
$\mu=\mu(\Lambda)=\E\tanh h$, $\tau=\E\tanh^3h$ and
$V=\operatorname{Var}(\lc h)$. By Lemma~\ref{lem:nm-NLid} we may write
\[
\bar\mu:=1-\mu=\E\sech^2h,\qquad v:=\mu-\tau=\E[\tanh^2h\sech^2h],\qquad
w:=1-2\mu+\tau=\E\sech^4h=\mu'(\Lambda),
\]
so that $B=\bar\mu-2v$ and $1-4\mu+3\tau=\bar\mu-3v=B-v$. We also use
$W(\Lambda)=\psi(\Lambda)-\Lambda/2$, so that $D_p=W-\frac{p-1}{2p}\Lambda\mu$, and $W'=\mu/2$ by
Lemma~\ref{lem:nm-NLid}.

\subsection{Closed forms at the triple point}\label{app:nm-closed}

\begin{proof}[Proof of Lemma~\ref{lem:nm-closed}]
All quantities are evaluated at $M$ and $z_0=(\mu,\Lambda_c,\mu)$, where $\Lambda=\Lambda_c$.
Differentiating $U$ with $\partial_h\E f(h+\sqrt r g)=\E f'$ and
$\partial_r\E f(h+\sqrt r g)=\frac12\E f''$, and using $(\tanh)''=-2\tanh\sech^2$ and
$(\tanh^2)''=2\sech^2(1-3\tanh^2)$, one finds that the Hessian of $U$ in $(m,h,q)$ is
\[
\mathbf H=\begin{pmatrix}\Omega&-1&0\\-1&\bar\mu&-\Omega v\\0&-\Omega v&\frac\Omega2\bigl(1-\Omega(B-v)\bigr)\end{pmatrix};
\]
the term $\frac12\xi'''(q)(q-\E\tanh^2)$ in $\partial_q^2U$ vanishes at $z_0$. This gives
\eqref{eq:nm-S}. Since $\bar\mu^2-3\bar\mu v+2v^2=(\bar\mu-v)(\bar\mu-2v)=wB$,
\[
\det S=\frac\Omega2\bigl(\bar\mu-\Omega wB\bigr)=\frac\Omega2\bigl(\lambda B+2v\bigr),\qquad
\det\mathbf H=\Omega\det S-S_{qq}=\frac\Omega2\bigl(\Omega\lambda B-\lambda\bigr)
=-\frac\Omega2\,\lambda\,(1-\Omega B),
\]
using $\lambda=1-\Omega w$ and $\bar\mu-v=w$. Hence
$\tilde B''(\mu)=\det\mathbf H/\det S=-\lambda(1-\Omega B)/(\lambda B+2v)$.

For $A_{\rm FM}$, let $z^*(\beta)$ be the critical point of Lemma~\ref{lem:nm-FMloc}(b) at
$j_0=j_{0M}$, which uses only (N2), so that
$\hat\varphi_{\rm FM}(\beta,j_{0M})=U(z^*(\beta);\beta)$. Differentiating twice
and using $\mathbf H\,z^{*\prime}+\partial_\beta\nabla U=0$ gives
\[
\partial_\beta^2\hat\varphi_{\rm FM}=\partial_\beta^2U-\mathbf d^{\top}\mathbf H^{-1}\mathbf d,
\qquad \mathbf d=\partial_\beta\nabla_{(m,h,q)}U ,
\]
all at $(z_0,\beta_1)$ and with $(m,h,q,j_0)$ held fixed in the partial derivatives. Write
$\xi_\beta=\beta^2\xi_1$ with $\xi_1(q)=q^p/2$, so that $\beta_1^2\xi_1'(\mu)=\Lambda_c$. From
$\partial_\beta U=j_0m^p+\beta\bigl[\xi_1(1)-\xi_1(q)+q\xi_1'(q)\bigr]-\beta\xi_1'(q)\,
\E\tanh^2\bigl(h+\beta\sqrt{\xi_1'(q)}\,g\bigr)$ one gets, at $(z_0,\beta_1)$,
\[
\partial_\beta^2U=\frac12-\frac{\mu^p}2-\frac{2\Lambda_c^2}{\beta_1^2}(B-v),\qquad
\beta_1\mathbf d=\Lambda_c\,\mathbf k,\qquad \mathbf k=\bigl(1,\,-2v,\,-\Omega(B-v)\bigr)^{\top}.
\]
For instance $\partial_\beta\partial_mU=pj_{0M}\mu^{p-1}=\Lambda_c/\beta_1$, and
$\partial_\beta\partial_hU=(2\Lambda_c/\beta_1)\cdot\frac12\E(\tanh)''(h_c)=-2\Lambda_cv/\beta_1$.
One checks directly, with $\bar\mu=B+2v$, that
\[
\mathbf H\mathbf y=\mathbf k\qquad\text{for}\qquad \mathbf y=-\frac{1}{1-\Omega B}\,(B,\,1,\,2B)^{\top} .
\]
Hence $\mathbf k^\top\mathbf H^{-1}\mathbf k=\mathbf k^\top\mathbf y
=\bigl(2v-B+2\Omega(B-v)B\bigr)/(1-\Omega B)$, and
\[
\beta_1^2\Bigl(\frac12-\partial_\beta^2\hat\varphi_{\rm FM}\Bigr)
=\frac{\beta_1^2\mu^p}2+\Lambda_c^2\bigl[2(B-v)+\mathbf k^\top\mathbf y\bigr]
=\frac{\beta_1^2\mu^p}2+\frac{\Lambda_c^2B}{1-\Omega B},
\]
which is $\beta_1^2A_{\rm FM}$ by~\eqref{eq:nm-A}. Under (N2), $\det\mathbf H=\det S\cdot\tilde
B''(\mu)\ne0$, so $\lambda(1-\Omega B)\ne0$ and every step is justified.
\end{proof}

The same second derivative computed in Nishimori's parametrization $h=p\beta j_0m^{p-1}$ gives his
(44)--(45); the on-shell second derivative does not depend on the parametrization of the
critical point. In that parametrization the Hessian is Nishimori's $2\times2$ Hessian $H$ of
$(m,q)\mapsto U_{\beta_1,j_{0M}}(m;p\beta_1j_{0M}m^{p-1},q)$ at $(\mu,\mu)$. Since $\partial_hU=0$
at $z_0$ and $\frac{\dd}{\dd m}p\beta_1j_{0M}m^{p-1}=\Omega$ at $m=\mu$, $H$ is the restriction of
$\mathbf H$ to the directions $(1,\Omega,0)$ and $(0,0,1)$, and a short computation gives
$\det H=\Omega\det\mathbf H=-\frac{\Omega^2}2\lambda(1-\Omega B)$. The certificates of
Appendix~\ref{app:nm-cert} do not code~\eqref{eq:nm-A} directly. They evaluate $A_{\rm FM}$
through $H$, in the three codings described there; the coding by closed forms reduced on the
Nishimori line equals~\eqref{eq:nm-A} by the identity for $\det H$.

\subsection{Single crossing and the triple point}\label{app:nm-U}

For a function $\varphi$ on $\R$ put $K_\varphi(s)=\int_\R e^{-h^2/(2s)}\varphi(h)\,\dd h$. The
density of $h\sim N(s,s)$ is $\nu(s)e^{-h^2/(2s)}e^h$ with $\nu(s)=e^{-s/2}/\sqrt{2\pi s}$, so for
even $F$ one has $\E F(h)=\nu(s)K_{F\cosh}(s)$. In particular, by Lemma~\ref{lem:nm-NLid}(b),
\[
\mu(s)=\E\tanh^2h=\nu(s)K_{\sinh\cdot\tanh}(s),\qquad \mu'(s)=\E\sech^4h=\nu(s)K_{\sech^3}(s).
\]

\begin{proof}[Proof of Lemma~\ref{lem:nm-U}]
Fix $\ell\in(0,1)$.

\emph{Representation.} Let $k(h)=\int_0^h\sech^3=\frac12(\sech h\tanh h+\arctan\sinh h)$, an odd
function with $|k|<\pi/4$. One Gaussian integration by parts,
$s\int e^{-h^2/(2s)}k'(h)\,\dd h=\int he^{-h^2/(2s)}k(h)\,\dd h$, gives
$s\mu'(s)=\nu(s)K_{hk}(s)$. Hence
\[
s\mu'(s)-\ell\mu(s)=\nu(s)K_{\varphi_\ell}(s),\qquad \varphi_\ell(h)=h\,k(h)-\ell\sinh h\tanh h,
\]
and $e(s)-\ell$ has the sign of $K_{\varphi_\ell}(s)$.

\emph{$\varphi_\ell$ changes sign once on $(0,\infty)$.} With $f=hk$ and $G=\sinh\cdot\tanh$ one
computes $f''=\sech^3h\,(2-3h\tanh h)$ and $G''=\sech^3h\,(\cosh^4h-\cosh^2h+2)$, so
$\varphi_\ell''=\sech^3h\cdot\Delta_\ell(h)$ with
$\Delta_\ell=2-3h\tanh h-\ell(\cosh^4h-\cosh^2h+2)$. On $(0,\infty)$, $\Delta_\ell'<0$,
$\Delta_\ell(0)=2(1-\ell)>0$ and $\Delta_\ell\to-\infty$, so $\varphi_\ell'$ first increases and
then decreases. Also $\varphi_\ell(0)=\varphi_\ell'(0)=0$, and $\varphi_\ell'\to-\infty$ because
$f'=k+h\sech^3h$ is bounded while $G'=\sinh h(1+\sech^2h)\to\infty$. So $\varphi_\ell'$
changes sign once on $(0,\infty)$, from $+$ to $-$, and $\varphi_\ell$ increases from $0$ and then
decreases to $-\infty$. Since $\varphi_\ell$ is even, there is $h_0>0$ with $\varphi_\ell>0$ for
$0<|h|<h_0$ and
$\varphi_\ell<0$ for $|h|>h_0$. For $\ell=1$,
$\Delta_1=-3h\tanh h-\cosh^2h\sinh^2h<0$ on $(0,\infty)$, so $\varphi_1<0$ on $\R\setminus\{0\}$.

\emph{Variation diminishing.} Let $\mathcal G(s)=e^{h_0^2/(2s)}K_{\varphi_\ell}(s)
=\int e^{-(h^2-h_0^2)/(2s)}\varphi_\ell(h)\,\dd h$. Differentiation under the integral is
justified on compact $s$-intervals, and
\[
\mathcal G'(s)=\frac1{2s^2}\int(h^2-h_0^2)\,e^{-(h^2-h_0^2)/(2s)}\varphi_\ell(h)\,\dd h<0,
\]
because $(h^2-h_0^2)\varphi_\ell(h)<0$ except at $h\in\{0,\pm h_0\}$. So $K_{\varphi_\ell}$, and
hence $e-\ell$, has at most one zero on $(0,\infty)$, with $e>\ell$ before it and $e<\ell$ after
it.

\emph{Endpoints.} Since $\mu(s)\le s$, $e(s)\ge\mu'(s)=\E\sech^4h\ge1-2s-2s^2$, using
$\sech^4y\ge1-2y^2$ and $\E h^2=s+s^2$; so $e>\ell$ near $0$. For $s\ge1$,
$e(s)\le s\cdot\frac\pi2\nu(s)/\mu(1)\to0$, since $\mu'=\nu K_{\sech^3}\le\frac\pi2\nu$ and $\mu$
increases. By continuity $e-\ell$ has exactly one zero.

\emph{Conclusion.} For $\ell=1$, $K_{\varphi_1}<0$, so $e<1$; also $e>0$. If $s_1<s_2$, take
$\ell=e(s_1)\in(0,1)$; then $s_1$ is the only zero of $e-\ell$, so $e(s_2)<e(s_1)$. Thus $e$ is
strictly decreasing, with limits $1$ at $0$ and $0$ at $\infty$.
\end{proof}

For integer $p\ge3$ let $s_p^*$ be the unique $s>0$ with $e(s)=1/(p-1)$
(Lemma~\ref{lem:nm-U}). For these $p$ the single crossing at level
$1/(p-1)$ also has a second, computer-assisted proof that combines a derivative bound for
$\log e$ with interval enclosures on $242$ cells of $[0.2,1.41]$.

\begin{proof}[Proof of Proposition~\ref{prop:nm-triple}]
\emph{Correspondence.} Put $s=\xi_\beta'(q)$. Then $\theta_\beta(q)=(p-1)\xi_\beta(q)=(p-1)sq/p$,
so $\Phi_\beta(q)=W(s)-\frac{p-1}{2p}sq$ and, by~\eqref{eq:nm-Phiprime},
$\Phi_\beta'(q)=\frac12\xi_\beta''(q)(\mu(s)-q)$. Hence $(\beta,q)$ with $\beta>0$ and
$q\in(0,1)$ solves~\eqref{eq:nm-system} exactly when $q=\mu(s)$ and $D_p(s)=0$. Conversely, a
zero $\Lambda>0$ of $D_p$ gives the solution $q=\mu(\Lambda)\in(0,1)$,
$\beta=(2\Lambda\mu(\Lambda)^{1-p}/p)^{1/2}$. This is a bijection.

\emph{Existence and uniqueness.} Since $W'=\mu/2$ and $\mu'=w$,
\[
D_p'=\frac\mu2-\frac{p-1}{2p}(\mu+\Lambda w)=\frac{\mu}{2p}\bigl(1-(p-1)e(\Lambda)\bigr).
\]
By Lemma~\ref{lem:nm-U}, $D_p'<0$ on $(0,s_p^*)$ and $D_p'>0$ on $(s_p^*,\infty)$. Also
$D_p(0^+)=0$, and $D_p(\Lambda)>\Lambda/(2p)-\log2$ because $\lc y\ge|y|-\log2$,
$\E|h|>\E h=\Lambda$ and $\mu<1$. So $D_p<0$ on $(0,s_p^*]$, $D_p$ increases strictly on
$[s_p^*,\infty)$, and $D_p(2p\log2)>0$. Hence $D_p$ has exactly one zero $\Lambda_c$, with
$s_p^*<\Lambda_c<2p\log2$, $D_p<0$ on $(0,\Lambda_c)$ and $D_p>0$ on $(\Lambda_c,\infty)$.

\emph{(M2).} $\xi_{\beta_1}''(q_1)=(p-1)\xi_{\beta_1}'(q_1)/q_1=(p-1)\Lambda_c/\mu(\Lambda_c)$, so
$\lambda=1-(p-1)\Lambda_cw(\Lambda_c)/\mu(\Lambda_c)=1-(p-1)e(\Lambda_c)>0$ because
$\Lambda_c>s_p^*$.

\emph{(M1).} For $q\in(0,1]$, $s(q)=\xi_{\beta_1}'(q)=\Lambda_c(q/q_1)^{p-1}$ is an increasing
bijection onto $(0,s(1)]$ with $s(1)>\Lambda_c$, and $\Gamma_{\beta_1}(q)=\mu(s(q))$. Hence
\[
\log\Gamma_{\beta_1}(q)-\log q=\chi_p(s(q))-\chi_p(\Lambda_c),\qquad
\chi_p(s)=\log\mu(s)-\frac{\log s}{p-1},\qquad \chi_p'(s)=\frac{e(s)-1/(p-1)}{s}.
\]
By Lemma~\ref{lem:nm-U}, $\chi_p$ increases strictly on $(0,s_p^*]$ and decreases strictly on
$[s_p^*,\infty)$, and $\chi_p(0^+)=-\infty$ because $\mu(s)\le s$ and $p>2$. So
$\chi_p(s)=\chi_p(\Lambda_c)$ has exactly one solution $s_1\ne\Lambda_c$, and $s_1<s_p^*$. With
$z_1=s^{-1}(s_1)<q_1$, $\Gamma_{\beta_1}-{\rm id}$ is negative on $(0,z_1)$, positive on
$(z_1,q_1)$ and negative on $(q_1,1]$. By~\eqref{eq:nm-Phiprime}, $\Phi_{\beta_1}$ decreases,
increases, then decreases, and $\Phi_{\beta_1}(0)=\Phi_{\beta_1}(q_1)=0$. This is (M1).
\end{proof}

\subsection{The tilted representation and the positive kernel}\label{app:nm-J}

Fix $\Lambda>0$, write $\nu=\nu(\Lambda)$ and $I[\varphi]=K_\varphi(\Lambda)$, and put
$r(h)=\log(1+e^{-2|h|})$, so that $\log(2\cosh h)=|h|+r(h)$, $L=r$ on $[0,\infty)$ and
$L(h)=2|h|+r(h)$ for $h<0$. For $G$ with $\E|G(h)|<\infty$,
\[
\E G(h)=\nu\,I\bigl[\operatorname{Even}(e^hG)\bigr],\qquad
\operatorname{Even}(F)(h)=\tfrac12\bigl(F(h)+F(-h)\bigr).
\]
Hence, using $0<e^{-h^2/(2\Lambda)}\le1$,
\begin{align*}
\bar\mu&=\nu I[\sech]\le\pi\nu, & w&=\nu I[\sech^3]\le\tfrac\pi2\nu,\\
\E|h|-\Lambda&=\nu I\bigl[|h|e^{-|h|}\bigr]\le2\nu, & \E r(h)&=\nu I[r\cosh]\le(\pi-2)\nu,\\
B&=\nu I[2\sech^3-\sech]=\frac\nu\Lambda\,I[h\tanh h\sech h]\in\Bigl(0,\frac{\pi\nu}\Lambda\Bigr].
\end{align*}
The integrals $\int\sech=\pi$, $\int\sech^3=\pi/2$, $\int|h|e^{-|h|}=2$, $\int r\cosh=\pi-2$ and
$\int h\tanh h\sech h=\pi$ over $\R$ are elementary. The identity for $B$ is one integration by
parts, since $(\sech\tanh)'=2\sech^3-\sech$ and
$\partial_he^{-h^2/(2\Lambda)}=-(h/\Lambda)e^{-h^2/(2\Lambda)}$. Put
$\varsigma=\E L(h)=(\E|h|-\Lambda)+\E r(h)\in(0,\pi\nu]$ and $T_1(\Lambda)=\Lambda\mu-2W$. Since
$W=\E\log(2\cosh h)-\log2-\Lambda/2$,
\begin{equation}\label{eq:nm-T1}
T_1=2\log2-\Lambda\bar\mu-2\varsigma\in\bigl[\,2\log2-\pi\nu(\Lambda+2),\,2\log2\bigr).
\end{equation}
The bound $\bar\mu\le\pi\nu$ is a variant of classical bounds on the minimum mean-square error for
binary input in Gaussian noise~\cite[Lemma~1]{CRU2001}.

\begin{proof}[Proof of Proposition~\ref{prop:nm-J}]
For $g$ with bounded derivatives of at most linear growth, Stein's identity
$\E[(h-\Lambda)g(h)]=\Lambda\E g'(h)$ and its iterate
$\E[(h-\Lambda)^2g(h)]=\Lambda\E g+\Lambda^2\E g''$ hold.
\begin{itemize}
\item Since $\lc h=h+L(h)-\log2$, $V=\operatorname{Var}(h+L)=\Lambda+2\operatorname{Cov}(h,L)
+\operatorname{Var}L$. As $L'=\tanh-1$, $\operatorname{Cov}(h,L)=\Lambda\E L'=-\Lambda\bar\mu$.
Hence $V-\Lambda\mu=\operatorname{Var}L-\Lambda\bar\mu$.
\item Take $g=\sech^2$ and write $h^2=(h-\Lambda)^2+2\Lambda(h-\Lambda)+\Lambda^2$. Then
$\E[h^2g]=\Lambda\E g+\Lambda^2\E[g''+2g'+g]$. In $t=\tanh h$,
$g''+2g'+g=-1-4t+7t^2+4t^3-6t^4$, whose expectation is $-B$ by
Lemma~\ref{lem:nm-NLid}(b). So $\Lambda^2B=\Lambda\bar\mu-\E[h^2\sech^2h]$.
\item Adding, $V-\Lambda\mu+\Lambda^2B=\E[L^2-h^2\sech^2h]-(\E L)^2$, and
$\E[L^2-h^2\sech^2h]=\nu I[\operatorname{Even}(e^hL^2)-h^2\sech h]=\nu I[J]$. For $h\ge0$,
$L(-h)=2h+r(h)$, and expanding with $2e^{-h}-\sech h=e^{-2h}\sech h$ gives the displayed form of
$J$, each of whose terms is nonnegative; $J(0)=(\log2)^2>0$.
\end{itemize}
For $c_0$, note $\int_\R J=\int_\R e^hL(h)^2\,\dd h-\int_\R h^2\sech h\,\dd h$. With $x=e^{-h}$
and one integration by parts,
\[
\int_\R e^hL^2\,\dd h=\int_0^\infty\frac{\log^2(1+x^2)}{x^2}\,\dd x
=4\int_0^\infty\frac{\log(1+x^2)}{1+x^2}\,\dd x
=-8\int_0^{\pi/2}\log\cos\theta\,\dd\theta=4\pi\log2,
\]
where the third step uses $x=\tan\theta$. Expanding
$\sech h=2\sum_{k\ge0}(-1)^ke^{-(2k+1)h}$ for $h>0$ gives
$\int_\R h^2\sech h\,\dd h=8\sum_{k\ge0}(-1)^k(2k+1)^{-3}=\pi^3/4$.
\end{proof}

These two integrals are Nishimori's: $4\pi\log2$ is his moment $I_0[r^2]$, and he notes that
$c_0$ is the difference of $4\pi\log2\approx8.710$ and $\pi^3/4\approx7.752$~\cite[(A112),
p.~40]{Nishimori2026}. The change of variables $h=\Lambda+\sqrt\Lambda z$ with the factor $e^h$
is his~\cite[(A110)--(A111)]{Nishimori2026}. Since $1-y\le e^{-y}\le1-y+y^2/2$ for $y\ge0$ and
$J\ge0$,
\begin{equation}\label{eq:nm-moments}
c_0-\frac{c_2}{2\Lambda}\le I[J]\le c_0-\frac{c_2}{2\Lambda}+\frac{c_4}{8\Lambda^2},\qquad
c_k=\int_\R h^kJ(h)\,\dd h .
\end{equation}

\subsection{Proof of Theorem~\ref{thm:nm-LP}}\label{app:nm-LP}

Put $Q(\Lambda)=\Lambda\mu/T_1$. By Lemma~\ref{lem:nm-NLid}, $T_1'=\Lambda w>0$, and since
$T_1(0^+)=0$, $T_1>0$ on $(0,\infty)$. Also
\[
D_p=\frac{T_1}{2p}\,(Q-p).
\]
So $D_p<0$ exactly where $Q<p$, and at the zero $\Lambda_c$ one has $T_1=\Lambda_c\mu/p$,
$p=\Lambda_c\mu/T_1\le\Lambda_c/T_1$ and $\Omega=(p-1)\Lambda_c/\mu<\Lambda_c^2/T_1$. Put
\[
T_1^{\rm lo}(\Lambda)=2\log2-\pi\nu(\Lambda+2),\qquad
\varphi_{\rm lo}(\Lambda)=2\log2-\pi\nu(2\Lambda+2+\pi\nu),\qquad
\varpi(\Lambda)=\frac{\pi\Lambda\nu}{T_1^{\rm lo}},
\]
\[
R_-(\Lambda)=\pi^2\nu+\frac{(\pi\Lambda+\pi^2\nu)^2\nu}{\varphi_{\rm lo}},\qquad
R_+(\Lambda)=\frac{\pi\Lambda\varpi}{1-\varpi},
\]
with $\nu=\nu(\Lambda)$. By~\eqref{eq:nm-T1}, $T_1\ge T_1^{\rm lo}$.

\emph{Monotonicity.} $\Lambda^k\nu(\Lambda)$ decreases for $\Lambda>2k-1$, and only $k\le2$
occurs. So on $[9,\infty)$ every lower bound below increases and every upper bound decreases in
$\Lambda$, and it suffices to evaluate them at the smallest admissible $\Lambda$. At $\Lambda=35$
interval arithmetic gives $\nu(35)\le1.6933\times10^{-9}$, $\pi\nu(35)\cdot37\le1.969\times10^{-7}$
and $T_1^{\rm lo}(35)\ge1.3862941$.

\emph{(a).} For $p\ge26$, $Q(35)\le35/T_1^{\rm lo}(35)\le25.2472<p$, so $D_p(35)<0$ and
$\Lambda_c>35$ by Proposition~\ref{prop:nm-triple}. Since $\Lambda_c=pT_1(\Lambda_c)/\mu\ge
pT_1^{\rm lo}(\Lambda_c)$ and $p\le\Lambda_c/T_1^{\rm lo}(\Lambda_c)$,
$2p\log2-\Lambda_c\le p\pi\nu(\Lambda_c)(\Lambda_c+2)
\le\Lambda_c\pi\nu(\Lambda_c)(\Lambda_c+2)/T_1^{\rm lo}(\Lambda_c)\le4.97\times10^{-6}$. Positivity
is Proposition~\ref{prop:nm-triple}.

\emph{(b), (c).} $\lambda=1-\Omega w\ge1-\frac\pi2\Lambda_c^2\nu/T_1^{\rm lo}\ge1-2.4\times10^{-6}$,
and $0<\Omega B\le(\Lambda_c^2/T_1)(\pi\nu/\Lambda_c)\le\varpi(\Lambda_c)\le1.35\times10^{-7}$. For
(N3), $\mu^k\ge1-k\bar\mu$ gives $p\mu^{p-1}+(p-1)\mu^p-1\ge(2p-2)(1-p\bar\mu)$, and
$p\bar\mu\le\pi\Lambda_c\nu/T_1\le1.35\times10^{-7}$.

\emph{(e).} $\varphi_{xx}=T_1+X$ with $X=V-\Lambda_c\mu$. By Proposition~\ref{prop:nm-J},
$X=\nu I[J]-\varsigma^2-\Lambda_c^2B\ge-\pi\Lambda_c\nu-\pi^2\nu^2$, so
$\varphi_{xx}\ge\varphi_{\rm lo}(\Lambda_c)\ge1.38629$. Next,
$\beta_1^2/2=T_1\mu^{-p}\le T_1e^{\Lambda_c\bar\mu/T_1}$, because $-\log(1-\bar\mu)\le\bar\mu/\mu$ and
$p\bar\mu/\mu=\Lambda_c\bar\mu/T_1$. Here $x:=\Lambda_c\bar\mu/T_1\le\varpi\le1$ and
$e^x\le1+x+x^2$, so $\beta_1^2/2\le2\log2-2\varsigma+\Lambda_c^2\bar\mu^2/T_1$. Finally
$\varsigma=\nu I[G]$ with $G=|h|e^{-|h|}+r\cosh\le(|h|+1)e^{-|h|}$, $\int G=\pi$ and
$\int h^2(|h|+1)e^{-|h|}=16$, so $\varsigma\ge\nu(\pi-8/\Lambda_c)$ and
$\beta_1^2/2\le2\log2-\nu[2(\pi-8/\Lambda_c)-\pi^2\Lambda_c^2\nu/T_1^{\rm lo}]<2\log2$.

\emph{(d) and the asymptotics.} By~\eqref{eq:nm-A1}, $\kappa_p=\kappa_0-X^2/\varphi_{xx}
+\Lambda_c^2B\,\Omega B/(1-\Omega B)$, and $\kappa_0=\nu I[J]-\varsigma^2$ by
Proposition~\ref{prop:nm-J}. Here $\varsigma^2/\nu\le\pi^2\nu$; $|X|\le\nu(\pi\Lambda_c+\pi^2\nu)$,
since $X\le\nu I[J]\le c_0\nu<\pi\Lambda_c\nu$; and $0\le\Lambda_c^2B\le\pi\Lambda_c\nu$ with
$\Omega B\le\varpi$. With~\eqref{eq:nm-moments},
\[
c_0-\frac{c_2}{2\Lambda_c}-R_-(\Lambda_c)\le\frac{\kappa_p}{\nu(\Lambda_c)}
\le c_0-\frac{c_2}{2\Lambda_c}+\frac{c_4}{8\Lambda_c^2}+R_+(\Lambda_c),
\]
which is the asymptotic statement; both $R_\pm$ are $O(\Lambda^{3/2}e^{-\Lambda/2})$ and at most
$1.5\times10^{-5}$ on $[35,\infty)$. The lower bound increases in $\Lambda$; at $\Lambda=35$,
with the cell enclosures of $c_0$ and $c_2$ below, it is at least $0.94616$. For the upper bound,
$-c_2/(2\Lambda)+c_4/(8\Lambda^2)\le0$ once $\Lambda\ge c_4/(4c_2)$, which the cell enclosures
certify for $\Lambda\ge35$, so $\kappa_p/\nu(\Lambda_c)\le c_0+R_+(35)$. Finally
$\Lambda_c=2p\log2-\Delta$ with $0<\Delta\le4.97\times10^{-6}$ gives
$N_p\nu(\Lambda_c)=e^{\Delta/2}(1-\Delta/(2p\log2))^{-1/2}\in[1,1+2.6\times10^{-6}]$. Together these
give $0.94616\le N_p\kappa_p\le0.96105$, and $0.94844\le N_p\kappa_p\le0.95880$ with the Arb
enclosures of $c_0$ and $c_2$.\qed

The first correction $-c_2/(2\Lambda_c)\approx-0.26/p$ carries no logarithm. Numerically
$N_p\kappa_p$ decreases from $1.24$ at $p=3$ to about $0.943$ near $p=14$ and then increases
toward $c_0$.

\subsection{Certificates}\label{app:nm-cert}

\paragraph{What is certified for $3\le p\le25$.} A program in ball arithmetic
(Arb~\cite{Johansson2017} through python-flint 0.9.0, FLINT 3.6.0) does the following for each
$p$.
\begin{itemize}
\item \emph{The zero of $D_p$.} Gaussian averages are enclosed by Arb's rigorous adaptive
Gauss--Legendre integration on $[-Z,Z]$ plus explicit tail balls, which use growth bounds on the
integrands proved by hand. The program certifies $D_p<0$ on $[3/16,\Lambda_{\rm lo}]$, $D_p'>0$ on
an interval $[\Lambda_{\rm lo},\Lambda_{\rm hi}]$ of width $1/32$ with
$D_p(\Lambda_{\rm lo})<0<D_p(\Lambda_{\rm hi})$, and $D_p>0$ on $[\Lambda_{\rm hi},2p\cdot355/512]$
when this interval is nonempty, by mean-value enclosures on an exact tiling. A one-dimensional Krawczyk test~\cite{Krawczyk1969} on a box of
radius $2^{-200}$ encloses $\Lambda_c$. Elementary estimates cover the rest: $D_p<0$ on
$(0,3/16]$, from $\mu(s)\le s$ and $\mu'(s)\ge1-2s-2s^2$, and $D_p>0$ on $[2p\log2,\infty)$
(proof of Proposition~\ref{prop:nm-triple}). When $\Lambda$ is a ball, each integrand is
evaluated over the ball, so the enclosure holds uniformly in $\Lambda$. The tiling runs at
$96$-bit precision (integration tolerance $2^{-64}$, cut-off $Z=12$); the Krawczyk test and the
evaluations at the enclosure of $\Lambda_c$ run at $256$ bits (tolerance $2^{-230}$, $Z=22$).
\item \emph{The signs.} At the enclosure of $\Lambda_c$ it evaluates $\kappa_p$, $\varphi_{xx}$,
$\lambda$, $B$ and $\det H$, where $H$ is Nishimori's $2\times2$ Hessian of
Appendix~\ref{app:nm-closed}, with $\det H=\Omega\det\mathbf H$ for the Hessian $\mathbf H$ in the
proof of Lemma~\ref{lem:nm-closed}. It does so in three codings within one program, sharing its
integration routine. The first uses closed forms reduced on the Nishimori line and written
through $H$, derived independently of~\cite{Nishimori2026} by the chain rule in $(m,q)$; they
equal~\eqref{eq:nm-A} by the identity $\det H=-\frac{\Omega^2}2\lambda(1-\Omega B)$ of
Appendix~\ref{app:nm-closed}. The
second is Nishimori's printed closed
forms~\cite[(21)--(23), (42)--(55), App.~A3--A4]{Nishimori2026}; the third is chain-rule
expressions taken before the Nishimori-line reductions, in which every Gaussian average is
integrated separately. All $713$ required signs hold ($31$ per $p$). The relative radius of the
$\kappa_p$ ball is at most $4\times10^{-55}$, and all $1610$ consistency checks (cross-coding
pairs and internal identities) overlap,
with differences below about $10^{-60}$. The printed values of~\cite[Tables~I--II]{Nishimori2026}
are correct roundings of the certified ones.
\item \emph{The remaining conditions.} A second program calls the same functions and certifies
$\det S\ge2.55\times10^{-6}$, $\tilde B''(\mu)\le-1.559$, the (N3) margin
$p\mu^{p-1}+(p-1)\mu^p-1\ge2.062$ and $\beta_1^2/2\le1.38629435<2\log2$, for all $3\le p\le25$;
$\det S$ and $\tilde B''(\mu)$ are computed both in the reduced form~\eqref{eq:nm-S} and without the
Nishimori-line reductions. Independently, (N2) follows from the certified signs of $\lambda$ and
$B$ and $\det H<0$ by the exact identities of Lemma~\ref{lem:nm-closed} and
$\det H=-\frac{\Omega^2}2\lambda(1-\Omega B)$, which were checked in SymPy.
\end{itemize}
A second run at $384$ bits (integration tolerance $2^{-350}$, cut-off $Z=28$) passes every check,
and all its balls overlap those of the $256$-bit run. A rerun of both in a fresh environment
reproduced every ball bit for bit. Table~\ref{tab:nm-cert} lists the certified bounds.

\begin{table}[!b]
\centering\small
\begin{tabular}{@{}rllllll@{}}
\toprule
$p$ & $\Lambda_c$ & $\kappa_p\ge$ & $\lambda\ge$ & $\det S\ge$ & $\tilde B''(\mu)\le$ & (N3) margin $\ge$\\
\midrule
3 & 2.3396416 & 0.0303915 & 0.330457 & 0.4471 & $-1.5592$ & 2.062\\
4 & 4.4788891 & 0.0131494 & 0.574280 & 0.3432 & $-10.484$ & 4.832\\
5 & 6.2996014 & 0.00536355 & 0.734818 & 0.2290 & $-38.830$ & 7.286\\
6 & 7.9416230 & 0.00227975 & 0.836921 & 0.1458 & $-111.40$ & 9.562\\
7 & 9.4806791 & 0.00100692 & 0.900932 & 0.09040 & $-279.87$ & 11.73\\
8 & 10.958556 & $4.57406\times10^{-4}$ & 0.940569 & 0.05495 & $-651.81$ & 13.83\\
9 & 12.399558 & $2.11906\times10^{-4}$ & 0.964795 & 0.03286 & $-1450.5$ & 15.90\\
10 & 13.818281 & $9.95243\times10^{-5}$ & 0.979400 & 0.01936 & $-3138.7$ & 17.94\\
11 & 15.223613 & $4.71947\times10^{-5}$ & 0.988085 & 0.01125 & $-6672.2$ & 19.96\\
12 & 16.620966 & $2.25347\times10^{-5}$ & 0.993182 & $6.470\times10^{-3}$ & $-14022$ & 21.98\\
13 & 18.013619 & $1.08147\times10^{-5}$ & 0.996137 & $3.680\times10^{-3}$ & $-29246$ & 23.98\\
14 & 19.403532 & $5.21015\times10^{-6}$ & 0.997830 & $2.073\times10^{-3}$ & $-60685$ & 25.99\\
15 & 20.791869 & $2.51766\times10^{-6}$ & 0.998791 & $1.158\times10^{-3}$ & $-1.2545\times10^{5}$ & 27.99\\
16 & 22.179305 & $1.21958\times10^{-6}$ & 0.999331 & $6.426\times10^{-4}$ & $-2.5865\times10^{5}$ & 29.99\\
17 & 23.566234 & $5.91990\times10^{-7}$ & 0.999632 & $3.541\times10^{-4}$ & $-5.3211\times10^{5}$ & 31.99\\
18 & 24.952879 & $2.87857\times10^{-7}$ & 0.999798 & $1.940\times10^{-4}$ & $-1.0927\times10^{6}$ & 33.99\\
19 & 26.339365 & $1.40185\times10^{-7}$ & 0.999890 & $1.057\times10^{-4}$ & $-2.2409\times10^{6}$ & 35.99\\
20 & 27.725764 & $6.83617\times10^{-8}$ & 0.999940 & $5.738\times10^{-5}$ & $-4.5896\times10^{6}$ & 37.99\\
21 & 29.112115 & $3.33771\times10^{-8}$ & 0.999968 & $3.100\times10^{-5}$ & $-9.3893\times10^{6}$ & 39.99\\
22 & 30.498440 & $1.63139\times10^{-8}$ & 0.999982 & $1.668\times10^{-5}$ & $-1.9189\times10^{7}$ & 41.99\\
23 & 31.884751 & $7.98175\times10^{-9}$ & 0.999990 & $8.951\times10^{-6}$ & $-3.9181\times10^{7}$ & 43.99\\
24 & 33.271055 & $3.90870\times10^{-9}$ & 0.999995 & $4.786\times10^{-6}$ & $-7.9936\times10^{7}$ & 45.99\\
25 & 34.657354 & $1.91570\times10^{-9}$ & 0.999997 & $2.552\times10^{-6}$ & $-1.6295\times10^{8}$ & 47.99\\
\bottomrule
\end{tabular}
\caption{Certified bounds at the triple point for $3\le p\le25$. The column $\Lambda_c$ gives
ball centres rounded to eight significant digits (the balls have radius below $10^{-60}$); every
other entry is a rigorous lower or upper bound, rounded in the safe direction. In addition
$\varphi_{xx}\ge0.4238$ for every $p$ in this range. The (N3) margin is
$p\mu^{p-1}+(p-1)\mu^p-1$ and is needed only for odd $p$.}
\label{tab:nm-cert}
\end{table}

\paragraph{What is certified for large $p$.} The numerical constants of
Theorem~\ref{thm:nm-LP} are certified in \texttt{mpmath} interval arithmetic with directed
rounding at $40$ digits. The inputs are elementary functions at $\Lambda=35$ and enclosures of
$c_0$, $c_2$, $c_4$ by cell sums, with the analytic tail bound $J(h)\le(2h^2+2h+1)e^{-3h}$ for
$h\ge12$: $c_0\in[0.9565281,0.9610269]$, $c_2\in[0.7202504,0.7244338]$ and
$c_4\in[2.0200,2.0319]$. These cell enclosures do not use the closed form of $c_0$. A separate
Arb computation gives $c_0\in0.95877519113945347813\pm3.7\times10^{-21}$, which contains
$4\pi\log2-\pi^3/4$, $c_2\in0.72233912108113666406\pm4.9\times10^{-21}$ and
$c_4\in2.0259419390462441765\pm4.1\times10^{-20}$; it lies inside the cell enclosures and is
used only for the sharper constants $0.94844$ and $0.95880$. The identities for $\det S$ and
$\tilde B''(\mu)$ in Lemma~\ref{lem:nm-closed}, the identity
$\det H=-\frac{\Omega^2}2\lambda(1-\Omega B)$ for Nishimori's Hessian, the algebra of
Proposition~\ref{prop:nm-J} and the steps in the proof of Lemma~\ref{lem:nm-U} are checked as
exact SymPy identities, with negative controls. A separate SymPy script differentiates $U$ for
symbolic $p$ and $\Lambda$ and checks the steps of the proof of Lemma~\ref{lem:nm-closed}: the
stationarity of $z_0$, the Hessian $\mathbf H$, $\partial_\beta^2U$, $\beta_1\mathbf d=\Lambda_c\mathbf k$,
$\mathbf H\mathbf y=\mathbf k$, $\det S$, $\det\mathbf H$ and the resulting $A_{\rm FM}$, which it
finds equal to~\eqref{eq:nm-A} and to the value computed through Nishimori's Hessian $H$ in the
first coding, with $\det H=\Omega\det\mathbf H$. It takes the Gaussian averages at $z_0$ from
Lemma~\ref{lem:nm-NLid}, and every difference simplifies to zero.

For $7\le p\le25$, the bounds of Appendix~\ref{app:nm-LP}, evaluated at one point
$\Lambda_p^-\ge9$ with $Q(\Lambda_p^-)<p$, hence below $\Lambda_c(p)$, give a second certificate
of (N1)--(N3) with no quadrature: the lower bound on $\kappa_p/\nu(\Lambda_c)$ ranges from
$0.1197$ at $p=7$ to $0.9460$ at $p=25$. For $p\le6$, where $\nu(\Lambda_c)$ is too large, these
elementary bounds are too weak.

\paragraph{Trust base.} The certificates for $3\le p\le25$ trust FLINT/Arb (ball arithmetic,
elementary functions and the error bounds of its integration routine) and the python-flint
binary distribution. The python-flint wrapper's forwarding of Arb's analyticity flag, which makes
integrals over boxes that meet a branch cut or pole non-finite, was checked in the 0.9.0 source
and by direct tests. The large-$p$ certificates trust \texttt{mpmath}'s directed rounding of
$\exp$, $\log$ and $\sqrt{\cdot}$, and the exact identities trust SymPy's simplification to zero.
The sharper constants $0.94844$ and $0.95880$ also trust the FLINT/Arb integration used for the
enclosures of $c_0$ and $c_2$; the constants $0.94616$ and $0.96105$, which suffice for (N1), do
not.
The identification of the certified formulas with the quantities of Section~\ref{sec:nearM} is
Lemma~\ref{lem:nm-closed} and Proposition~\ref{prop:nm-J}; the three codings for $p\le25$ guard
against transcription errors larger than about $10^{-60}$.

\paragraph{Reproduction.} The programs, their outputs and run logs are in the directory
\texttt{pspin/} of the companion repository~\cite{PeiCode2026} (subdirectories \texttt{kappa-cert/},
\texttt{cert-check/}, \texttt{triple-point/} and \texttt{large-order/}); its file \texttt{README.md}
maps each statement of this appendix to the files that certify it (Appendix~\ref{app:repro}).
\begin{itemize}
\item \texttt{kappa\_cert.py} (the zero of $D_p$ and the signs, $3\le p\le25$) and
\texttt{compare.py} (cross-run and printed-value comparisons). With python-flint 0.9.0,
\texttt{kappa\_cert.py -{}-pmin 3 -{}-pmax 25} runs in about $27$~s single-threaded;
\texttt{-{}-prec 384 -{}-tol-bits 350 -{}-zcut 28} gives the second run in about $54$~s. Each run exits
with status $0$ only if every check passes. Output hashes change only through embedded
timestamps.
\item \texttt{cert\_extra.py} ($\det S$, $\tilde B''(\mu)$, the (N3) margin and $\beta_1^2/2$,
together with checks of the wrapper), run next to an unmodified \texttt{kappa\_cert.py}.
\item \texttt{ns\_identities.py} and \texttt{lp\_identities.py} (exact identities, SymPy),
\texttt{afm\_sympy.py} (the steps of the proof of Lemma~\ref{lem:nm-closed}, SymPy; directly in
\texttt{pspin/}),
\texttt{lp\_bounds.py} (Theorem~\ref{thm:nm-LP} and the second certificate for
$7\le p\le25$) and \texttt{lp\_unimodal.py} (the cell proof of the crossing), each running in
a few seconds, and \texttt{lpc\_arb.py} (the Arb enclosures of $c_0$, $c_2$, $c_4$, which the programs call $c_0$, $m_2$, $m_4$).
\end{itemize}

\subsection{Remaining proofs for Section~\ref{sec:nearM}}\label{app:nm-proofs}

This subsection gives the proofs that Section~\ref{sec:nearM} defers, in the order of that section.

\begin{proof}[Proof of Lemma~\ref{lem:nm-NLid}]
Write $h=r+\sqrt r\,z$. Differentiating in $r$ and integrating by parts in $z$ gives (a). The
density of $h$ is $e^{h-r/2}$ times the $N(0,r)$ density, which gives (c). For (b), expand
$e^h=\cosh h+\sinh h$: the odd parts integrate to zero, and $\tanh h\,\sinh h=\tanh^2h\cosh h$.
By (a) and (b) with $F=\sech^2$,
$\mu'=\E[\sech^2h\,(1-\tanh h)]=\E[\sech^2h\,(1-\tanh^2h)]$, and
$\psi'=\E[\tanh h+\frac12\sech^2h]=\mu+\frac12(1-\mu)$. Since $0<\sech^4h<1$ almost surely,
$0<\mu'<1$. Finally $\mu(0^+)=0$ by bounded convergence, since $h\to0$ in probability as $r\to0$;
then $\mu'<1$ gives $\mu(r)<r$, and $|\tanh|<1$ gives $\mu(r)<1$.
\end{proof}

\begin{proof}[Proof of Lemma~\ref{lem:nm-G}]
With $\mathcal L(x)=\log\E e^{x\ell(\sqrt r\,g)}$ we have
\[
\Par(x,q;\beta)=\log2+\frac{\mathcal L(x)}x+\frac12\bigl[\xi_\beta(1)-r+(1-x)\theta_\beta(q)\bigr],
\]
so $\partial_x^2\Par|_{x=1}=\mathcal L''(1)-2\mathcal L'(1)+2\mathcal L(1)$. By
Lemma~\ref{lem:nm-NLid}(c), $\mathcal L(1)=r/2$, $\mathcal L'(1)=\E_Q\ell$ and
$\mathcal L''(1)=\operatorname{Var}_Q\ell$. By Lemma~\ref{lem:nm-NLid}(a),
$\frac{\dd}{\dd r}\E_Q\ell=\frac12(1+t_r)$ with $t_r=\E_Q\tanh^2=\E_Q\tanh$, and
$\frac{\dd}{\dd r}\operatorname{Var}_Q\ell=\operatorname{Cov}_Q(\ell,2\tanh-\tanh^2)+t_r$.
Hence, by Lemma~\ref{lem:nm-NLid}(b),
\[
\frac{\dd}{\dd r}\varphi_{xx}^{(r)}=\operatorname{Cov}_Q(\ell,2\tanh-\tanh^2)
=\operatorname{Cov}_Q(\ell,\tanh^2).
\]
Both $\ell$ and $\tanh^2$ are
even and strictly increasing in $|h|$, and $|h|$ is non-degenerate under $Q$, so Chebyshev's
association inequality makes the covariance positive. Since $\varphi_{xx}^{(0)}=0$,
$\varphi_{xx}^{(r)}>0$ for $r>0$. At $M$, $\Phi_{\beta_1}(q_1)=0$ gives
$\E_Q\ell=\frac12(\Lambda_c+\theta_{\beta_1}(\mu))$, so
$\varphi_{xx}^{(\Lambda_c)}=V-\theta_{\beta_1}(\mu)$, and
$\theta_{\beta_1}(\mu)=(p-1)\beta_1^2\mu^p/2=(p-1)\Lambda_c\mu/p$.
\end{proof}

\begin{proof}[Proof of Lemma~\ref{lem:nm-LS}, continued]
Let $\Gamma$, $\tilde f$ and $D_{\mu^*}$ be as in the first part of the proof
(Section~\ref{sec:nm-sg}); it remains to show $\tilde f\le0$ on $[0,1]$.

\emph{Joint continuity.} Let
$\mathcal K=[0,1]\times[\frac12,1]\times[\frac{q_1}2,\frac{1+q_1}2]\times[\frac{\beta_1}2,2\beta_1]$,
with points $(u,x,q,\beta)$. For $u\le q$ put $\sigma^2=\xi_\beta'(q)-\xi_\beta'(u)$,
$\mathcal B(y)=\E\cosh^x(y+\sigma g)$, $\mathcal B_1(y)=\E[\cosh^x\tanh](y+\sigma g)$ and
$\mathcal B_1'(y)=\E[\cosh^x(x\tanh^2+\sech^2)](y+\sigma g)$; $\mathcal B$ and $\mathcal B_1$ are
the functions $B$ and $B_1$ of Lemma~\ref{lem:D}, renamed here to avoid the constant $B$
of~\eqref{eq:nm-const}. By Lemma~\ref{lem:D}(iii) and its proof
(Appendix~\ref{app:D}),
\[
\Gamma(u)=\frac{\E[\mathcal B_1(Y_u)^2/\mathcal B(Y_u)]}{A},\qquad
\Gamma'(u)=\frac{\xi_\beta''(u)}{A}\,\E\Bigl[\mathcal B\Bigl(\frac{\mathcal B_1'}{\mathcal B}-x\frac{\mathcal B_1^2}{\mathcal B^2}\Bigr)^2\Bigr](Y_u)
\qquad(u\le q).
\]
For $u\ge q$, with $s^2=\xi_\beta'(u)-\xi_\beta'(q)$ and $Y_u=Y_q+sg$,
$\Gamma'(u)=\xi_\beta''(u)\,\E[\cosh^{x-1}Y_q\,\sech^3Y_u]/(Ae^{s^2/2})$. Since $\mathcal B\ge1$,
$|\mathcal B_1|\le\mathcal B$, $|\mathcal B_1'|\le(1+x)\mathcal B$ and $\cosh^x\le\cosh$, every integrand is dominated on
$\mathcal K$ by $e^{c|y|}$ against Gaussians of variance at most $\xi_\beta'(1)\le2p\beta_1^2$.
Dominated convergence makes each formula jointly continuous on the closed pieces $\{u\le q\}$
and $\{u\ge q\}$ of $\mathcal K$. At $u=q$ both formulas for $\Gamma$ give $T$ and both formulas
for $\Gamma'$ give $\xi_\beta''(q)\E[\cosh^xY_q\sech^4Y_q]/A$, because
$\mathcal B_1'/\mathcal B-x\mathcal B_1^2/\mathcal B^2=\sech^2$ when $\sigma=0$. Hence $\Gamma$, $\Gamma'$ and $\tilde f$ are jointly
continuous on $\mathcal K$. At $(x,q,\beta)=(1,q_1,\beta_1)$, $\Gamma=\Gamma_{\beta_1}$ and
$\tilde f=\Phi_{\beta_1}$ by~\eqref{eq:nm-Phiprime}.

\emph{Order of constants.} By (M2), $\Gamma_{\beta_1}'(q_1)=\xi_{\beta_1}''(q_1)\E\sech^4h_c
=1-\lambda$. Fix $\rho_1>0$ with $\Gamma_{\beta_1}'(s)\le1-\frac34\lambda$ for
$s\in[0,1]$ with $|s-q_1|\le2\rho_1$. Let $u_0(\beta)=(2/(p\beta^2))^{1/(p-2)}$ be as in
Lemma~\ref{lem:D}(iv). Then take $\rho\le\rho_1/2$ so small that every $(x,q,\beta)$ allowed in the
lemma satisfies $(u,x,q,\beta)\in\mathcal K$ for $u\in[0,1]$, that $u_0(\beta)\ge u_0(\beta_1)/2$ for
$|\beta-\beta_1|<\rho$, and that Steps~2 and~3 below apply.

\emph{Step 1 (near $0$).} By Lemma~\ref{lem:D}(iv), $\Gamma\le\xi_\beta'$, and
$\xi_\beta'(s)<s$ on $(0,u_0(\beta))$. So $\tilde f'<0$ there, and $\tilde f<\tilde f(0)=0$ on
$(0,u_0(\beta)]$.

\emph{Step 2 (bulk).} By (M1),
$\eta:=-\max\{\Phi_{\beta_1}(u):u\in[u_0(\beta_1)/2,1],\,|u-q_1|\ge\rho_1/2\}>0$. By uniform
continuity on $\mathcal K$, $|\tilde f-\Phi_{\beta_1}|<\eta/2$ on $[0,1]$ once $\rho$ is small.
Since $|q-q_1|<\rho_1/2$, the condition $|u-q|\ge\rho_1$ implies $|u-q_1|\ge\rho_1/2$. Hence
$\tilde f\le-\eta/2$ on $\{u\in[u_0(\beta_1)/2,1]:|u-q|\ge\rho_1\}$.

\emph{Step 3 (near $q$).} For $s\in[0,1]$ with $|s-q|\le\rho_1$ we have $|s-q_1|\le2\rho_1$, so
joint continuity of $\Gamma'$ gives $\Gamma'(s)\le1-\lambda/2$ there once $\rho$ is small. With
$\Gamma(q)=q$, $\Gamma(s)-s\ge\frac\lambda2(q-s)>0$ for $s\in{[q-\rho_1,q)}\cap[0,1]$ and
$\Gamma(s)-s\le-\frac\lambda2(s-q)<0$ for $s\in{(q,q+\rho_1]}\cap[0,1]$. So $\tilde f$ increases
to $0$ on ${[q-\rho_1,q]}\cap[0,1]$ and decreases from $0$ on ${[q,q+\rho_1]}\cap[0,1]$.

\emph{Cover.} Since $u_0(\beta)\ge u_0(\beta_1)/2$, the sets
\[
[0,u_0(\beta)],\qquad[u_0(\beta_1)/2,1]\setminus(q-\rho_1,q+\rho_1),\qquad[q-\rho_1,q+\rho_1]\cap[0,1]
\]
cover $[0,1]$. Hence $\tilde f\le0$ on $[0,1]$.
\end{proof}

\begin{proof}[Proof of the analytic extension in Lemma~\ref{lem:nm-branch}]
For $r>0$ the function $(x,r)\mapsto\E\cosh^x(\sqrt r\,g)=\int\cosh^x(y)e^{-y^2/(2r)}\,
\dd y/\sqrt{2\pi r}$ extends holomorphically to a neighbourhood of $(1,\Lambda_c)$, because the
integrand is dominated there by $e^{2|y|-cy^2}$; the same holds with $\lc$ or $\tanh^2$
inserted. Also $r=\xi_\beta'(q)$ is analytic near $q_1>0$. This gives the extension, including
$x>1$.
\end{proof}

\begin{proof}[Proof of Lemma~\ref{lem:nm-FMloc}(b), (c)]
(b) The Hessian of $U$ in $(m,h,q)$ at the critical point $(\mu,\Lambda_c,\mu)$ of~(a) is
\[
\begin{pmatrix}\Omega&-1&0\\-1&S_{hh}&S_{hq}\\0&S_{hq}&S_{qq}\end{pmatrix},
\]
with determinant $\Omega\det S-S_{qq}=\det S\cdot\tilde B''(\mu)\ne0$. The implicit function
theorem gives $z^*$, and $\hat\varphi_{\rm FM}=U(z^*)$ is $C^\infty$. By the envelope theorem
$\partial_{j_0}\hat\varphi_{\rm FM}=\beta m^p$ and
$\partial_\beta\hat\varphi_{\rm FM}=j_0m^p+\beta^{-1}\bigl[\xi_\beta(1)-\xi_\beta(q)
+q\xi_\beta'(q)-\xi_\beta'(q)\E\tanh^2(h+\sqrt{\xi_\beta'(q)}g)\bigr]$ at $z^*$. At $M$ these are
$\beta_1\mu^p$ and $j_{0M}\mu^p+(\xi_{\beta_1}(1)-\xi_{\beta_1}(\mu))/\beta_1=\beta_1/2$.

(c) Since $\det S\ne0$, the implicit function theorem gives smooth $(h^*(m),q^*(m))$ solving
$\nabla_{(h,q)}U=0$ near $(\Lambda_c,\mu)$, for $m$ near $\mu$ and $(\beta,j_0)$ near
$(\beta_1,j_{0M})$. Put $\tilde B(m)=U(m;h^*(m),q^*(m))$. By the envelope theorem
$\tilde B'(m)=p\beta j_0m^{p-1}-h^*(m)$, and differentiating $\nabla_{(h,q)}U=0$ gives
$\tilde B''=p(p-1)\beta j_0m^{p-2}-(S_m^{-1})_{hh}$, where $S_m$ is the $(h,q)$-Hessian at
$(h^*(m),q^*(m))$. At $M$ and $m=\mu$ this is $\tilde B''(\mu)<0$. Choose $r_0$ and $\mathcal V$
so that on $[\mu-r_0,\mu+r_0]\times\mathcal V$ the maps $h^*,q^*$ are defined,
$0<h^*\le2\Lambda_c$, $\tilde B''<0$, and, for odd $p$, $q^*(m)\in Q_p$ by (N3). Since
$\tilde B'(\mu)=0$ at $M$, for $(\beta,j_0)\in\mathcal V$ the concave function $\tilde B$ has a
unique maximizer $m^*$ in the interior of $[\mu-r_0,\mu+r_0]$. There
$(m^*,h^*(m^*),q^*(m^*))$ is a critical point of $U$, hence equals $z^*$, and
$\max\tilde B=\hat\varphi_{\rm FM}$. For a bin $[m',m'']$ in this interval,
\eqref{eq:nm-binU} with $(h,q)=(h^*(m''),q^*(m''))$ gives the bound
$\tilde B(m'')+h^*(m'')(m''-m')$. For odd $p$ it needs only the replica-symmetric interpolation,
since $q^*(m'')\in Q_p$.
\end{proof}

\section{Proofs and certified inputs for Section~\ref{sec:lattice}}\label{app:lattice}

This appendix proves the statements of Section~\ref{sec:lattice}. We first treat two-terminal graphs
(Appendix~\ref{app:lat-two}), then the infinite-volume reduction (Appendix~\ref{app:lat-infinite}),
contours and non-uniqueness (Appendix~\ref{app:lat-contours}), uniqueness
(Appendix~\ref{app:lat-unique}), order at the Nishimori temperature (Appendix~\ref{app:lat-order})
and the Nishimori line (Appendix~\ref{app:lat-line}). The certified inputs are in
Appendix~\ref{app:lat-cert}; the proof of Theorem~\ref{thm:lat-main} from them is at the end of
Section~\ref{sec:lat-proof}. Throughout, $\Pp$ and $\E$ refer
to iid couplings with $\Pp(J_e=-1)=p$, and $\beta_N=\beta_N(p)$ when $p$ is fixed.

\subsection{Two-terminal graphs}\label{app:lat-two}

Let $B$ be a two-terminal graph with vertex set $V$, edge set $E$, poles $\pi_0,\pi_1$ and real
couplings $a=(a_e)$, and put
$Z_B(s_0,s_1):=\sum_{\sigma\in\{\pm1\}^V:\,\sigma_{\pi_0}=s_0,\,\sigma_{\pi_1}=s_1}
\exp(\sum_{e=\{u,v\}}a_e\sigma_u\sigma_v)$.

\begin{lemma}[Two-terminal reduction]\label{lem:lat-reduction}\leavevmode
\begin{enumerate}
\item[\textup{(i)}] There are $A_B>0$ and $K_B\in\R$ with $Z_B(s_0,s_1)=A_Be^{K_Bs_0s_1}$,
namely $A_B^2=Z_B(+,+)Z_B(+,-)$ and $e^{2K_B}=Z_B(+,+)/Z_B(+,-)$. The Gibbs measure of $B$ has
$\gibbs{\sigma_{\pi_0}\sigma_{\pi_1}}_B=\tanh K_B$.
\item[\textup{(ii)}] $\tanh K_{B_1\cdot B_2}=\tanh K_{B_1}\tanh K_{B_2}$ and
$K_{B_1\parallel B_2}=K_{B_1}+K_{B_2}$.
\item[\textup{(iii)}] For $\tau\in\{\pm1\}^V$ and $(a^\tau)_e:=\tau_u\tau_va_e$,
$K_B(a^\tau)=\tau_{\pi_0}\tau_{\pi_1}K_B(a)$ and $A_B(a^\tau)=A_B(a)$.
\end{enumerate}
\end{lemma}

\begin{proof}
(i) The global flip gives $Z_B(s_0,s_1)=Z_B(-s_0,-s_1)$, so $Z_B$ depends on $s_0s_1$ only, and both
values are positive. Then $\gibbs{\sigma_{\pi_0}\sigma_{\pi_1}}_B=(Z_B(+,+)-Z_B(+,-))/(Z_B(+,+)+Z_B(+,-))
=\tanh K_B$. (ii) In series, summing the junction spin gives
$\sum_se^{K_1s_0s+K_2ss_1}=2\cosh(K_1s_0+K_2s_1)$, so
$e^{2K}=\cosh(K_1+K_2)/\cosh(K_1-K_2)$ and $\tanh K=\tanh K_1\tanh K_2$. In parallel the partition
functions multiply. (iii) The substitution $\sigma_v\mapsto\tau_v\sigma_v$ at every vertex gives
$Z_B(s_0,s_1)(a^\tau)=Z_B(\tau_{\pi_0}s_0,\tau_{\pi_1}s_1)(a)$.
\end{proof}

For $\pm J$ couplings at inverse temperature $\beta$ we write $K_B(J,\beta):=K_B(\beta J)$.

\begin{lemma}[Nishimori law of a two-terminal graph]\label{lem:lat-nishimori-law}
Let $J$ be iid on $E$ and $K:=K_B(J,\beta_N)$ with law $\mu$. Then
$\int f(-k)\,\mu(\dd k)=\int f(k)e^{-2k}\,\mu(\dd k)$ for bounded $f$. In particular
$\E e^{-K}=\E\sech K$ and $\E\tanh K=\E\tanh^2K\ge0$.
\end{lemma}

\begin{proof}
$\Pp(J)=C\prod_ee^{\beta_NJ_e}$ with $C=(2\cosh\beta_N)^{-|E|}$. For $\tau$ with $\tau_{\pi_0}=1$ the
map $J\mapsto J^\tau$ is a bijection and $K(J^\tau)=\tau_{\pi_1}K(J)$ by
Lemma~\ref{lem:lat-reduction}(iii). Averaging $\E f(K)=\sum_J\Pp(J^\tau)f(K(J^\tau))$ over these
$2^{|V|-1}$ gauges and summing $\prod_ee^{\beta_NJ_e\tau_u\tau_v}$ over $\tau$ with
$\tau_{\pi_1}=\pm1$ gives
\[
\E f(K)=2^{1-|V|}C\sum_J\bigl(Z_B(+,+)f(K)+Z_B(+,-)f(-K)\bigr)=\int f(k)e^k\lambda(\dd k),
\]
with the reflection-symmetric measure $\lambda:=2^{1-|V|}C\sum_JA_B(\delta_{K(J)}+\delta_{-K(J)})$,
where $Z_B$, $A_B$ and $K$ are evaluated at $\beta_NJ$. So $\mu=e^k\lambda$, which is the identity. Both sides of the first consequence equal $\lambda(\R)$, and
both sides of the second equal $\int\tanh k\sinh k\,\lambda(\dd k)$.
\end{proof}

This is Nishimori's gauge argument~\cite{Nishimori1981} applied to the effective coupling. For
series--parallel graphs it also follows from the observation of Le Doussal and Harris that Nishimori's
condition on a coupling law is preserved by series and parallel
composition~\cite[pp.~626--627]{LeDoussalHarris1988}.

\begin{proof}[Proof of Proposition~\ref{prop:lat-law}]
A path of $b_0$ edges has $\tanh K=\varepsilon\,\theta^{b_0}$ with $\varepsilon$ the product of its edge signs, so
$K=\varepsilon\chi$ by Lemma~\ref{lem:lat-reduction}(ii). Two paths in parallel give $(\varepsilon_1+\varepsilon_2)\chi\in
\{2\chi,0,-2\chi\}$; adding the connectors in series gives $\tanh K_{\rm br}=s_1s_2\theta^2
\tanh((\varepsilon_1+\varepsilon_2)\chi)$, where $s_1,s_2$ are the connector signs; and the direct edge adds $\beta s_d$.
Hence $K_U=\beta s_d+s_1s_2\operatorname{artanh}(\theta^2\tanh((\varepsilon_1+\varepsilon_2)\chi))$, and $|K_{\rm br}|\in
\{0,\chi_2\}$ with $\chi_2<\beta$ because $\tau=\theta\cdot\theta\tanh2\chi<\theta$. So $\operatorname{sgn}K_U
=s_d$, and $|K_U|=\beta$ if $\varepsilon_1\ne\varepsilon_2$, while for $\varepsilon_1=\varepsilon_2$ one has $|K_U|=\beta+\chi_2$
if $s_ds_1s_2\varepsilon_1=1$ and $|K_U|=\beta-\chi_2$ otherwise. The signs are independent, $\Pp(\varepsilon=-1)
=q$, and $\Pp(s_ds_1s_2=1)=(1+\vartheta^3)/2$; this gives the $\alpha_c$ and proves~(i). Units have
disjoint edge sets and the terminal edges contribute $\tanh(\beta J_a)\tanh(\beta J_b)$, which
gives~(ii), with $\operatorname{sgn}K_H=J_aJ_b\prod_is_d^{(i)}$ and $\E J_aJ_b=\vartheta^2$.
For~(iii), put $x:=e^{-2\beta}$. Then $e^{2(\beta-\chi_2)}=(1-\tau)/(x(1+\tau))$ and
$1-\tau=(1-\theta^2)+\theta^2(1-\tanh2\chi)$. Here $(1-\theta^2)/x=4/(1+x)^2\to4$, and
$1-\tanh2\chi\le2e^{-4\chi}\le8b_0^2x^2$ because $e^{-2\chi}=(1-\theta^{b_0})/(1+\theta^{b_0})\le
b_0(1-\theta)\le2b_0x$. So $e^{2(\beta-\chi_2)}\to2$, and $y_3\to\tanh(\frac12\log2)=\frac13$.
\end{proof}

\subsection{Infinite volume and the backbone reduction}\label{app:lat-infinite}

Write $V_\infty$ for the vertex set of $\Z^2[H]$ and ${\rm int}(H^g)$ for the
interior vertices of $H^g$. For finite $\Lambda\subset\Z^2$ let $E_\Lambda$ be the set of edges of $\Z^2$
meeting $\Lambda$, $\partial\Lambda$ the set of vertices of $\Z^2\setminus\Lambda$ adjacent to $\Lambda$,
and $\hat\Lambda:=\Lambda\cup\bigcup_{g\in E_\Lambda}{\rm int}(H^g)$ the \emph{gadget-closed box}; put
$\Lambda_n:=\{-n,\dots,n\}^2$. For finite $\Delta\subset V_\infty$ and $\eta\in\{\pm1\}^{V_\infty}$ let
$\mu^{J,\beta}_\Delta(\cdot\mid\eta)$ be the finite-volume Gibbs measure on $\{\pm1\}^\Delta$ with
weight $\exp(\beta\sum_{e\cap\Delta\ne\emptyset}J_e\sigma_u\sigma_v)$ and $\sigma=\eta$ off $\Delta$;
for couplings $K$ on $E(\Z^2)$ let $\nu^K_\Lambda(\cdot\mid\eta)$ be the backbone Ising measure on
$\{\pm1\}^\Lambda$ with weight $\exp(\sum_{g\in E_\Lambda}K_g\sigma_x\sigma_y)$ and $\sigma=\eta$ off
$\Lambda$. The gadget kernel $\Psi_g(\cdot\mid s,t)$ is the Gibbs law of the interior of $H^g$ with pole
values $(s,t)$.

\begin{lemma}[Decimation]\label{lem:lat-decimation}
For finite $\Lambda\subset\Z^2$ and every $\eta$,
\[
\mu^{J,\beta}_{\hat\Lambda}(\sigma_{\hat\Lambda}\mid\eta)=\nu^K_\Lambda(\sigma_\Lambda\mid\eta)
\prod_{g\in E_\Lambda}\Psi_g(\sigma_{{\rm int}(H^g)}\mid\sigma_g),\qquad K_g=K_H(J|_{H^g},\beta),
\]
where $\sigma_g$ is the pair of pole spins of $g$. In particular the left side depends on $\eta$ only
through $\eta_{\partial\Lambda}$. Likewise, on $\Lambda_L[H]$ and $T_L[H]$ the backbone marginal of
$\mu_{J,\beta}$ is the Ising measure on $\Lambda_L$ or $T_L$ with couplings $K_g$, and given the backbone
the gadget interiors are independent with laws $\Psi_g$. For every $\beta$ the $K_g$ are iid with the
law of $K_H(\beta)$.
\end{lemma}

\begin{proof}
Every microscopic edge with an endpoint in $\hat\Lambda$ lies in some $H^g$ with $g\in E_\Lambda$, and
every vertex outside $\hat\Lambda$ adjacent to $\hat\Lambda$ is in $\partial\Lambda$. So the weight is
$\prod_{g\in E_\Lambda}\exp(\beta\sum_{e\in H^g}J_e\sigma_u\sigma_v)$, and each factor equals
$A_ge^{K_g\sigma_x\sigma_y}\Psi_g(\sigma_{{\rm int}}\mid\sigma_g)$ by Lemma~\ref{lem:lat-reduction}(i).
The $A_g$ do not depend on $\sigma$. The finite graphs are identical. The $K_g$ are functions of the
disorder on disjoint edge sets.
\end{proof}

The next two facts are standard for nearest-neighbour specifications with spins $\pm1$ on a countable
locally finite graph (see e.g.~\cite{Georgii2011}); they apply both to $\mu^{J,\beta}$ on $\Z^2[H]$
and to $\nu^K$ on $\Z^2$. Write $\mathcal S_\Delta(\cdot\mid\eta)$ for such a specification.

\begin{lemma}[Existence and a uniqueness criterion]\label{lem:lat-standard}
Let $\Delta_n$ be finite sets increasing to the vertex set. (i) For any $\eta_n$, every weak limit point
of $\mathcal S_{\Delta_n}(\cdot\mid\eta_n)$ is a Gibbs measure; hence Gibbs measures exist. (ii) Put
\[
\delta_n(f):=\max_{\eta,\eta'}\bigl|\mathcal S_{\Delta_n}(f\mid\eta)-\mathcal S_{\Delta_n}(f\mid\eta')\bigr|.
\]
Then $\delta_n(f)$ is nonincreasing in $n$, and there is exactly one Gibbs measure if and only if
$\delta_n(f)\to0$ for every $f=\ind\{\sigma_S=\zeta\}$ with $S$ finite.
\end{lemma}

\begin{proof}
(i) For finite $\Delta$ and local $f$, $\mathcal S_\Delta f$ is local, and for $n$ large consistency gives
$\mathcal S_{\Delta_n}(\mathcal S_\Delta f\mid\eta_n)=\mathcal S_{\Delta_n}(f\mid\eta_n)$; pass to the
limit and use density of local functions. (ii) Consistency writes $\mathcal S_{\Delta_{n+1}}(f\mid\eta)$
as an average of $\mathcal S_{\Delta_n}(f\mid\cdot)$, which gives monotonicity. If $\delta_n\to0$, two
Gibbs measures $\mu,\mu'$ satisfy $|\mu(f)-\mu'(f)|\le\delta_n(f)$ by the DLR equation, and the
indicators determine the measure. If $\delta_n(f)\ge\varepsilon$ for all $n$, maximizing pairs have limit
points along a common subsequence, which are Gibbs by~(i) and differ on $f$.
\end{proof}

\begin{proof}[Proof of Lemma~\ref{lem:lat-BR}]
Let $\mu\in\mathcal G(J,\beta)$. For finite $F\subset E(\Z^2)$ the set $\Delta_F:=\bigcup_{g\in F}
{\rm int}(H^g)$ is finite, and as in Lemma~\ref{lem:lat-decimation} the specification on $\Delta_F$ is
$\prod_{g\in F}\Psi_g(\cdot\mid\sigma_g)$. The DLR equation on $\Delta_F$ therefore shows that under
$\mu$, given the backbone spins, the interiors are independent with laws $\Psi_g$; so
$\mu=\Pi\mu\otimes\Psi$, where $\Pi$ is the backbone marginal. The DLR equation on $\hat\Lambda$ and
Lemma~\ref{lem:lat-decimation} give $\mu(\sigma_\Lambda\in\cdot\mid\mathcal F_{V_\infty\setminus\hat\Lambda})
=\nu^K_\Lambda(\cdot\mid\sigma_{\partial\Lambda})$, which is measurable with respect to the backbone
spins off $\Lambda$; by the tower property $\Pi\mu\in\mathcal G^{\rm bb}(K)$. Conversely let
$\nu\in\mathcal G^{\rm bb}(K)$ and $\mu:=\nu\otimes\Psi$. Gadget-closed boxes are cofinal, so it suffices to
check the DLR equation on $\hat\Lambda$. Integrating out the interiors of the gadgets in $E_\Lambda$ with
Lemma~\ref{lem:lat-decimation}, and those of the other gadgets, whose poles lie off $\Lambda$, reduces it
to the DLR equation of $\nu$ on $\Lambda$. So $\Pi$ is a bijection with inverse $\nu\mapsto\nu\otimes\Psi$.
Measurability: each $K_g$ depends on finitely many $J_e$, the quantities $\delta^{\rm bb}_n(f,K)$ of
Lemma~\ref{lem:lat-standard} for $\Delta_n=\Lambda_n$ are continuous in finitely many $K_g$, and
$\{|\mathcal G^{\rm bb}(K)|=1\}=\bigcap_f\{\lim_n\delta^{\rm bb}_n(f,K)=0\}$ over countably many $f$.
Zero--one law: the lift uses the translation-invariant orientation, so each backbone shift $a\in\Z^2$
induces an automorphism $T_a$ of $\Z^2[H]$ and a shift $\theta_a$ of the disorder with
$|\mathcal G(\theta_aJ,\beta)|=|\mathcal G(J,\beta)|$. The product measure $\Pp$ is mixing under
$(\theta_a)$, hence ergodic, and the event $\{|\mathcal G(J,\beta)|\ge2\}$ is invariant.
\end{proof}

\subsection{Contours and non-uniqueness}\label{app:lat-contours}

For a spin configuration $\sigma$ on a graph, $D(\sigma)$ is the set of edges $\{x,y\}$ with
$\sigma_x\sigma_y=-1$, and $\partial A$ is the set of edges with exactly one endpoint in $A$.

\begin{lemma}[Flip bound]\label{lem:lat-flip}
Let $\mu$ be the Ising measure on a finite graph with couplings $K$, with the spins on a set $R$ of
vertices fixed. For every $A$ disjoint from $R$ and every $s\in[0,1]$,
\[
\mu(\partial A\subset D(\sigma))\le\min\bigl(1,e^{-2\sum_{g\in\partial A}K_g}\bigr)\le
e^{-2s\sum_{g\in\partial A}K_g}.
\]
\end{lemma}

\begin{proof}
Flipping the spins in $A$ is an injective map of the allowed configurations that changes
$\sigma_x\sigma_y$ exactly on $\partial A$. On $\{\partial A\subset D\}$ it multiplies the weight by
$e^{2S}$ with $S:=\sum_{\partial A}K_g$, which gives the first inequality. For the second,
$e^{-2S}\le e^{-2sS}$ if $S\ge0$, and $1\le e^{-2sS}$ if $S<0$. No sign condition on $K$ is used.
\end{proof}

The constraint $s\le1$ is needed: for $S>0$ and $s>1$ the second inequality fails. Write $(\Z^2)^*$ for
the dual lattice, $\Lambda_L^*$ for the planar dual of $\Lambda_L$ (its vertices are the unit squares and
the outer face $o$), and call a dual cycle of $T_L$ null-homologous if it is a $\Z_2$-boundary of a set of
dual faces.

\begin{lemma}[Contours]\label{lem:lat-extraction}\leavevmode
\begin{enumerate}
\item[\textup{(a)}] Let $\Lambda\subset\Z^2$ be finite with $\Z^2\setminus\Lambda$ connected, $x\in\Lambda$,
and $\sigma_x=-1$ with $\sigma\equiv+1$ off $\Lambda$. There is $A\subset\Lambda$ with $x\in A$ and
$\partial A\subset D(\sigma)$ whose dual $(\partial A)^*$ is a simple cycle of $(\Z^2)^*$ surrounding $x$.
\item[\textup{(b)}] On $\Lambda_L$, if $\sigma_x\ne\sigma_y$ there is $A$ containing exactly one of $x,y$
with $\partial A\subset D(\sigma)$ and $(\partial A)^*$ a simple cycle of $\Lambda_L^*$.
\item[\textup{(c)}] On $T_L$, if $\sigma_x\ne\sigma_y$ there is $A$ with $\partial A\subset D(\sigma)$ such
that either $A$ contains exactly one of $x,y$ and $(\partial A)^*$ is a simple null-homologous cycle, or
$(\partial A)^*$ is an edge-disjoint union of $j\ge2$ simple cycles that are not null-homologous.
\end{enumerate}
\end{lemma}

\begin{proof}
(a) Let $A_0$ be the connected component of $x$ in $\{\sigma=-1\}$, let $O$ be the infinite component
of $\Z^2\setminus A_0$, and put $A:=\Z^2\setminus O$. The set $\Z^2\setminus\Lambda$ is connected,
infinite and disjoint from $A_0$, so it lies in $O$ and $A\subset\Lambda$. An edge $\{a,b\}\in\partial A$
with $a\in A$, $b\in O$ has $a\in A_0$ (a finite component of $\Z^2\setminus A_0$ adjacent to $b$ would be
$O$), so $\sigma_a=-1$, and $\sigma_b=+1$ by maximality of $A_0$ or by the boundary condition. Both $A$
and $O$ are connected, so $\partial A$ is a minimal edge cut and its dual is a simple cycle surrounding
$A\ni x$. (b) The dual edge set $D^*$ has even degree at every dual vertex, since the product of
$\sigma_u\sigma_v$ around every face boundary is $1$. It decomposes into edge-disjoint simple cycles, and a
primal path from $x$ to $y$ crosses $D$ an odd number of times, hence crosses one of these cycles, $C^*$,
an odd number of times. By planar duality $C=\partial A$ with $A$ and its complement connected, and the
parity of the crossings is $\ind[x\in A]+\ind[y\in A]$. (c) As in (b), $D^*=(\partial\{\sigma=-1\})^*$ is
even and null-homologous; decompose it and pick a cycle crossed an odd number of times by a path from $x$
to $y$. If it is null-homologous it equals $(\partial A)^*$ for some $A$, and separation follows as in (b).
Otherwise let $W^*$ be the union of the non-null-homologous cycles of the decomposition; $W^*$ is
null-homologous because $D^*$ and the remaining cycles are, so $W=\partial A$ for some $A$, and $j\ge2$.
\end{proof}

\begin{lemma}[Counting]\label{lem:lat-counting}
For even $\ell\ge4$:
\begin{enumerate}
\item[\textup{(N)}] the number $N_\ell$ of simple cycles of $(\Z^2)^*$ of length $\ell$ surrounding a given
vertex of $\Z^2$ is at most $\frac{\ell-2}2\,3^{\ell-2}$;
\item[\textup{(M)}] the number of simple cycles of length $\ell$ through a given dual edge, in $(\Z^2)^*$ or
in the dual of $T_L$, is at most $3^{\ell-2}$;
\item[\textup{(B)}] the number of simple cycles of $\Lambda_L^*$ of length $\ell$ through $o$ is at most
$8L\cdot3^{\ell-2}$;
\item[\textup{(T1)}] the number of simple non-null-homologous dual cycles of $T_L$ of length $\ell$ is at
most $2L\cdot3^{\ell-2}$, and each has $\ell\ge L$;
\item[\textup{(T2)}] a simple null-homologous dual cycle of $T_L$ of length $\ell<L$ lifts to a simple cycle
of $(\Z^2)^*$, and it is $(\partial A)^*$ for the projection $A$ of the vertices it surrounds.
\end{enumerate}
Consequently, for $3w<1$, $\sum_\ell N_\ell w^\ell\le\mathsf P(w)$ and
$\sum_{\ell}3^{\ell-2}w^{\ell-1}\le\mathsf Q(w)$, the sums running over even $\ell\ge4$.
\end{lemma}

\begin{proof}
(M) A cycle through the dual edge $\{f_1,f_2\}$ is that edge plus a self-avoiding path from $f_1$ to $f_2$;
its first $\ell-2$ steps have at most three choices each and the last is forced. (N) The cycle crosses the
horizontal rays from $x$ to the right and to the left, at offsets $k+\frac12$ and $k'+\frac12$; its two
arcs between these crossings need $2(k+k'+1)$ horizontal steps and at least two vertical ones, so
$\ell\ge2k+4$, which leaves at most $(\ell-2)/2$ choices of the right crossing edge; then apply (M).
(B) There are $4(L-1)$ first dual edges out of $o$, at most three choices for each of the next $\ell-2$
steps, and at most two dual edges back into $o$. (T1) A dual cycle whose $\Z_2$-class has an odd first
(second) coordinate crosses the vertical (horizontal) primal circle through the origin an odd number of
times, so it contains one of $2L$ dual edges; apply (M). Its lift has displacement $L(w_1,w_2)\ne0$, so
$\ell\ge L$. (T2) The winding numbers vanish because $L|w_i|\le\ell<L$, so the lift closes; it surrounds a
region of diameter less than $L/2$, on which the projection is injective. The sums: with $z=9w^2$,
$\sum_{m\ge2}(m-1)9^{m-1}w^{2m}=z^2/(9(1-z)^2)=\mathsf P(w)$ and
$\sum_{m\ge2}9^{m-1}w^{2m-1}=z^2/(9w(1-z))=\mathsf Q(w)$.
\end{proof}

\begin{proof}[Proof of Proposition~\ref{prop:lat-CP}]
By Lemma~\ref{lem:lat-extraction}(a), $\{\sigma_x=-1\}\subset\bigcup_A\{\partial A\subset D(\sigma)\}$, the
union over sets $A\subset\Lambda$ whose boundary is dual to a simple cycle of $(\Z^2)^*$ surrounding $x$;
distinct $A$ give distinct cycles. By Lemma~\ref{lem:lat-flip} on the graph $(\Lambda\cup\partial\Lambda,
E_\Lambda)$ with the spins of $R=\partial\Lambda$ fixed to $+1$, independence, and distinctness of the edges
of $\partial A$,
\[
\E\,\nu^+_{\Lambda,K}(\sigma_x=-1)\le\sum_A\prod_{g\in\partial A}\E e^{-2sK_g}\le\sum_{\ell\ge4}N_\ell
w^\ell\le\mathsf P(w)
\]
by Lemma~\ref{lem:lat-counting}. Finally $\E\nu^+_{\Lambda,K}(\sigma_x)=1-2\E\nu^+_{\Lambda,K}(\sigma_x=-1)$.
\end{proof}

\begin{proposition}[Two-point bound with at most two defects]\label{prop:lat-twopoint}
Let $\Gamma_L\in\{\Lambda_L,T_L\}$ carry independent couplings $K_g$ and let $F_0$ be a set of at most two
edges, with $\E e^{-K_g}\le v$ for $g\notin F_0$, $\E e^{-K_g}\le1$ for $g\in F_0$, and $3v<1$. Fix
$\delta\in(0,\frac12)$ and let $C^\delta_L:=\{x\in\Lambda_L:{\rm dist}_\infty(x,\Z^2\setminus\Lambda_L)\ge\delta
L\}$. On $\Lambda_L$ let $x,y$ and the endpoints of the edges of $F_0$ lie in $C^\delta_L$; on $T_L$ they are
arbitrary. If $|F_0|=2$ let $D_0$ be the $\ell^\infty$-distance between the midpoints of its two edges. Then
\[
\E\gibbs{\sigma_x\sigma_y}_{\Gamma_L,K}\ge1-4\mathsf P(v)-2|F_0|\,\mathsf Q(v)
-\ind[|F_0|=2]\frac{2(3v)^{2D_0-2}}{1-3v}-2\varepsilon_L(v),
\]
with $\varepsilon_L(v):=v^{-2}\sum_{\ell\ge2\delta L-4}8L\,3^{\ell-2}v^\ell$ on $\Lambda_L$ and
$\varepsilon_L(v):=v^{-2}\bigl[\sum_{\ell\ge L}L\,3^{\ell-2}v^\ell+a_L^2/(1-a_L)\bigr]$,
$a_L:=2L(3v)^L/(9(1-3v))$, on $T_L$ for $L$ large enough that $a_L<1$ (for other $L$ put
$\varepsilon_L(v):=+\infty$). In both cases $\varepsilon_L(v)\to0$ as $L\to\infty$.
\end{proposition}

\begin{proof}
By Lemma~\ref{lem:lat-extraction}(b),(c), Lemma~\ref{lem:lat-flip} with $s=\frac12$ and a union bound,
$\mu_K(\sigma_x\ne\sigma_y)\le\sum_Ye^{-\sum_{g\in Y}K_g}$ over the structures $Y=\partial A$ listed
there, and $\E e^{-\sum_YK}\le v^{|Y|-|Y\cap F_0|}$ by independence. \emph{Short structures} are the cycles
of $\Lambda_L^*$ avoiding $o$ and the null-homologous torus cycles of length $\ell<L$; by
Lemma~\ref{lem:lat-counting}(T2) both are cycles of $(\Z^2)^*$ surrounding $x$ or $y$. Those avoiding $F_0$
contribute at most $2\mathsf P(v)$; those through exactly one given defect edge at most $\mathsf Q(v)$;
those through both have $\ell\ge2D_0$ and contribute at most $\sum_{\ell\ge2D_0}(3v)^{\ell-2}$.
\emph{Long structures} carry a factor at most $v^{-2}$. On $\Lambda_L$ they are the cycles through $o$; such
a cycle separating $x$ from $y$ crosses a monotone lattice path from $x$ to $y$ inside the square
$C^\delta_L$, at a dual edge at dual distance at least $\delta L-2$ from $o$, so $\ell\ge2\delta L-4$, and
(B) applies. On $T_L$ they are the null-homologous cycles with $\ell\ge L$, each containing one of at most
$L$ dual edges crossing a geodesic from $x$ to $y$ (apply (M)), and the unions $W$ of
Lemma~\ref{lem:lat-extraction}(c), bounded by summing over ordered $j$-tuples of non-null-homologous cycles
with (T1). Finally $\gibbs{\sigma_x\sigma_y}=1-2\mu_K(\sigma_x\ne\sigma_y)$.
\end{proof}

\begin{proof}[Proof of Proposition~\ref{prop:lat-NU}]
Put $X_n(J):=\mu^{J,\beta}_{\hat\Lambda_n}(\sigma_0\mid+)$, a continuous function of finitely many $J_e$.
By Lemma~\ref{lem:lat-decimation} it equals $\nu^+_{\Lambda_n,K}(\sigma_0)$, where the $K_g$, $g\in
E_{\Lambda_n}$, are iid with $\E e^{-2sK_g}=F_H(\beta,s)\le W$; the square $\Lambda_n$ has connected
complement, so Proposition~\ref{prop:lat-CP} gives $\E X_n\ge1-2\mathsf P(W)$. Let $Y:=\limsup_nX_n$.
Since $X_n\le1$, reverse Fatou gives $\E Y\ge\limsup_n\E X_n\ge1-2\mathsf P(W)$. For fixed $J$ choose
$n_k$ with $X_{n_k}(J)\to Y(J)$ and a further subsequence along which
$\mu^{J,\beta}_{\hat\Lambda_{n_k}}(\cdot\mid+)$ converges weakly to some $\mu^+_J$, which is Gibbs by
Lemma~\ref{lem:lat-standard}(i) and has $\mu^+_J(\sigma_0)=Y(J)$. The specification is flip-invariant, so
the flipped measure $\mu^-_J$ is Gibbs with $\mu^-_J(\sigma_0)=-Y(J)$, and $\mu^+_J\ne\mu^-_J$ when
$Y(J)\ne0$. Since $\Pp(Y>0)>0$, $\Pp(|\mathcal G(J,\beta)|\ge2)>0$, and the zero--one law of
Lemma~\ref{lem:lat-BR} gives probability one. The equivalence with $w_H(\beta)<w_+$: $\mathsf P$ increases
on $(0,\frac13)$, $2\mathsf P(w_+)=1$, and $w_+<\frac13$.
\end{proof}

No measurable selection of $\mu^\pm_J$ is claimed, and nothing is claimed about the number of Gibbs
states or about interior magnetizations off the Nishimori line.

\subsection{Uniqueness}\label{app:lat-unique}

This subsection gives the locators behind Lemma~\ref{lem:lat-newman} and proves
Corollaries~\ref{cor:lat-UH} and~\ref{rem:lat-finite}.

With the connection probability of the wired ferromagnetic random-cluster measure in place of
$\Phi_{p(K)}$, Lemma~\ref{lem:lat-newman} is the decomposition~\cite[(3.3), p.~255]{Newman1994} in the proof
of~\cite[Theorem~2, p.~253]{Newman1994}, a proof that takes $\Lambda$ to be any finite subset of $\Z^d$
(p.~254); that measure is dominated by $\Phi_{p(K)}$~\cite[(1.11), p.~250]{Newman1994}. See
also~\cite[Propositions~3.5--3.6, pp.~32--33, and (3.88), p.~46]{Newman1997}. Newman's edge parameter
$1-e^{-\beta|J_b|}$ belongs to the weight $\exp(\frac\beta2\sum_bJ_bs_xs_y)$~\cite[(1.1), (1.4)]{Newman1994};
in our normalization it is $1-e^{-2|K_g|}$.

\begin{proof}[Proof of Corollary~\ref{cor:lat-UH}]
By Lemmas~\ref{lem:lat-BR} and~\ref{lem:lat-standard}(ii) it suffices that, almost surely,
$\delta^{\rm bb}_n(f,K)\to0$ for the countably many $f=\ind\{\sigma_S=\zeta\}$ with $S\subset\Z^2$ finite.
Take $n$ with $S\subset\Lambda_n$ and write $f=2^{-|S|}\sum_{A\subset S}\zeta_A\sigma_A$. By
Lemma~\ref{lem:lat-newman}, $\delta^{\rm bb}_n(f,K)\le2\,\Phi_{p(K)}(S\leftrightarrow\partial\Lambda_n)=:2B_n(K)$,
and $B_n(K)$ equals the $\Z^2$-probability that $S$ is joined to $\Z^2\setminus\Lambda_n$ (a path stopped at
its first vertex outside $\Lambda_n$ uses only edges of $E_{\Lambda_n}$). Since the $K_g$ are independent,
the $K$-mixture of independent Bernoulli$(p_g(K_g))$ edges is Bernoulli$(\bar p_H(\beta))$, so
$\E B_n(K)=\Phi_{\bar p_H(\beta)}(S\leftrightarrow\Z^2\setminus\Lambda_n)$. As $n\to\infty$ this decreases
to $\Phi_{\bar p_H(\beta)}(S\leftrightarrow\infty)\le|S|\,\Phi_{1/2}(0\leftrightarrow\infty)=0$ by monotonicity
in the density and Harris's theorem~\cite{Harris1960}. $B_n$ is nonincreasing, so by monotone convergence
$\lim_nB_n=0$ almost surely. The image of the unique Gibbs measure under the global flip is Gibbs, hence
equal to it.
\end{proof}

\begin{proof}[Proof of Corollary~\ref{rem:lat-finite}]
Fix $\beta$ with $\bar p_H(\beta)\le\frac12$ (this is the certified input, Table~\ref{tab:lat-inputs}). Give
each interior vertex of $H^g$ the \emph{home edge} $g$ and each backbone vertex a fixed incident home edge;
each edge is home to at most $I+2$ vertices. Two facts are needed. First, for $u'$ interior to $H^e$ with
poles $a,b$, the conditional expectation of $\sigma_{u'}$ given the backbone spins is $c_1\sigma_a+c_2\sigma_b$
with $|c_1|+|c_2|\le1$ (it is odd under the global flip and bounded by $1$ at $(+,+)$ and $(+,-)$); given the
backbone, spins in distinct gadgets are independent (Lemma~\ref{lem:lat-decimation}). Hence for $u',v'$ with
distinct home edges $e,f$,
$|\gibbs{\sigma_{u'}\sigma_{v'}}|\le\sum_{x\in e,y\in f}|\gibbs{\sigma_x\sigma_y}_{\Gamma_L,K}|$.
Second, on any finite graph and for couplings of either sign,
$|\gibbs{\sigma_x\sigma_y}_K|\le\Phi_{p(K)}(x\leftrightarrow y)$~\cite[(1.8), (1.11)--(1.13)]{Newman1994}. For
completeness: the Edwards--Sokal type measure
$P(\sigma,\omega)\propto\prod_g[(1-p_g)\ind\{\omega_g=0\}+p_g\ind\{\omega_g=1,\,\operatorname{sgn}(K_g)
\sigma_x\sigma_y=1\}]$ has $\sigma$-marginal $\mu_K$ and $\omega$-marginal proportional to
$\Phi_{p(K)}(\omega)f(\omega)$, where $f=2^{\#\text{clusters}}\ind\{\text{every open cycle has sign product }
+1\}$ is decreasing; given $\omega$, $\sigma_x\sigma_y$ has mean zero unless $x\leftrightarrow y$, and is
$\pm1$ otherwise; and Harris's correlation inequality for the product measure~\cite{Harris1960} gives
$\Phi_{p}(f\ind_{x\leftrightarrow y})\le\Phi_p(f)\Phi_p(x\leftrightarrow y)$. Averaging over the independent $K_g$ gives
$\E|\gibbs{\sigma_x\sigma_y}_{\Gamma_L,K}|\le\Phi_{\bar p_H(\beta)}(x\leftrightarrow y)\le
\Phi^{\Z^2}_{\bar p_H(\beta)}(0\leftrightarrow\partial B_{r\wedge\rho_L})$ for $|x-y|_\infty\ge r$, where
$\rho_L=\infty$ on $\Lambda_L$ and $\rho_L=\lfloor(L-1)/2\rfloor$ on $T_L$ (on the torus use the event that $x$
is joined to the boundary of its ball of radius $n\le\rho_L$ by edges with an endpoint in the ball of radius
$n-1$, which is isomorphic to the corresponding event in $\Z^2$). Pairs whose home edges are at distance less
than $R$ form a fraction at most $4(2R+2)^2(I+2)/|V|$ of all pairs. Hence
\[
\limsup_{L\to\infty}\E\gibbs{M_L^2}_\beta\le4\,\Phi^{\Z^2}_{\bar p_H(\beta)}(0\leftrightarrow\partial B_R)
\xrightarrow[R\to\infty]{}4\,\Phi_{\bar p_H(\beta)}(0\leftrightarrow\infty)=0
\]
by Harris's theorem. The same bound applies to
$\E\gibbs{Q_L^2}_\beta=|V|^{-2}\sum_{u,v}\E\gibbs{\sigma_u\sigma_v}^2$.
\end{proof}

\subsection{Order at the Nishimori temperature}\label{app:lat-order}

Throughout this subsection $\beta=\beta_N(p)$. The \emph{planted law} on a finite graph $G=(V,E)$ draws
$\tau$ uniformly from $\{\pm1\}^V$ and sets $J_e=\tau_u\tau_v\eta_e$ with $\eta$ iid, $\Pp(\eta_e=-1)=p$.

\begin{lemma}[Gauge identities]\label{lem:lat-gauge}
Let $G=(V,E)$ be finite and $u,v\in V$.
\begin{enumerate}
\item[\textup{(i)}] $\E\gibbs{\sigma_u\sigma_v}_{G,J}=\E\gibbs{\sigma_u\sigma_v}^2_{G,J}\ge0$.
\item[\textup{(ii)}] If $F\subset E$, then $\E\gibbs{\sigma_u\sigma_v}_{(V,E)}\ge\E\gibbs{\sigma_u\sigma_v}_{(V,F)}$.
\item[\textup{(iii)}] If $G=G_1\cup G_2$ with $V(G_1)\cap V(G_2)=\{a\}$, $u\in V(G_1)$ and $v\in V(G_2)$, then
$\gibbs{\sigma_u\sigma_v}_G=\gibbs{\sigma_u\sigma_a}_{G_1}\gibbs{\sigma_a\sigma_v}_{G_2}$ for all couplings.
\end{enumerate}
\end{lemma}

\begin{proof}
(i) Under the planted law, $\Pp(\tau,J)=2^{-|V|}\prod_ee^{\beta_NJ_e\tau_u\tau_v}/(2\cosh\beta_N)$, so given $J$
the configuration $\tau$ has the Gibbs law and $\gibbs{\sigma_u\sigma_v}_J=\E[\tau_u\tau_v\mid J]$. Given
$\tau$, $J$ has the law of the gauge transform $J'^{\,\tau}$ of an iid $J'$, and
$\gibbs{\sigma_u\sigma_v}_{J'^{\,\tau}}=\tau_u\tau_v\gibbs{\sigma_u\sigma_v}_{J'}$. Hence
$\E_{\rm pl}[\tau_u\tau_v\gibbs{\sigma_u\sigma_v}_J]=\E\gibbs{\sigma_u\sigma_v}$, while the left side is
$\E_{\rm pl}\gibbs{\sigma_u\sigma_v}_J^2=\E\gibbs{\sigma_u\sigma_v}^2$ by gauge invariance of the square.
(ii) Under the planted law on $(V,E)$, $(\tau,J_F)$ has the planted law on $(V,F)$, and conditional Jensen
gives $\E[(\E[\tau_u\tau_v\mid J_F])^2]\le\E[(\E[\tau_u\tau_v\mid J])^2]$; apply (i) on both graphs.
(iii) Given $\sigma_a$ the two sides are independent, and by flip symmetry the conditional mean of
$\sigma_u$ is $\sigma_a\gibbs{\sigma_u\sigma_a}_{G_1}$.
\end{proof}

Part (i) is Nishimori's identity for correlations~\cite{Nishimori1981}, and part (ii) is Kitatani's
inequality~\cite{Kitatani1994}, in the form stated in~\cite[(3)]{Kitatani2009}.

For a two-terminal subgraph $B$ with terminals $a\ne b$ put $t(B):=\E\tanh K_B=\E\gibbs{\sigma_a\sigma_b}_B
\in[0,1]$ (Lemma~\ref{lem:lat-nishimori-law}). An \emph{admissible chain} from $\pi_i$ to $u$ in $H^{(i)}$ is a
sequence of two-terminal subgraphs $B_1,\dots,B_s$ of $H^{(i)}$ with terminals $(a_{j-1},a_j)$, $a_{j-1}\ne
a_j$, $a_0=\pi_i$, $a_s=u$, pairwise edge-disjoint, with $V(B_j)\cap V(B_{j+1})=\{a_j\}$ and $V(B_j)\cap
V(B_l)=\emptyset$ for $|j-l|\ge2$.

\begin{lemma}[Chains]\label{lem:lat-chain}\leavevmode
\begin{enumerate}
\item[\textup{(i)}] For every admissible chain, $\E\gibbs{\sigma_u\sigma_{\pi_i}}_{H^{(i)}}\ge\prod_jt(B_j)$.
\item[\textup{(ii)}] A path of $\ell$ edges has $t=\vartheta^{2\ell}$. Two internally disjoint paths of
$\ell_1,\ell_2$ edges between the same terminals have
$t=\mathcal C(\vartheta^{2\ell_1},\vartheta^{2\ell_2})$ with
$\mathcal C(t_1,t_2):=(t_1+t_2-2t_1t_2)/(1-t_1t_2)$, which is nondecreasing in each argument on $[0,1)^2$.
\end{enumerate}
\end{lemma}

\begin{proof}
(i) By Lemma~\ref{lem:lat-gauge}(ii), deleting the edges not in $\bigcup_jB_j$ does not increase the average,
and isolated vertices do not matter. By the intersection conditions and Lemma~\ref{lem:lat-gauge}(iii),
applied inductively, $\gibbs{\sigma_u\sigma_{\pi_i}}=\prod_j\tanh K_{B_j}$ samplewise
(Lemma~\ref{lem:lat-reduction}(i)); the factors are independent. (ii) A path has $\tanh K=(\prod_eJ_e)
\vartheta^\ell$ and $\E\prod_eJ_e=\vartheta^\ell$. For two paths put $\theta_i:=\vartheta^{\ell_i}$; path $i$
has $\tanh K_i=s_i\theta_i$ with independent signs, $\E s_i=\theta_i$, and $\tanh(K_1+K_2)=(s_1\theta_1+
s_2\theta_2)/(1+s_1s_2\theta_1\theta_2)$. Averaging over the four sign pairs gives
$(\theta_1+\theta_2)^2/(2(1+\theta_1\theta_2))+(\theta_1-\theta_2)^2/(2(1-\theta_1\theta_2))=
\mathcal C(\theta_1^2,\theta_2^2)$, and $\partial_{t_1}\mathcal C=(1-t_2)^2/(1-t_1t_2)^2\ge0$.
\end{proof}

\begin{proof}[Proof of Theorem~\ref{thm:lat-A}]
Write $v_*:=v_H(\beta_N)<v_P<\frac13$ and $G:=\Gamma_L[H]$ with vertex set $V$. Fix $\delta\in(0,\frac14)$
(unused on $T_L$) and $R\ge3$. For $w\in V$ let ${\rm ext}(w):=\{w\}$, $\kappa(w):=1$ and $a(w):=w$ if $w$ is a
backbone vertex; if $w$ is interior to $H^g$, let ${\rm ext}(w)$ be the endpoints of $g$,
$\kappa(w):=\kappa_H(w)$ computed in $H^g$, and $a(w)$ the pole of $g$ with a maximizing index $i(w)$.
We call $a(w)$ the \emph{anchor} of $w$; the other pole of $g$ is its non-anchor pole. Call
$w$ central if ${\rm ext}(w)\subset C^\delta_L$ (always on $T_L$), and a pair far if
${\rm dist}_\infty({\rm ext}(u),{\rm ext}(v))\ge R$.

\emph{One far central pair.} Let $F_0$ be the set of home edges of the interior vertices among $u,v$
(as in the proof of Corollary~\ref{rem:lat-finite}, the home edge of a vertex interior to $H^g$ is
$g$); its edges are vertex-disjoint. For $g\in F_0$ delete the edges of $H^g$ at its non-anchor pole; by
Lemma~\ref{lem:lat-gauge}(ii) this does not increase $\E\gibbs{\sigma_u\sigma_v}$. The resulting graph is the
union of $R_*$, which is $G$ without the edges and interior vertices of the gadgets in $F_0$, and of the pieces
$(H^g)^{(i)}$ with their now isolated non-anchor poles removed; each piece meets $R_*$ exactly in its anchor.
Lemma~\ref{lem:lat-gauge}(iii), applied once or twice, and independence give
\[
\E\gibbs{\sigma_u\sigma_v}_G\ge\kappa(u)\,\E\gibbs{\sigma_{a(u)}\sigma_{a(v)}}_{R_*}\,\kappa(v).
\]
By Lemma~\ref{lem:lat-decimation}, $\gibbs{\sigma_a\sigma_c}_{R_*}$ is the backbone correlation with $K_g=0$ on
$F_0$ and $K_g$ iid of law $K_H(\beta_N)$ elsewhere. The anchors are distinct and lie with $F_0$ in
$C^\delta_L$, and $D_0\ge R-1$. Proposition~\ref{prop:lat-twopoint} with $v=v_*$ and $4\mathsf P+2|F_0|\mathsf Q\le\mathsf h$
give
\[
\E\gibbs{\sigma_u\sigma_v}_G\ge\kappa(u)\kappa(v)\,\Pi_{R,L},\qquad
\Pi_{R,L}:=1-\mathsf h(v_*)-\frac{2(3v_*)^{2R-4}}{1-3v_*}-2\varepsilon_L(v_*).
\]

\emph{Summation.} All terms of $\E\gibbs{M_L^2}=|V|^{-2}\sum_{u,v}\E\gibbs{\sigma_u\sigma_v}$ are nonnegative
(Lemma~\ref{lem:lat-gauge}(i)), so non-central and non-far pairs may be dropped. Each $w$ has at most
$N_R:=(2R+2)^2(1+4I)$ partners that are not far, independently of $L$. With $S_c:=\sum_{w\text{ central}}\kappa(w)$
and $\kappa\le1$, whenever $\Pi_{R,L}\ge0$,
\[
\E\gibbs{M_L^2}_{\beta_N}\ge|V|^{-2}\bigl(S_c^2-|V|N_R\bigr)\,\Pi_{R,L}.
\]
On $T_L$, $S_c=L^2(1+2I\bar\kappa_H)$ and $|V|=L^2(1+2I)$, so $S_c/|V|=\hat\kappa_H$. On $\Lambda_L$ about
$(1-2\delta)^2L^2$ backbone vertices and $2(1-2\delta)^2L^2$ backbone edges are central, and
$|V|=L^2(1+2I)+O(L)$, so $S_c/|V|\to(1-2\delta)^2\hat\kappa_H$. Letting $L\to\infty$, then $R\to\infty$, then
$\delta\to0$ proves the bound. Finally $\kappa_H(u)\ge0$ by Lemma~\ref{lem:lat-gauge}(i).
\end{proof}

\begin{lemma}[Interior constant of $H_4$]\label{lem:lat-kappa}
Let $H=H_4(b_0,k)$, $r:=\vartheta^2$ and $t_U:=\E\tanh K_U=\alpha_1y_1^2+\alpha_2y_2^2+\alpha_3y_3^2$ at
$\beta_N$, and write $\mathcal C_r(\ell_1,\ell_2):=\mathcal C(r^{\ell_1},r^{\ell_2})$ for integers
$\ell_1,\ell_2$. Label the path vertices
of unit $i$ by $(i,c,j)$, with $c\in\{1,2\}$ the path and $j$ the distance from $w_1^{(i)}$, so that
$j=0$ is $w_1^{(i)}$ and $j=b_0$ is $w_2^{(i)}$. Then $\kappa_H(z_i)\ge r\,t_U^{\min(i,k-i)}$ and
\[
\kappa_H(i,c,j)\ge g_i(j):=\max\bigl\{r\,t_U^{\,i-1}\,\mathcal C_r(j+1,b_0-j+2),\ r\,t_U^{\,k-i}\,
\mathcal C_r(b_0-j+1,j+2)\bigr\}.
\]
Hence $I\bar\kappa_H\ge\sum_{i=0}^kr\,t_U^{\min(i,k-i)}+\sum_{i=1}^k\bigl(g_i(0)+g_i(b_0)+2\sum_{j=1}^{b_0-1}
g_i(j)\bigr)$. For $H_4(1000,10)$ at $p_0$ this gives $\bar\kappa_H\ge0.4914640678$, and with
$1-\mathsf h(v_H(\beta_N))\ge0.332172142$ (Table~\ref{tab:lat-inputs}),
$\hat\kappa_H^2(1-\mathsf h(v_H(\beta_N)))\ge0.0802$.
\end{lemma}

\begin{proof}
Apply Lemma~\ref{lem:lat-chain}. For $z_i$ toward $\pi_0$ take the edge $e_a$ and the units
$U^{(1)},\dots,U^{(i)}$ (value $r\,t_U^i$), and toward $\pi_1$ the units $U^{(i+1)},\dots,U^{(k)}$ and $e_b$.
For $(i,c,j)$ toward $\pi_0$ take $e_a$, the units $U^{(1)},\dots,U^{(i-1)}$, and the cycle with terminals
$z_{i-1}$ and $u$ formed by $e_1$ followed by the arc of path $c$ from $w_1$ to $u$ ($j+1$ edges) and by
$e_d$, $e_2$ and the arc from $w_2$ to $u$ ($b_0-j+2$ edges); it meets $U^{(i-1)}$ only in $z_{i-1}$ and uses
no edge at $\pi_1$. Toward $\pi_1$ the symmetric chain uses the cycle with terminals $u$ and $z_i$ formed by
the arc to $w_2$ and $e_2$ ($b_0-j+1$ edges) and by the arc to $w_1$, $e_1$ and $e_d$ ($j+2$ edges), then
$U^{(i+1)},\dots,U^{(k)}$ and $e_b$. The values follow from Lemma~\ref{lem:lat-chain}(ii), and $t_U$ from
Lemma~\ref{lem:lat-nishimori-law} and Proposition~\ref{prop:lat-law}(i). The numerical values are certified
inputs (C5), (C6) of Table~\ref{tab:lat-inputs}; with $I=20011$,
$\hat\kappa_H\ge(1+2I\cdot0.4914640678)/(1+2I)\ge0.491476$ and $0.491476^2\cdot0.332172142\ge0.08023$.
\end{proof}

\subsection{The Nishimori line}\label{app:lat-line}

\begin{proof}[Proof of Lemma~\ref{lem:lat-degradation}]
A binary-input channel $W$ from $\{\pm1\}$ to a finite set has Bhattacharyya parameter
$Z(W):=\sum_y\sqrt{W(y\mid+)W(y\mid-)}$; a degradation $W'=D\circ W$ by a stochastic matrix $D$ satisfies
$Z(W')\ge Z(W)$, since by Cauchy--Schwarz $\sqrt{\sum_yD(y'|y)W(y|+)\sum_yD(y'|y)W(y|-)}\ge\sum_yD(y'|y)
\sqrt{W(y|+)W(y|-)}$ for each $y'$. Let $W_p$ be the channel that draws $\tau$ uniformly on $\{\pm1\}^{V(H)}$
given $\tau_{\pi_0}\tau_{\pi_1}=x$ and outputs $J_e=\tau_u\tau_v\eta_e$ with $\eta$ iid, $\Pp(\eta_e=-1)=p$.
For $p\le p'<\frac12$, flipping each output sign independently with probability
$\delta=(p'-p)/(1-2p)$ turns $W_p$ into $W_{p'}$, so $W_{p'}$ is a degradation of $W_p$. With
$C_H=2^{1-|V|}(2\cosh\beta_N)^{-|E|}$ and the gauge sum of Proposition~\ref{prop:lat-floor},
$W_p(J\mid x)=2C_H\,A_H(J,\beta_N)e^{xK_H(J,\beta_N)}$, so $Z(W_p)=\sum_JW_p(J\mid+)e^{-K_H(J,\beta_N)}$.
Given $x=+1$, $J$ is a gauge transform of an iid $\eta$ with $\tau_{\pi_0}\tau_{\pi_1}=1$, so
$K_H(J,\beta_N)=K_H(\eta,\beta_N)$ by Lemma~\ref{lem:lat-reduction}(iii), and $Z(W_p)=v_H(\beta_N(p))$.
\end{proof}

\subsection{Certified inputs}\label{app:lat-cert}

\paragraph{What is computed.} Every input is an inequality about the law of Proposition~\ref{prop:lat-law}
for $H=H_4(1000,10)$: the $66$ class weights (multinomial in exact rational $\alpha_c$), the magnitudes $x_n$,
and the sign biases $A_n$. For example
\[
F_H(\beta,s)=\sum_n w_n\Bigl[\frac{1+A_n}2e^{-2sk_n}+\frac{1-A_n}2e^{2sk_n}\Bigr],\qquad
\bar p_H(\beta)=\sum_nw_n\frac{2x_n}{1+x_n},
\]
with $w_n$ the multinomial weights and $k_n=\operatorname{artanh}x_n$. At $\beta=\beta_N(p)$ with rational $p$,
$\tanh\beta_N=1-2p$ is rational, so every $x_n$ is rational and $v_H=\E\sech K_H=\sum_nw_n(1-x_n^2)^{1/2}$
(Lemma~\ref{lem:lat-nishimori-law}) needs only certified square roots. The thresholds are handled exactly.
The rational test $\mathsf h(\text{lower})<1<\mathsf h(\text{upper})$ places $v_P$ in
$[0.222407439452191,\,0.222407439452192]$. For $w_+=(9+3\sqrt2)^{-1/2}$ one has $W_\star:=274797/10^6<w_+$,
because $9W_\star^2<1$ and $2\mathsf P(W_\star)<1$ exactly.

\begin{table}[!htb]
\centering\small
\begin{tabular}{@{}lll@{}}
\toprule
& statement & used for\\
\midrule
(C1) & $\bar p_H(\beta)<\frac12$ for all $\beta\in(0,\beta_{\rm hot}]$ & (a), Cor.~\ref{rem:lat-finite}\\
(C2) & $1202$ cells cover $[\beta_-,\beta_+]$; on each, $F_H(\beta,s_j)\le W_\star:=274797/10^6$ & (b)\\
& for one rational $s_j\in[0.457,1]$ & \\
(C3) & $9W_\star^2<1$ and $2\mathsf P(W_\star)<1$ & (b)\\
(C4) & $\bar p_H(\beta)<\frac12$ for all $\beta\ge\beta_{\rm cold}$ & (a), Cor.~\ref{rem:lat-finite}\\
(C5) & $v_H(\beta_N)\le V_0:=20529504599/10^{11}$, & (b), (d)\\
& $1-\mathsf h(V_0)\ge0.332172142$, $1-2\mathsf P(V_0)\ge0.91700$ & \\
(C6) & chain bound of Lemma~\ref{lem:lat-kappa}: $\bar\kappa_H\ge0.4914640678$ & (d)\\
(C7) & $v_H(\beta_N(1/1000))\le V_1:=54076342043/(2.5\cdot10^{11})$, & (e)\\
& $\mathsf h(V_1)\le0.864508105378$ & \\
(C8) & $v_H(\beta_N(3/2500))\le W_N:=236586916634239/10^{15}$, & (e)\\
& $9W_N^2<1$, $1-2\mathsf P(W_N)\ge0.770991$ & \\
\bottomrule
\end{tabular}
\caption{Certified inputs for $H=H_4(1000,10)$ and $p_0=9/10000$ (for (C7) and (C8) the value of
$p$ is stated). The last column lists the parts of Theorem~\ref{thm:lat-main} that use each input.
Decimals are rounded in the safe direction.}
\label{tab:lat-inputs}
\end{table}

\needspace{4\baselineskip}
\paragraph{How each input is certified.}
\begin{itemize}
\item \emph{(C2).} For each cell $[\beta_l,\beta_r]$ the program chooses $s_j\in(1/1000)\Z\cap(0,1]$ by a
floating-point minimization, computes a ball-arithmetic upper bound $F_{\rm up}$ of $F_H(\beta_l,s_j)$ at the
rational point $\beta_l$, and checks exactly that $F_{\rm up}(2+y)/(2-y)\le W_\star$ with
$y:=6s_j(\beta_r-\beta_l)<2$. This bounds
$F_H(\beta,s_j)$ on the whole cell by the rule below. Consecutive cells share endpoints exactly; widths range
from $4\cdot10^{-7}$ to $3.6\cdot10^{-3}$. As a cross-check that avoids the rule, $F_H(\cdot,s_j)$ was also
enclosed over each whole cell.
\item \emph{(C1), (C4).} The intervals are covered by finitely many $\beta$-intervals, and $\bar p_H$ is
enclosed on each; no monotonicity of $\bar p_H$ in $\beta$ is used. On the cold side the enclosure of the last piece
$[64,\infty)$ uses the variables $e^{-2\beta}$, $e^{-2\chi}$ and $e^{-4\chi+2\beta}$, in which the law of
Proposition~\ref{prop:lat-law} has denominators bounded below by $1$, so one enclosure covers the limit
$\beta=\infty$.
\item \emph{(C5), (C7), (C8).} Exact rational magnitudes and weights, certified square roots, and exact
comparison with the thresholds.
\item \emph{(C6).} The chain sum of Lemma~\ref{lem:lat-kappa} in exact rational arithmetic, with every product
and quotient rounded down, using the monotonicity of $\mathcal C$.
\end{itemize}

\begin{lemma}[Cell rule]\label{lem:lat-cell}
Let $H$ be series--parallel with a fixed decomposition, and define $\lambda(e)=1$ for an edge,
$\lambda(B_1\cdot B_2)=\max(\lambda(B_1),\lambda(B_2))$ and $\lambda(B_1\parallel B_2)=\lambda(B_1)+\lambda(B_2)$. Then
$|\partial_\beta K_H(J,\beta)|\le\lambda(H)$ for every $J\in\{\pm1\}^E$, and $\lambda(H_4)=3$. Consequently, for
$s\ge0$, $\beta_l\le\beta\le\beta_r$ and $y:=6s(\beta_r-\beta_l)<2$, one has
$F_{H_4}(\beta,s)\le F_{H_4}(\beta_l,s)\,(2+y)/(2-y)$.
\end{lemma}

\begin{proof}
For $S(a,b):=\operatorname{artanh}(\tanh a\tanh b)$ and $X=|\tanh a|$, $Y=|\tanh b|$,
$|\partial_aS|+|\partial_bS|=[(1-X^2)Y+(1-Y^2)X]/(1-X^2Y^2)=(X+Y)/(1+XY)\le1$. Induction on the decomposition
tree, using the chain rule at series nodes and additivity at parallel nodes, gives the bound; for $H_4$,
$\lambda(P_{b_0}\parallel P_{b_0})=2$, $\lambda(U)=1+2$ and $\lambda(H_4)=3$. Then $K_{H_4}(J,\beta)\ge
K_{H_4}(J,\beta_l)-3(\beta-\beta_l)$, so $e^{-2sK(\beta)}\le e^{-2sK(\beta_l)}e^{y}$, and $e^y\le(2+y)/(2-y)$
for $0\le y<2$ because $\log\frac{2+y}{2-y}=2\operatorname{artanh}\frac y2\ge y$.
\end{proof}

\paragraph{Implementations.} The inputs were computed by the following programs. They are included, with
their outputs and run logs, in the directory \texttt{lattice/} of the companion repository~\cite{PeiCode2026}, whose file
\texttt{README.md} maps each input to the files that certify it and gives the commands that reproduce it.
\begin{itemize}
\item (C4), (C5) (\texttt{lattice/cert/}): a primary certificate in Arb ball arithmetic at $160$ bits (via python-flint), with adaptive
bisection of rational $\beta$-intervals, and a second implementation in pure Python with integer intervals
scaled by $2^{480}$ and outward rounding, rational Taylor bounds for $e^{-2\beta}$, and the law obtained by
enumerating the unit sign configurations and convolving along the chain rather than from the multinomial
formula.
\item (C1)--(C3), (C7), (C8) (\texttt{lattice/window/}): a certificate in Arb ball arithmetic and exact rationals, whose output is
reproducible bit for bit, and two re-certifications that share no code with it: one with its own fixed-point
interval arithmetic and \texttt{mpmath} intervals at $256$ bits, which derives the law of $K_H$ independently
from the gadget and re-proves every recorded inequality (for (C1) and (C2) with its own covers and monotone
brackets, not with Lemma~\ref{lem:lat-cell}); and one with \texttt{mpmath} intervals at $45$ digits and exact
rationals, which re-certifies (C1)--(C3) and (C7), also without Lemma~\ref{lem:lat-cell}. Input (C8) is
re-certified by the first of these and reproduced by two further enclosures.
\item (C6) (\texttt{lattice/theory/}): exact rational arithmetic with directed rounding, reproduced by an
independent $384$-bit fixed-point implementation.
\end{itemize}
None of the arithmetic engines (Arb, \texttt{mpmath} intervals, the fixed-point codes) has been audited by
us; the chains use different engines. The law of Proposition~\ref{prop:lat-law}, which is proved, was also
compared with exhaustive enumeration of small gadgets and with a numerical decimation of the full
$20032$-edge gadget on sample disorders.

\begin{remark}[The criteria fail just outside]\label{rem:lat-sharp}
The following are also certified. $w_H(98303/50000)\ge0.274829$ and $w_H(98233/20000)\ge0.274811$, both
above $w_+$, at about $0.56075\beta_N$ and $1.40088\beta_N$. $\bar p_H(38559/25000)\ge0.500006548$ at about
$0.43991\beta_N$, and $\bar p_H(698553/100000)>\frac12$. Hence the criteria used here cannot extend the
window $[\beta_-,\beta_+]$ contiguously, or the intervals $(0,\beta_{\rm hot}]$ and $[\beta_{\rm cold},\infty)$, by
more than $10^{-4}$ in $\beta$. These are failures at single points; they do not exclude certified points
farther out. Closing the gaps $(\beta_{\rm hot},\beta_-)$ and $(\beta_+,\beta_{\rm cold})$ needs different
estimates.
\end{remark}

The cell bounds of (C2) and the value of (C5) also give the lower bounds on the backbone
plus-state magnetization of Proposition~\ref{prop:lat-NU} that are listed in
Table~\ref{tab:lat-bands}. At $\beta=\beta_N$ the certified value of $v_H(\beta_N)$ gives
$\E Y\ge0.917$, and on $[24543/10^4,42073/10^4]\supset[0.70001\beta_N,1.19999\beta_N]$ the certified
cell bounds give $\E Y\ge0.885$ (Table~\ref{tab:lat-bands}).

\begin{table}[!htb]
\centering\small
\begin{tabular}{@{}lll@{}}
\toprule
$\beta$-band & in units of $\beta_N$ & $\E Y\ge$\\
\midrule
$[21037/10^4,\ 11833/2500]$ & $[0.600011,\ 1.349987]$ & $0.593675$\\
$[2279/1000,\ 45579/10^4]$ & $[0.650009,\ 1.299988]$ & $0.780507$\\
$[24543/10^4,\ 42073/10^4]$ & $[0.700008,\ 1.199991]$ & $0.885170$\\
$\beta_N$ & $1$ & $0.91700$\\
\bottomrule
\end{tabular}
\caption{Lower bounds on the backbone plus-state magnetization $\E Y$ of Proposition~\ref{prop:lat-NU}, from
the largest certified cell bound $W$ meeting each band ($1-2\mathsf P(W)$), and at $\beta_N$ from (C5). The
$\beta_N$-unit endpoints are rounded inward. Near the ends of $[\beta_-,\beta_+]$ only
$\E Y\ge1-2\mathsf P(W_\star)\ge2.05\cdot10^{-5}$ is available.}
\label{tab:lat-bands}
\end{table}

\section{Computer-assisted steps and reproducibility}\label{app:repro}

\paragraph{Computer-assisted steps.} The proofs of the theorems use computer assistance at four
places.
\begin{enumerate}
\item The accumulation theorem (Theorem~\ref{thm:accum}), which is due to
Lopatto~\cite{Lopatto2026}: our alternative proof of it verifies one polynomial inequality, (V1),
in exact rational arithmetic (Appendix~\ref{sec:accum-cert}).
\item The inputs of Theorem~\ref{thm:nm-main} that depend on $p$
(Theorem~\ref{thm:nm-nondeg}): for $3\le p\le25$ they are verified in interval arithmetic, and
for $p\ge26$ they follow from the explicit bounds of Theorem~\ref{thm:nm-LP}, whose constants are
also certified in interval arithmetic (Appendices~\ref{app:nm-LP}--\ref{app:nm-cert} and
Table~\ref{tab:nm-cert}).
\item The predicates (W), (F) and (L) for $p=3$ and $p=4$, behind Theorems~\ref{thm:p3}
and~\ref{thm:p4}: every predicate is certified in interval arithmetic and certified again by an
independent implementation (Appendix~\ref{app:cert34} and Table~\ref{tab:cert}).
\item The inputs (C1)--(C8) of Theorem~\ref{thm:lat-main}: finitely many inequalities about the
explicit law of the decimated coupling, certified in interval arithmetic, each by at least two
implementations that share no code (Appendix~\ref{app:lat-cert} and
Tables~\ref{tab:lat-inputs}--\ref{tab:lat-bands}).
\end{enumerate}
Remarks~\ref{rem:nm-chen} and~\ref{rem:lat-sharp}, Section~\ref{sec:nm-route} (the sign of
$C_{\rm FM}$ for $3\le p\le25$) and Appendices~\ref{app:nm-U} and~\ref{app:nm-cert} contain further
computer-assisted checks that no proof of a theorem needs.

The companion repository~\cite{PeiCode2026} (tag \texttt{arxiv-v1.1}, commit \texttt{1fc26e0}) contains the programs, frozen inputs, certified outputs and run logs
behind every computer-assisted statement of this paper, in the directories
\texttt{accumulation/}, \texttt{pspin/} and \texttt{lattice/}. Its file \texttt{README.md} maps
each such statement to the files that certify it, gives the commands that reproduce it, and
records the trust bases; the file \texttt{SHA256SUMS} lists the hashes of all other files.
Sections~\ref{sec:sk} and~\ref{sec:skline}, and their proofs in
Appendices~\ref{app:mf-sk}--\ref{sec:spiked}, contain no computer-assisted statement; the directory
\texttt{sk/} holds a SymPy script that rechecks the algebra and the constants of the hand proofs
of Lemma~\ref{lem:y}(c),(d) and Corollary~\ref{cor:envelope}, and the bound
$4(2\delta+\delta^2)/(1+\delta)^4\le8\delta$ of Remark~\ref{rem:spiked}. The
certificates for Sections~\ref{sec:nearM} and~\ref{sec:lattice} are described in
Appendices~\ref{app:nm-cert} and~\ref{app:lat-cert}; those for Section~\ref{sec:accum} and for
Section~\ref{sec:certificates} are in Appendices~\ref{sec:accum-cert} and~\ref{app:cert34}, and
their programs are described below.

\paragraph{Section~\ref{sec:accum} and Appendix~\ref{sec:accum-cert}.} The six coefficients of~\eqref{eq:cert-P} are stored in
the file \texttt{certificate\_\allowbreak P\_\allowbreak1x2.json} of the repository directory
\texttt{accumulation/}, together with two scripts.
\texttt{accum\_cert.py} uses only the Python standard library and exact rational arithmetic. It
derives $\mathcal B$ from $P$ by~\eqref{eq:cert-B}, compares it coefficient by coefficient with
the polynomial displayed in Appendix~\ref{sec:accum-cert}, and proves (V1), (V2), (V3) and the
Step~1 inequality of Lemma~\ref{lem:accumulator} by Bernstein subdivision; it runs in a few
seconds. \texttt{accum\_sym.py} checks with SymPy the identities of Lemmas~\ref{lem:omega}
and~\ref{lem:reduction}, the completion of the square behind~\eqref{eq:cert-B}, and the
constants of Lemma~\ref{lem:accumulator} and of (V2)--(V3).

\paragraph{Section~\ref{sec:certificates} and Appendix~\ref{app:cert34}.} The certificate programs, the
frozen trial data and the certified outputs (JSON files, with run logs where recorded) are in the
repository directories \texttt{pspin/cert-warm/} (predicate (W), both implementations),
\texttt{pspin/cert-p3/} and \texttt{pspin/cert-p4/} (predicates (L) and (F), with the independent
implementation in the subdirectory \texttt{second-implementation/}). For each $p$, the
primary certificate consists of the following programs:
\begin{itemize}
\item an interval-arithmetic library with Taylor-model derivatives (\texttt{iv.py});
\item quadrature with rigorous remainders (\texttt{gauss.py}; for the nested integrals in (F),
\texttt{nested.py} for $p=3$ and \texttt{trap4.py} for $p=4$), with the integrands $D$ and
$\Gamma$ built in \texttt{sg\_lb.py} ($p=3$) and in \texttt{sg\_lb4t.py} and
\texttt{common4.py} ($p=4$);
\item one driver per predicate, i.e.\ (W) (\texttt{warm\_cert\_A.py}, with the
enclosure of $\psi$ in \texttt{enc.py}), (L) with a separate replay of the frozen cover,
and (F): the grid computation and the final assembly.
\end{itemize}
The independent implementation consists of \texttt{rig.py} and \texttt{oned.py}, shared by both
values of $p$, with the drivers \texttt{warm\_cert\_B.py} for (W)
and both $p$, \texttt{L.py}, \texttt{fconst.py},
\texttt{fgrid.py}, \texttt{fcells.py}, \texttt{final.py} for $p=3$ and
\texttt{L4.py}, \texttt{fconst4.py}, \texttt{fgrid4.py}, \texttt{fcells4.py},
\texttt{final4.py} for $p=4$. It reads only the frozen trial parameters and recomputes every
predicate from the definitions in Sections~\ref{sec:prelim} and~\ref{sec:nishimori-point}.

All rational parameters are listed in Table~\ref{tab:cert} and in the frozen JSON files. Each
predicate takes at most about half an hour of CPU time. The nested integrals for (F) take
$11$--$29$ seconds per grid point for $p=3$ ($73$ points, about $25$ CPU-minutes in total) and
at most about $10$ seconds per point for $p=4$ ($48$ points, under four minutes). The
independent (F) grids take at most about $11$ seconds per point.

\end{document}